\documentclass[12pt, a4]{amsart}

\usepackage{latexsym,amsmath,amssymb,amsthm,mathrsfs,color}

\usepackage[bbgreekl]{mathbbol}
\usepackage{bbm}
\usepackage{tikz-cd}
\usetikzlibrary{shapes,arrows}

\usepackage[all]{xy}
\usepackage{stmaryrd}
\usepackage{hyperref}
\hypersetup{colorlinks,%
    citecolor=brown,%
    filecolor=black,%
   linkcolor=blue,%
   urlcolor=black}

\usepackage[toc,page]{appendix}

\usepackage[shortalphabetic]{amsrefs}
 \makeatletter
\def\append@label@year@{%
    \safe@set\@tempcnta\bib@year
    \edef\bib@citeyear{\the\@tempcnta}%
    \ifnum\bib@citeyear>9
      \append@to@stem{%
          \ifx\bib@year\@empty
          \else
            \@xp\year@short \bib@citeyear \@nil
          \fi
      }%
    \fi
}
\makeatother

\let\oldtocsection=\tocsection
\renewcommand{\tocsection}[2]{\hspace{0em}\oldtocsection{#1}{#2}}

\usepackage[a4paper]{geometry}

\def\upddots{\mathinner{\mkern 1mu\raise 1pt \hbox{.}\mkern 2mu
\mkern 2mu \raise 4pt\hbox{.}\mkern 1mu \raise 7pt\vbox {\kern 7
pt\hbox{.}}} }

\numberwithin{equation}{section}

\begin{document}
\setlength{\unitlength}{2.5cm}

\newtheorem{thm}{Theorem}[section]
\newtheorem{lm}[thm]{Lemma}
\newtheorem{prop}[thm]{Proposition}
\newtheorem{cor}[thm]{Corollary}
\newtheorem{conj}[thm]{Conjecture}
\newtheorem{specu}[thm]{Speculation}

\theoremstyle{definition}
\newtheorem{dfn}[thm]{Definition}
\newtheorem{eg}[thm]{Example}
\newtheorem{rmk}[thm]{Remark}

\newcommand{\F}{\mathbf{F}}
\newcommand{\N}{\mathbf{N}}
\newcommand{\R}{\mathbf{R}}
\newcommand{\C}{\mathbf{C}}
\newcommand{\Z}{\mathbf{Z}}
\newcommand{\Q}{\mathbf{Q}}

\newcommand{\Mp}{{\rm Mp}}
\newcommand{\Sp}{{\rm Sp}}
\newcommand{\GSp}{{\rm GSp}}
\newcommand{\GL}{{\rm GL}}
\newcommand{\PGL}{{\rm PGL}}
\newcommand{\SL}{{\rm SL}}
\newcommand{\SO}{{\rm SO}}
\newcommand{\Spin}{{\rm Spin}}
\newcommand{\GSpin}{{\rm GSpin}}
\newcommand{\Ind}{{\rm Ind}}
\newcommand{\Res}{{\rm Res}}
\newcommand{\Hom}{{\rm Hom}}
\newcommand{\End}{{\rm End}}
\newcommand{\msc}[1]{\mathscr{#1}}
\newcommand{\mfr}[1]{\mathfrak{#1}}
\newcommand{\mca}[1]{\mathcal{#1}}
\newcommand{\mbf}[1]{{\bf #1}}

\newcommand{\mbm}[1]{\mathbbm{#1}}

\newcommand{\into}{\hookrightarrow}
\newcommand{\onto}{\twoheadrightarrow}

\newcommand{\s}{\mathbf{s}}
\newcommand{\cc}{\mathbf{c}}
\newcommand{\bfa}{\mathbf{a}}
\newcommand{\id}{{\rm id}}
\newcommand{\g}{\mathbf{g}_{\psi^{-1}}}
\newcommand{\w}{\mathbbm{w}}
\newcommand{\Ftn}{{\sf Ftn}}
\newcommand{\p}{\mathbf{p}}
\newcommand{\bq}{\mathbf{q}}
\newcommand{\WD}{\text{WD}}
\newcommand{\W}{\text{W}}
\newcommand{\Wh}{{\rm Wh}}
\newcommand{\ggma}{\omega}
\newcommand{\sct}{\text{\rm sc}}
\newcommand{\Of}{\mca{O}^\digamma}
\newcommand{\gk}{c_{\sf gk}}
\newcommand{\Irr}{ {\rm Irr} }
\newcommand{\Irrg}{ {\rm Irr}_{\rm gen} }
\newcommand{\diag}{{\rm diag}}
\newcommand{\uchi}{ \underline{\chi} }
\newcommand{\Tr}{ {\rm Tr} }
\newcommand{\der}\de
\newcommand{\Stab}{{\rm Stab}}
\newcommand{\Ker}{{\rm Ker}}
\newcommand{\bfp}{\mathbf{p}}
\newcommand{\bfq}{\mathbf{q}}
\newcommand{\KP}{{\rm KP}}
\newcommand{\Sav}{{\rm Sav}}
\newcommand{\de}{{\rm der}}
\newcommand{\lest}{\leqslant}
\newcommand{\gest}{\geqslant}

\newcommand{\tchi}{\tilde{\chi}}
\newcommand{\tomega}{\tilde{\omega}}

\newcommand{\cu}[1]{\textsc{\underline{#1}}}
\newcommand{\set}[1]{\left\{#1\right\}}
\newcommand{\ul}[1]{\underline{#1}}
\newcommand{\wt}[1]{\overline{#1}}
\newcommand{\wtsf}[1]{\wt{\sf #1}}
\newcommand{\anga}[1]{{\left\langle #1 \right\rangle}}
\newcommand{\angb}[2]{{\left\langle #1, #2 \right\rangle}}
\newcommand{\wm}[1]{\wt{\mbf{#1}}}
\newcommand{\elt}[1]{\pmb{\big[} #1\pmb{\big]} }
\newcommand{\ceil}[1]{\left\lceil #1 \right\rceil}
\newcommand{\val}[1]{\left| #1 \right|}

\newcommand{\exc}{ {\rm exc} }

\newcommand{\motimes}{\text{\raisebox{0.25ex}{\scalebox{0.8}{$\bigotimes$}}}}

\newcommand{\nequiv}{\not \equiv}
\newcommand{\half}{\frac{1}{2}}
\newcommand{\psii}{\widetilde{\psi}}
\newcommand{\ab} {|\!|}
\newcommand{\mb}{{\widetilde{B(\F)}}}

\title[Restrictions, L-parameters, and local coefficients]{Restrictions, L-parameters, and local coefficients for genuine representations}

\author{Fan Gao, Freydoon Shahidi, and Dani Szpruch}
\address{Fan Gao: School of Mathematical Sciences, Yuquan Campus, Zhejiang University, 38 Zheda Road, Hangzhou, China 310027}
\email{gaofan@zju.edu.cn}
\address{Freydoon Shahidi: Department of Mathematics, Purdue University, 150 N. University Street, West Lafayette, IN 47907}
\email{shahidi@math.purdue.edu}
\address{Dani Szpruch: Department of Mathematics and Computer Science, Open University of Israel, Raanana 43107, Israel}
\email{dszpruch@openu.ac.il}

\maketitle

\begin{abstract} 
We consider the restriction and induction of representations between a covering group and its derived subgroup, both on the representation-theoretic side and the L-parameter side. In particular, restriction of a genuine principal series is analyzed in detail. We also discuss a metaplectic tensor product construction for covers of the symplectic similitudes groups, and remark on the generality of such a construction for other groups. Furthermore, working with an arbitrary irreducible constituent of a unitary unramified principal series, we prove a multiplicity formula for its restriction to the derived subgroup in terms of three associated R-groups. Later in the paper, we study an unramified L-packet on how the parametrization of elements inside such a packet varies along with different choices of hyperspecial maximal compact subgroups and their splittings. We also investigate the genericity of elements inside such an L-packet with respect to varying Whittaker datum. Pertaining to the above two problems, covers of the symplectic similitudes groups are discussed in detail in the last part of the paper.
\end{abstract}
\tableofcontents

\section{Introduction}
Let $F$ be a $p$-adic field. Let $G=\mbf{G}(F)$ be the group of $F$-rational points of a connected split reductive group $\mbf{G}$ over $F$. 
Let $H \subset G$ be a closed subgroup, which might be the $F$-rational points of a subgroup  $\mbf{H} \subset \mbf{G}$ over $F$ but not always so. Let $\pi \in \Irr(G)$ be an irreducible admissible representation. It is an important question to determine the restriction $\pi|_{H}$. As the literature on this problem is vast, we only mention three representative examples. The first one originated from the theory of Gelfand pair \cite{GK, Gro1}, and it includes the spherical theory when $H \subset G$ is a maximal compact subgroup, and the Whittaker theory when $H$ is the unipotent radical of a Borel subgroup.  The second example concerns the Gan--Gross--Prasad conjecture, see \cite{GP1, GGP1} and \cite{Gan14-2} including references therein. The third example arises from the special case where $H$ is closely related to the derived subgroup $G_\der$ of $G$.

If $\pi|_H$ is of finite length, then we have the semisimplification
$$(\pi|_{H})^{\rm ss} = \sum_i m_i \pi_i  \in \msc{R}(H),$$
where $\pi_i \in \Irr(H)$ and $\msc{R}(H)$ denotes the Grothendieck group of $\Irr(H)$. Part of the problem is then on determining the multiplicity $m_i$, which usually manifests in terms of some representation-theoretic and arithmetic invariants associated with the pair $(\pi, \pi_i) \in \Irr(G) \times \Irr(H)$.  A special case is the last example mentioned above. Now we elaborate on this, as its covering analogue is the main focus of our paper. 

It was shown by Silberger \cite{Sil79} that for every $\pi \in \Irr(G)$, the restriction $\pi|_{G_\der}$ is semisimple of finite length and thus
$$\pi|_{G_\der} = \sum_{i=1}^k m_i \cdot \pi_i$$
with $\pi_i \in \Irr(G_\der)$ and $m_i \gest 1$. This is a weaker form of a conjecture by Borel \cite{Bor} that all the constituents $\pi_i$ belong to the same L-packet associated with the parameter 
$$f\circ \phi_\pi: \WD_F \to {}^L G \to {}^L G_\der,$$
where $f$ is the natural quotient map. In fact, the conjecture of Borel concerns slightly more general group homomorphism $H \to G$ where $H$ may not be $G_\der$. A special case of this is when $G_\der \subset H \subset G$, and for such $H$ Adler and Prasad proposed a unified formula for the multiplicity $m_i$, see \cite{AP1}. The original conjecture by Borel was also generalized in another direction using the enchanced L-parameters. By an approach via the Hecke algebras, such generalized version was proved for principal series and unipotent representations of $G$, see the work of Aubert--Baum--Plymen--Solleveld \cite{ABPS17} and Solleveld \cite{Sol20}. We note that there are other works on various families of groups and representations for the restriction problem from $G$ to $G_\der$, see \cite{Key3, Nev15, BCG18, Cho19}. 

The starting point of this paper is motivated from the above (third) example in the setting of covering groups. More precisely, assume that $F$ contains the full group $\mu_n$ of $n$-th roots of unity, and consider a central cover
$$\begin{tikzcd}
\mu_n \ar[r, hook] & \wt{G} \ar[r, two heads] & G
\end{tikzcd}
$$
of $G$ by $\mu_n$ arising from the Brylisnki--Deligne framework \cite{BD}. Denote by 
$$\Irrg(\wt{G})$$ the set of isomorphism classes of irreducible genuine representations of $\wt{G}$, where $\mu_n$ always acts by a fixed embedding $\mu_n \into \C^\times$. The derived subgroup of $\wt{G}$ is equal to the cover $\wt{G}_\der$ of $G_\der$ obtained from the pull-back of $\wt{G}$ via the canonical inclusion $G_\der \subset G$. While the restriction problem for representations from $G$ is $G_\der$ is already a nontrivial one involving data from the L-parameter side, this is even more the case for central covers. 

Indeed, it is already instructive to consider the restriction of the principal series representations of $G$ and $\wt{G}$ to $G_\der$ and $\wt{G}_\der$ respectively. Let $I(\chi)$ be a principal series of $G$, where $\chi$ is a character of the torus $T$. It is well-known that
$$I(\chi)|_{G_\der} = I(\chi_0),$$
where $\chi_0 = \chi|_{\wt{T}_0}$ with $T_0 \subset G_\der$ being the split torus. However, if $I(\chi)$ is a genuine principal series of $\wt{G}$, where $\chi$ is a genuine character of the center $Z(\wt{T}) \subset \wt{T}$, then the restriction $I(\chi)|_{G_\der}$, which is still semisimple, may not be an isotypic sum of genuine principal series of $G_\der$. In fact, this phenomenon arises from its counterpart for the genuine representations of covering tori. That is, the decomposition of an irreducible genuine representation of $\wt{T}$, when restricted to $\wt{T}_0$, is more delicate than the linear case, which is deceptively simple. Such delicacy is a consequence of the fact that 
$$Z(\wt{T}_0) \nsubseteq Z(\wt{T})$$
in general.

Thus, we expect multiple layers of subtleties arising from the multiplicity formula for the restriction $\pi|_{\wt{G}_\der}$
for $\pi\in \Irrg(\wt{G})$. This also renders many new phenomena regarding the local coefficients associated with $\pi$ and certain intertwining operators. To illustrate, we consider  first the linear case. Let $\pi$ be a $\psi$-generic representation of $G$, where 
$$\psi: U \to \C^\times$$
 is a nondegenerate character of the unipotent radical $U$ of a Borel subgroup $B=TU$.  We have the canonical identification
$$\Wh_\psi(\pi) = \Wh_\psi(\pi|_{G_\der}).$$
It follows from the multiplicity-one property of Whittaker functionals that one has  
$$\Wh_\psi(\pi) = \Wh_\psi(\pi_0)$$ for a unique constituent $\pi_0 \subset \pi|_{G_\der}$. Thus, the arithmetic information encoded in the local coefficients of $\pi$ for $G$ is completely elucidated from this unique $\pi_0$. 
The simplest such example is when $\pi=I(\chi)$ is a principal series of $G$ as above. This is in contrast with the situation for covers. Indeed, although for $\pi \in \Irrg(\wt{G})$ we still have
$$\Wh_\psi(\pi) = \Wh_\psi(\pi|_{\wt{G}_\der}),$$
the two sides might be of high dimensions. In fact, it is possible that $\dim \Wh_\psi(\pi_i) \gest 1$ for every constituent $\pi_i \subset \pi|_{\wt{G}_\der}$. Hence, it is necessary to study every space $\Wh_\psi(\pi_i)$, as there may not be any distinguished one. Already as a preliminary step, it is essential to study the multiplicity formula in the decomposition $\pi|_{\wt{G}_\der}$. 

In fact, the relation between the Whittaker spaces and restriction problems can be deepened if one considers the  more ``restrictive'' Whittaker space. More precisely, the $\psi$-Whittaker space ${\rm Wh}_\psi(\pi)$ of $\pi$ is naturally a genuine $\wt{Z(G)}$-module, where $\wt{Z(G)}$ is a Heisenberg type group and thus its genuine irreducible representations are finite dimensional with the same dimension. If the uniqueness of $\psi$-Whittaker functionals for $\pi$ holds, then $\wt{Z(G)}$ is necessarily  abelian. The converse may not hold. However, assume that $\wt{Z(G)}$ is abelian and $\mu$ a genuine character of it, and let $\psi \times \mu$ be the corresponding representation of $U\times \wt{Z(G)}$, then one can ask for the dimension of the $\mu$-isotypic subspace  
$$\Wh_{\psi \times \mu}(\pi) \subset \Wh_\psi(\pi).$$
For the two-fold Kazhdan--Patterson covering $\wt{\GL}_2^{(2)}$ with twisting parameter $c=0$ (in the notation of \cite{KP}), it was shown by Gelbart, Howe, and Piatetski-Shapiro  \cite{GHPS} that $\dim \Wh_{\psi \times \mu}(\pi) \lest 1$ for every $\pi \in \Irrg(\wt{\GL}_2^{(2)})$. This fact was used in \cite{GPS} to study distinguished theta representations. Such uniqueness results for $\Wh_{\psi \times \mu}(\pi)$ do not hold for higher degree cover of $\GL_2$, nor for covers of $\GL_r$ in general. On the other hand, similar  results were obtained for certain covers of $\GSp_{2r}$ (see \cite[Theorem 5.1]{Szp4-1}), which follow from analyzing the restriction of $\pi \in \Irrg(\wt{\GSp}_{2r}^{(2)})$ to $\wt{\Sp}_{2r}^{(2)}$ and the uniqueness of Whittaker models for the metaplectic group $\wt{\Sp}_{2r}$ as shown in \cite{Szp1}. 

We also mention that some restriction and induction results between the above double cover $\wt{\GL}_2^{(2)}$ and the derived subgroup $\wt{\SL}_2^{(2)}$ were explicated in the work \cite{GPS83} by Gelbart and Piatetski-Shapiro, for both local and global groups. 
For untwisted (i.e., $c=0$ as above) Kazhdan--Patterson covers of $\GL_r$, similar results were obtained by Adams \cite{Ada03}. In a more recent work \cite{PatPr1}, Patel and Prasad refined the multiplicity formula for the restriction problem for $\wt{\GL}_2^{(2)}$, by utilizing the Shimura--Waldspurger correspondence.

It is also worthwhile to point out that in the tame setting when $p\nmid n$,  even if $I(\chi)$ is a $(K, s_K)$-unramified representation of $\wt{G}$, where $K\subset G$ is a hyperspecial maximal compact subgroup of $G$ and
$$s_K: K \into \wt{G}$$
a fixed splitting (assuming its existence), there might exist constituents in $I(\chi)|_{\wt{G}_\der}$ which are not unramified with respect to 
$$K_0:= K \cap G_\der$$
 and its inherited splitting into $\wt{G}_\der$ from $s_K$. This already occurs for even-fold cover of $G=\GL_2$ and was discussed in some depth in \cite{Szp6, GSS2}.

Our paper is thus partly motivated from all the above and we investigate several aspects of the representation theory of covering groups centered around such restriction problem, mainly on the following themes:
\begin{enumerate}
\item[(A1)] the restriction and induction functor between $\wt{G}$ and $\wt{G}_\der$ on the representation side,
\item[(A2)] natural speculations and results on the L-parameter side corresponding to the restriction and induction in (A1),
\item[(A3)] the behaviour of the Whittaker spaces and the local coefficients matrices for representations of $G$ and $G_\der$ with respect to restriction.
\end{enumerate}
Since at this stage we do not have a full understanding of the $\psi$-Whittaker space (or even its dimension) of an arbitrary genuine representation of a covering group, in order to obtain precise results, at various places we focus only on genuine principal series, or even on the unramified ones. Even with such confinement on the class of representations under investigation, exhibited are many new and interesting results which do not seem to exist for linear algebraic groups.

\subsection{Main results} 
We give a brief summary on the content and main results proved in this paper. While some notations are standard, we refer the reader to the actual theorems where full elaboration is given not only on the notations and also on their content. The three themes in (A1)--(A3) are not studied in a strictly linear order and are interwoven with the results proved in various sections.

\subsubsection{} \label{SSS:0-RI}
In \S  \ref{S:RI}, we study the restriction and induction on the representation side between $\wt{G}$ and a certain normal subgroup $\wt{H}$ of $\wt{G}$. For this purpose, we first recall the general Clifford--Mackey theory. There are essentially two families of such $\wt{H}$ considered in this paper, which are of different nature. 
\begin{enumerate}
\item[(i)] In the first case, $\wt{H} \subset \wt{G}$ is such that every $\pi \in \Irrg(\wt{G})$ is $\wt{H}$-concentrated (see Definition \ref{D:conc}). This implies that the support of the Harish-Chandra character distribution $\Theta_\pi$ lies in $\wt{H}$. Such normal subgroup $\wt{H}$ will be the focus of discussion in \S \ref{SS:mtp-GSp}, where we study the metapletic tensor product construction of $\wt{\GSp}_{2r}$, as an analogue to that of $\wt{\GL}_r$ studied in \cite{Kab1, Mez04, Tak3, Tak4, Cai1}.
\item[(ii)] For the second type, we take  
$$\wt{H} = \wt{Z(G)} \cdot \wt{G}_\der,$$
where $Z(G) \subset G$ denotes the center of $G$. Here $\wt{Z(G)}$ and $\wt{G}_\der$ are of dual-pair alike inside $\wt{G}$. Every irreducible representation of $\wt{Z(G)} \cdot \wt{G}_\der$ is of the form $\tau \boxtimes \rho$, where $\tau$ and $\rho$ agree on the intersection of $\wt{Z(G)}$ and $\wt{G}_\der$. Since $\wt{Z(G)}$ is a Heisenberg-type group, $\tau \in \Irrg(\wt{Z(G)})$ is of finite dimension and determined by its central character. The study of restriction and induction between representations of $\wt{G}_\der$ and $\wt{G}$ essentially amounts to those between $\wt{H}$ and $\wt{G}$.
\end{enumerate}

In fact, the consideration in (ii) above applies to the Levi subgroups $\wt{M}$ of $\wt{G}$, and concerns the remaining discussion in \S \ref{S:RI}. For convenience, we write  henceforth
$$M_0:= M_\der$$
for every Levi subgroup $M \subset G$, and denote 
$$\wt{M}^\dag:=\wt{Z(G)} \cdot \wt{M}_0 \subset \wt{Z(M)} \cdot \wt{M}_0.$$
Thus, $G_0$ denotes $G_\der$, and $\wt{G}^\dag =\wt{Z(G)} \cdot \wt{G}_0$.
Note that we use $\wt{Z(G)}$ (instead of $\wt{Z(M)}$) in the definition of $\wt{M}^\dag$ for all Levi subgroup $M\subset G$.
 In \S \ref{SSS:2pic}, we give a sufficient criterion for the induced representation 
$$\Ind_{\wt{M}^\dag}^{\wt{M}} (\tau \boxtimes \rho)$$
to be irreducible. The idea is simple as we explain now. The group
$$\mfr{Q}^\dag:=\wt{M}/\wt{M}^\dag$$
acts naturally on 
$$\mfr{Q}_M:=Z(\wt{M}^\dag)/Z(\wt{M})$$
 by conjugation  and thus on its Pontryagin dual $\widehat{\mfr{Q}_M}$. If the action of $\mfr{Q}^\dag$ on $\widehat{\mfr{Q}_M}$ is free, then the Mackey's theory gives the irreducibility of $\Ind_{\wt{M}^\dag}^{\wt{M}} (\tau \boxtimes \rho)$.  Similarly, we can define $\mfr{Q}_Z$ and $\mfr{Q}_{M_0}$ associated with $\wt{Z(G)}$ and $\wt{M}_0$ which gives a $\mfr{Q}^\dag$-equivariant embedding
$$\begin{tikzcd}
\widehat{\mfr{Q}_M} \ar[r, hook] & \widehat{\mfr{Q}_Z} \times \widehat{\mfr{Q}_{M_0}}.
\end{tikzcd}$$
We have

\begin{thm}[{Theorem \ref{T:in-re}}] \label{T:0-irr}
Let $\wt{M} \subset \wt{G}$ be a covering Levi subgroup.
\begin{enumerate}
\item[(i)] Suppose the action of $\mfr{Q}^\dag$ on either $\widehat{\mfr{Q}_Z}$ or $\widehat{\mfr{Q}_{M_0}}$ is free. Then $\Ind_{\wt{M}^\dag}^{\wt{M}} (\tau \boxtimes \rho)$ is irreducible for every $\tau\boxtimes \rho \in \Irrg(\wt{M}^\dag)$.
\item[(ii)] Specializing to the case $M=T$, if the action of $\mfr{Q}^\dag$ on both $\widehat{\mfr{Q}_Z}$ and $\widehat{\mfr{Q}_{T_0}}$ are trivial, then for every $\pi \in \Irrg(\wt{T})$ we have 
$$\pi|_{\wt{T}^\dag} = (\tau\boxtimes \rho)^{\oplus e}$$ for some
$e\in \N$. On the other hand, if $Z(\wt{T}) \subset Z(\wt{Z(G)}) \cdot Z(\wt{T}_0)$, then $\Ind_{\wt{T}^\dag}^{\wt{T}}(\tau \boxtimes \rho)$ is an isotypic sum of an irreducible genuine representation of $\wt{T}$.
\end{enumerate}
\end{thm}
In Corollary \ref{C:Zinc}, we give several explicit and equivalent conditions for the restriction $\pi|_{\wt{T}_0}$ to be an isotypic sum. Following an idea from \cite{Szp4-1}, we study the irreducibility of parabolic induction for $\wt{G}$ and $\wt{G}_0$, see Theorem \ref{T:irre-c}.

Several families of covers are worked out in detail in \S \ref{SS:RI-eg}. In particular, we analyze the structure of  the restriction and induction of representations for covers of $\GL_r, \GSp_{2r}$ and $\GSpin_{2r+1}$, and apply Theorem \ref{T:in-re} and Corollary \ref{C:Zinc}  to obtain concrete results, see Propositions \ref{P:GLsum}, \ref{P:RI-GSp} and \ref{P:RI-GSpin}. We note that the example of covers of $\GSpin_{2r+1}$ was already discussed by Kaplan in \cite{Kap004}. The scheme of analysis in \S \ref{SS:RI-eg} essentially follows \cite{Szp4-1, Szp5}. Indeed, the results in \cite{GPS83} for $\wt{\GL}_2^{(2)}$ were generalized in \cite{Szp4-1} for $\wt{\GSp}_{2r}^{(2)}$, where the latter also includes a discussion on the analogous problem of restriction and induction on covering Levi subgroups of $\wt{\GSp}_{2r}^{(2)}$. It should be noted that a common feature of all the covers studied in \cite{Ada03, GPS83, Szp4-1} is that $\wt{Z(G)}$ is abelian. However, in this paper (and also the previous study \cite{GSS2}), we do not impose such constraints and treat covers in generality.

In the last part of \S \ref{S:RI}, we work out explicitly in the tame case (i.e., 
when $p\nmid n$) the decomposition of a genuine principal series $I(\chi)$ of $\wt{G}$ when restricted to $\wt{G}_0$. Recall that in the tame case, if $\wt{G}$ splits over the hyperspecial maximal compact subgroup $K \subset G$ generated by $\mbf{T}(O)$ and $e_{\alpha}(O)$, then we could fix such a splitting $s_K: K \into \wt{G}$ and consider a $(K, s_K)$-unramified genuine principal series $I(\chi)$. Such a $(K, s_K)$-unramified genuine principal series is parabolically induced from $i(\chi) \in \Irrg(\wt{T})$. We set $K_0 = K \cap G_\der$.

\begin{thm}[Theorem \ref{T:decT}] \label{T:0-dec}
Let $i(\chi):=\Ind_{\wt{A}}^{\wt{T}}(\tchi) \in \Irrg(\wt{T})$ be with central character $\chi$. Let $I(\chi)$ be the associated genuine principal series of $\wt{G}$. 
\begin{enumerate}
\item[(i)] One has
$$i(\chi)|_{\wt{T}_0}= \val{\msc{X}_{Q,n}^\mfr{c}} \cdot \bigoplus_{\gamma \in \msc{X}_{Q,n}^\Gamma/\msc{X}_{Q,n}^\mfr{c}} \bigoplus_{\omega_{\gamma, j} \in \msc{E}(\chi, {}^\gamma \tchi_O; Z(\wt{T}_0))} i(\omega_{\gamma, j}),$$
where every $i(\omega_{\gamma, j}) \in \Irrg(\wt{T}_0)$ appears with multiplicity one in the double sum of the right hand side. 
\item[(ii)] Consequently,
$$I(\chi)|_{\wt{G}_0}= \val{\msc{X}_{Q,n}^\mfr{c}} \cdot \bigoplus_{\gamma \in \msc{X}_{Q,n}^\Gamma/\msc{X}_{Q,n}^\mfr{c}} \bigoplus_{\omega_{\gamma, j} \in \msc{E}(\chi, {}^\gamma \tchi_O; Z(\wt{T}_0))} I(\omega_{\gamma, j})$$
for the restriction of the genuine principal series $I(\chi)$. Here generically, the $I(\omega_{\gamma, j})$'s appear multiplicity-free in the double sum of the right hand side. 
\item[(iii)] Assume that $I(\chi)$ is an $(K, s_K)$-unramified genuine principal series. Then $I(\omega_{\gamma, j})$ is $(K_0, s_K)$-unramified if and only if $\omega_{\gamma, j}$ belongs to the set $\msc{E}(\chi, \tchi_O; Z(\wt{T}_0))$, i.e., with $\gamma =0 \in \msc{X}_{Q,n}^\Gamma /\msc{X}_{Q,n}^\mfr{c}$ being the trivial class.
\end{enumerate}
\end{thm}
In the above theorem, the finite abelian group $\msc{X}_{Q,n}^\Gamma$ is a natural quotient of $\msc{X}_{Q,n}$, the latter of which is the ``moduli space" of the Whittaker space $\Wh_\psi(I(\chi))$ when $I(\chi)$ is unramified, as discussed in detail in \cite{GSS2}. We highlight that from Theorem \ref{T:0-dec} (iii) above, the genuine principal series $I(\omega_{\gamma, j})$ of $\wt{G}_0$ with $\gamma \ne 0$  is not $(K_0, s_K)$-unramified.
In fact, such $I(\omega_{\gamma, j}), \gamma \ne 0$ is unramified with respect to $(K_0', s_{K'})$ inherited from a different pair $(K', s_{K'})$ for $\wt{G}$. 

It is easy to work out from  Theorem \ref{T:0-dec} when the restriction $I(\chi)|_{\wt{G}_0}$ is an isotypic sum of a genuine principal series: this is exactly when $(\wt{G}, \wt{G}_0)$ is an isotypic pair in the sense of Definition \ref{D:isop}, i.e., when
$$Z(\wt{T}_0) \subset Z(\wt{T}),$$
or equivalently by Corollary \ref{C:Zinc}, $Y_0 \cap Y_{Q,n} = Y_{0,Q,n}$.

\subsubsection{}  In \S \ref{S:func}, we study the dual side of the restriction and induction discussed on the representation side in \S \ref{S:RI}. Some anomalies in the covering setting were already discussed in \cite{GG} and part of the work here may be considered as a generalization and further elaboration of some speculations there. To motivate the general setup, it is instructive to consider the Kazhdan--Patterson double cover $\wt{G}=\wt{\GL}_2^{(2)}$, which by restriction gives the classical metaplectic double cover $\wt{G}_0= \wt{\SL}_2^{(2)}$. Regarding the dual groups, one has
$$\wt{G}^\vee=\GL_2 \text{ and } \wt{G}_0^\vee = \SL_2.$$
Clearly, there is no natural nontrivial homomorphism of algebraic groups from $\wt{G}^\vee$ to $\wt{G}_0^\vee$, which supposedly encodes from the dual side the restriction functor from $\wt{G}$ to $\wt{G}_0$. On the other hand, both $\wt{G}^\vee$ and $\wt{G}_0^\vee$ map naturally to $\PGL_2$, which thus serves as the bridge between the two dual groups. 

For general $\wt{G}$, there is a natural linear algebraic group $H$ and a cover $\wt{H}$ whose dual group $\wt{H}^\vee$ is endowed with two natural maps:
$$f_{G, H}: \wt{G}^\vee \longrightarrow \wt{H}^\vee \text{ and } f_{G_0, H}: \wt{G}^\vee_0 \longrightarrow \wt{H}^\vee.$$
These two maps can be extended to homormophisms of the corresponding L-groups. Thus, we expect a natural commutative diagram 
$$
\begin{tikzcd}
 & & {}^L \wt{Z(G)} \ar[r, "{f_{c,z}}"] & {}^L Z(\wt{G}) \\
 \WD_F \ar[rr, "{\phi_\pi}"] \ar[rru, "{\phi_\tau}"] \ar[rrd, "{\phi_\rho}"'] & & {}^L \wt{G}  \ar[ru, "{f_{G,z}}"']  \ar[rd, "{f_{G,H}}"] \\
&  & {}^L \wt{G}_0 \ar[r, "{f_{G_o, H}}"'] & {}^L \wt{H}
\end{tikzcd}
$$
indicating some natural speculations on the pertinent parameters for representations involved in the restriction and induction between $\wt{G}$ and $\wt{G}_0$. This we explain below.

Denote by 
$$\phi: \WD_F \longrightarrow {}^L\wt{G}$$
an L-parameter of $\wt{G}$. Let $\mca{L}(\phi)$ be the associated hypothetical L-packet. We assume the strong form of the local Langlands correspondence that every $\phi \in \mca{L}(\phi)$ is associated with an enhanced parameter
$$(\phi, \theta) \in \Phi^{\rm en}({}^L\wt{G}),$$
where $\theta$ is an irreducible representation of the component group $\mca{S}_\phi:=\mca{S}(\phi)$ of $\phi$. Denote 
$$\phi^\diamondsuit = f_{G,H} \circ \phi,$$
and consider the naturally induced map
$$\mca{S}_\phi \into \mca{S}_{\phi^\diamondsuit},$$
which is an embedding (see \cite[Proposition 5.4]{Sol20}). Let $\Phi({}^L\wt{G}_0; \phi)$ be the set of parameters $\phi_0: \WD_F \to {}^L\wt{G}_0$ such that $f_{G_0, H}\circ \phi_0 = \phi^\diamondsuit$. For every $\phi_0 \in \Phi({}^L\wt{G}_0; \phi)$, we have an embedding 
$$\mca{S}_{\phi_0} \into \mca{S}_{\phi^\diamondsuit}$$
 as well.

\begin{conj}[Conjecture \ref{C:res}] \label{C:0-res}
Let $\pi \in \Irrg(\wt{G})$ be  with central character $\omega_\pi$ and associated enhanced parameter $(\phi_\pi, \theta_\pi) \in \Phi^{\rm en}({}^L\wt{G})$. Let $\tau \boxtimes \rho$ be an irreducible representation of $\wt{G}^\dag=\wt{Z(G)} \cdot \wt{G}_0$ occurring in the restriction of $\pi$. 
 \begin{enumerate}
 \item[(i)] The L-parameters  $\phi_{\omega_\pi}, \phi_\tau$ and $\phi_\rho$ are such that the following hold:
\begin{enumerate}
\item[$\bullet$] $f_{G,z} \circ \phi_\pi = f_{c, z} \circ \phi_\tau$ and is equal to the parameter $\phi_{\omega_\pi}$ associated with the central character $\omega_\pi$ of $\pi$;
\item[$\bullet$] $\phi_\pi^\diamondsuit = f_{G_0, H} \circ \phi_\rho$.
\end{enumerate}
\item[(ii)] 
There exists $e(\pi)\in \N$ such that for every $\tau \in \Irrg(\wt{G}_0)$ with enhanced parameter $(\phi_\tau, \theta_\tau) \in \Phi^{\rm en}({}^L\wt{G}_0)$, one has
$$\dim \Hom_{\wt{G}_0}(\pi, \tau) =
\begin{cases}
0 & \text{ if } \phi_\tau \notin \Phi({}^L\wt{G}_0; \phi_\pi), \\
e(\pi) \cdot \angb{ \Ind_{\mca{S}_{\phi_\pi}}^{\mca{S}_{\phi^\diamondsuit}} \theta_\pi }{ \Ind_{\mca{S}_{\phi_\tau}}^{\mca{S}_{\phi^\diamondsuit}} \theta_\tau }_{\mca{S}_{\phi^\diamondsuit}} & \text{ if }  \phi_\tau \in \Phi({}^L\wt{G}_0; \phi_\pi),
\end{cases}
$$
where $\angb{-}{-}$ denotes the pairing of two representations of $\mca{S}_{\phi^\diamondsuit}$; in particular, if $(\wt{G}, \wt{G}_0)$ is an isotypic-pair, then
$$\pi|_{\wt{G}_0} =e(\pi) \cdot \bigoplus_{\tau \in \mca{L}(\phi_\pi^\diamondsuit)} \angb{ \theta_\tau|_{\mca{S}_{\phi_\pi}} }{ \theta_\pi}_{\mca{S}_{\phi_\pi}} \cdot \tau,$$
where $(\phi_\pi^\diamondsuit, \theta_\tau) \in \Phi^{\rm en}({}^L\wt{G}_0)$ is the enhanced parameter associated with $\tau \in \mca{L}(\phi_\pi^\diamondsuit)$.
\item[(iii)] Let $\tau \in \Irrg(\wt{Z(G)})$ and $\rho \in \Irrg(\wt{G}_0)$ be compatible representations. Then every irreducible constituent $\pi \in \Irrg(G)$ of $\Ind_{\wt{G}^\dag}^{\wt{G}} (\tau\boxtimes \rho)$ has a parameter $\phi_\pi$ which fits into a commutative diagram \eqref{L-comp}.
\end{enumerate}
\end{conj}
Part (ii) above is clearly motivated from the case of linear algebraic groups, see \cite{AP1} and references therein. Conjecture \ref{C:0-res} is not a statement that can be tackled at the moment for general $\pi\in \Irrg(\wt{G})$, as it relies on an established local Langlands correspondence (LLC). For example, it would be interesting to see if the results as in \cite{PatPr1} for $\wt{\GL}_2^{(2)}$ fit into Conjecture \ref{C:0-res}, if one has an established LLC for $\wt{\GL}_2^{(2)}$. For general covering torus for which LLC has been proved (see \cite{KP, Mc1, We1, We5, We6, GG}), we verify this in \S \ref{SS:func-ps}, see especially Proposition \ref{P:func-tor}.

In the remaining of \S \ref{S:func}, we study the metaplectic tensor product and a metaplectic restriction (the latter construction as a counterpart of the former) for covers $\wt{\GSp}_{2r}$. For Kazhdan--Patterson covers $\GL_r$, this problem was studied by P. Mezo \cite{Mez04} and also in several works following it, as mentioned in \S \ref{SSS:0-RI}. For a general Levi subgroup $$M_{\mbf{r}} = \GL_{r_1} \times ... \times \GL_{r_k},$$
the covering subgroups $\wt{\GL}_{r_i}$ and $\wt{\GL}_{r_j}$ of $\wt{\GL}_r$ do not commute in general, and thus the representation theory of $\wt{M}_\mbf{r}$ is not obtained from the naive tensor product of a genuine representation of each $\wt{\GL}_{r_i}$. However, there is still a surjective map (see \cite{Mez04})
$$\begin{tikzcd}
\tilde{\otimes}: ( \Irrg(\wt{\GL}_{r_1}) \times ... \times \Irrg(\wt{\GL}_{r_k}) \times \Irrg(Z(\wt{\GL}_{r})))^\heartsuit \ar[r, two heads] & \Irrg(\wt{M}_\mbf{r}),
\end{tikzcd} 
$$ 
where the superscript $(-)^\heartsuit$ indicates the subset of 
$$(\pi_1, ..., \pi_r, \omega)$$
satisfying certain compatible relation. The map $\tilde{\otimes}$ is called the metaplectic tensor product construction. The functorial interpretation of $\tilde{\otimes}$ on the dual side was given by Gan \cite{Gan17}.

As an analogue of $\wt{\GL}_r$ above, in \S \ref{SS:mtp-GSp} we investigate the case of $\wt{\GSp}_{2r}$. Let
$\mbf{r}=(r_1, r_2, ..., r_k; r_0)$ be a partition of $r$. Let
$$M_\mbf{r} = \GL_{r_1} \times \GL_{r_2} \times ... \times \GL_{r_k} \times \GSp_{2r_0}$$
be the associated Levi subgroup. We show the following.

\begin{thm}[Theorem \ref{T:GSp}]
Let $\wt{\GSp}_{2r}$ be a similitudes-splitting $n$-fold cover of $\GSp_{2r}$ with $n$ odd. There is a well-defined bijective map
$$\begin{tikzcd}
\tilde{\otimes}_i: ( \prod_{i=1}^k\Irrg(\wt{\GL}_{r_i}) \times \Irrg(\wt{\GSp}_{2r_0}) \times \Irrg(Z(\wt{\GSp}_{2r})))^\heartsuit \ar[r] & \Irrg(\wt{M}_\mbf{r}),
\end{tikzcd} $$
where the superscript $(-)^\heartsuit$ indicates the subset of  $(\pi_1, ..., \pi_k, \pi_0, \omega)$  satisfying a certain equality \eqref{E:comp01}. The inverse of $\tilde{\otimes}_i$ is given by a natural and explicit ``metaplectic restriction" map $\tilde{\rm R}$.
\end{thm}

One has a conjectural dual side description of the two maps $\tilde{\otimes}_i$ and $\tilde{\rm R}$ given in the theorem above, on the relations among the parameters of the representations involved. This is Conjecture \ref{C:F-odd}, and is clearly motivated from that in \cite{Gan17} for Kazhdan--Patterson covers; in fact, it is  a simpler situation dealing with $\wt{\GSp}_{2r}$ here. For genuine principal series, we verify Conjecture \ref{C:F-odd} in Proposition \ref{P:F-odd-ps}. We also explain the obstacles to realizing a similar metaplectic tensor product construction for even-fold cover of $\GSp_{2r}$.

In fact, such metaplectic tensor product construction is ubiquitous for covering Levi subgroups. Indeed,  it essentially amounts to the existence of a certain subgroup $\wt{H} \subset \wt{M}$ satisfying two conditions which counteract against each other. On the one hand, we would like that every $\pi\in \Irrg(\wt{M})$ is $\wt{H}$-concentrated, and thus $\wt{H}$ should be taken as large as possible. On the other hand, suppose $\wt{M}_i \subset \wt{M}$ are the covering blocks, then it is a desired property that $\wt{H} \cap \wt{M}_i$ commutes with each other: this is satisfied only when $\wt{H}$ is as small as possible. Whenever a subgroup $\wt{H} \subset \wt{M}$ satisfying these two conditions exist, one could have a metaplectic tensor product. For $\wt{M}_\mbf{r}$ of the Kazhdan--Patterson covers, we take $\wt{H}$ to be essentially the subgroup of elements with  determinants lying in $F^{\times n}$; while for $\wt{\GSp}_{2r}^{(n)}$ with $n$ odd, we take the subgroup with similitudes in $F^{\times n}$. For more details, see \S \ref{SS:gen-mtp}. 

\subsubsection{} 
In \S \ref{S:LCM}, we study how the local coefficients matrices associated with certain intertwining operators behave with respect to the restriction. More precisely, suppose
$$\pi|_{\wt{G}_\der} = \bigoplus_i m_i \cdot \pi_i,$$
with $\pi_i\in \Irrg(\wt{G}_\der)$. Then one has
$$\Wh_\psi(\pi) = \bigoplus_i m_i  \cdot\Wh_\psi(\pi_i).$$
Let $\wt{M} \subset \wt{P} \subset \wt{G}$ be a Levi subgroup and $\sigma \in \Irrg(\wt{M})$. Let 
$$T(w, \sigma): I_{\wt{P}}^{\wt{G}}(\sigma) \longrightarrow I_{\wt{P}'}^{\wt{G}}({}^w \sigma)$$
be the standard intertwining operator, where $\wt{P}'$ is a parabolic subgroup associated with $\wt{P}$ by $w$. It is shown in \cite[\S 3.2]{GSS2} that there is a naturally induced homomorphism
$$\mca{T}(w, \sigma)^*: \Wh_\psi( I_{\wt{P}}^{\wt{G}}(\sigma) ) \longrightarrow \Wh_\psi( I_{\wt{P}}^{\wt{G}}(\sigma) ).$$
With the choice of a basis $\mfr{B}$ for $\Wh_\psi( I_{\wt{P}}^{\wt{G}}(\sigma) )$, one obtains a local coefficients matrix
$$\mca{M}_\mfr{B}(w, \sigma)$$
representing $\mca{T}(w, \sigma)^*$.

For every $\sigma_0 \in \Irrg(\wt{M}_0)$, one has an analogous map
$$\mca{T}(w, \sigma_0)^*: \Wh_\psi( I_{\wt{P}_0}^{\wt{G}_0}(\sigma) ) \longrightarrow \Wh_\psi( I_{\wt{P}_0}^{\wt{G}_0}(\sigma) ).$$
We show that 
$$\mca{T}(w, \sigma)^* = \bigoplus_{\sigma_i \subset \sigma|_{\wt{M}_0}} \mca{T}(w, \sigma_i)^*.$$
This implies that there is a local coefficients matrix $\mca{M}_\mfr{B}(w, \sigma)$ taking the form
$$\bigoplus_{\sigma_i \subset \sigma|_{\wt{M}_0}} \mca{M}_{\mfr{B}_i}(w, \sigma_i),$$
where $\mca{M}_{\mfr{B}_i}(w, \sigma_i)$ is a local coefficients matrix associated with the triple $(\wt{P}_0, w, \sigma_i)$ and a basis $\mfr{B}_i$ of $\Wh_\psi( I_{\wt{P}_0}^{\wt{G}_0} (\sigma_i))$, see Proposition \ref{P:decWh}. In particular, the arithmetic invariants (trace and determinant, for example) of $\mca{M}_\mfr{B}(w, \sigma)$ are completely determined by those of $\mca{M}_{\mfr{B}_i}(w, \sigma_i)$.

Since the Whittaker space $\Wh_\psi(\pi)$ is best understood for a genuine principal series, we specialize to this case and investigate the relation between $\Wh_\psi(\pi)$ and $\Wh_\psi(\pi_i), \pi_i \subset \pi|_{\wt{G}_\der}$, where $\pi \in {\rm JH}(I(\chi))$ is an irreducible constituent of a $(K, s_K)$-unramified genuine principal series. In fact, we consider the two special cases:
\begin{enumerate}
\item[--] when $I(\chi)$ is a regular unramified principal series in \S \ref{SS:reg-ps}, and
\item[--] when $I(\chi)$ is a unitary unramified principal series in \S \ref{SS:uni-ps}.
\end{enumerate}
The analysis relies on certain results and conjectures from \cite{Ga6, Ga7}. We elaborate below these two cases.

First, suppose $I(\chi)$ is a regular unramified principal series. Let $\Phi(\chi)$
be the set of rank-one reducibility points. Then every irreducible constituent of $I(\chi)$ is uniquely parametrized by a subset $S \subset \Phi(\chi)$, which we denote by $\pi_{\chi, S}$ or $\pi(\chi)_S$.  If $(\wt{G}, \wt{G}_0)$ is an isotypic pair, then one has $\Phi(\chi) = \Phi(\omega)$ where $\omega = \chi|_{Z(\wt{T}_0)}$; in this case, 
$$(\pi_{\chi, S})|_{\wt{G}_0} = \val{\msc{X}_{Q,n}^\Gamma} \cdot \pi_{\omega, S},$$
which immediately yields the equality
\begin{equation} \label{E:0-dimWh}
\dim \Wh_\psi(\pi_{\chi, S})= \val{\msc{X}_{Q,n}^\Gamma} \cdot \dim \Wh_\psi(\pi_{\omega, S}).
\end{equation}

Assume 
$$\Phi(\chi) \subset \Delta,$$
i.e., it is a subset of simple roots. Then we have a formula for $\dim \Wh_\psi(\pi_{\chi, S})$ and $\dim \Wh_\psi(\pi_{\omega, S})$  in terms of certain permutation representations $\sigma_\msc{X}$ and $\sigma_{\msc{X}_0}$ of $W$ respectively. It is expected that the equality \eqref{E:0-dimWh} of Whittaker dimensions actually originates from a relation between $\sigma_\msc{X}$ and $\sigma_{\msc{X}_0}$, see Conjecture \ref{C:dec}.
If $(\wt{G}, \wt{G}_0)$ is not an isotypic pair, then the relation between $\dim \Wh_\psi(\pi_{\chi, S})$ and $\dim \Wh_\psi(\pi_{\omega, S})$ for an arbitrary $I(\omega) \subset I(\chi)|_{\wt{G}_0}$ is quite complicated, since $I(\omega)$ may not be $(K_0, s_K)$-unramified. This already occurs for even-fold cover of $\GL_2$ (odd-fold cover of $\GL_2$ gives an isotypic pair). However, for $\wt{\GL}_2$, by using the work in \cite{GSS2} we can show the following:

\begin{thm}[Theorem \ref{T:GLSL2}] \label{T:0-GL2}
Consider the Kazhdan-Patterson $n$-fold cover $\wt{\GL}_2$. Let $\chi: Z(\wt{T}) \to \C^\times$ be an unramified genuine character such that $\Phi(\chi) = \Delta$. Assume $n=2m$ is even and $\mfr{f}(\psi)=O_F$. We always have 
$$\dim \Wh_\psi(\pi(\omega_{\gamma_i, j})_\emptyset) + \dim \Wh_\psi(\pi(\omega_{\gamma_i, j})_\Delta) = \dim \Wh_\psi(I(\omega_{\gamma_i, j})) = m.$$
Moreover, we have an explicit formula for every $\dim \Wh_\psi(\pi(\omega_{\gamma_i, j})_\Delta)$, and thus also $\dim \Wh_\psi(\pi(\omega_{\gamma_i, j})_\emptyset$.
\end{thm}

Second, assume $I(\chi)$ is a unitary $(K, s_K)$-unramified principal series. Then we have a decomposition 
$$I(\chi) = \bigoplus_{\tau \in \Irr(\mca{S}_\chi)} \pi_\tau$$
of $I(\chi)$ into irreducible constituents, where 
$$\mca{S}_\chi:=\mca{S}_{\phi_\chi}$$
 is the component group of the centralizer of ${\rm Im}(\phi_\chi) \subset {}^L\wt{G}$. If $I(\omega^\flat) \subset I(\chi)|_{\wt{G}_0}$ is a $(K_0, s_K)$-unramified constituent as in Theorem \ref{T:0-dec}, then one has analogously
$$I(\omega^\flat) = \bigoplus_{\rho \in \Irr(\mca{S}_{\omega^\flat})} \pi_\rho.$$
Recall the map
$$\phi^\diamondsuit:= f_{G, H} \circ \phi_\chi = f_{G_0, H} \circ \phi_{\omega^\flat}$$
and the associated component group $\mca{S}^\diamondsuit=\mca{S}_{\phi^\diamond}$. We have the inclusions
$$\begin{tikzcd}
\mca{S}_\chi \ar[rd, hook] \\
\mca{S}_{\omega^\flat} \ar[r, hook] &  \mca{S}^\diamondsuit.
\end{tikzcd}
$$
The main result of \S \ref{SS:uni-ps} is regarding the multiplicity $\dim \Hom_{\wt{G}_0}(\pi_\rho, \pi_\tau)$. 

\begin{thm}[Theorem \ref{T:uni-func}] \label{T:0-mul-uni-ps}
Let $I(\chi)$ be a unitary $(K, s_K)$-unramified genuine principal series of $\wt{G}$ and $I(\omega^\flat)$ a $(K_0, s_K)$-unramified constituent of $I(\chi)|_{\wt{G}_0}$. Then
\begin{equation} \label{E:uni-mul}
\dim \Hom_{\wt{G}_0}(\pi_\rho, \pi_\tau) =\val{ \msc{X}_{Q,n}^\mfr{c} } \cdot \angb{ \Ind_{\mca{S}_{\omega^\flat}}^{ \mca{S}^\diamondsuit } \rho }{  \Ind_{\mca{S}_\chi}^{ \mca{S}^\diamondsuit } \tau }_{\mca{S}^\diamondsuit}
\end{equation}
for every $\tau \in \Irr(\mca{S}_\chi)$ and $\rho \in \Irr(\mca{S}_{\omega^\flat})$. 
\end{thm}
In fact, even if $I(\omega^\flat) \subset I(\chi)|_{\wt{G}_0}$ is not $(K_0, s_K)$-unramified, the equality holds as well, see Corollary \ref{C:uni-ac}. 
In particular, this verifies Conjecture \ref{C:0-res} (ii) for unitary unramified principal series of $\wt{G}$ with 
$$e(I(\chi)) = \val{ \msc{X}_{Q,n}^\mfr{c} }.$$
If $n=1$ and $\wt{G}=G$ is a split linear group, then the above formula recovers that in \cite{Key3}, the proof of which in fact motivates that for Theorem \ref{T:0-mul-uni-ps} above.

\subsubsection{} In \S \ref{S:unLP}, we concentrate on unramified L-packets and consider the following two problems:
\begin{enumerate}
\item[(i)] the variation of the parametrization of elements inside an unramified L-packet  with respect to the change of hyperspecial maximal compact subgroup $K$ of $G$ and different choices of splittings,
\item[(ii)] the variation of the Whittaker dimensions of elements inside an unramified L-packet with respect to changing orbits of the Whittaker datum. 
\end{enumerate}
Let $\chi: Z(\wt{T}) \to \C^\times$ be a genuine central character. Fix a splitting $s_{T}: \mbf{T}(O) \into \wt{T}$, and assume that $i(\chi)$ is $(\mbf{T}(O), s_T)$-unramified. By the local Langlands correspondence for covering tori, one has a parameter
$$\phi_\chi: \WD_F \longrightarrow {}^L\wt{T} \into {}^L\wt{G}.$$
It is postulated that the L-packet $\mca{L}(\phi_\chi)$ associated with $\phi_\chi$ consists of exactly the subquotients of the principal series $I(\chi)$, which are $(K, s_K)$-unramified with respect to a hyperspecial maximal compact subgroup $K \subset G$ and a splitting $s_K: K \into \wt{G}$ such that $s_K$ restricts to $s_T$ on $\mbf{T}(O)$. It is also expected that there is a bijection between
$$\mca{L}(\phi_\chi) \longleftrightarrow \Irr(\mca{S}_{\phi_\chi})$$
which we denote by
$$\pi(\phi_\chi, \rho) \leftrightarrow \rho.$$
Thus, we want to investigate how the $(K, s_K)$-unramifiedness of an element $\pi(\phi_\chi, \rho) \in \mca{L}(\phi_\chi)$ is reflected from its parameter $\rho$.

For this purpose, we briefly recall the linear algebraic case. Let
$$\mca{K}=\set{G_x: \ x \text{ is a hyperspecial point in } \msc{B}(G) }/_\sim$$
 be the set of conjugacy classes of hyperspecial maximal compact subgroups of $G$, where $\msc{B}(G)$ is the Bruhat--Tits building associated with $G$. It is known that $\mca{K}$ is a torsor over 
 $$\widehat{\Gamma_G^{\rm tor}}:=\Hom((X/X^{sc})_{\rm tor}, \Q/\Z),$$
and there is a natural surjective map
$$f_\Gamma: \widehat{\Gamma_G^{\rm tor}} \onto \Irr(\mca{S}_{\phi_\chi}),$$
which gives a natural action of $\widehat{\Gamma_G^{\rm tor}}$ on $\Irr(\mca{S}_{\phi_\chi})$ given by $f_\Gamma(y) \otimes \rho$.
It follows from \cite[Theorem 1]{Mis2} that for every $K \in \mca{K}$ and $y\in \widehat{\Gamma_G^{\rm tor}}$, the representation $\pi(\phi_\chi, \rho) \in \mca{L}(\phi_\chi)$ is $K$-unramified if and only if $\pi(\phi_\chi, f_\Gamma(y) \otimes \rho)$ is $y\cdot K$-unramified.

To generalize the above results to covering groups, we note the following obstacles:
\begin{enumerate}
\item[$\bullet$] first, not every splitting of $K$ into $\wt{G}$ is compatible with $s_T$, and thus it is necessary to filter only those compatible ones,
\item[$\bullet$] second, although $\mca{K}$ stays as a torsor over $\widehat{\Gamma_G^{\rm tor}}$, there is no longer any obvious natural map from $\widehat{\Gamma_G^{\rm tor}}$ to $\Irr(\mca{S}_{\phi_\chi})$. 
\end{enumerate}
For the second obstacle above, since $\mca{S}_{\phi_\chi}$ is the component group of the image of $\phi_\chi$ in ${}^L\wt{G}$, there is a natural finite abelian group $\widehat{\Gamma_{G_{Q,n}}^{\rm tor}}$ arising from the principal endoscopic group $G_{Q,n}$ of $\wt{G}$ and a surjection 
$$\tilde{f}_\Gamma: \widehat{\Gamma_{G_{Q,n}}^{\rm tor}} \onto \Irr(\mca{S}_{\phi_\chi}).$$
The map $\tilde{f}_\Gamma$ is defined analogous to $f_\Gamma$ in the linear case. With this in view, it is  important to bridge the two groups $\widehat{\Gamma_G^{\rm tor}}$ and $\widehat{\Gamma_{G_{Q,n}}^{\rm tor}}$. A substantial part of \S \ref{SS:varK} is to analyze such a relation, which gives Conjecture \ref{C:varK}.

Regarding the Whittaker space, we fix a Borel subgroup $B=TU$ and vary the non-degenerate character  $\psi: U \to \C^\times$. Let $T_{ad} \subset G_{ad}$ be the split torus of $G_{ad}$. The set of $T$-orbits of non-degenerate characters of $U$ is a torsor over $T_{ad}/T$ under the conjugation action. There is a natural quotient
$$T_{ad}/T \onto \Irr(\mca{S}_{\phi_\chi}),$$
and thus an action of $ t\in T_{ad}/T$ on $\rho  \in \Irr(\mca{S}_{\phi_\chi})$, which we denote by $t\cdot \rho$. In this case, it was shown in \cite{Mis1} that $\pi(\phi_\chi, \rho)$ is $\psi$-generic if and only if $\pi(\phi_\chi, t\cdot \rho)$ is ${}^t \psi$-generic.
Now for a covering group $\wt{G}$, there is no obvious map from $T_{ad}/T$ to $\Irr(\mca{S}_{\phi_\chi})$. Similar to the above discussion on the unramified-ness, there is a natural surjection
$$T_{Q,n, ad}/T_{Q,n} \onto \Irr(\mca{S}_{\phi_\chi}),$$
where $T_{Q,n}$ is the torus of $G_{Q,n}$ and $T_{Q,n, ad}$ is the torus for its adjoint group $G_{Q,n, ad}$. Again, it is important to inspect any natural relation between $T_{Q,n, ad}/T_{Q,n}$ and $T_{ad}/T$, or at the level of character lattices, to relate $X_{Q,n}/X_{Q,n}^{sc}$ and $X/X^{sc}$. We discuss this in \S \ref{SS:psi-var} and postulate Conjecture \ref{C:varW} on the  $\psi$-Whittaker dimension of  $\pi(\phi_\chi, \rho)$ with respect to the action of $t$ on the character $\psi$.

The following is the main result of \S \ref{S:unLP}, amalgamated from the discussions in \S \ref{SS:varKW-eg}, especially the main Propositions therein.
\begin{thm}
Conjecture \ref{C:varK} and Conjecture \ref{C:varW} hold for unitary unramified principal series of covers of $\Sp_{2r}, \SO_3$, and also for Kazhdan--Patterson covers of $\GL_r$.
\end{thm}

\subsubsection{}  In \S \ref{S:GSp}, we consider exclusively covers of $\GSp_{2r}$ and its unitary unramified principal series $I(\chi)$.
The discussion is oriented towards Conjecture \ref{C:varK} and Conjecture \ref{C:varW}. In particular, we 
\begin{enumerate}
\item[$\bullet$] determine the component group $\mca{S}_{\phi_\chi}$ and the pair $(K, s_K)$ with respect to which $\pi(\phi_\chi, \rho)$ is unramified;
\item[$\bullet$] determine the $\psi$-Whittaker dimension of each constituent $\pi(\phi_\chi, \rho)$, when $\mfr{f}(\psi) =\mfr{p}_F$ or $O_F$;
\item[$\bullet$] investigate the restriction of each $\pi(\phi_\chi, \rho)$ to the derived subgroup $\wt{G}_0=\wt{\Sp}_{2r}$ and the pair $(K_0, s_{K_0})$ with respect to which an irreducible constituent of $\pi(\phi_\chi, \rho)|_{\wt{\Sp}_{2r}}$ is unramified, and also study the $\psi$-Whittaker dimension of such a constituent.
\end{enumerate}
We summarize (in a compressed form) the main results for \S \ref{S:GSp} obtained from Theorems \ref{T:GSp-odd}, \ref{T:nev-rod} and \ref{T:GSp-ee}.
Below $K$ denotes the standard hyperspecial maximal compact subgroup and $K_0:=K \cap \Sp_{2r}$. Also, $\set{K_0, K_0'}$ denotes the set of conjugacy classes of hyperspecial maximal compact subgroup of $\Sp_{2r}$, and $s_0, s_0'$ denote the unique splittings of $K_0$ and $K_0'$ into $\wt{\Sp}_{2r}$ respectively. Let $\psi$ and $t\in T_{ad}/T$ be such that
$$\mfr{f}(\psi)=O_F \text{ and } \mfr{f}({}^t \psi) = \mfr{p}_F.$$
We also write 
$$n_{(r)}:=n/\gcd(n, r).$$

\begin{thm}[Theorems \ref{T:GSp-odd},  \ref{T:nev-rod} and \ref{T:GSp-ee}]
Consider the $n$-fold cover $\wt{\GSp}_{2r}$ of similitudes-splitting type. Let $I(\chi)$ be a unitary $(K, s_K)$-unramified genuine principal series of $\wt{\GSp}_{2r}$.
\begin{enumerate}
\item[(i)] If $n$ is odd, then $I(\chi)|_{\wt{G}_0} = \val{\msc{X}_{Q,n}^\Gamma} \cdot I(\omega)$, where $\omega = \chi|_{Z(\wt{T}_0)}$. If $I(\omega) = \pi(\phi_\omega, \mbm{1}) \oplus \pi(\phi_\omega, \varepsilon)$ is reducible, and we assume that $\pi(\phi_\omega, \mbm{1})$ is $(K_0, s_0)$-unramified, then $\pi(\phi_\omega, \varepsilon)$ is $(K_0', s_0')$-unramified.
\item[(ii)] If $n$ is even and $r$ is odd, then $I(\chi)$ is always irreducible and 
$$I(\chi)|_{\wt{G}_0}= \frac{n_{(r)}}{2}\cdot (I(\omega_{0, 0}) \oplus I(\omega_{0, 1}))\bigoplus \frac{n_{(r)}}{2}\cdot (I(\omega_{e_0, 0}) \oplus I(\omega_{e_0, 1})),$$
where every $I(\omega_{\gamma, j})$ is irreducible. Moreover, $I(\omega_{0, j}), j=0, 1$ is $(K_0, s_0)$-unramified and $I(\omega_{e_0, j}), j=0, 1$ is $(K_0', s_0')$-unramified.
\item[(iii)] If $n$ is even and $r$ is even, then 
$$I(\chi)|_{\wt{G}_0}= n_{(r)}\cdot (I(\omega_{0, 0}) \oplus I(\omega_{0, 1}))\bigoplus n_{(r)}\cdot (I(\omega_{e_0, 0}) \oplus I(\omega_{e_0, 1})),$$
where $I(\omega_{0, j})$ is $(K_0, s_0)$-unramified and $I(\omega_{e_0, j})$ is $(K_0', s_0')$-unramified. Here every $I(\omega_{\gamma, j})$ is irreducible. If $I(\chi) =\pi(\phi_\chi, \mbm{1}) \oplus \pi(\phi_\chi, \varepsilon)$ is reducible, then 
$$\begin{aligned}
& \pi(\phi_\chi, \mbm{1})|_{\wt{G}_0} \simeq \pi(\phi_\chi, \varepsilon)|_{\wt{G}_0} \\
\simeq & \  n_{(r)} \cdot I(\omega_{0, 0}) \bigoplus n_{(r)} \cdot I(\omega_{e_0, 0}) \simeq n_{(r)} \cdot I(\omega_{0, 1}) \bigoplus n_{(r)} \cdot I(\omega_{e_0, 1}).
\end{aligned} $$
\end{enumerate}
In all cases above, $\dim \Wh_\psi(\pi(\phi_{\omega_{\gamma, j}}, \rho))$ and $\dim \Wh_{^t\psi}(\pi(\phi_{\omega_{\gamma, j}}, \rho))$ are determined explicitly for every $\rho \in \mca{S}(\phi_{\omega_{\gamma, j}})$.
\end{thm}
The above theorem generalizes part of the results in \cite{Szp5}.

\subsection{Several remarks} Some comments are in order pertaining to the results proved above.

\begin{enumerate}
\item[(i)] In this paper, we mostly concentrate on genuine principal series of $\wt{G}$ or even restrict to the unramified ones. It is desirable to work out for other classes of representations for $\wt{G}$ in view of the conjectural formulas studied in \S \ref{S:func} regarding the multiplicity formula for the restriction. For example, the depth-zero supercuspidal representations for covers studied in \cite{GW} form another test ground for such formulas. For restriction of certain supercuspidal representations from a linear group to its derived subgroup, see \cite{Nev15}; for general discussion and results, see \cite{AP1, Sol20} as well.
\item[(ii)] Even for genuine principal series, we do not yet have a complete understanding of the local coefficients matrices or scattering matrices, which also play important roles in some other recent works such as \cite{BBBF, Kap01}. Indeed, for general character (not necessarily unramified character) in the  tame case, such a local coefficients matrix was already used in the proof of Theorem \ref{T:0-GL2}, and subtleties were observed when comparing it with the unramified local coefficients matrix. Also, for the proof of Conjecture \ref{C:varW} on the $\psi$-Whittaker space of $\pi(\phi_\chi, \rho)$ for varying $\psi$, it seems to be helpful, or perhaps even essential, to have an explicit description of the local coefficients matrix in the tame case. We leave the investigation of tame local coefficients matrix to a future work.
\item[(iii)] To the best of our knowledge, in the literature there has not been any computation of the archimedean local coefficients matrix or scattering matrix for double covers $\wt{G}_\R$ of $\mbf{G}(\R)$ in the real case, except for the metaplectic double cover ${\rm Mp}_{2r}$ of $\Sp_{2r}(\R)$ considered in \cite{Szp2, Szp4}. In general, the $\psi$-Whittaker space of a genuine principal series of $\wt{G}_\R$ may also be of high dimension. Besides ${\rm Mp}_{2r}$, the first examples to consider are the double covers $\wt{\SL}_{3,\R}$ and $\wt{\Spin}_{7,\R}$. It is expected that analogous results regarding the invariants of such a local coefficients matrix hold in this archimedean setting, by employing the methods of partial zeta integrals as developed in \cite{Szp6} and already used in \cite{GSS2}.
\item[(iv)] When the representation is not generic, an analogue of the local coefficients for certain classical groups was studied by Friedberg and Goldberg \cite{FrGo} by considering the generalized Bessel models. The simplest family of genuine representations for covering groups concerns the theta representations $\Theta(\chi)$. If the degree of covering is ``small" relative to the rank of the group, then $\Theta(\chi)$ is not generic. It is interesting to investigate the Bessel models for these theta representations. On the other hand, the leading nilpotent orbit in the Harish-Chandra character expansion governs the generalized or degenerate Whittaker functionals for $\Theta(\chi)$, see \cite{MW1, Var1, Pate}. It is interesting to see further  how such generalized Bessel models and generalized Whittaker functionals are related to each other, in pursuit of the representation-theoretic invariants.
\item[(v)] Especially lacking from the dual side is an interpretation of the high dimension of Whittaker models. Recall that for linear algebraic groups, it was a conjecture of the second-named author \cite{Sha3} that there is a bijection between L-packets and tempered parameters, which now has been proved in many cases, see \cite{Sha5} and references therein. Moreover, it follows from Gross--Reeder \cite{GrRe2} that the adjoint $L$-function of a generic discrete series is necessarily holomorphic at $s=0$. For covering groups, it is expected that every L-packet associated with a tempered parameter also contains a generic representation; however, the converse already fails for theta representations mentioned in (iv) above. The L-parameter for the theta representation $\Theta(\chi)$ is never tempered, although we could have $\dim \Wh_\psi(\Theta(\chi))>0$ if $n\gg r$. We do not yet know an interpretation from the dual side of  the quantity $\dim \Wh_\psi(\Theta(\chi))$, which was shown to be equal to the number of certain Weyl orbits. In fact, we are still very lacking any evidence for a general formula for $\dim \Wh_\psi(\pi), \pi \in \Irrg(\wt{G})$.
\end{enumerate}

It is clear that several speculations and results proved in this paper are motivated from its linear algebraic counterpart, while some others are new only in the covering setting, for example the metaplectic tensor product. We hope that the paper could provide a motivation to a further study on various aspects of the restriction problem for covers and the theory of local coefficients beyond the spectrum of genuine principal series. Again, as in \cite{GSS2}, we strive to work with the most general setup if possible, and assume only the minimal requirement. Hopefully, this could provide some convenience for any further study on the topic.

\subsection{Acknowledgement} We would like to thank Wee Teck Gan, Eyal Kaplan, and Dipendra Prasad for some very helpful communications on certain topics discussed in this paper. The second author was partially supported by NSF grant DMS-1801273.

\section{Restriction and induction of genuine representations} \label{S:RI}

\subsection{Covering groups and L-groups} \label{SS:cov-l}
We follow the exposition in \cite[\S 2]{GSS2} to summarize some important results on the structure of covering groups and the construction of their dual groups and L-groups. For more details, the reader is referred to \cite{BD, Mc1, We2, We3, We6, We7, GG}.

\subsubsection{Covering groups}
Let $\mbf{G}$ be a split connected linear reductive group over a $p$-adic field $F$ with root datum
$$(X, \Phi, \Delta; Y, \Phi^\vee, \Delta^\vee).$$
Here we fix a maximal split torus $\mbf{T} \subset \mbf{G}$, and $X$ (resp. $Y$) is the character (resp. cocharacter) lattice of $\mbf{T}$. Choose a set of simple roots $\Delta$ from the set $\Phi$ of roots. One has the corresponding simple coroots $\Delta^\vee$. Let $\mbf{B}=\mbf{T}\mbf{U}$ be the Borel subgroup associated with $\Delta$. Let 
$$W=N(\mbf{T})/\mbf{T}$$
be the Weyl group of $(\mbf{G}, \mbf{T})$, which we identify with the Weyl group of the coroot system generated by the simply reflections $w_\alpha, \alpha\in \Delta$. We also fix a Chevalley--Steinberg system of pinnings, i.e., a compatible system
$$\set{e_\alpha: \mbf{G}_a \to \mbf{U}_\alpha}_{\alpha \in \Phi},$$
where $\mbf{U}_\alpha \subset \mbf{G}$ is the root subgroup associated with $\alpha$. Denote by $G, B, T$ the $F$-rational points of $\mbf{G}, \mbf{B}, \mbf{T}$ respectively. 

Let 
$$D: Y\times Y \longrightarrow \Z$$
be a bilinear form such that
\begin{equation} \label{E:D-Q}
Q(y)= D(y, y)
\end{equation}
is a Weyl-invariant bilinear form of $Y$. Here $D$ may not be symmetric. If $Q$ is an integral Weyl-invariant bilinear form, then it is known (see \cite{We3}) that there exists $D$ such that  \eqref{E:D-Q} holds. Let 
$$B_Q: Y \times Y \longrightarrow \Z$$ 
be the Weyl-invariant bilinear form given by 
$$B_Q(y, z)= D(y, z) + D(z, y),$$
which actually only depends on $Q$ and thus justifies the notation used here. Furthermore, let
$$\eta: Y^{sc} \longrightarrow F^\times$$
be a homomorphism of the coroot lattice $Y^{sc} \subset Y$ into $F^\times$. Denote by 
$$\eta_n: Y^{sc} \longrightarrow F^\times \longrightarrow F^\times/F^{\times n}$$
 the composite of $\eta$ with the obvious quotient.

Every couple $(D, \eta)$ gives rise to a $\mbf{K}_2$-extension $\wm{G}$ of $\mbf{G}$. Assuming $F^\times$ contains the the full group of $n$-th roots of unity, denoted by $\mu_n$, one has an $n$-fold cover 
$$\begin{tikzcd}
\mu_n \ar[r, hook] & \wt{G} \ar[r, two heads, "p"] & G
\end{tikzcd}$$
obtained from the $n$-th Hilbert symbol 
$$(-,-)_n: \mbf{K}_2(F) \longrightarrow \mu_n.$$
A representation $(\pi, V_\pi)$ of $\wt{G}$ is called genuine if $\mu_n$ acts on $V_\pi$ by a fixed embedding $\mu_n \into \C^\times$. In this paper, we denote by
$$\Irrg(\wt{G})$$
the set of isomorphism classes of irreducible genuine representations of $\wt{G}$.

The covering group $\wt{G}$ splits over unipotent subgroups canonically and $G$-equivariantly, where the action of $G$ on $\wt{G}$ is by conjugation. Namely, if $U \subset G$ is a unipotent subgroup, then there is a unique splitting 
$$s_U: U \longrightarrow \wt{G}$$ satisfying
$$s_U(g u g^{-1}) = \wt{g} s_U(u) \wt{g}^{-1},$$
where the right hand side is independent of the choice of lifting $\wt{g} \in \wt{G}$ of $g \in G$.

 Denote by $\wt{e}_\alpha(F)$ the splitting of $e_\alpha(F)$ in $\wt{G}$. For any $\alpha \in \Phi$ and $x\in F^\times$, define
\begin{equation} \label{E:w}
\wt{w}_\alpha(x):=\wt{e}_\alpha(x) \wt{e}_{-\alpha}(-x^{-1}) \wt{e}_\alpha(x)
\end{equation}
and
\begin{equation} \label{E:h}
\wt{h}_\alpha(x):=\wt{w}_\alpha (x) \wt{w}_{\alpha}(-1).
\end{equation}
For any $\alpha\in \Phi$, for simplicity we write
$$\wt{w}_\alpha:=\wt{w}_\alpha(1).$$
When we consider the case $n=1$ (i.e., $\wt{G}=G$), we will use the notations $e_\alpha, w_\alpha, h_\alpha$ for $\wt{e}_\alpha, \wt{w}_\alpha, \wt{h}_\alpha$ respectively.

The group $\wt{G}$ is generated by the union of the sets $\mu_n$, $\set{\wt{e}_\alpha(F)}_{\alpha\in \Phi}$ and $\set{y(F^\times)}_{y\in Y}$ (see \cite[Theorem 3]{BLS}). Relations among the generators include the following:
\begin{enumerate}
\item[(A)] $\wt{e}_\alpha(x) $ is additive in $x$.
\item[(B)] If $\alpha$ and $\beta$ are roots with $\alpha + \beta \ne 0$, then the commutator
$$[\wt{e}_\alpha(x), \wt{e}_\beta(y)]=\prod \wt{e}_{i\alpha + j\beta}(c_{i,j} x^i y^j),$$
where $i$ and $j$ are positive integers and $c_{i,j}$'s are certain integers.
\item[(B)'] For any $\alpha\in \Phi$ and $x\in F^\times$:
$$\wt{w}_\alpha(x) \wt{e}_\alpha(u) \wt{w}_\alpha(x)^{-1} = \wt{e}_{-\alpha}(-x^2 u).$$
\item[(C)] There exists a section $\s$ of $\wt{T}$ over $T$ such that
$$\s(y_1(a)) \cdot \s(y_2(b)) = \s(y_1(a) \cdot y_2(b)) \cdot (a, b)_n^{D(y_1, y_2)}$$
for any $y_1, y_2\in Y$ and $a, b\in F^\times$. For any $\alpha \in \Delta$ and $x\in F^\times$, one has
$$\wt{h}_\alpha(x)= \s(\alpha^\vee(x)) \cdot (\eta(\alpha^\vee), x)_n^{Q(\alpha^\vee)}.$$
\item[(D)] For $\wt{t} \in \wt{T}$ whose image in $T$ is denoted by $t$, one has
\begin{equation} \label{W-act}
\wt{w}_\alpha \cdot \wt{t}  \cdot \wt{w}_\alpha^{-1} = \wt{t} \cdot \wt{h}_\alpha(\alpha(t)^{-1}).
\end{equation}
\end{enumerate}
By (C) above, the commutator 
$$[-,-]: T \times T \longrightarrow \mu_n$$
 is given by
$$[y_1(a), y_2(b)]=(a, b)_n^{B_Q(y_1, y_2)},$$
where $y_i \in Y$ and $a, b\in F^\times$.   Fix a uniformizer $\varpi \in F$. For any $y\in Y$, we write
$$\varepsilon:=(-1, \varpi)_n \in \mu_n, \quad \s_y:=\s(y(\varpi)) \in \wt{T}.$$
For every $n, r\in \N$, we also denote in this paper
$$n_{(r)} = n/\gcd(n, r).$$
It is clear that $n_{(r)}$ and $r_{(n)}$ are coprime.
\subsubsection{Dual group}
For a cover $(\wm{G}, n)$ associated to $(D, \eta)$, with $Q$ and $B_Q$ arising from $D$, we define
$$
Y_{Q,n}:= Y\cap nY^*,
$$
where $Y^* \subset Y\otimes \Q$ is the dual lattice of $Y$ with respect to $B_Q$; more explicitly,
\begin{equation} \label{YQn}
Y_{Q,n}= \set{y\in Y: B_Q(y, y')\in n\Z \text{ for all } y'\in Y} \subset Y.
\end{equation}
For every $\alpha^\vee\in \Phi^\vee$, denote
$$n_\alpha:= n_{(Q(\alpha^\vee))}= \frac{n}{\text{gcd}(n, Q(\alpha^\vee))}$$
and
$$ \alpha_{Q,n}^\vee=n_\alpha \alpha^\vee, \quad \alpha_{Q,n}=\frac{\alpha}{n_\alpha} .$$
Let
$$Y_{Q,n}^{sc} \subset Y_{Q,n}$$
be the sublattice generated by $\Phi_{Q,n}^\vee=\{\alpha_{Q,n}^\vee: \alpha^\vee \in \Phi^\vee \}$.  Denote $X_{Q,n}=\text{Hom}_\Z(Y_{Q,n}, \Z)$ and $\Phi_{Q,n}=\set{\alpha_{Q,n}: \alpha \in \Phi }$. We also write
$$\Delta_{Q,n}^\vee=\{ \alpha_{Q,n}^\vee: \alpha^\vee \in \Delta^\vee \} \text{ and } \Delta_{Q,n}=\set{\alpha_{Q,n}: \alpha\in \Delta}.$$
Then
$$\big( Y_{Q,n}, \ \Phi_{Q,n}^\vee, \ \Delta_{Q,n}^\vee;\  X_{Q,n},\  \Phi_{Q,n}^\vee, \Delta_{Q,n} \big)$$
forms a root datum with a choice of simple roots $\Delta_{Q,n}$. It gives a unique (up to unique isomorphism) pinned reductive group $\wm{G}^\vee$ over $\Z$, called the dual group of $(\wm{G}, n)$. In particular, $Y_{Q,n}$ is the character lattice for $\wt{G}^\vee$ and $\Delta_{Q,n}^\vee$ the set of simple roots. Let
$$\wt{G}^\vee:=\wm{G}^\vee(\C)$$
be the associated complex dual group. The center of $\wt{G}^\vee$ is
$$Z(\wt{G}^\vee) = \Hom(Y_{Q,n}/Y_{Q,n}^{sc}, \C^\times).$$
Let $\mbf{G}_{Q,n}$ be the pinned split reductive group over $F$ such that
$$\mbf{G}_{Q,n}^\vee \simeq \wm{G}^\vee,$$
where $\mbf{G}_{Q,n}^\vee$ is the Langlands dual group of $\mbf{G}_{Q,n}$. Then $\mbf{G}_{Q,n}$ is the principal endoscopic group of $(\wm{G},n)$, and clearly $\mbf{G}_{Q,n} = \mbf{G}$ if $n=1$.

\subsubsection{L-group} \label{SS:L-g}
In \cite{We3, We6}, Weissman constructed the global L-group as well as the  local L-group extension
 $$\begin{tikzcd}
\wt{G}^\vee_{Q,n} \ar[r, hook] & {}^L\wt{G} \ar[r, two heads] & \WD_{F},
\end{tikzcd}$$
which is compatible with the global L-group. Here $\WD_F$ is the Weil--Deligne group of $F$. His construction of L-group is functorial, and in particular it behaves well with respect to the restriction of $\wm{G}$ to parabolic subgroups. More precisely, let $\mbf{M} \subset \mbf{G}$ be a Levi subgroup. By restriction, one has the $n$-fold cover $\wt{M}$ of $M$. Then the L-groups ${}^L\wt{M}$ and ${}^L\wt{G}$ are compatible, i.e., there are natural homomorphisms of extensions:
$$\begin{tikzcd}
\wt{G}^\vee_{Q,n} \ar[r, hook] & {}^L\wt{G} \ar[r, two heads] & \WD_{F} \\
\wt{M}^\vee_{Q,n} \ar[r, hook] \ar[u, hook] & {}^L\wt{M} \ar[u, hook] \ar[r, two heads] & \WD_{F}  \ar[u, equal] .
\end{tikzcd}$$
For details on the construction and some properties regarding the L-group, we refer the reader to \cite{We3, We6, GG}.

If $\eta_n=\mbm{1}$, then there exists a so-called distinguished genuine character 
$$\chi_\psi: Z(\wt{T}) \longrightarrow \C^\times,$$
depending on a nontrivial additive character $\psi$ of $F$, such that the following properties hold:
\begin{enumerate}
\item[$\bullet$] the character $\chi_\psi$ takes values in $\mu_4\subseteq \C^\times$ and is Weyl-invariant, i.e., $\chi_\psi(w^{-1} \cdot \wt{t} \cdot w) = \chi_\psi(\wt{t})$ for all $\wt{t} \in Z(\wt{T})$ and $w$;
\item[$\bullet$] $\chi_\psi$ gives rise to a splitting of ${}^L\wt{G}$ over $W_F$, with respect to which one has an isomorphism 
$${}^L\wt{G} \simeq_{\chi_\psi} \wt{G}^\vee \times \WD_F.$$
\end{enumerate}

\subsection{The Clifford--Mackey theory}
We recall some results by Gelbart and Knapp \cite{GeKn1, GeKn2} (see also \cite{Tad92}) on the Clifford--Mackey theory of restriction and induction of representations in the setting of $l$-adic groups (see \cite{BZ1}). In particular, it applies to the $\mu_n$-cover of a linear algebraic group $G$ discussed in \S \ref{SS:cov-l}. These results could be considered as a proper generalization of the finite group situation, as discussed in \cite[\S 11]{CuRe1}.

Let $M$ be a totally disconnected group, i.e., a separable locally compact topological group whose open compact subgroups form a neighborhood base at the identity.  For an admissible representation $\pi$ and $e \in \N$, we will use
$$e\cdot \pi = \pi^{\oplus e}$$
interchangeably to denote the isotypic sum of $\pi$ with multiplicity $e$. Let 
$$H \subset M$$ 
be an open normal subgroup such that 
$$M/H$$ is a finite abelian group. One of two operations we consider in this paper is the restriction of representations from $M$ to $H$, while the other one is the induction from $H$ to $M$. For every admissible representation $(\rho, V_\rho)$ of $H$ and $g \in M$, there is the representation 
$${}^g \rho(h):= \rho(g^{-1} h g)$$ 
of $H$ afforded by the same vector space $V_\rho$ but  with the action twisted.  We denote
$$S_\rho:=\Stab(\rho, M)=\set{g\in M:  {}^g\rho \simeq \rho},$$
whenever the ambient group $M$ is clear from the context. Denote by 
$$\omega_\rho:  Z(H) \longrightarrow \C^\times$$
the central character of $\rho$. Also, for every subgroup $K \subset M$ and $g\in K$, we define a map
$$\varphi_{K,g}: K \longrightarrow K$$
by 
$$\varphi_g(x):=[x, g]_K,$$
where  $[g, g']_K:=g g' g^{-1} g'^{-1}$  denotes the commutator of $K$.

\begin{dfn} \label{D:conc}
For $H \subset K \subset M$, a representation $\sigma \in \Irr(K)$ is called $H$-concentrated if
$$\set{ g\in K:  \varphi_{K,g}^{-1}\big(Z(K) - \Ker(\omega_\sigma|_{Z(K)}) \big)= \emptyset } \subseteq H.$$
\end{dfn}

For every $\sigma \in \Irr(K)$, denote by $\Theta_\sigma$ the Harish-Chandra character distribution of $\sigma$, which is conjugation invariant and can be represented by a locally-integrable function. The above definition is motivated from the following:
\begin{lm} \label{L:supp}
Assume that $\sigma \in \Irr(K)$ is $H$-concentrated. Then
$${\rm supp}(\Theta_\sigma) \subset H.$$
\end{lm}
\begin{proof}
Let $g \in K - H$. Since $\sigma$ is $H$-concentrated by assumption, there exists $x\in K$ such that 
$$[x, g] \in Z(K) - \Ker(\omega_\sigma|_{Z(K)}).$$
One has
$$\Theta_\sigma(g) = \Theta_\sigma([x, g]\cdot g)= \omega_\sigma([x, g]) \cdot \Theta_\sigma(g),$$
which shows that $\Theta_\sigma(g)=0$ since $\omega_\sigma([x, g]) \ne 1$. Thus, ${\rm supp}(\Theta_\sigma) \subset H$.
\end{proof}

\subsubsection{Restriction} \label{SSS:res}
Let $\pi \in \Irr(M)$ be an irreducible admissible representation of $M$. One has (see \cite[Lemma 2.1]{GeKn1})
$$\pi|_{H} = \bigoplus_{g_i \in M/S_\rho} ({}^{g_i} \rho)^{\oplus e},$$
where
\begin{enumerate}
\item[--] $\rho$ is an irreducible representation of $H$ that occurs in $\pi|_{H}$;
\item[--] ${}^{g} \rho$ is the twist of $\rho$ by $g\in M$, and elements in $\set{{}^{g_i}\rho: g_i \in M/S_\rho}$ are mutually non-isomorphic representations of $H$;
\item[--] $e \in \N$ is a certain ramification index (i.e., multiplicity) of ${}^g \rho$, and in particular $\rho^{\oplus e}$ is an isotypic component of $\rho$ in $\pi|_{H}$.
\end{enumerate}
One has
$$S_\rho = S_{\rho^{\oplus e}},$$
and there is a natural irreducible representation of $S_\rho$ on $\rho^{\oplus e}$. Moreover,
$$\pi = \Ind_{S_\rho}^{M} (\rho^{\oplus e}).$$
Note that the irreducible representation of $S_\rho$ on $\rho^{\oplus e}$ depends on the $\pi$ one started with. Regarding the restriction, there are two extreme cases:
\begin{enumerate}
\item[--] if $S_\rho =M$, then $\pi|_{H}$ is an isotypic sum of $\rho$, i.e., $\pi|_{H} = \rho^{\oplus e}$;
\item[--] if $S_\rho = H$, then $e=1$ and in this case $\pi = \Ind_{H}^{M} (\rho)$.
\end{enumerate}

\subsubsection{Induction} Let $(\rho, V_\rho)$ be an irreducible admissible representation of $H$. The induced representation of $\rho$ to $M$ might not be irreducible. To obtain an irreducible representation of $M$ from $\rho$, one approach is as follows.

\begin{lm} \label{L:001}
One can extend $\rho$ to a representation $\rho^\flat$ of a subgroup $H^\flat \subset M$ (still acting on the same vector space $V_\rho$) such that $S_{\rho^\flat}= H^\flat$. In this case, $\Ind_{H^\flat}^{M} (\rho^\flat)$ is irreducible.
\end{lm}
\begin{proof}
The proof is the same as in \cite[Page 89-90]{Mez04} almost word for word, by noting that $M/H$ is an abelian group of finite order by assumption.
\end{proof}
Using notations from the above lemma, one has
\begin{equation} \label{E:ext}
\Ind_{H}^{H^\flat} (\rho) = \bigoplus_{\chi \in \Hom(H^\flat/H, \C^\times)} \chi\otimes \rho^\flat,
\end{equation}
where $$\set{\chi\otimes \rho^\flat: \chi \in \Hom(H^\flat/H, \C^\times)}$$ 
gives all the possible extensions of $\rho$ to $H^\flat$. For a general $\chi \in \Hom(H^\flat/H, \C^\times)$, one can extend it to a character 
$$\chi': M/H \longrightarrow \C^\times.$$
Then one has
$$\Ind_{H^\flat}^M(\chi\otimes \rho^\flat) \simeq \chi' \cdot \Ind_{H^\flat}^M(\rho^\flat),$$
which is an irreducible representation of $\wt{M}$. Therefore, we have a decomposition
$$\Ind_H^M(\rho) = \bigoplus_{\chi \in \Hom(H^\flat/H, \C^\times)} \Ind_{H^\flat}^M (\chi\otimes \rho^\flat)$$
of $\Ind_H^M(\rho)$ into irreducible representations of $M$.

\begin{prop} \label{P:con-con}
Assume that every irreducible representation of $M$ is $H$-concentrated.  Let $\rho \in \Irr(H)$. Then the representation 
$$\Ind_H^M(\rho)= \pi^{\oplus \val{H^\flat/H}}$$
is an isotypic sum with $\pi:=\Ind_{H^\flat}^H(\rho^\flat) \in \Irr(M)$ given as in Lemma \ref{L:001}.
\end{prop}
\begin{proof} Again, the argument is essentially that of \cite[Lemma 4.2]{Mez04} which deals with covers of $\GL_r$; the desired result in that context is given explicitly in \cite[Proposition 4.6]{Tak3}. We sketch the argument for completeness.

Denote $\sigma:=\Ind_{H^\flat}^M (\chi \otimes \rho^\flat)$ for an arbitrary $\chi\in \Hom(H^\flat/H, \C^\times)$, we want to show $\Theta_\sigma = \Theta_\pi$. First, applying Lemma \ref{L:supp} to $K=M$ shows that it suffices to prove $\Theta_\sigma(x) = \Theta_\pi(x)$ for $x\in H$. Since 
$$\Theta_\sigma(x) = \sum_{\gamma \in M/H^\flat} (\chi \cdot \Theta_{\rho^\flat})(\gamma^{-1} x \gamma),$$
it gives that for $x\in H$,
$$\Theta_{\sigma}(x) = \sum_{\gamma \in M/H^\flat} \Theta_{\rho^\flat}(\gamma^{-1} x \gamma) = \Theta_\pi(x).$$
This shows $\Theta_\sigma = \Theta_\pi$ and thus
$$\sigma \simeq \pi,$$ 
see \cite[Corollary 2.20]{BZ1}. Since $\chi$ is arbitrary, the proof is completed.
\end{proof}

If every irreducible representation of $M$ is $H$-concentrated, then for given $\rho \in \Irr(H)$ we denote by 
$$\pi[\rho] \in \Irr(M)$$
the unique irreducible representation appearing in $\Ind_H^M(\rho)$, as in Proposition \ref{P:con-con}. Since $\pi[\rho] \in \Irr(M)$ is irreducible, by the discussion in \S \ref{SSS:res} and the Frobenius reciprocity, there exists $\tilde{\rho} \in \Irr(S_\rho)$
such that 
\begin{equation} \label{E:2con}
\pi[\rho] = \Ind_{S_\rho}^M (\tilde{\rho}),
\end{equation}
where $\tilde{\rho}|_{H} = \rho^{\oplus e}$ for some $e\in \N$. The relations among the two constructions are depicted in the following diagram
\begin{equation} \label{D:2ind}
\begin{tikzcd}
& (M, \pi[\rho]) \\
(S_\rho, \tilde{\rho}) \ar[ru] & & (H^\flat, \rho^\flat) \ar[lu] \\
& (H, \rho). \ar[lu] \ar[ru]
\end{tikzcd}
\end{equation}
Note that here $\tilde{\rho}$ arises from the restriction of the irreducible representation $\pi[\rho]$. The left two arrows follow from the Clifford--Mackey theory. On the other hand, the right two arrows in \eqref{D:2ind} together give an analogue of the Stone--von Neumann construction of irreducible representations of Heisenberg groups, as illustrated by Proposition \ref{P:con-con}.

\begin{eg}
Consider the covering torus 
$$\begin{tikzcd}
\mu_n \ar[r, hook] & \wt{T} \ar[r, two heads] & T,
\end{tikzcd}$$
which is a Heisenberg type group. The center $Z(\wt{T})$ is of finite index in $T$. Consider the pair $(M, H)=(\wt{T}, Z(\wt{T}))$. Every irreducible representation of $\wt{T}$ is $Z(\wt{T})$-concentrated. One has the diagram \eqref{D:2ind}. To describe the right arrows of \eqref{D:2ind}, let $\chi: Z(\wt{T}) \longrightarrow \C^\times$ be a genuine central character. We choose a maximal abelian subgroup $\wt{A} \subset \wt{T}$ and an extension $\tilde{\chi}: \wt{A} \longrightarrow \C^\times$ of $\chi$. Then $\pi[\chi] = \Ind_{\wt{A}}^{\wt{T}} (\tchi)$ is irreducible and independent of the choice of $\wt{A}$ and $\tchi$; also, one has
$$\dim \pi[\chi] = \sqrt{[\wt{T}: Z(\wt{T})]},$$
see \cite{We1, We5}. For the left arrows in \eqref{D:2ind}, we have
$$S_\chi = \wt{T}$$
and thus $\pi[\chi]|_{Z(\wt{T})} = \chi^{\oplus e}$, where $e=\sqrt{[\wt{T}: Z(\wt{T})]}$. We also use the notation
$$i(\chi):= \pi[\chi]$$
throughout the paper.
\end{eg}

\subsection{Two special pictures} \label{SS:2pic}
Henceforth, we consider covers arising from the Brylinski--Deligne framework exclusively. Recall that $T=\Hom(X, F^\times) = Y\otimes F^\times$. 

\begin{dfn} \label{D:assL}
A subgroup $\wt{H} \subset \wt{T}$ is said to be associated with a sublattice $\mbm{L} \subset Y$ if 
$$\wt{H} = p^{-1}({\rm Im}(\iota)),$$
where $\iota: \mbm{L} \otimes F^\times \to T$ is the map induced from the inclusion $\mbm{L} \subset Y$. 
\end{dfn}

Write $\mbm{L}(\wt{H})$ for the unique sublattice of $Y$ associated with $\wt{H}$, whenever it exists. Assume that $Z(\mbf{G})$ is connected i.e., $X/X^{sc}$ is a free $\Z$-module, or equivalently, the exact sequence
$$\begin{tikzcd}
X^{sc} \ar[r, hook] & X \ar[r, two heads] & X/X^{sc}
\end{tikzcd}$$
of $\Z$-modules split. We thus have a $W$-equivariant embedding
$$\begin{tikzcd}
\Hom(X/X^{sc}, \Z)  \ar[r, hook] & \Hom(X, \Z) \simeq Y,
\end{tikzcd} $$
the image of which we denote by $Y_c \subset Y$. In particular, $\wt{Z(G)}$ is associated with the sublattice  $Y_c \subset Y$, see Definition \ref{D:assL}. Let
$$i_{Q,n}:  Y_{Q,n}\otimes F^\times \longrightarrow T= Y\otimes F^\times$$
be the isogeny induced from the inclusion $Y_{Q,n} \into Y$. It is known that  (see \cite{We1})
$$Z(\wt{T}) = p^{-1}({\rm Im}(i_{Q,n}));$$
that is, $Y_{Q,n}= \mbm{L}(Z(\wt{T}))$.
Now we define
$$Y_z := Y_{Q,n} \cap Y_c$$
and let
$$i_{Q,n}^z: Y_z \otimes F^\times \longrightarrow T$$
be the map induced from the inclusion $Y_z \subset Y$.

The following has been observed and used in many places in the literature (for example see \cite{KP, CO, Gan17})
\begin{lm}  \label{L:Gz}
Keeping notations as above, we have $Z(\wt{G}) = \wt{Z(G)} \cap Z(\wt{T})$ and 
$$Z(\wt{G}) =p^{-1}({\rm Im}(i_{Q,n}^z)),$$
that is, $Y_z = \mbm{L}(Z(\wt{G}))$.
\end{lm}
\begin{proof}
It suffices to prove the first equality. The inclusion 
$$Z(\wt{G}) \subset \wt{Z(G)} \cap Z(\wt{T})$$ is clear. On the other hand, since $\wt{G} = \wt{T} \cdot \wt{G}_\de$, we see that if $g\in \wt{Z(G)}  \cap Z(\wt{T})$, then it centralizes $\wt{G}_\de$ as well as $\wt{T}$. Therefore $\wt{Z(G)} \cap Z(\wt{T}) \subset Z(\wt{G})$.
\end{proof}

We want to give a (hopefully proper) subgroup $\wt{H} \subset \wt{G}$ such that every $\pi\in \Irrg(\wt{G})$ is $\wt{H}$-concentrated, and therefore
$${\rm supp}(\Theta_\pi) \subset \wt{H}.$$
For this purpose, we assume that
\begin{enumerate}
\item[--] there exists $e_s \in Y$ such that $\wt{G} = \wt{G}_\de \rtimes e_s(F^\times)$;
\item[--] one has $Y_c = \Z e_c$ for some $e_c \in Y$.
\end{enumerate}
We denote 
$$n_s = n/\gcd(n, B_Q(e_c, e_s)).$$

\begin{prop} \label{P:Gns}
Keep notations as above and denote $\wt{G}_{n_s}:= \wt{G}_\de \rtimes \wt{e_s(F^{\times n_s})}$. Then every  irreducible genuine representation $\pi$ of $\wt{G}$ is $\wt{G}_{n_s}$-concentrated, and thus ${\rm supp}(\Theta_\pi) \subset \wt{G}_{n_s}$. Consequently,
$$\Ind_{\wt{G}_{n_s}}^{\wt{G}}(\sigma) = m\cdot \pi$$
is an isotypic sum for every $\sigma \in \Irrg(\wt{G}_{n_s})$.
\end{prop}
\begin{proof} 
Given with $e_s(b) \in e_s(F^\times ) \subset G$, we claim that if $[e_c(a), e_s(b)] = 1$ for all $a\in F^\times$, then $b \in F^{\times n_s}$. Noting that
$$[e_c(a), e_s(b)] = (a, b)_n^{B(e_c, e_s)},$$
the claim follows from the non-degeneracy of the Hilbert symbol. Therefore, if $b\notin F^{\times n_s}$, then there exists $e_c(a)$ such that 
$$1\ne [e_c(a), e_s(b)] \in \mu_n.$$
Thus, if $\wt{g} \notin \wt{G}_\de \rtimes \wt{e_s(F^{\times n_s})}$, then we have $\wt{g} = \wt{g}_0 \cdot \wt{e_s(b)}$ with $\wt{g}_0\in \wt{G}_\der$ and $b \notin F^{\times n_s}$. It then gives
$$[\wt{g}, \wt{e_c(a)}] = [\wt{g}_0 \cdot \wt{e_s(b)}, \wt{e_c(a)}] = [e_c(a), e_s(b)] \in \varphi_{\wt{G}, \wt{g}}^{-1}\left( Z(\wt{G}) - \Ker(\omega_\pi|_{Z(\wt{G})})\right),$$
and this shows that $\pi$ is $\wt{G}_{n_s}$-concentrated. The rest is clear in view of Proposition \ref{P:con-con}.
\end{proof}

\begin{rmk} \label{R:suppGL}
Proposition \ref{P:Gns} certainly does not give the sharpest bound for the support of all $\Theta_\pi$. For example, for Kazhdan--Patterson covers of $G= \GL_r$, it is shown in \cite[Proposition 0.1.4]{KP} that 
$${\rm supp}(\Theta_\pi) \subset Z(\wt{G}) \cdot (\wt{G})^n$$
 for every $\pi \in \Irrg(\wt{G})$. This fact plays a crucial role in the metaplectic tensor product construction by Mezo \cite{Mez04}, Takeda \cite{Tak3, Tak4} and the functorial interpretation by Gan \cite{Gan17}, when blocks in a covering Levi subgroup do not commute. In \S  \ref{SS:mtp-GSp}, we will discuss an analogue of the metaplectic tensor product (and in fact also a metaplectic restriction) for odd fold cover of $\GSp_{2r}$, relying on Proposition \ref{P:Gns}.
\end{rmk}

In this paper, we will consider restrictions and inductions from two types of normal subgroups of $\wt{G}$. Each of these two types has some advantage and we place focus on one of them in different contexts regarding different problems. To elaborate, in the first case we let 
$$\wt{H} \subset \wt{G}$$ be a normal subgroup such that $\wt{G}/\wt{H}$ is a finite abelian group, and that every genuine representation of $\wt{G}$ is $\wt{H}$-concentrated. In the second case, we consider the subgroup 
$$\wt{Z}_0 \cdot \wt{G}_0 \subset \wt{G},$$
which is normal of finite index, satisfying $\wt{Z}_0 \subset \wt{Z(G)}, \wt{G}_\de \subset \wt{G}_0$ and that $\wt{Z}_0$ commutes with $\wt{G}_0$, as depicted from the following diagram
$$\begin{tikzcd}
\wt{Z(G)}   \ar[rd, no head, dashed] & \wt{G}_0 \\
\wt{Z}_0 \ar[u, hook] \ar[ru, no head, dashed]  & \wt{G}_\de \ar[u, hook].
\end{tikzcd}$$

We have
$$\begin{tikzcd}
\wt{H} \subset \wt{G} \supset \wt{Z}_0\cdot \wt{G}_0,
\end{tikzcd}$$
and in general none of the two groups $\wt{H}$ and $\wt{Z}_0\cdot \wt{G}_0$ is contained in the other. These two types of groups play different roles in our paper as follows.
\begin{enumerate}
\item[--] (First picture) Most often, $\wt{H}$ is not equal to the $F$-rational points of an algebraic subgroup of $\mbf{G}$, and thus the construction of dual group and L-group of $\wt{H}$ does not apply directly. However, on the representation side, since $\wt{H}$ contains the support of an arbitrary irreducible genuine representation of $\wt{G}$, it follows from Proposition \ref{P:con-con} that the representation $\Ind_{\wt{H}}^{\wt{G}}(\rho)$ is an isotypic sum of a certain $\pi[\rho] \in \Irrg(\wt{G})$. In this case, the induction functor from $\wt{H}$ to $\wt{G}$ is simple. As mentioned in Remark \ref{R:suppGL}, some application of this includes the ``metaplectic tensor product" construction of representations of the covering Levi subgroups of $\wt{\GL}_r$ as in \cite{Mez04, Tak3, Tak4, Cai1, Gan17}.
\item[--] (Second picture) The group $\wt{Z}_0 \cdot \wt{G}_0$ is almost a direct product of the commuting pair $\wt{Z}_0$ and $\wt{G}_0$, as the intersection is a small finite group. For example, the two groups $\wt{Z(G)}$ and $\wt{G}_\de$ form a commuting pair inside $\wt{G}$. We consider both the restriction and induction of representations between $\Irrg(\wt{G})$ and $\Irrg(\wt{Z}_0 \cdot \wt{G}_0)$. One important aspect of the restriction of representations from $\wt{G}$ is  how a local coefficients matrix for $\wt{G}$ behaves with respect to the restriction to the subgroup $\wt{G}_0$. For the groups 
$$\wt{G}=\wt{\GL}_2,\ \wt{G}_0=\wt{G}_\de \text{ and }\wt{Z}_0 = \wt{Z(G)},$$
we carried out an extensive study in \cite{GSS2}. For the induction, the representation $\Ind_{\wt{Z(G)}\cdot \wt{G}_\de}^{\wt{G}}(\rho)$ is not an isotypic sum in general; in fact, it is useful already to determine when it is irreducible. For double covers of $\GSp_{2r}$, this was investigated in \cite{Szp4-1, Szp5}.
\end{enumerate}

\subsubsection{An irreducibility criterion} \label{SSS:2pic}
Let $\wt{G}$ be an $n$-fold cover of $G$. We continue to assume that $\wt{Z(G)}$ is associated with a sublattice $Y_{G,c} = \Z e_c \subset Y$. Henceforth, we will denote 
$$\wt{G}_0:=\wt{G}_\de,$$
whenever there is no possibility of confusion. Let $\wt{M} \subset \wt{G}$ be a covering Levi subgroup. Define
$$\wt{M}_0:= \wt{M} \cap \wt{G}_\de,$$
which is then a Levi subgroup of $\wt{G}_\de$. We have
$$\wt{M}_\de \subset \wt{M}_0 \subset \wt{M},$$
where the first inclusion however may not be an equality in general. We consider
$$\wt{M}^\dag:= \wt{Z(G)} \cdot \wt{M}_0.$$
There are inclusions of normal subgroups
$$\wt{M}^\dag \subset \wt{Z(M)} \cdot \wt{M}_0 \subset \wt{M}$$
with $\wt{M}/\wt{M}^\dag$ a finite abelian group. The following result is straightforward.

\begin{lm}
For every Levi subgroup $M\subset G$, one has $\wt{M}/\wt{M}^\dag \simeq \wt{T}/\wt{T}^\dag$; in particular, the quotient is independent of the Levi subgroup.
\end{lm}

From now on, we write
$$\mfr{Q}^\dag:=\wt{T}/\wt{T}^\dag.$$
We are interested in both the induction and restriction of representations between $\wt{M}^\dag$ and $\wt{M}$. It is natural to expect that any answer will depend on the specific representation chosen. Nonetheless, for special covering groups, it is possible to obtain a uniform description of the induction and restriction process, to describe the irreducbility of an induction and the decomposition from a restriction. We briefly explain the main idea following \cite{Szp4-1}.

The two groups $\wt{Z(G)}$ and $\wt{M}_0$ have a finite intersection $\wt{Z(G)} \cap \wt{M}_0$ which lies in the center of both groups. Thus, $\wt{M}^\dag$ is almost a direct product of $\wt{Z(G)}$ and $\wt{M}_0$. Every representation of $\wt{M}^\dag$ is of the form
$$\tau\boxtimes \rho,$$
where $\tau \in \Irrg(\wt{Z(G)})$ and $\rho \in \Irrg(\wt{M}_0)$ are compatible, i.e., they agree on the intersection $\wt{Z(G)} \cap \wt{M}_0$. Thus, 
$$\tau\boxtimes \rho \simeq \tau'\boxtimes \rho'$$
if and only if 
$$\tau \simeq \tau'\text{ and }\rho \simeq \rho'.$$
Denote by $\omega_\tau$ the central character of $\tau$, and similarly $\omega_\rho$ for $\rho$. Then we can use the central characters $\omega_\tau$ or $\omega_\rho$ to distinguish representations. More precisely, if $\omega_\tau \ne \omega_{\tau'}$ or $\omega_\rho \ne \omega_{\rho'}$, then $\tau\boxtimes \rho$ and $\tau'\boxtimes \rho'$ are not isomorphic. In fact, since $\wt{Z(G)}$ is a Heisenberg-type group, one has
$$\tau \simeq \tau' \text{ if and only if } \omega_\tau = \omega_{\tau'}.$$

We denote
$$\mfr{Q}_M := Z(\wt{M}^\dag)/Z(\wt{M})$$ 
and
$$\mfr{Q}_Z:=Z(\wt{Z(G)})/Z(\wt{G}), \quad \mfr{Q}_{M_0}:=Z(\wt{M}_0)/(\wt{M}_0 \cap Z(\wt{M})).$$
Note that here $\wt{M}_0 \cap Z(\wt{M}) = Z(\wt{M}_0) \cap Z(\wt{M})$. The natural surjective map
\begin{equation*}
\begin{tikzcd}
\mfr{Q}_Z \times \mfr{Q}_{M_0} \ar[r, two heads] & \mfr{Q}_M
\end{tikzcd}
\end{equation*}
has a finite kernel. Let
$$\widehat{\mfr{Q}_M}:= \Hom(\mfr{Q}_M, \C^\times), \ \widehat{\mfr{Q}_Z}:=\Hom(\mfr{Q}_Z, \C^\times) \text{ and }  \widehat{\mfr{Q}_{M_0}}:=\Hom(\mfr{Q}_{M_0}, \C^\times)$$
be the Pontryagin duals. One has an embedding
\begin{equation}  \label{E:hatQ}
\begin{tikzcd}
\widehat{\mfr{Q}_M} \ar[r, hook] & \widehat{\mfr{Q}_Z} \times \widehat{\mfr{Q}_{M_0}}.
\end{tikzcd}
\end{equation}
There is a natural conjugation action of $\wt{M}$ on $\mfr{Q}_Z, \mfr{Q}_{M_0}$ and $\mfr{Q}_M$, and thus also on their Pontryagin duals $\widehat{\mfr{Q}_Z}, \widehat{\mfr{Q}_{M_0}}$ and $\widehat{\mfr{Q}_M}$ respectively. Since $\wt{M}^\dag$ acts trivially on all $\mfr{Q}_Z, \mfr{Q}_{M_0}$ and $\mfr{Q}_M$, we have a well-defined action of $\mfr{Q}^\dag$ on the three Pontryagin duals.
The embedding in \eqref{E:hatQ} is $\mfr{Q}^\dag$-equivariant with respect to this natural action.

It is clear that the action of $\mfr{Q}^\dag$ on $\widehat{\mfr{Q}_Z}$ (resp. $\widehat{\mfr{Q}_{M_0}}$) is trivial if and only if $Z(\wt{Z(G)}) \subset Z(\wt{T})$ (resp. $Z(\wt{M}_0) \subset Z(\wt{T})$). If $Z(\wt{M}_0)$ is associated with a sublattice $Y_{M_0,c} \subset Y$, then it follows from Lemma \ref{L:Gz} that these two inclusions are equivalent to 
$$Y_{G,c,Q,n} \subset Y_{Q,n}$$
and 
$$(Y_{M_0,c} \cap Y_{0,Q,n}) \subset Y_{Q,n}$$
respectively. Here $Y_{G,c,Q,n} \subset Y_{G,c}$ is the sublattice given by \eqref{YQn} with respect to the restriction of $B_Q$ on $Y_{G,c}$, and $Y_0=Y^{sc}$ is the cocharacter lattice of $G_0=G_\de$.

\begin{thm} \label{T:in-re}
Keep notations as above. 
\begin{enumerate}
\item[(i)] Suppose the action of $\mfr{Q}^\dag$ on either $\widehat{\mfr{Q}_Z}$ or $\widehat{\mfr{Q}_{M_0}}$ is free. Then $\Ind_{\wt{M}^\dag}^{\wt{M}} (\tau \boxtimes \rho)$ is irreducible for every $\tau\boxtimes \rho \in \Irrg(\wt{M}^\dag)$.
\item[(ii)] Specializing to the case $M=T$, if the action of $\mfr{Q}^\dag$ on both $\widehat{\mfr{Q}_Z}$ and $\widehat{\mfr{Q}_T}$ are trivial, then for every $\pi \in \Irrg(\wt{T})$ we have 
$$\pi|_{\wt{T}^\dag} = (\tau\boxtimes \rho)^{\oplus e}$$ for some
$e\in \N$. On the other hand, if $Z(\wt{T}) \subset Z(\wt{Z(G)}) \cdot Z(\wt{T}_0)$, then $\Ind_{\wt{T}^\dag}^{\wt{T}}(\tau \boxtimes \rho)$ is an isotypic sum of an irreducible representation of $\wt{T}$.
\end{enumerate}
\end{thm}
\begin{proof}
For (i), in view of \eqref{E:hatQ}, the assumption implies that the action of $\mfr{Q}^\dag$ on $\widehat{\mfr{Q}_M}$ is free. Thus, the assertion is a direct consequence of Mackey's irreducibility criterion for induced representations. 

For the first assertion in (ii), let $(\tau \boxtimes \rho) \in \Irrg(\wt{T}^\dag)$ be a constituent of $\pi|_{\wt{T}^\dag}$. The assumption implies that 
$$S_{\tau\boxtimes \rho} = \wt{T}$$
 and thus by the discussion in \S \ref{SSS:res} the restriction $\pi|_{\wt{T}^\dag}$ is an isotypic sum. We give an alternate argument as follows. The assumption implies that 
\begin{equation} \label{Inc-1}
Z(\wt{Z(G)}) \cdot Z(\wt{T}_0) \subset Z(\wt{T}).
\end{equation}
The representation $\pi$ is determined by its central character $\omega_\pi$. Every constituent $\tau\boxtimes \rho \subset \pi|_{\wt{T}^\dag}$ is also determined by $\omega_\tau$ and $\omega_\rho$. However, the inclusion  \eqref{Inc-1} implies that $\omega_\tau$ and $\omega_\rho$ are uniquely determined by $\omega_\pi$. That is, $\pi|_{\wt{T}^\dag}$ is an isotypic sum in this case.

For the second assertion in (ii), it is exactly the same argument as the above with the inclusion in \eqref{Inc-1} reversed. Indeed, every constituent $\pi$ of $\Ind_{\wt{T}^\dag}^{\wt{T}}(\tau \boxtimes \rho)$ is determined by its central character $\omega_\pi$. Since $Z(\wt{T}) \subset Z(\wt{Z(G)}) \cdot Z(\wt{T}_0)$, we have
$$\omega_\pi = (\omega_\tau \cdot \omega_\rho)|_{Z(\wt{T})},$$
which is uniquely determined by $\tau\boxtimes \rho$. Thus, in this case $\Ind_{\wt{T}^\dag}^{\wt{T}}(\tau \boxtimes \rho) = \pi^{\oplus e}$ for some $e\in \N$.
\end{proof}


\begin{cor} \label{C:Zinc}
The following three statements are equivalent:
\begin{enumerate}
\item[--] $\mfr{Q}^\dag$ acts trivially on $\widehat{\mfr{Q}_{T_0}}$;
\item[--] the inclusion $Z(\wt{T}_0) \subset Z(\wt{T})$ holds;
\item[--] the equality $Y_0 \cap Y_{Q,n} = Y_{0, Q,n}$ holds.
\end{enumerate}
If one (and thus every) of the above holds, then for every $\pi\in \Irrg(\wt{T})$ the restriction $\pi|_{\wt{T}_0} = \rho^{\oplus e}$ is an isotypic sum of a certain $\rho \in \Irrg(\wt{T}_0)$.
\end{cor}
\begin{proof}
The equivalence between the first two assertions follows from the definition of $\mfr{Q}^\dag$ and $\widehat{\mfr{Q}_{T_0}}$.
Note $Y_{Q,n} = \mbm{L}(Z(\wt{T}))$ and $Y_{0, Q,n} =\mbm{L}(Z(\wt{T}_0))$. One always has $Y_0 \cap Y_{Q,n} \subset Y_{0, Q,n}$, which corresponds to the inclusion
$$\wt{T}_0 \cap Z(\wt{T}) \subset Z(\wt{T}_0).$$
This gives the equivalence between the second and the third assertions. The rest follows from the same argument as in Theorem \ref{T:in-re} by noting that every irreducible genuine representation of $\wt{T}$ or $\wt{T}_0$ is determined by its central character.
\end{proof}

Now we consider two Levi subgroups $M \subset M' \subset G$. It is easy to check that there are natural maps
$$\begin{tikzcd}
Z(\wt{M}_0')/Z(\wt{M}_0') \cap Z(\wt{M}) \ar[r, hook] & Z(\wt{M}_0)/Z(\wt{M}_0) \cap Z(\wt{M}) = \mfr{Q}_{M_0} \\
\mfr{Q}_{M_0'}=Z(\wt{M}_0')/Z(\wt{M}_0') \cap Z(\wt{M}') \ar[u, two heads].
\end{tikzcd}$$
Thus, we have a $\mfr{Q}^\dag$-equivaraint map
$$\widehat{\mfr{Q}_{M_0}} \longrightarrow \widehat{\mfr{Q}_{M_0'}}.$$
This immediately gives

\begin{cor} \label{L:her}
For a Levi subgroup $\wt{M}' \subset \wt{G}$, if $\mfr{Q}^\dag$ acts freely on $\widehat{\mfr{Q}_{M_0'}}$, then:
\begin{enumerate}
\item[--] $\mfr{Q}^\dag$ acts freely on $\widehat{\mfr{Q}_{M_0}}$ for every Levi subgroup $\wt{M} \subset \wt{M}' \subset \wt{G}$; 
\item[--] hence, $\Ind_{\wt{M}^\dag}^{\wt{M}} (\tau \boxtimes \rho)$ is always irreducible.
\end{enumerate}
\end{cor}

\subsubsection{Reducibility of parabolic inductions}
Let $(\sigma, V_\sigma) \in \Irrg(\wt{M})$, where $V_\sigma$ is a space of realization of $\sigma$.

Let $\mbf{P}=\mbf{M}\mbf{N} \subset \mbf{G}$ be a parabolic subgroup associated with $\theta \subset \Delta$. One has the parabolic subgroup $\wt{P} \subset \wt{G}$ by restriction from $G$ to $P$. Consider the  normalized induced representation
$$I_{\wt{P}}^{\wt{G}} (\sigma):=\Ind_{\wt{P}}^{\wt{G}}(\sigma)$$
of $\wt{G}$. Let $\mbf{P}'= \mbf{M}' \mbf{N}'$ be another parabolic subgroup corresponding to $\theta'\subset \Delta$.
Define
$$W^{\theta, \theta'}=\set{ w \in W:  w(\theta) =\theta' }.$$
Call $\mbf{P}$ and $\mbf{P}'$ associated if $W^{\theta, \theta'} \ne \emptyset$. Let $w \in W^{\theta, \theta'}$. We always take the representatives 
$$\wt{w} = \wt{w}_{\alpha_1} ... \wt{w}_{\alpha_l} \in \wt{G}$$
 of $w=w_{\alpha_1} ... w_{\alpha_l}$ with $\wt{w}_\alpha$ given in \S \ref{SS:cov-l}. One has $w \mbf{M} w^{-1} = \mbf{M}'$ and thus a representation ${}^w\sigma$ of $\wt{M}'$  given by
$${}^w\sigma(m')(v):= \sigma(w^{-1} m w) (v)$$
for any $m'\in \wt{M}'$ and $v\in V_\sigma$. In particular, for the  underlying vector spaces, we have $V_{{}^w\sigma}=V_\sigma$.


Consider the intertwining operator
$$ T(w,\sigma): \ \Ind_{\wt{P}}^{\wt{G}}(\sigma) \longrightarrow  \Ind_{\wt{P}'}^{\wt{G}}( {}^w \sigma) $$
given by the meromorphic continuation of the integral
\begin{equation} \label{T(w)}
T(w, \sigma)(f)(g)=\int_{N_w} f(\wt{w}^{-1} n g) dn
\end{equation}
where $N_w=U\cap (w N^- w^{-1})$ with $N^-$ the unipotent opposite $N$. Recall that we write $\wt{G}_0 = \wt{G}_\de$ whenever there is no confusion. There is a similar intertwining operator
$$T(w,\sigma_0): \ \Ind_{\wt{P}_0}^{\wt{G}_0}(\sigma_0) \longrightarrow  \Ind_{\wt{P}'_0}^{\wt{G}_0}( {}^w \sigma_0)$$
defined for every $\sigma_0 \in \Irrg(\wt{M}_0)$. The following is well-known (see \cite[Page 409]{BJ04} for example).

\begin{lm} \label{L:kc}
Let $\sigma \in \Irrg(\wt{M})$. We have
$$(\Ind_{\wt{P}}^{\wt{G}} (\sigma))|_{\wt{G}_0} = \bigoplus_{\sigma_0 \subset \sigma|_{ \wt{M}_0}} \Ind_{\wt{P}_0}^{\wt{G}_0} (\sigma_0).$$
On the other hand,
$$\Ind_{\wt{P}}^{\wt{G}}\  \Ind_{\wt{M}^\dag}^{\wt{M}} (\tau \boxtimes \sigma_0) = \Ind_{\wt{G}^\dag}^{\wt{G}} \big( \tau \boxtimes \Ind_{\wt{P}_0}^{\wt{G}_0}(\sigma_0)\big),$$
where $\tau \in \Irrg(\wt{Z(G)})$ and $\sigma_0 \in \Irrg(\wt{M}_0)$. Moreover, the above two equalities are compatible with the standard intertwining operators $T(w, \sigma)$ and $T(w, \sigma_0)$ for every $\sigma_0 \subset \sigma|_{\wt{M}_0}$.
\end{lm}

\begin{thm} \label{T:irre-c}
Let $\wt{P} = \wt{M} N \subset \wt{G}$ be a parabolic subgroup. Assume $\mfr{Q}^\dag$ acts freely on either $\widehat{\mfr{Q}_Z}$ or $\widehat{\mfr{Q}_{M_0}}$. Let $\sigma \in \Irrg(\wt{M})$, and let $\sigma_0 \subset \sigma|_{\wt{M}_0}$ be an irreducible constituent.
\begin{enumerate}
\item[(i)] If $\mfr{Q}^\dag$ acts freely on either $\widehat{\mfr{Q}_Z}$ or $\widehat{\mfr{Q}_{G_0}}$, then $I_{\wt{P}}^{\wt{G}}(\sigma)$ is irreducible if and only if $I_{\wt{P}_0}^{\wt{G}_0}(\sigma_0)$ is irreducible.
\item[(ii)] Assume $\sigma$ is supercuspidal. If $\mfr{Q}^\dag$ acts trivially on $\widehat{\mfr{Q}_Z}$ and does not act freely on $\widehat{\mfr{Q}_{G_0}}$, then $I_{\wt{P}}^{\wt{G}}(\sigma)$ is irreducible if and only if $I_{\wt{P}_0}^{\wt{G}_0}(\sigma_0)$ is irreducible and
$$(W \cdot \sigma_0) \cap \set{{}^{t}(\sigma_0): t\in \mfr{Q}^\dag} = \set{\sigma_0}$$
holds, i.e., $\sigma_0$ is not Weyl-conjugate to any representation of the form ${}^t\sigma_0$ with $1\ne t \in \mfr{Q}^\dag$.
\end{enumerate}
\end{thm}
\begin{proof}
Since $\mfr{Q}^\dag$ acts freely on either $\widehat{\mfr{Q}_Z}$ or  $\widehat{\mfr{Q}_{M_0}}$, we see that
$$\sigma = I_{\wt{M}^\dag}^{\wt{M}} (\tau \boxtimes \sigma_0)$$
for some $\tau \in \Irrg(\wt{Z(G)})$. Lemma \ref{L:kc} gives 
\begin{equation} \label{E:iind-2}
I_{\wt{P}}^{\wt{G}}(\sigma) = I_{\wt{G}^\dag}^{\wt{G}} (\tau \boxtimes I_{\wt{P}_0}^{\wt{G}_0}(\sigma)).
\end{equation}
Thus it is clear that if $I_{\wt{P}}^{\wt{G}}(\sigma)$ is irreducible, then $I_{\wt{P}_0}^{\wt{G}_0}(\sigma)$ is irreducible.

For (i), it suffices to prove the if part. If $I_{\wt{P}_0}^{\wt{G}_0}(\sigma)$ is irreducible and $\mfr{Q}^\dag$ acts freely on $\widehat{\mfr{Q}_Z}$ or $\widehat{\mfr{Q}_{G_0}}$, then $I_{\wt{P}}^{\wt{G}}(\sigma)$ is  irreducible by Theorem \ref{T:in-re}. 

For (ii), in view of Proposition \ref{E:iind-2}, $I_{\wt{P}}^{\wt{G}}(\sigma)$ is irreducible if and only if $\tau \boxtimes I_{\wt{P}_0}^{\wt{G}_0}(\sigma)$ is not isomorphic to $({}^t\tau) \boxtimes {}^t(I_{\wt{P}_0}^{\wt{G}_0}(\sigma_0))$ for any $1\ne t \in \mfr{Q}^\dag$; since 
$${}^t(I_{\wt{P}_0}^{\wt{G}_0}(\sigma_0)) \simeq I_{\wt{P}_0}^{\wt{G}_0}({}^t\sigma_0))$$ and the assumption implies that ${}^t\tau \simeq \tau$, this condition is equivalent to that $I_{\wt{P}_0}^{\wt{G}_0}(\sigma_0)$ is not equivalent to $I_{\wt{P}_0}^{\wt{G}_0}({}^t\sigma_0)$ for any $1\ne t \in \mfr{Q}^\dag$. Again, as $\sigma$ is assumed to be supercuspidal, it follows from \cite[Theorem 2.9]{BZ2} that this amounts to that $\sigma_0$ is not in the $W$-orbit of ${}^t \sigma_0$ for any $1\ne t\in  \mfr{Q}^\dag$. The proof is thus completed.
\end{proof}

\subsection{Covers of $\GL_r, \GSp_{2r}$ and $\GSpin_{2r+1}$} \label{SS:RI-eg}
In this subsection, we work out explicitly covers of $\GL_r, \GSp_{2r}$ and $\GSpin_{2r+1}$ in view of the above discussion . Some results thus generalize earlier work as in \cite{Szp4-1, Szp5}, and they will also be used in later sections of the paper.

\subsubsection{Covers of $\GL_r$} \label{SSS:GL}
Every Brylinski--Deligne cover $\wt{\GL}_r$ arises from two parameters $\bfp, \bfq \in \Z$ such that
$$
B_Q(e_i, e_j) =
\begin{cases}
2\bfp    & \text{ if } i=j, \\
 \bfq   & \text{ otherwise.}
\end{cases}
$$
Here $\set{e_i: 1\lest i \lest r}$ is a standard basis for the cocharacter lattice $Y$ of $\GL_r$ such that 
$$\alpha_i^\vee:= e_i - e_{i+1}, 1\lest i \lest r-1$$
 are the simple coroots. The Dyndin diagram for the simple coroots is given as follows:

$$\qquad
\begin{picture}(4.7,0.2)(0,0)
\put(1,0){\circle{0.08}}
\put(1.5,0){\circle{0.08}}
\put(2,0){\circle{0.08}}
\put(2.5,0){\circle{0.08}}
\put(3,0){\circle{0.08}}
\put(1.04,0){\line(1,0){0.42}}
\multiput(1.55,0)(0.05,0){9}{\circle{0.02}}
\put(2.04,0){\line(1,0){0.42}}
\put(2.54,0){\line(1,0){0.42}}
\put(1,0.1){\footnotesize $\alpha_{1}^\vee$}
\put(1.5,0.1){\footnotesize $\alpha_{2}^\vee$}
\put(2,0.1){\footnotesize $\alpha_{r-2}^\vee$}
\put(2.5,0.1){\footnotesize $\alpha_{r-1}^\vee$}
\put(3,0.1){\footnotesize $\alpha_r^\vee$}
\end{picture}
$$
\vskip 10pt

 One has 
$$Q(\alpha^\vee)= 2\bfp - \bfq$$
 for every coroot $\alpha^\vee$.  There are two special families of Brylinski--Deligne covers we highlight below (see \cite{Ga3, GW}):
 \begin{enumerate}
 \item[$\bullet$] (Kazhdan--Patterson covers \cite{KP}) A Kazhdan--Patterson covering $\wt{\GL}_{r,\KP}$ is a Brylinski--Deligne cover $\wt{\GL}_r$ such that 
 $$2\bfp - \bfq = -1.$$
 Here $\bfp$ equals the twisting parameter $c$ in \cite{KP}. We have $n_\alpha = n$ for every root $\alpha$, and also the inclusion 
 $$nY \subset Y_{Q,n},$$
 which however may not be an equality in general.
 \item[$\bullet$] (Savin covers \cite{Sav2}) Consider the Brylinski--Deligne covers $\wt{\GL}_r^{(n)}$, which are associated with the pairs $(\bfp, \bfq)$ such that 
$$Q(\alpha^\vee)= 2\bfp - \bfq = -2.$$
In particular, the cover parametrized by $(\bfp, \bfq)= (-1, 0)$ was first studied by Savin \cite{Sav2}, and we denote it by $\wt{\GL}_{r,\Sav}$. In fact, $\wt{\GL}_{r,\Sav}$ arises from the restriction of the $n$-fold cover of $\Sp_{2r}$ to its Siegel Levi subgroup $\GL_r$.
The fact that $\bfq=0$ accounts for the block commutativity of the covering Levi subgroups of $\wt{\GL}_{r,\Sav}$.
In this case, we have
$$n_\alpha = \frac{n}{\text{gcd}(2, n)}$$
and $Y_{Q,n}= n_\alpha \cdot Y$.
 \end{enumerate}

Let $\mbf{r}=(r_1, r_2, ..., r_k)$ be a partition of $r$. Let 
$$M:=M_\mbf{r} =\GL_{r_1} \times ... \times \GL_{r_k} \subset \GL_r$$
be the associated Levi subgroup of $\GL_r$. One has
$$M_{0}:=M_{\mbf{r},0} =\set{(g_i): g_i \in \GL_{r_i} \text{ and } \prod \det(g_i)=1 } \subset M_\mbf{r}.$$
In order to determine $\mfr{Q}^\dag, \mfr{Q}_Z$ and $\mfr{Q}_{M_0}$, we compute the sublattice $\mbm{L} \subset Y$ associated with the following subgroups of $\wt{T}$:
$$\wt{T}^\dag, \ Z(\wt{Z(G)}), \ Z(\wt{G}), \ Z(\wt{M}_0) \text{ and } \wt{M}_0 \cap Z(\wt{M}),$$
whenever they exist. 

First, we have
$$Y_{c}=\Z e_c \text{ with } e_c = \sum_{i=1}^r e_i.$$
It is clear that
$$\mbm{L}(\wt{T}^\dag) = Y_c + Y_0 \text{ and thus } \mfr{Q}^\dag \simeq e_1(F^\times)/e_1(F^{\times r}).$$
An easy computation gives that
$$Q(e_c)= r\bfp + \frac{r(r-1)}{2}\bfq, \quad 2Q(e_c) = r\cdot B(e_c, e_i) \text{ for every i},$$
where
$$B(e_c, e_i) = 2\bfp + r(r-1)\bfq = (2\bfp + 1)r -1.$$
It follows that
$$\mbm{L}(Z(\wt{Z(G)})) =Y_{c,Q,n}= \frac{\Z\cdot ne_c}{\gcd(n, 2Q(e_c))}=: \Z(n_1 e_c),$$
and also
$$\mbm{L}(Z(\wt{G})) = Y_c \cap Y_{Q,n} = \frac{\Z\cdot ne_c}{\gcd(n, B(e_c, e_1))} =: \Z(n_2 e_c).$$ 
We have
$$n_2 = n_1 \cdot \gcd(n, r),$$
since $r$ and $B(e_c, e_i)$ are coprime. Hence,
$$\mfr{Q}_Z \simeq e_c(F^{\times n_1})/ e_c(F^{\times n_2}).$$
Recall that the action of $g\in \mfr{Q}^\dag$ on $\chi_0 \in \widehat{\mfr{Q}_Z}$ is given by the multiplication $\chi_0 \cdot \chi_g$, where
$$\chi_g: \mfr{Q}_Z \longrightarrow \C^\times$$
with 
$$\chi_g(h) = [g^{-1}, h].$$
In the case of $\GL_r$, we take $g= e_1(a), h= e_c(b^{n_1}), a, b \in F^\times$ to obtain
\begin{equation} \label{GL-comm}
\chi_{e_1(a)}(e_c(b^{n_1})) = (b, a)_n^{n_1 \cdot B_Q(e_c, e_1)}.
\end{equation}
Two special cases to note:
\begin{enumerate}
\item[--] if $n$ and $r$ are coprime, then $Z(\wt{Z(G)}) = Z(\wt{G})$. In particular, $\mfr{Q}^\dag$ acts trivially on $\widehat{\mfr{Q}_Z}$ in this case.
\item[--] if $r|n$, then it is easy to see from \eqref{GL-comm} that $\chi_{e_1(a)} $ is trivial on $\mfr{Q}_Z$ if and only if $a\in F^{\times r}$, i.e., $e_1(a) =1 \in \mfr{Q}^\dag$. Thus, $\mfr{Q}^\dag$ acts freely on $\widehat{\mfr{Q}_Z}$. However, since in this case $\val{\mfr{Q}^\dag}= \val{\mfr{Q}_Z}$, the group $\widehat{\mfr{O}_Z}$ is a torsor over $\mfr{O}^\dag$.
\end{enumerate}

On the other hand, to compute $\widehat{\mfr{Q}_{M_0}}$, we note that
$$Y_{M_0} = Y_0 = Y^{sc}.$$
Recalling that $n_\alpha = n/\gcd(n, Q(\alpha^\vee))$, if $\wt{Z(M_0)}$ is associated with a sublattice $Y_{M_0, c}$ of $Y$, then we have 
$$Y_{M_0, c} = Y_0 \cap Y_{M, c},$$
 where $Y_{M, c}$ is the lattice associated with $\wt{Z(M)}$. In this case, it is easy to obtain
$$
\mbm{L}(Z(\wt{M}_0)) = Y_{M_0, c} \cap Y_{0,Q,n} =
\left\{
\begin{array}{cc}
 (y_1, ..., y_1;  y_2, ..., y_2; ...; y_k, ..., y_k) \in \oplus_{i=1}^{r} \Z e_i: \\
 \bullet \quad  n_\alpha |(y_i - y_{i+1}) \text{ for every } i, \\
 \bullet \quad  \sum_{i=1}^k r_i y_i=0.
\end{array}
\right\},
$$
where $y_i$ appears with multiplicity $r_i$ for every $i$. On the other hand, we have
$$
\mbm{L}(\wt{M}_0 \cap Z(\wt{M})) = Y_{M_0, c} \cap Y_{Q,n} =
\left\{
\begin{array}{cc}
(y_1, ..., y_1;  y_2, ..., y_2; ...; y_k, ..., y_k) \in \oplus_{i=1}^{r} \Z e_i: \\
\bullet \quad  n_\alpha |y_i \text{ for every } i, \\
\bullet \quad  \sum_{i=1}^k r_i y_i=0.
\end{array}
\right\},
$$
For simplicity, we specialize to the case $M=T$ and obtain
$$\mbm{L}(Z(\wt{T}_0)) = Y_{0,Q,n} =
\left\{
\begin{array}{cc}
 (y_1, y_2, ..., y_r) \in \oplus_{i=1}^{r} \Z e_i: \\
 \bullet \quad  n_\alpha |(y_i - y_j) \text{ for every } i, j, \\
 \bullet \quad  \sum_{i=1}^r y_i=0.
\end{array}
\right\}
$$
and
$$ \mbm{L}(\wt{T}_0 \cap Z(\wt{T})) = Y_0 \cap Y_{Q,n} = nY_0.$$
Setting
$$v_i= e_i - e_r \text{ for } 1\lest i \lest r-2, \text{ and } v_{r-1} = (e_1 + ... + e_{r-1} + e_r) - r \cdot e_r,$$
it is clear that $Y_{0,Q,n}$ has a basis
$$\set{n_\alpha v_i: 1\lest i \lest r_2} \cup \set{n_{\alpha, (r)} v_{r-1}}$$
where $n_{\alpha, (r)}= n_\alpha/\gcd(n_\alpha, r)$, and $Y_{Q,n}$ has a basis
$$\set{nv_i: 1\lest i \lest r-1}.$$
It thus follows that
$$\mfr{Q}_{T_0} = v_{r-1}(F^{\times n_{\alpha, (r)}})/v_{r-1}(F^{\times n_\alpha}).$$
The action of $g\in \mfr{Q}^\dag$ on $\chi_0 \in \widehat{\mfr{Q}_{T_0}}$ is given by $\chi_0 \cdot \chi_g$, where
$$\chi_g: \mfr{Q}_{T_0} \longrightarrow \C^\times$$
takes the form
$$\chi_g(h) = [g^{-1}, h].$$
We take $g= e_1(a), h= v_{r-1}(b^{n_{\alpha, (r)}}), a, b \in F^\times$ to obtain
\begin{equation} \label{GL-com2}
\chi_{e_1(a)}(v_{r-1}(b^{n_{\alpha, (r)}})) = (b, a)_n^{Q(\alpha^\vee) \cdot n_{\alpha, (r)}}.
\end{equation}
Again, there are two special cases:
\begin{enumerate}
\item[--] if $\gcd(n_\alpha, r)=1$, then $\mfr{Q}^\dag$ acts trivially on $\widehat{\mfr{Q}_{T_0}}$.
\item[--] if $\gcd(n_\alpha, r)=r$ (i.e., $r|n_\alpha$), then $\widehat{\mfr{Q}_{T_0}}$ is a torsor over $\mfr{Q}^\dag$.
\end{enumerate}
We give a partial summary of some results for covers of $\GL_r$, concentrating on the special cases when the induced representation $\Ind_{\wt{M}^\dag}^{\wt{M}}(\tau \boxtimes \rho)$ is irreducible, and when the restriction $\pi|_{\wt{T}_0}$ is an isotypic sum for $\pi \in \Irrg(\wt{T})$.

\begin{prop} \label{P:GLsum}
Let $\wt{\GL}_r$ be a Brylinski--Deligne cover associated with $\bfp, \bfq \in \Z$.
\begin{enumerate}
\item[(i)] If $r|n$, then every $\Ind_{\wt{M}^\dag}^{\wt{M}}(\tau \boxtimes \rho)$ is irreducible for $\tau \boxtimes \rho \in \Irrg(\wt{M}^\dag)$.
\item[(ii)] If $\gcd(n_\alpha, r)=1$, then for every $\pi \in \Irrg(\wt{T})$, the restriction $\pi|_{\wt{T}_0}$ is an isotypic sum. If we assume the stronger equality $\gcd(n, r)=1$, then $\pi|_{\wt{T}^\dag}$ is already an isotypic sum.
\end{enumerate}
\end{prop}
\begin{proof}
The assertion (i) follows from Theorem \ref{T:in-re}. The first part of (ii) follows from Corollary \ref{C:Zinc}, while the second part of (ii) is an immediate consequence of second statement (ii) of Theorem \ref{T:in-re}.
\end{proof}

The case of $n$-fold Kazhdan--Patterson covers of $\GL_2$ was already discussed extensively in \cite[\S 8]{GSS2}. For example, (ii) in Proposition  \ref{P:GLsum} occurs exactly when one considers restriction of a genuine principal series of odd-degree covers  $\wt{\GL}_2$ to $\wt{\SL}_2$.

\subsubsection{Covers of $\GSp_{2r}$} \label{SSS:GSp}
Let $\GSp_{2r}$ be the group of similitudes of symplectic type, and let $(X, \Delta, Y, \Delta^\vee)$ be its root data given as follows. The character lattice $X\simeq \Z^{r+1}$ has a standard basis 
$$\{e_i^*: 1\lest i\lest r \} \cup \{e_0^* \},$$
where the simple roots are 
$$\Delta=\{e_i^*-e_{i+1}^*: 1\lest i \lest r-1 \} \cup \{ 2e_r^*-e_0^* \}.$$
The cocharacter lattice $Y\simeq \Z^{r+1}$ is given with a basis 
$$\{e_i: 1\lest i\lest r \} \cup \{e_0 \},$$
such that the paring between $X$ and $Y$ is $\angb{e_i}{e_j^*} = \delta_{ij}$. The simple coroots are 
$$\Delta^\vee=\{ e_i-e_{i+1}: 1\lest i \lest r-1 \} \cup \{ e_r \}.$$
Write 
$$\alpha_i=e_i^* - e_{i+1}^*,\  \alpha^\vee_i=e_i-e_{i+1}$$
for $1\lest i\lest r-1$, and also 
$$\alpha_r=2e_r^*-e_0^*, \ \alpha_r^\vee=e_r.$$
The Dynkin diagram for the simple coroots is as follows:

$$ \qquad
\begin{picture}(4.7,0.2)(0,0)
\put(1,0){\circle{0.08}}
\put(1.5,0){\circle{0.08}}
\put(2,0){\circle{0.08}}
\put(2.5,0){\circle{0.08}}
\put(3,0){\circle{0.08}}
\put(1.04,0){\line(1,0){0.42}}
\multiput(1.55,0)(0.05,0){9}{\circle*{0.02}}
\put(2.04,0){\line(1,0){0.42}}
\put(2.54,0.015){\line(1,0){0.42}}
\put(2.54,-0.015){\line(1,0){0.42}}
\put(2.74,-0.04){$>$}
\put(1,0.1){\footnotesize $\alpha_1^\vee$}
\put(1.5,0.1){\footnotesize $\alpha_2^\vee$}
\put(2,0.1){\footnotesize $\alpha_{r-2}^\vee$}
\put(2.5,0.1){\footnotesize $\alpha_{r-1}^\vee$}
\put(3,0.1){\footnotesize $\alpha_r^\vee$}
\end{picture}
$$
\vskip 10pt

We have 
$$Y_c = \Z e_c, \text{ where } e_c:=2e_0 + \sum_{1\lest i \lest r} e_i.$$
Consider the covering $\overline{\GSp}_{2r}$ incarnated by $(D, \mbm{1})$. The following relations are easy to check:
$$\begin{aligned}
 Q(e_i) & = -B(e_i, e_0) \text{ for every } i;\\
 Q(e_c) & = B_Q(e_c, e_0) =  4Q(e_0) - rQ(e_i) \text{ for any } i.
\end{aligned}$$
Here $B(e_i, e_0)$ dictates the (non-)commutativity of blocks in a covering Levi subgroup of $\GSp_{2r}$. More precisely, if a Levi subgroup $M \subset \GSp_{2r}$ contains blocks $\GL_k$'s and $\GSp_{2m}$, then $\wt{\GL}_k$ and $\wt{\GL}_{k'}$ always commute; however, $\wt{\GL}_k$ and $\wt{\GSp}_{2m}$ commute if and only if 
$$B(e_i, e_0) = 0 \in \Z/n\Z.$$

We are interested in those $\overline{\GSp}_{2r}$ whose restriction to $\Sp_{2r}$ is the one with 
$$Q(\alpha_r^\vee)=Q(e_i)=-1$$
for every $i$. Thus, we assume
$$Q(\alpha_i^\vee)=-2 \text{ for } 1\lest i \lest r-1, \text{ and } Q(\alpha_r^\vee)=-1.$$
In this case,
$$Q(e_c)= 4Q(e_0) + r.$$
Since $\Delta^\vee \cup \{e_0\}$ gives a basis for $Y$, to determine $Q$ it suffices to specify $Q(e_0)$. There are two special families of $\wt{\GSp}_{2r}$ (depending on the choice of $Q(e_0)$) we want to highlight below. 
\begin{enumerate}
\item[$\bullet$] (Type I: similitudes-splitting type) If $Q(e_0) =0$, then the similitude factor $F^\times$ corresponding to the cocharacter $e_0$ splits into $\overline{\GSp}_{2r}$, and we have 
$$\wt{\GSp}_{2r} \simeq \wt{\Sp}_{2r} \rtimes F^\times.$$
For $n=2$, this recovers the classical double cover $\wt{\GSp}_{2r}^{(2)} \simeq \wt{\Sp}_{2r}^{(2)} \rtimes F^\times$ as discussed in \cite{Szp5}. We note that it suffices to take $Q(e_0) = 0 \in \Z/n\Z$ to ensure the splitting of $\wt{\GSp}_{2r}$ over $e_0(F^\times)$. In any case, if $Q(e_0)=0$, then
$$Q(e_c) = r.$$
\item[$\bullet$] (Type II: Kazhdan--Patterson type) The second family of covers of $\GSp_{2r}$ arises from restricting the Kazhdan--Patterson covers of $\wt{\GL}_r$. To avoid confusion, we let 
$$\set{v_i: 1\lest i \lest 2r} \subset Y$$ be the standard basis of $\GL_{2r}$. The one has a natural embedding
$$\GSp_{2r} \subset \GL_{2r}$$
such that
$$\alpha_r^\vee= v_r - v_{r+1}, \ \alpha_i^\vee = v_i - v_{i+1} - v_{2r+1 - i} + v_{2r-i} \text{ and } e_0=\sum_{r+1 \lest i \lest 2r} v_i.$$
Let $\wt{\GL}_r$ be a Brylinski--Deligne cover associated with $\bfp, \bfq \in \Z $ such that $Q(v_i)=\bfp, B_Q(v_i, v_j)=\bfq, i\ne j$, as in \S \ref{SSS:GL}. It is then easy to see that
$$Q(\alpha_r^\vee) = 2\bfp - \bfq, \ Q(\alpha_i^\vee)= 2Q(\alpha_r^\vee) \text{ for } 1\lest i< r$$
and 
$$Q(e_0)=r \cdot \frac{(2\bfp + 1)r -1}{2}, \quad Q(e_c)=r\cdot (2(2\bfp+1)r - 1).$$
In particular, the $n$-fold cover $\wt{\GSp}_{2r}$ obtained from such restriction may not split over the similitude factor $F^\times \subset \GSp_{2r}$.
\end{enumerate}
We will carry out the computation mainly for type (I) and (II) covers of $\GSp_{2r}$ as above.

Let $\mbf{r}=(r_1, r_2, ..., r_k, m)$ be an ordered partition of $r$. Let 
$$M_\mbf{r} = \GL_{r_1} \times ... \times \GL_{r_k} \times \GSp_{2m} \subset \GSp_{2r}$$
be the associated Levi subgroup. We have $G_0 = \Sp_{2r}$, and 
$$M_0 = \GL_{r_1} \times ... \times \GL_{r_k} \times \Sp_{2m} \subset \Sp_{2r}$$
is the Levi subgroup of $\Sp_{2r}$. Every element in $T \subset \GSp_{2r}$ is of the form
$$(a; \lambda):=(a_1, ..., a_r; \lambda):= \Big(\prod_{i=1}^r e_i(a_i)\Big)\cdot e_0(\lambda) \text{ with } a_i, \lambda \in F^\times.$$
Here $(a; \lambda) \in T_0$ if and only if $\lambda=1$.

\begin{lm}
For $(a; \lambda), (b; \delta) \in T$, one has the commutator
$$[(a; \lambda), (b; \delta)] =(\lambda, \det(b))_n^{-Q(\alpha_r^\vee)} \cdot  (\det(a), \delta)_n^{-Q(\alpha^\vee_r)} \cdot  (\lambda, \delta)_n^{2Q(e_0)} \cdot   \prod_i (a_i, b_i)_n^{2Q(\alpha_r^\vee)}.$$
\end{lm}
\begin{proof}
By abuse of notation, we will write $[a, b]$ for $[(a; 1), (b; 1)]$, and $[\lambda, b]$ for $[(1; \lambda), (b; 1)]$. One has
$$[(a; \lambda), (b; \delta)]= [\lambda, b] \cdot [a, b] \cdot [\lambda, \delta] \cdot [a, \delta] \in \mu_n.$$
Now, it is easy to obtain
$$
\begin{aligned}[t]
 [\lambda, b]  & = (\lambda, \det(b))_n^{B(e_0, e_1)}= (\lambda, \det(b))_n^{-Q(\alpha_r^\vee)}, \\
 [a, b]  & = \prod_i (a_i, b_i)_n^{2Q(e_i)} = \prod_i (a_i, b_i)_n^{2Q(\alpha_r^\vee)},\\
 [\lambda, \delta] &  = (\lambda, \delta)_n^{2Q(e_0)}, \\
 [a, \delta]  & = (\det(a), \delta)_n^{B(e_1, e_0)} =  (\det(a), \delta)_n^{-Q(\alpha^\vee_r)}. 
\end{aligned}
$$
The result immediately follows from combining these equalities.
\end{proof}
For $\wt{\GSp}_{2r}$ of type (I) or (II), we have $Q(\alpha_r^\vee)=-1$ and thus
$$[(a; \lambda), (b; \delta)] =(\lambda, \det(b))_n \cdot  (\det(a), \delta)_n \cdot  (\lambda, \delta)_n^{2Q(e_0)} \cdot  \prod_i (a_i, b_i)_n^{-2}.$$
Now, we want to explicate the three groups $\mfr{Q}^\dag, \mfr{Q}_G$ and $\mfr{Q}_{M_0}$. Equivalently, we want to determine the sublattices of $Y$ associated with the following groups
$$\wt{T}^\dag, \ Z(\wt{Z(G)}), \ Z(\wt{G}), \ Z(\wt{M}_0) \text{ and } \wt{M}_0 \cap Z(\wt{M}).$$

As $\mbm{L}(\wt{T}^\dag) = Y_c + Y_0$, we see that
$$\mfr{Q}^\dag \simeq e_0(F^\times)/e_0(F^{\times 2}),$$
where representatives are taken from $e_0(F^\times)$. We have
$$\mbm{L}(Z(\wt{Z(G)})) = Y_{c,Q,n} = \frac{\Z\cdot ne_c}{\gcd(n, 2Q(e_c))}=: \Z(n_1 e_c)$$
and 
$$\mbm{L}(Z(\wt{G}))= Y_c \cap Y_{Q,n} = \frac{\Z\cdot ne_c}{\gcd(n, Q(e_c))}=: \Z(n_2 e_c).$$
Thus,
$$\mfr{Q}_Z = e_c(F^{\times n_1})/e_c(F^{\times n_2}).$$
For every natural number $n$, we let 
$$\wp(n) \in \N_{\gest 0}$$
 be the 2-exponent in $n$ such that
$$n= 2^{\wp(n)} \cdot n'$$
with $n'$ odd.

\begin{lm} \label{L:GSp-Z}
For $\wt{\GSp}_{2r}$ of type (I) or (II), we have:
\begin{enumerate}
\item[(i)] if $\wp(n) \lest \wp(r)$, then $\mfr{Q}_Z =\set{1}$;
\item[(ii)] if $\wp(n) > \wp(r)$, then $\widehat{\mfr{Q}_Z}$ is a torsor over $\mfr{Q}^\dag$.
\end{enumerate}
\end{lm}
\begin{proof}
One has
$$Q(e_c) =
\begin{cases}
r & \text{ if $\wt{\GSp}_{2r}$ is of type (I) } ,\\
r\cdot (2(2\bfp+1)r - 1) & \text{ if $\wt{\GSp}_{2r}$ is of type (II)}.
\end{cases}
$$
If $\wp(n) \lest \wp(r)$, the assertion is clear. Now we assume $\wp(n) > \wp(r)$. The action of $e_0(a) \in \mfr{Q}^\dag$ on $\widehat{\mfr{Q}_Z}$ is given by the multiplication of the character
$$\chi_{e_0(a)}: \mfr{Q}_Z \longrightarrow \C^\times$$
given by
$$\chi_{e_0(a)}(e_c(b^{n_1})) = (b, a)_n^{n_1\cdot B(e_0, e_c)} = (b, a)_n^{n_1\cdot Q(e_c)}.$$
If $\chi_{e_0(a)} \in \widehat{\mfr{Q}_Z}$ is the trivial character, then $a\in F^{\times 2}$, i.e., $e_0(a) =1 \in \mfr{Q}^\dag$. That is, $\mfr{Q}^\dag$ acts freely on $\widehat{\mfr{Q}_Z}$. Since in this case $n_2= 2 n_1$, the two groups $\mfr{Q}^\dag \simeq F^\times/F^{\times 2}$ and $\mfr{Q}_Z \simeq F^{\times n_1} /F^{\times n_2}$ have the same size. Hence, $\widehat{\mfr{Q}_Z}$ is a torsor over $\mfr{Q}^\dag$.
\end{proof}

To determine $Z(\wt{M}_0)$, we first note that for every $I \subset (\N \cap [1, r])$, one has 
$$Q\Big(\sum_{i\in I} e_i\Big) = \val{I}.$$
Let $e_{c,i} \in Y$ be such that $Y_{\GL_{r_i}, c} = \Z e_{c,i}, 1\lest i \lest k$.  We have
$$\begin{aligned}[t]
\mbm{L}(Z(\wt{M}_0)) & = Y_{M_0, c} \cap Y_{0,Q,n}= \bigoplus_{i=1}^k \frac{\Z \cdot ne_{c,i}}{\gcd(n, 2Q(e_i))}= \bigoplus_{i=1}^k \frac{\Z \cdot ne_{c,i}}{\gcd(n, Q(\alpha_1^\vee))} \\
& =\set{\sum_{i=1}^k y_i e_{c,i}: n_{\alpha_1} |(y_i) \text{ for every } i}.
\end{aligned}
$$
It is not hard to see that
$$
\mbm{L}(\wt{M}_0 \cap Z(\wt{M})) = Y_{M_0, c} \cap Y_{Q,n} =
\left\{
\begin{array}{cc}
\sum_{i=1}^k y_i e_{c,i} \in Y_{M_0, c}: \\
\bullet \quad  n_{\alpha_1} |y_i \text{ for every } i, \\
\bullet \quad  n_{\alpha_r}| (\sum_{i=1}^k y_i r_i).
\end{array}
\right\}.
$$

\begin{lm} \label{L:GSp-M}
Let $M \subset \GSp_{2r}$ be the Levi subgroup associated with the partition $\mbf{r}=(r_1, r_2, ..., r_k; m)$ of $r$. There are three cases:
\begin{enumerate}
\item[(i)] if $n_{\alpha_1} = n_{\alpha_r}$, then $\mfr{Q}_{M_0}=\set{1}$;
\item[(ii)] if $n_{\alpha_r}=2\cdot n_{\alpha_1}$ and $r_i$ is even for every $1\lest i\lest k$, then $\mfr{Q}_{M_0}=\set{1}$;
\item[(iii)] if $n_{\alpha_r}=2\cdot n_{\alpha_1}$ and $r_i$ is odd for some $i$, then $\widehat{\mfr{Q}_{M_0}}$ is a torsor over $\mfr{Q}^\dag$.
\end{enumerate}
\end{lm}
\begin{proof}
First, (i) and (ii) are clear since in this case 
$$n_{\alpha_r}| (\sum_{i=1}^k y_i r_i)$$
 holds whenever $n_{\alpha_1}|y_i$ for all $1\lest i\lest k$. 
 Now for (iii), the assumption implies that $\mbm{L}(\wt{M}_0 \cap Z(\wt{M}))$ is of index two inside $\mbm{L}(Z(\wt{M}_0))$, and representatives of $\mfr{Q}_{M_0}$ can be chosen from $e_{c,i}(b^{n_{\alpha_1}}), b\in F^\times$. The action of $ e_0(a) \in \mfr{Q}^\dag$ on $\widehat{\mfr{Q}_{M_0}}$ is given by the multiplication with the character
 $$\chi_{e_0(a)}: \mfr{Q}_{M_0} \longrightarrow \C^\times$$
which takes the form
$$\chi_{e_0(a)( e_{c,i}(b^{n_{\alpha_1}}) )} = (b, a)_n^{r_i n_{\alpha_1} \cdot B(e_0, e_i)} = (b, a)_n^{-r_i n_{\alpha_1} Q(\alpha_r^\vee)}.$$
If $\chi_{e_0(a)} \in \widehat{\mfr{Q}_{M_0}}$ is trivial, then $a\in F^{\times 2}$, i.e., $e_0(a)=1 \in \mfr{Q}^\dag$. On the other hand, $\mfr{Q}^\dag$ and $\mfr{Q}_{M_0}$ have the same size. Hence, $\widehat{\mfr{Q}_{M_0}}$ is a torsor over $\mfr{Q}^\dag$.
\end{proof}

For $n=2$, this agrees with the results in \cite{Szp4-1}, as expected. Note that if $(a_1, a_2, ..., a_k)$ is a partition of an odd $a \in \N$, then one of the $a_j$'s is odd. It follows that if the assertion (iii) in Lemma \ref{L:GSp-M} holds for $M$, then it holds for every Levi subgroup $M' \subset M \subset G$. This is consistent with Corollary \ref{L:her}.

Similar as Proposition \ref{P:GLsum}, we give a partial summary for the above discussion on $\wt{\GSp}_{2r}$. Combining Theorem \ref{T:in-re},  Lemma \ref{L:GSp-Z} and Lemma \ref{L:GSp-M}, we have the following.
\begin{prop}  \label{P:RI-GSp}
Let $\wt{\GSp}_{2r}$ be a cover of $\GSp_{2r}$ either of type (I) or (II).
\begin{enumerate}
\item[(i)] For a Levi subgroup $M$ associated with $(r_1, ..., r_k, m)$, if 
\begin{enumerate}
\item[--] either $\wp(n) > \wp(r)$, or
\item[--] $n$ is even and $r_i$ is odd for some $i$,
\end{enumerate}
then $\Ind_{\wt{M}^\dag}^{\wt{M}}(\tau \boxtimes \rho)$ is irreducible for every $\tau\boxtimes \rho \in \Irrg(\wt{M}^\dag)$.
\item[(ii)] If $n$ is odd, then for every $\pi\in \Irrg(\wt{T})$, its restriction to $\wt{T}^\dag$ is an isotypic sum; hence the restriction $\pi|_{\wt{T}_0}$ is also an isotypic sum.
\end{enumerate} 
\end{prop}

\subsubsection{Covers of $\GSpin_{2r+1}$}
We consider the simply-connected algebraic group $\Spin_{2r+1}$ which sits in the exact sequence
$$\begin{tikzcd}
Z \ar[r, hook] & \Spin_{2r+1} \ar[r, two heads] & \SO_{2r+1},
\end{tikzcd}$$
where $Z =\set{1, z}$ is the center of $\Spin_{2r+1}$. 
The Dynkin diagram for the simple coroots of the group $\Spin_{2r+1}$ is as follows:

$$ \qquad
\begin{picture}(4.7,0.2)(0,0)
\put(1,0){\circle{0.08}}
\put(1.5,0){\circle{0.08}}
\put(2,0){\circle{0.08}}
\put(2.5,0){\circle{0.08}}
\put(3,0){\circle{0.08}}
\put(1.04,0){\line(1,0){0.42}}
\multiput(1.55,0)(0.05,0){9}{\circle*{0.02}}
\put(2.04,0){\line(1,0){0.42}}
\put(2.54,0.015){\line(1,0){0.42}}
\put(2.54,-0.015){\line(1,0){0.42}}
\put(2.74,-0.04){$<$}
\put(1,0.1){\footnotesize $\alpha_1^\vee$}
\put(1.5,0.1){\footnotesize $\alpha_2^\vee$}
\put(2,0.1){\footnotesize $\alpha_{r-2}^\vee$}
\put(2.5,0.1){\footnotesize $\alpha_{r-1}^\vee$}
\put(3,0.1){\footnotesize $\alpha_r^\vee$}
\end{picture}
$$
\vskip 10pt

Pushing out the above exact sequence by the embedding $Z \hookrightarrow \GL_1$ with $z$ sent to $-1$, one obtains an exact sequence (see \cite{Asg02})
$$\begin{tikzcd}
\GL_1 \ar[r, hook] & \GSpin_{2r+1} \ar[r, two heads] & \SO_{2r+1},
\end{tikzcd}$$
where by definition
$$\GSpin_{2r+1} = \frac{ \GL_1 \times \Spin_{2r+1}}{\set{(1, 1), (-1, z)}}.$$
In \cite{ASh1}, the group $\GSpin_{2r}$ is also discussed. One of the advantegeous properties of $\GSpin_{2r+1}$ is that its Levi subgroups are associated with partitions $(r_1, ..., r_k, m)$ of $r$ and take the form of $\GL_{r_1} \times ... \times \GL_{r_k} \times \GSpin_{2m+1}$; similarly for $\GSpin_{2r}$. However, the Levi subgroups of $\Spin_{2r+1}$ are complicated to describe explicitly, as explained in \cite{Asg02} and discussed with details in \cite{Mat09}.

The Langlands dual group of $\GSpin_{2r+1}$ is $\GSp_{2r}$, and thus the root datum 
$$(X, \Delta, Y, \Delta^\vee)$$ of the former is obtained from inverting that of the latter, which is given in \S \ref{SSS:GSp}. The character group $X\simeq \Z^{r+1}$ is given with a basis 
$$\{e_i^*: 1\lest i\lest r \} \cup \{e_0^* \}.$$
The simple roots are 
$$\Delta=\{ e_i^*-e_{i+1}^*: 1\lest i \lest r-1 \} \cup \{ e_r^* \}.$$
The cocharacter group $Y\simeq \Z^{r+1}$ has a standard basis 
$$\{e_i: 1\lest i\lest r \} \cup \{e_0 \},$$
where the simple coroots are 
$$\Delta^\vee=\{e_i-e_{i+1}: 1\lest i \lest r-1 \} \cup \{ 2e_r-e_0 \}.$$

Write 
$$\alpha_i=e_i^* - e_{i+1}^*, \quad \alpha^\vee_i=e_i-e_{i+1}$$
for $1\lest i\lest r-1$, and also 
$$\alpha_r=e_r^*,\quad  \alpha_r^\vee=2e_r-e_0.$$
Consider the covering $\overline{\GSpin}_{2r+1}$ incarnated by $(D, \mbm{1})$. We are interested in those $\overline{\GSpin}_{2r}$ whose restriction to $\Spin_{2r+1}$ is the one with $Q(\alpha_1^\vee)=1$. That is, we assume
$$Q(\alpha_i^\vee)=1 \text{ for } 1\lest i \lest r-1, \text{ and } Q(\alpha_r^\vee)=2.$$

We have 
$$Y_c = \Z e_0$$
and thus
$$0=B_Q(\alpha_r^\vee, e_0)= 2B_Q(e_r, e_0) - 2Q(e_0).$$
Hence,
\begin{equation}
B_Q(e_i, e_0) = Q(e_0)
\end{equation}
for every $i$. It also follows from the two equalities
$$Q(e_i) = Q(e_j) \text{ for every } i, j \text{ and } Q(2e_r - e_0) = 2 Q(e_i - e_{i+1})$$
that
 \begin{equation}
 Q(e_0) = 2 B(e_i, e_j) \text{ for every } i \ne j.
 \end{equation}
 On the other hand, we have 
 $$Q(\alpha_i^\vee) = 2 Q(e_i) - B(e_i, e_{i+1})$$
 for $1\lest i \lest r-1$.  Thus, the cover $\wt{\GSp}_{2r+1}$ we are interested in is determined by the two numbers 
 $$\bfp:=Q(e_i) \text{ and } \bfq:=B(e_i, e_j) \text{ for } i \ne j$$
constrained by
 $$2\bfp - \bfq = 1.$$
 
Again, we want to determine the three groups $\mfr{Q}^\dag, \mfr{Q}_Z$ and $\mfr{Q}_{M_0}$ and the structure of the Pontryagin duals of latter two under the action of the first. For $\mfr{Q}^\dag$ and $\mfr{Q}_Z$, we want to determine the sublattices of $Y$ associated with the following groups
$$\wt{T}^\dag, \ Z(\wt{Z(G)}), \ Z(\wt{G}).$$
To avoid technical difficulties while dealing with general $M_0$ (see \cite{Mat09}), for $\mfr{Q}_{M_0}$ we only illustrate the minimal parabolic case when $\wt{M} = \wt{T}$, and thus consider
$$\ Z(\wt{T}_0) \text{ and } \wt{T}_0 \cap Z(\wt{T}).$$
First, $\mbm{L}(\wt{T}^\dag) = Y_c + Y_0$ with 
$$Y/(Y_c + Y_0) = \set{0, e_r}.$$
Hence,
$$\mfr{Q}^\dag \simeq e_r(F^\times)/e_r(F^{\times 2}),$$
where representatives are taken from $e_r(F^\times)$. We have
$$\mbm{L}(Z(\wt{Z(G)})) = Y_{c,Q,n} = \frac{\Z\cdot ne_0}{\gcd(n, 2Q(e_0))}=: \Z(n_1 e_0)$$
and 
$$\mbm{L}(Z(\wt{G}))= Y_c \cap Y_{Q,n} = \frac{\Z\cdot ne_0}{\gcd(n, Q(e_0))}=: \Z(n_2 e_0).$$
This gives that
$$\mfr{Q}_Z = e_0(F^{\times n_1})/e_0(F^{\times n_2}).$$

\begin{lm} \label{L:GSpinZ}
We have two cases for the action of $\mfr{Q}^\dag$ on $\widehat{\mfr{Q}_Z}$:
\begin{enumerate}
\item[(i)] if $\wp(n_1) = 0$, or equivalently $n_2 = n_1$, then $\mfr{Q}_Z=\set{1}$; 
\item[(ii)] if $\wp(n_1)\gest 1$, or equivalently $n_2 = 2 n_1$, then $\widehat{\mfr{Q}_Z}$ is a torsor over $\mfr{Q}^\dag$.
\end{enumerate}
\end{lm}
\begin{proof}
The assertion (i) is clear. For (ii), we have $\mfr{Q}_Z \simeq F^{\times n_1}/F^{\times 2n_1}$, which has the same size as $\mfr{Q}^\dag \simeq F^{\times} /F^{\times 2}$. It suffices to show that the action of $\mfr{Q}^\dag$ on $\widehat{\mfr{Q}_Z}$ is free. Consider the character $\chi_{e_r(a)} \in \widehat{\mfr{Q}_Z}$ given by
$$\chi_{e_r(a)}( e_0(b^{n_1})) = (b, a)_n^{n_1 \cdot B(e_r, e_0)}= (b, a)_n^{n_1 \cdot Q(e_0)}.$$
If $\chi_{e_r(a)}$ is the trivial character of $\mfr{Q}_Z$, then $a\in F^{\times 2}$, i.e., $e_r(a) =1 \in \mfr{Q}^\dag$. This shows that the action of $\mfr{Q}^\dag$ on $\widehat{\mfr{Q}_Z}$ is free and completes the proof.
\end{proof}

To compute $\mfr{Q}_{T_0}$, we have
$$\mbm{L}(Z(\wt{T}_0)) = Y_{0,Q,n}$$
and 
$$\mbm{L}(\wt{T}_0 \cap Z(\wt{T})) = Y_0 \cap Y_{Q,n}.$$
A simple computation gives that 
$$Y_0 \cap Y_{Q,n} = Y_{Q,n}^{sc} = \bigoplus_{i=1}^r \Z(n_{\alpha_i} \alpha_i^\vee),$$
where $n_{\alpha_i} = n/\gcd(n, Q(\alpha_i^\vee))$. Note that this shows that the cover $\wt{\GSpin}_{2r+1}$ is always saturated in the sense of \cite[Definition 2.1]{Ga6}. Regarding the lattice $Y_{0,Q,n}$, we have two cases:
\begin{enumerate}
\item[--] if $n$ is odd or $4|(2r-n)$, then $Y_{0,Q,n} = Y_{Q,n}^{sc}$, and we have the dual group
$$\wt{\Spin}_{2r+1}^\vee =
\begin{cases}
{\rm PGSp}_{2r} & \text{ if $n$ is odd}, \\
\SO_{2r+1} & \text{ if $4|(2r-n)$}.
\end{cases}
$$
\item[--]  if $4\nmid (2r-n)$ with $n=2m$, for each $1\lest i \lest r$ defining
$$v_i:= m(2e_i - e_0) \text{ and } v_0:=\frac{v_1 + v_2 + ... + v_r}{2},$$
then one can check easily that
$$\set{v_i: 1\lest i \lest r-1} \cup \set{ v_0 }$$
constitutes a basis of $Y_{0,Q,n}$. In this case,
$$\wt{\Spin}_{2r+1}^\vee = \Spin_{2r+1}.$$
\end{enumerate}
It is clear that in the first case above ($n$ is odd or $4|(2r-n)$), one has
$$\mfr{Q}_{T_0} = \set{1}.$$
On the other hand, if $4\nmid(2r-n)$, then 
$$\mfr{Q}_{T_0} = v_0(F^\times)/v_0(F^{\times 2}) \simeq F^\times/F^{\times 2},$$
where more explicitly  $v_0= m(e_1 + e_2 + ... + e_r) -(mr\cdot e_0)/2$. 
The action of $ e_r(a) \in \mfr{Q}^\dag$ on $\widehat{\mfr{Q}_{T_0}}$ is given multiplication by  the character
 $$\chi_{e_r(a)}: \mfr{Q}_{T_0} \longrightarrow \C^\times,$$
which takes the form
$$\chi_{e_r(a)}(v_0(b)) = (b, a)_n^{B_Q(e_r, v_0)} = (b, a)_n^{mQ(\alpha_1^\vee)} = (b, a)_2.$$
From this it follows that $\widehat{\mfr{Q}_{T_0}}$ is a $\mfr{Q}^\dag$-torsor in this case.

We give a summary of the results for $\GSpin_{2r+1}$ below.
\begin{prop} \label{P:RI-GSpin}
Let $\wt{\GSpin}_{2r+1}$ be a Brylinski--Deligne cover associated with a quadratic form $Q$ such that $Q(\alpha_r^\vee)=2$.
\begin{enumerate}
\item[(i)] If $\wp(n)\gest 3$, then $\Ind_{\wt{M}^\dag}^{\wt{M}} \tau \boxtimes \rho$ is irreducible for every $\tau \boxtimes \rho \in \Irrg(\wt{M}^\dag)$. Moreover, if $4\nmid (2r-n)$ with $n$ even, then $\Ind_{\wt{T}^\dag}^{\wt{T}} \tau\boxtimes \rho$ is irreducible for $\tau\boxtimes \rho \in \Irrg(\wt{T}^\dag)$.
\item[(ii)] If $n$ is odd or $4|(2r-n)$, then the restriction of every $\pi \in \Irrg(\wt{T})$ to $\wt{T}_0$ is an isotypic sum. Moreover, if $\wp(n)=2$ and $2|r$, then the restriction $\pi|_{\wt{T}^\dag}$ is already an isotypic sum.
\end{enumerate}
\end{prop}

\subsection{Restriction of genuine principal series} \label{SS:res-ps}
In this subsection, we consider exclusively the restriction of genuine principal series (especially in the tame case) by using a direct analysis from the perspective of Mackey's theory.

The notations are the same as before, $\wt{G}$ is a cover of $G$ and $\wt{G}_0 =\wt{G}_\de$ the derived subgroup of $\wt{G}$. Let $\wt{T}_0 \subset \wt{G}_0$ be the covering torus. 

Every genuine principal series is parabolically induced from an irreducible representation of $\wt{T}$ of the form 
$$\Ind_{\wt{A}}^{\wt{T}}(\tchi)$$ 
where $\wt{A} \subset \wt{T}$ is a maximal abelian subgroup, necessarily containing $Z(\wt{T})$, and 
$$\tchi: \ \wt{A} \longrightarrow \C^\times$$
 is an extension of the central character
$$\chi: Z(\wt{T}) \longrightarrow \C^\times.$$
For every $\gamma \in \wt{T}_0\backslash \wt{T} / \wt{A}$, denote
$$\wt{A}_\gamma = (\gamma \wt{A} \gamma^{-1}) \cap \wt{T}_0.$$
Define the character
$${}^\gamma\tchi: \wt{A}_\gamma \longrightarrow \C^\times$$
by
$${}^\gamma\tchi(g) = \tchi(\gamma^{-1} g \gamma) \text{ for } g\in \wt{A}_\gamma.$$
The Mackey theory gives a decomposition
$$\left(\Ind_{\wt{A}}^{\wt{T}}(\tchi) \right)|_{\wt{T}_0} = \bigoplus_{\gamma \in \wt{T}_0 \backslash \wt{T} /\wt{A}} \Ind_{\wt{A}_\gamma}^{\wt{T}_0} ({}^\gamma \tchi).$$
We have the following inclusions between various groups
$$\begin{tikzcd}
\wt{A} \ar[r, hook] & \wt{T}  & \wt{T}_0 \ar[l, hook'] & \wt{A}_\gamma \ar[l, hook'] \\
 Z(\wt{T}) \ar[u, hook] & Z(\wt{T}) \cap \wt{T}_0 \ar[u, hook] \ar[l, hook'] \ar[r, hook] &   Z(\wt{T}_0) \ar[u, hook] \ar[ru, hook, dashed].
\end{tikzcd}$$
Here the dashed arrow indicates that there may not be an inclusion $Z(\wt{T}_0) \subset \wt{A}_\gamma$; in particular, $\wt{A}_\gamma$ may not be a maximal abelian subgroup of $\wt{T}_0$.

\begin{dfn} \label{D:isop}
A pair $(\wt{G}, \wt{G}_0)$ is called an isotypic-pair if any of the following equivalent conditions (see Corollary \ref{C:Zinc}) is satisfied:
\begin{enumerate}
\item[(i)] the equality $Y_0 \cap Y_{Q,n} = Y_{0,Q,n}$ holds;
\item[(ii)] the inclusion $Z(\wt{T}_0) \subset Z(\wt{T})$ holds.
\end{enumerate}
\end{dfn}

To continue the analysis, we specialize to the tame case when $p\nmid n$. In this case, we fix a splitting of $\wt{G}$ over $K:=\mbf{G}(O)$ denoted by
$$s_K: \ K \longrightarrow \wt{G}.$$
This gives inherited splittings of $\mbf{G}_0(O), \mbf{T}(O)$ and $\mbf{T}_0(O)$ in the respective covering groups. If no confusion arises, we will omit $s_K$ and simply write $K \subset \wt{G}$. The group
$$\wt{A}= Z(\wt{T}) \cdot \mbf{T}(O)$$
is a maximal abelian subgroup of $\wt{T}$. One has
$$\wt{A}_\gamma = (Z(\wt{T}) \cap \wt{T}_0) \cdot ({}^\gamma \mbf{T}(O) \cap \wt{T}_0)= (Z(\wt{T}) \cap \wt{T}_0) \cdot {}^\gamma \mbf{T}_0(O) \subset Z(\wt{T}_0) \cdot \mbf{T}_0(O),$$
where ${}^\gamma \mbf{T}_0(O) = \gamma \mbf{T}_0(O) \gamma^{-1}$ for every $\gamma$. Note that 
$$\wt{A}_0 :=Z(\wt{T}_0) \cdot \mbf{T}_0(O)$$ 
is a maximal abelian subgroup of $\wt{T}_0$. Defining 
$$\wt{A}^\natural= (Z(\wt{T}) \cap \wt{T}_0) \cdot \mbf{T}_0(O),$$
We write every 
$$\tchi: \wt{A} \longrightarrow \C^\times$$
 as
$$\tchi = \chi \boxtimes \tchi_O,$$
where $\chi$ is the central character mentioned above, and 
$$\tchi_O: \mbf{T}(O) \longrightarrow \C^\times$$
is a character such that $\chi$ and $\tchi_O$ agree on the intersection $Z(\wt{T}) \cap \mbf{T}(O)$. 
Thus the character ${}^\gamma \tchi$ of $\wt{A}_\gamma$ is just 
 $${}^\gamma \tchi = ({}^\gamma \chi) \boxtimes ({}^\gamma \tchi_O)= \chi \boxtimes ({}^\gamma \tchi_O),$$
 where 
 $${}^\gamma \tchi_O: {}^\gamma \mbf{T}(O) \cap \wt{T}_0 = {}^\gamma \mbf{T}_0(O) \longrightarrow \C^\times $$
  is given by 
  $${}^\gamma\tchi_O ({}^\gamma k_0) = \tchi_O(k_0)$$ 
  for every $k_0 \in \mbf{T}_0(O)$.   In view of the identification
  $$\wt{A}_\gamma = \wt{A}^\natural \text{ given by } z \cdot {}^\gamma k_0 \mapsto (z \cdot [\gamma, k_0]) \cdot k_0,$$
 where $z\in Z(\wt{T}) \cap \wt{T}_0$ and $k_0 \in \mbf{T}_0(O)$, we can identify 
 $${}^\gamma \tchi = \chi \boxtimes {}^\gamma \tchi_O: \wt{A}_\gamma \longrightarrow \C^\times$$
 with a character
 $${}^\gamma \tchi = \chi \boxtimes ([\gamma, -]^{-1} \tchi_O): \wt{A}^\natural \longrightarrow \C^\times.$$
Thus, we have 
$$\left(\Ind_{\wt{A}}^{\wt{T}}(\tchi) \right)|_{\wt{T}_0} = \bigoplus_{\gamma \in \wt{T} /(\wt{T}_0\wt{A})} \Ind_{\wt{A}^\natural}^{\wt{T}_0} ({}^\gamma \tchi),$$
where
\begin{equation} \label{E:XGa}
\wt{T}/(\wt{T}_0 \wt{A}) \simeq Y/(Y_0 + Y_{Q,n})=: \msc{X}_{Q,n}^\Gamma.
\end{equation}

Now we analyze the decomposition of $\Ind_{\wt{A}^\natural}^{\wt{T}_0} ({}^\gamma \tchi)$. By induction in stages, one has
$$\Ind_{\wt{A}^\natural}^{\wt{T}_0} ({}^\gamma \tchi) = \Ind_{\wt{A}_0}^{\wt{T}_0} \ \Ind_{\wt{A}^\natural}^{\wt{A}_0} ({}^\gamma \tchi).$$
Denote by 
$$\msc{E}({}^\gamma \tchi; \wt{A}^\natural, \wt{A}_0)$$
the set of all the possible extensions of ${}^\gamma \tchi$ to $\wt{A}_0$; then it is exactly the set of irreducible constituents of $\Ind_{\wt{A}^\natural}^{\wt{A}_0} ({}^\gamma \tchi)$. We thus have
$$\Ind_{\wt{A}^\natural}^{\wt{T}_0} ({}^\gamma \tchi) = \bigoplus_{({}^\gamma\tchi)' \in \msc{E}({}^\gamma \tchi; \wt{A}^\natural, \wt{A}_0)} \Ind_{\wt{A}_0}^{\wt{T}_0} ( ({}^\gamma \tchi) '  ).$$

Recall that $\wt{T}_0$ is a Heisenberg-type group and thus irreducible genuine representation is determined by the central character. In particular, the isomorphism class of $\Ind_{\wt{A}_0}^{\wt{T}_0} ( ({}^\gamma \tchi) '  )$ is in fact determined by the restriction of $({}^\gamma \tchi)'$ to $Z(\wt{T}_0)$. We thus consider
$$\msc{E}(\chi, {}^\gamma \tchi_O; Z(\wt{T}_0)) = 
\left\{
\begin{array}{cc}
\omega_{\gamma, j}: Z(\wt{T}_0) \longrightarrow \C^\times: \\
\bullet \quad  \omega_{\gamma, j} \text{ extends } \chi|_{Z(\wt{T})\cap \wt{T}_0}, \\
\bullet \quad  \omega_{\gamma, j} \text{ extends } ([\gamma, \cdot]^{-1} \tchi_O)|_{\mbf{T}_0(O)\cap Z(\wt{T}_0)}.
\end{array}
\right\}.$$
We note that $\chi$ and $[\gamma, \cdot]^{-1}\tchi_O$ are compatible on the intersection of $Z(\wt{T})\cap \wt{T}_0$ and $\mbf{T}_0(O) \cap Z(\wt{T}_0)$. 
It is easy to see that there is a natural bijection
$$\msc{E}(\chi, {}^\gamma \tchi_O; Z(\wt{T}_0)) \longrightarrow \msc{E}({}^\gamma \chi; \wt{A}^\natural, \wt{A}_0)$$
with size
$$\val{ \msc{E}(\chi, {}^\gamma \tchi_O; Z(\wt{T}_0)) } = \val{Y_{0,Q,n}/(Y_0 \cap Y_{Q,n})}.$$
This gives the bounds for the index $j$ in $\omega_{\gamma, j}$ as
$$1\lest j \lest \val{Y_{0,Q,n}/(Y_0 \cap Y_{Q,n})}.$$
We also have
\begin{equation} \label{E:inds2}
\Ind_{\wt{A}^\natural}^{\wt{T}_0} ({}^\gamma \tchi) = \bigoplus_{\omega_{\gamma, j} \in \msc{E}(\chi, {}^\gamma \tchi_O; Z(\wt{T}_0))} \Ind_{\wt{A}_0}^{\wt{T}_0} (\tomega_{\gamma, j}) = \bigoplus_{\omega_{\gamma, j} \in \msc{E}(\chi, {}^\gamma \tchi_O; Z(\wt{T}_0))} i(\omega_{\gamma, j}),
\end{equation}
where in the middle term 
$$\tomega_{\gamma, j}: \wt{A}_0 \longrightarrow \C^\times$$
is any extension of $\omega_{\gamma, j}$ to $\wt{A}_0$. Here $i(\omega_{\gamma, j})$ denotes the isomorphism class of $\Ind_{\wt{A}_0}^{\wt{T}_0} (\tomega_{\gamma, j})$, which only depends on the central character $\omega_{\gamma, j}$. Note that the decomposition in \eqref{E:inds2} is multiplicity-free.
We thus get
\begin{equation} \label{E:decT}
\left(\Ind_{\wt{A}}^{\wt{T}}(\tchi) \right)|_{\wt{T}_0}= \bigoplus_{\gamma \in \msc{X}_{Q,n}^\Gamma} \bigoplus_{\omega_{\gamma, j} \in \msc{E}(\chi, {}^\gamma \tchi_O; Z(\wt{T}_0))} i(\omega_{\gamma, j}),
\end{equation}
where we recall that 
$$\msc{X}_{Q,n}^\Gamma= Y/(Y_0 + Y_{Q,n})$$
by definition in \eqref{E:XGa}.

To simplify the formula \eqref{E:decT}, we consider the group homomorphism
\begin{equation} \label{D:mfr-c}
\begin{tikzcd}
\mfr{c}: \ \msc{X}_{Q,n}^\Gamma \ar[r] & \Hom(\mbf{T}_0(O) \cap Z(\wt{T}_0), \mu_n)
\end{tikzcd}
\end{equation}
given by
$$\mfr{c}(y)(k_0):=[y(\varpi), k_0] \in \mu_n.$$
Define
$$Y^\mfr{c} = \set{y\in Y: B_Q(y, y') \in n\Z \text{ for all } y'\in Y_{0,Q,n}} \subset Y.$$
Note that $\mbf{T}_0(O) \cap Z(\wt{T}_0) = Y_{0,Q,n}(O)$, and thus
$$\mfr{c}(y)(y'(u))=(\varpi, u)_n^{B_Q(y, y')}$$
for every $y' \in Y_{0,Q,n}$. Hence,
$$\mfr{c}(y) = \mbm{1}$$ if and only if $y\in Y^\mfr{c}$.
It is easy to see that 
$$Y_0 + Y_{Q,n} \subset Y^\mfr{c}$$ 
and thus we have the well-defined group
\begin{equation} \label{D:Xc}
\msc{X}_{Q,n}^\mfr{c} := Y^\mfr{c}/(Y_{Q,n} + Y_0) \subset \msc{X}_{Q,n}^\Gamma.
\end{equation}

We summarize the above discussion to give the following:
\begin{thm} \label{T:decT}
Assume $p\nmid n$. Let $i(\chi)=\Ind_{\wt{A}}^{\wt{T}}(\tchi) \in \Irrg(\wt{T})$ be with central character $\chi$. Let $I(\chi)$ be the associated genuine principal series of $\wt{G}$. 
\begin{enumerate}
\item[(i)] One has $\Ker(\mfr{c}) = \msc{X}_{Q,n}^\mfr{c}$ and thus
$$i(\chi)|_{\wt{T}_0}= \val{\msc{X}_{Q,n}^\mfr{c}} \cdot \bigoplus_{\gamma \in \msc{X}_{Q,n}^\Gamma/\msc{X}_{Q,n}^\mfr{c}} \bigoplus_{\omega_{\gamma, j} \in \msc{E}(\chi, {}^\gamma \tchi_O; Z(\wt{T}_0))} i(\omega_{\gamma, j}),$$
where $i(\omega_{\gamma, j})$ appears with multiplicity one in the double sum of the right hand side. 
\item[(ii)] Consequently,
$$I(\chi)|_{\wt{G}_0}= \val{\msc{X}_{Q,n}^\mfr{c}} \cdot \bigoplus_{\gamma \in \msc{X}_{Q,n}^\Gamma/\msc{X}_{Q,n}^\mfr{c}} \bigoplus_{\omega_{\gamma, j} \in \msc{E}(\chi, {}^\gamma \tchi_O; Z(\wt{T}_0))} I(\omega_{\gamma, j})$$
for the restriction of genuine principal series $I(\chi)$. Here generically, the $I(\omega_{\gamma, j})$'s appear multiplicity-free in the above decompostion. 
\item[(iii)] Assume that $I(\chi)$ is a $(K, s_K)$-unramified genuine principal series. Then $I(\omega_{\gamma, j})$ is $(K_0, s_K)$-unramified if and only if $\omega_{\gamma, j}$ belongs to the set $\msc{E}(\chi, \tchi_O; Z(\wt{T}_0))$, i.e., with $\gamma =0 \in \msc{X}_{Q,n}^\Gamma /\msc{X}_{Q,n}^\mfr{c}$ being the trivial class.
\end{enumerate}
\end{thm}

\begin{cor} \label{C:decTi}
Assume $p\nmid n$. If $(\wt{G}, \wt{G}_0)$ is an isotypic-pair, then 
$$\msc{X}_{Q,n}^\Gamma = \msc{X}_{Q,n}^\mfr{c};$$
in this case 
$$i(\chi)|_{\wt{T}_0} =i(\omega)^{\oplus \val{\msc{X}_{Q,n}^\Gamma} } \text{ and } I(\chi)|_{\wt{G}_0} =I(\omega)^{\oplus \val{\msc{X}_{Q,n}^\Gamma} },$$
are both isotypic sums with multiplicity $\val{\msc{X}_{Q,n}^\Gamma}$, where $\omega =\chi|_{Z(\wt{T}_0)}$ is the unique character in $\msc{E}(\chi, \tchi_O; Z(\wt{T}_0))$.
\end{cor}
\begin{proof}
Under the assumption, we have $Y^\mfr{c}= Y$ and thus $\msc{X}_{Q,n}^\mfr{c} = \msc{X}_{Q,n}^\Gamma$.
The rest is clear in view of Theorem \ref{T:decT}.
\end{proof}

\begin{rmk} \label{R:all-unram}
If $I(\chi)$ is $(K, s_K)$-unramified, then it is $(\gamma\cdot K, \gamma\cdot s_K)$-unramified as well. Here $\gamma\cdot K = \gamma K \gamma^{-1}$ and $\gamma \cdot s_K$ is a splitting of $\wt{G}$ over $\gamma \cdot K$ given by 
$$(\gamma \cdot s_K)(\gamma \cdot k) := \gamma \cdot s_K(k) \cdot \gamma^{-1}.$$
Then every $I(\omega_{\gamma, j})$ is in fact $(\gamma \cdot K_0, \gamma\cdot s_K)$-unramified. Note that $\gamma \cdot K_0$ and $K_0$ may not be in the same $G_0$-conjugacy classes of maximal compact subgroup of $G_0$. See \S \ref{SS:varK} for a general discussion of this, and the reader may also refer to \S \ref{SS:varKW-eg} and \S \ref{S:GSp} for examples.
\end{rmk}

\section{L-parameters and functoriality} \label{S:func}
We continue to assume in this section that $Z(\mbf{G})$ is connected as in \S \ref{SS:2pic}.

\subsection{Relations among several L-groups}
Recall that from Lemma \ref{L:Gz} the group $Z(\wt{G})$ is associated with the sublattice 
$$Y_z:=Y_{Q,n} \cap Y_c.$$
We have the L-group extension
$$\begin{tikzcd}
Z(\wt{G})^\vee \ar[r, hook] & {}^L Z(\wt{G}) \ar[r, two heads] & \WD_F
\end{tikzcd}$$
where $Z(\wt{G})^\vee:=\Hom(Y_z, \C^\times)$ denotes the dual group of $Z(\wt{G})$.
Similarly, we have an L-group for $\wt{Z(G)}$ as an extension
$$ \begin{tikzcd}
\wt{Z(G)}^\vee \ar[r, hook] & {}^L \wt{Z(G)} \ar[r, two heads] & \WD_F.
\end{tikzcd}$$
Here $\wt{Z(G)}^\vee:= \Hom(Y_{c, Q,n}, \C^\times)$. It is easy to see that
\begin{equation} \label{Y-zc}
Y_z \subset Y_{c, Q,n}
\end{equation}
and thus there is a natural homormorphism
\begin{equation} \label{E:f-cz}
f_{c, z}: {}^L \wt{Z(G)} \longrightarrow {}^L Z(\wt{G})
\end{equation}
such that
$$f_{c,z}|_{\wt{Z(G)}^\vee}:  \wt{Z(G)}^\vee \longrightarrow Z(\wt{G})^\vee$$
is just the restriction map 
$$\Hom(Y_{c,Q,n}, \C^\times) \longrightarrow \Hom(Y_z, \C^\times)$$
 induced from  \eqref{Y-zc}.

On the other hand, the dual group for $\wt{G}$ has root datum $(Y_{Q,n}, \Phi_{Q,n}^\vee; X_{Q,n}, \Phi_{Q,n})$, where $Y_{Q,n}$ is the character lattice for $\wt{G}^\vee$. We see that there is a natural homomorphism
$$f_{G, z}: \wt{G}^\vee \longrightarrow Z(\wt{G})^\vee$$
which extends to a homomorphism of L-groups
\begin{equation} \label{E:f-Gz}
f_{G,z}:  {}^L \wt{G} \longrightarrow {}^L Z(\wt{G}),
\end{equation}
depending on the choice of a distinguished character $\chi_\psi$. We give some details of the map $f_{G,z}$ as follows. Let $\mbf{G}_{Q,n}$ be the split linear algebraic group over  $F$ such that the Langlands dual group of $G_{Q,n}=\mbf{G}_{Q,n}(F)$ is equal to $\wt{G}^\vee$, i.e., 
$$G_{Q,n}^\vee =\wt{G}^\vee.$$
One can restrict the Brylisnki--Deligne data $(D, \eta)$ to the cocharacter lattice $Y_{Q,n}$ of $G_{Q,n}$ and to the coroot lattice $Y_{Q,n}^{sc}$ respectively to obtain an $n$-fold cover $\wt{G}_{Q,n}$. As the restriction of $B_Q$ to $Y_{Q,n}$ is trivial modulo $n$, the cover $\wt{G}_{Q,n}$ is almost a trivial cover over $G_{Q,n}$. In particular, the covering torus $\wt{T}_{Q,n} \subset \wt{G}_{Q,n}$ is abelian. In fact, we have 
$$\wt{G}_{Q,n}^\vee = G_{Q,n}^\vee = \wt{G}^\vee,$$
where $\wt{G}_{Q,n}^\vee$ is the dual group for the $n$-fold cover $\wt{G}_{Q,n}$. One also has
$$Z(\wt{G}_{Q,n}) = \wt{Z(G_{Q,n})} \cap Z(\wt{T}_{Q,n}) = \wt{Z(G_{Q,n})}.$$
Recall the isogeny
$$i_{Q,n}: T_{Q,n} \longrightarrow T$$
induced from the inclusion $Y_{Q,n} \subset Y$. By pull-back, one naturally has the map (by abuse of notation) denoted by
$$i_{Q,n}: \wt{T}_{Q,n} \longrightarrow \wt{T}.$$

\begin{lm}
Keeping notations as above, we have
\begin{equation} \label{E:zEG}
i_{Q,n}(Z(\wt{G}_{Q,n})) = Z(\wt{G}).
\end{equation}
Hence, $Z(\wt{G})^\vee = Z(G_{Q,n})^\vee$ and the map $f_{G,z}: \wt{G}^\vee \longrightarrow Z(\wt{G})^\vee$ is just
$$f_{G,z}: \wt{G}^\vee \longrightarrow \wt{G}^\vee/[\wt{G}^\vee, \wt{G}^\vee],$$
where $[\wt{G}^\vee, \wt{G}^\vee] \subset \wt{G}^\vee$ is the derived subgroup.
\end{lm}
\begin{proof}
For simplicity, writing $i$ for $i_{Q,n}$. One has a commutative diagram
$$\begin{tikzcd}
X/X^{sc}  \ar[d]     & X \ar[l, two heads] \ar[d, hook] \\
X_{Q,n}/X_{Q,n}^{sc}  & X_{Q,n}, \ar[l, two heads]
\end{tikzcd}$$
which gives the commutative diagram
$$\begin{tikzcd}
Z(G)  \ar[r, hook]     & T \\
Z(G_{Q,n}) \ar[u, "i"]  \ar[r, hook] & T_{Q,n} \ar[u, "i"] .
\end{tikzcd}$$
It follows that  $i(Z(\wt{G}_{Q,n})) \subset Z(\wt{G})$. To show the other inclusion, we consider the exact sequence
$$\begin{tikzcd}
Z(G_{Q,n}) \ar[r, hook] & T_{Q,n}  \ar[r, "r"] & \Hom(X_{Q,n}^{sc}, F^\times) \ar[r] & {\rm Ext}(X_{Q,n}/X_{Q,n}^{sc}, F^\times)\\
 & Y_z\otimes F^\times \ar[u, "{i_z}"],
\end{tikzcd}$$
where $i_z$ is the map induced from the inclusion $Y_z \subset Y_{Q,n}$. It suffices to show that $r\circ i_z = 0$. For every $y\otimes a \in Y_z\otimes F^\times$ and $x\in X_{Q,n}^{sc}$, one has
$$r\circ i_z(y\otimes a)(x) =a^{\angb{x}{y}},$$
where $\angb{-}{-}: X_{Q,n} \times Y_{Q,n} \longrightarrow \Z$ denotes the canonical pairing. However, since $x\in X_{Q,n}^{sc}$ and $y\in Y_c\cap Y_{Q,n}$, we have $\angb{x}{y}=0$. This shows that $r\circ i_z=0$ and completes the proof of \eqref{E:zEG}. It follows that $Z(\wt{G})^\vee = Z(G_{Q,n})^\vee$; as $\wt{G}^\vee = G_{Q,n}^\vee$, the rest is clear.
\end{proof}
To extend $f_{G,z}$ to be an L-homomorphism. We choose a distinguished character 
$$\chi_\psi: Z(\wt{T}) \longrightarrow \C^\times$$
which gives an isomorphism ${}^L\wt{G} \simeq \wt{G}^\vee \times \WD_F$, see \S \ref{SS:L-g}. Note that by pull-back to $\wt{T}_{Q,n} \subset \wt{G}_{Q,n}$ and further restriction to $Z(\wt{G}_{Q,n})$, one has a distinguished genuine character of $Z(\wt{G}_{Q,n})$, with respect to which we also have 
$${}^L Z(\wt{G}) = {}^L Z(\wt{G}_{Q,n}) = Z(\wt{G})^\vee \times \WD_F.$$
From this, we have an L-map as in \eqref{E:f-Gz}.

Next,  we want to understand the relation between ${}^L \wt{G}$ and ${}^L \wt{G}_0$. Our exposition here follows closely that of \cite[\S 12.3]{GG}. Let $Y_0$ be the cocharacter lattice of $G_0$. One has
$$Y^{sc} \subset Y_0 \subset Y.$$
The root datum of $G_0$ is 
$$(X_0, \Phi; Y_0, \Phi^\vee),$$
where $X_0 \supset X^{sc}$ is the character lattice. By pull-back, one has a cover $\wt{G}_0$ of $G_0$. Then dual group $\wt{G}_0^\vee$ of $\wt{G}_0$ has root datum
$$(Y_{0,Q,n}, \Phi_{Q,n}^\vee; X_{0,Q,n}, \Phi_{Q,n}).$$
Unlike the linear algebraic case, there might be no group homomorphism from $\wt{G}^\vee$ to $\wt{G}_0^\vee$ in general. However, these two groups are related as follows.

First, we have
$$Y_0 \cap Y_{Q,n} \subset Y_{0,Q,n}.$$
Let $\mbf{H}$ be the split algebraic group over $F$ whose Langlands dual group $H^\vee$ has root datum
$$(Y_0\cap Y_{Q,n},\  \Phi_{Q,n}^\vee; \ \text{dual lattice of } Y_0 \cap Y_{Q,n}, \ \Phi_{Q,n}).$$
We obtain two homomorphisms of algebraic groups valued in $H^\vee$
\begin{equation} \label{E:H-vee}
\wt{G}^\vee \longrightarrow H^\vee \longleftarrow \wt{G}_0^\vee.
\end{equation}
By restricting the Brylinski--Deligne data $D$ and $\eta$ to $Y_0\cap Y_{Q,n}$ and $Y_{Q,n}^{sc}$ respectively, one obtains an $n$-fold cover $\wt{H}$ of $H$ such that 
$$\wt{H}^\vee \simeq H^\vee \simeq \wt{G}^\vee.$$
Suppose there exist distinguished genuine characters for $\wt{G}$ and $\wt{G}_0$, then we can extend the homomorphism in \eqref{E:H-vee} to obtain two homomorphisms of L-groups,
$$f_{G,H}: {}^L \wt{G} \longrightarrow {}^L \wt{H}$$
and
$$f_{G_0, H}: {}^L \wt{G}_0 \longrightarrow {}^L\wt{H}.$$
Note that if $(\wt{G}, \wt{G}_0)$ is an isotypic-pair, then 
$$\wt{G}_0^\vee = \wt{H}^\vee.$$
In particular, this is true when $\wt{G}_0$ is a saturated cover, i.e., when $\wt{G}_0^\vee$ is of adjoint type; in this case, $f_{G_0, H}$ is an isomorphism.

\subsection{L-parameters and some speculations} \label{SS:Lspec}
To every parameter 
$$\phi: \WD_F \longrightarrow {}^L \wt{G}$$
we associate the group of connected components of the centralizer of ${\rm Im}(\phi)$:
$$\mca{S}_\phi = \mca{S}(\phi):= \pi_0\left({\rm Cent}_{\wt{G}^\vee}({\rm Im}(\phi)/Z(\wt{G}^\vee)) \right).$$
Let $\Phi({}^L\wt{G})$ be the set of $\wt{G}^\vee$-conjugacy classes of L-parameters of $\WD_F$ valued in ${}^L\wt{G}$. Let $\Phi^{\rm en}({}^L \wt{G})$ be the set of enhanced parameters of the form
$$(\phi, \theta) \text{ with } \phi \in \Phi({}^L\wt{G}) \text{ and }  \theta \in \Irr(\mca{S}_\phi).$$
We assume that there is a local Langlands correspondence
$$\begin{tikzcd}
\Irrg(\wt{G}) \ar[r] &  \Phi^{\rm en}({}^L\wt{G}),
\end{tikzcd}$$
assigning to every $\pi \in \Irrg(\wt{G})$ a pair $(\phi_\pi, \theta_\pi) \in \Phi^{\rm en}({}^L\wt{G})$. For every $\phi\in \Phi({}^L\wt{G})$, denote by
$$\mca{L}(\phi):=\set{\pi \in \Irrg(\wt{G}): \ \phi_\pi = \phi} \subset \Irrg(\wt{G})$$
the hypothetical L-packet associated with $\phi$. Besides the one for linear algebraic groups given in \cite{Bor}, there are several additional desiderata for such a correspondence as given in \cite[\S 12]{GG}, some of which we will also explain below. These desiderata encode natural properties expected on both the representation and parameter side.

We expect on the parameter side relations among various L-groups, and also the compatibility of certain L-parameters, as depicted in the following diagram.
\begin{equation} \label{L-comp}
\begin{tikzcd}
 & & {}^L \wt{Z(G)} \ar[r, "{f_{c,z}}"] & {}^L Z(\wt{G}) \\
 \WD_F \ar[rr, "{\phi_\pi}"] \ar[rru, "{\phi_\tau}"] \ar[rrd, "{\phi_\rho}"'] & & {}^L \wt{G}  \ar[ru, "{f_{G,z}}"']  \ar[rd, "{f_{G,H}}"] \\
&  & {}^L \wt{G}_0 \ar[r, "{f_{G_o, H}}"'] & {}^L H.
\end{tikzcd}
\end{equation}
To explain this, for every $\phi \in \Phi({}^L\wt{G})$, we denote
$$\phi^\diamondsuit = f_{G,H}\circ \phi$$
and set
$$\Phi({}^L\wt{G}_0; \phi) = \set{\phi_0 \in \Phi({}^L\wt{G}_0): \  f_{G_0, H} \circ \phi_0 = \phi^\diamondsuit}.$$
Consider the induced homomorphism 
$$\mca{S}_\phi \into \mca{S}_{\phi^\diamondsuit},$$
which is actually injective (see \cite[Proposition 5.4]{Sol20}). For every $\phi_0 \in \Phi({}^L\wt{G}_0; \phi)$, there is also an embedding 
$$\mca{S}_{\phi_0} \into \mca{S}_{\phi^\diamondsuit}.$$

The following are some (refined) speculations from \cite[\S 12.1--\S 12.3]{GG} on central characters and restriction of representations to the derived subgroup.

\begin{conj} \label{C:res}
Let $\pi \in \Irrg(\wt{G})$ be  with central character $\omega_\pi$ and associated enhanced parameter $(\phi_\pi, \theta_\pi) \in \Phi^{\rm en}({}^L\wt{G})$. Let $\tau \boxtimes \rho$ be an irreducible representation of $\wt{G}^\dag=\wt{Z(G)} \cdot \wt{G}_0$ that occurs in the restriction of $\pi$. 
 \begin{enumerate}
 \item[(i)] The L-parameters  $\phi_{\omega_\pi}, \phi_\tau$ and $\phi_\rho$ satisfy the following:
\begin{enumerate}
\item[$\bullet$] $f_{G,z} \circ \phi_\pi = f_{c, z} \circ \phi_\tau$ and is equal to the parameter $\phi_{\omega_\pi}$ associated with the central character $\omega_\pi$ of $\pi$;
\item[$\bullet$] $\phi_\pi^\diamondsuit = f_{G_0, H} \circ \phi_\rho$.
\end{enumerate}
\item[(ii)] 
There exists $e(\pi)\in \N$ such that for every $\tau \in \Irrg(\wt{G}_0)$, one has
$$\dim \Hom_{\wt{G}_0}(\pi, \tau) =
\begin{cases}
0 & \text{ if } \phi_\tau \notin \Phi({}^L\wt{G}_0; \phi_\pi), \\
e(\pi) \cdot \angb{ \Ind_{\mca{S}_{\phi_\pi}}^{\mca{S}_{\phi^\diamondsuit}} \theta_\pi }{ \Ind_{\mca{S}_{\phi_\tau}}^{\mca{S}_{\phi^\diamondsuit}} \theta_\tau }_{\mca{S}_{\phi^\diamondsuit}} & \text{ if }  \phi_\tau \in \Phi({}^L\wt{G}_0; \phi_\pi),
\end{cases}
$$
where $\angb{-}{-}$ denotes the pairing of two representations of $\mca{S}_{\phi^\diamond}$. In particular, if $(\wt{G}, \wt{G}_0)$ is an isotypic-pair, then
$$\pi|_{\wt{G}_0} =e(\pi) \cdot \bigoplus_{\tau \in \mca{L}(\phi_\pi^\diamondsuit)} \angb{ \theta_\tau|_{\mca{S}_{\phi_\pi}} }{ \theta_\pi}_{\mca{S}_{\phi_\pi}} \cdot \tau,$$
where $(\phi_\pi^\diamondsuit, \theta_\tau) \in \Phi^{\rm en}({}^L\wt{G}_0)$ is the enhanced parameter associated with $\tau \in \mca{L}(\phi_\pi^\diamondsuit)$.
\item[(iii)] Let $\tau \in \Irrg(\wt{Z(G)})$ and $\rho \in \Irrg(\wt{G}_0)$ be compatible representations. Then every irreducible constituent $\pi \in \Irrg(G)$ of $\Ind_{\wt{G}^\dag}^{\wt{G}} (\tau\boxtimes \rho)$ has a parameter $\phi_\pi$ which fits into a commutative diagram \eqref{L-comp}.
\end{enumerate}
\end{conj}
Part (ii) is clearly motivated from the case of linear algebraic groups, especially the work of Silberger \cite{Sil79}, Keys \cite{Key3}, Gelbart--Knapp \cite{GeKn1, GeKn2}, Adler--Prasad \cite{AP1}; see also that of Ban, Choiy and Goldberg \cite{BCG18, Cho19}.

\subsection{Functoriality for genuine principal series} \label{SS:func-ps}
We first briefly recall the local Langlands correspondence for covering tori. The description here follows closely \cite[\S 10]{We6}. It gives the construction of the L-group ${}^L\wt{T}$ of $\wt{T}$ and establish the local Langlands correspondence for $\wt{T}$ along the same line of the construction.

Let $T = Y\otimes F^\times$ be a linear torus and $\wt{T}$ an $n$-fold cover associated with a quadratic form $Q: Y\to \Z$. Let 
$$i_{Q,n}: T_{Q,n} \to T$$
 be the isogeny induced from the inclusion $Y_{Q,n} \subset Y$. Recall that the preimage of 
 $$T^\dag:={\rm Im}(i_{Q,n}) \subset T$$
 inside $\wt{T}$ is just the center $Z(\wt{T})$.  Using the fixed embedding $\varepsilon: \mu_n \into \C^\times$ we obtain the push-out $\varepsilon_*(Z(\wt{T}))$ of $Z(\wt{T})$. At the same time, any genuine character $\chi$ of $Z(\wt{T})$ gives rise to a splitting

$$\begin{tikzcd}
s_\chi:  \varepsilon_*(Z(\wt{T})) \ar[r] & \C^\times
\end{tikzcd}$$ 
given by
$$\begin{tikzcd}
s_\chi:  [(z, \wt{t})] \ar[r, mapsto] & z\cdot \chi^{-1}(\wt{t}),
\end{tikzcd}$$
where $[(z, \wt{t})]$ denote the class of 
$$(z, \wt{t}) \in \C^\times \times Z(\wt{T})$$
 in $\varepsilon_*(Z(\wt{T}))$. We illustrate this by using the following diagram
$$\begin{tikzcd}
\mu_n \ar[d, hook, "\varepsilon"] \ar[r, hook] & Z(\wt{T}) \ar[r, two heads] \ar[ld, "{\chi}"']  \ar[d] & T^\dag \ar[d, equal] \\
\C^\times \ar[r, hook] & \varepsilon_*(Z(\wt{T})) \ar[l, bend left=50, "{s_\chi}"] \ar[r, two heads] & T^\dag.
\end{tikzcd}$$ 
 Clearly $s_\chi$ entails a splitting $\phi_\chi$ as in
 $$\begin{tikzcd}
 \C^\times \ar[r, hook] & \varepsilon_*(Z(\wt{T})) \ar[r, two heads] & T^\dag \ar[l, bend left=50, "\phi_\chi"]
 \end{tikzcd}$$
 given by
 $$\begin{tikzcd}
\phi_\chi:\  t \ar[r] &  {[(\chi(\wt{t}), \wt{t})]},
\end{tikzcd}$$
where $\wt{t}\in Z(\wt{T})$ is any lifting of $t \in \wt{T}^\dag$.

By definition 
$$\wt{T}^\vee =X_{Q,n} \otimes \C^\times,$$ 
where $X_{Q,n}=\Hom(Y_{Q,n}, \Z)$ is the lattice dual to $Y_{Q,n}$. By abuse of notation, we still use $i_{Q,n}$ to denote the naturally induced map 
$$i_{Q,n}: X_{Q,n} \otimes T_{Q,n} \longrightarrow X_{Q,n} \otimes T^\dag.$$
Consider the composite
$$\begin{tikzcd}
\mfr{m}: \WD_F \ar[r, "{\rm rec}"] & F^\times \ar[r, "f"] & X_{Q,n}\otimes T_{Q,n}  \simeq \Hom(Y_{Q,n}, T_{Q,n}),
\end{tikzcd}$$
where the first map is the reciprocity map of class field theory sending a geometric Frobenius to the uniformizer $\varpi \in F^\times$ and trivial on $\SL_2(\C) \subset \WD_F$, and the second map is given by 
$$f(a)(y)=y\otimes {\rm rec}(a), \ y\in Y_{Q,n}.$$

The L-group ${}^L \wt{T}$ is defined to be the pull-back of $X_{Q,n}\otimes \varepsilon_*(Z(\wt{T}))$ via $i_{Q,n} \circ \mfr{m}$:
\begin{equation}\label{LT}
\begin{tikzcd}
\wt{T}^\vee \ar[r, hook]  & X_{Q,n} \otimes \varepsilon_*(Z(\wt{T})) \ar[r, two heads] & X_{Q,n}\otimes T^\dag \ar[l, bend right=30, "\phi_\chi"']  \\
\wt{T}^\vee \ar[u, equal] \ar[r, hook]  & {}^L\wt{T} \ar[u] \ar[r, two heads] & \WD_F \ar[u, "{i_{Q,n}\circ \mfr{m}}"'] \ar[l, bend left=30, "\phi_\chi"].
\end{tikzcd}
\end{equation}
Here, the bottom splitting $s_\chi$ of ${}^L\wt{T}$ over $\WD_F$ is the one inherited from the splitting of the top extension.

The above construction
$$\begin{tikzcd}
\Irrg(\wt{T}) \ar[r, hook] & \Phi({}^L\wt{T})
\end{tikzcd}$$
given by 
$$i(\chi) \mapsto \phi_\chi$$
is the local Langlands correspondence (LLC) for covering torus, which is an injective map.

We want to obtain functoriality of LLC with respect to the restriction to a subgroup of $\wt{T}$, which is associated with a $W$-stable lattice $J \subset Y$. Note that $J_{Q,n}$ may not be a sublattice of $Y_{Q,n}$ in general, and we have
$$J_{Q,n} \supset (J \cap Y_{Q,n}) \subset Y_{Q,n}.$$
Let $T_J$ be the torus associated with $J$ and $T_{J,\natural}$ the torus associated with $J\cap Y_{Q,n}$. We have a relation
$$\begin{tikzcd}
Z(\wt{T}_J) & Z(\wt{T}_{J,\natural}) = \wt{T}_{J,\natural} \ar[l, hook'] \ar[r, hook] & Z(\wt{T}).
\end{tikzcd}$$

\begin{lm} \label{L:incs}
Let $i(\chi_J) \in \Irrg(\wt{T}_J)$ and $i(\chi) \in \Irrg(\wt{T})$. Then $i(\chi_J) \subset i(\chi)|_{\wt{T}_J}$ if and only if $\chi_J$ and $\chi$ agree on $\wt{T}_{J,\natural}$, or equivalently, the following diagram
$$\begin{tikzcd}
Z(\wt{T}_J) \ar[rd, "\chi_J"'] & \wt{T}_{J,\natural} \ar[d, dashed] \ar[l, hook'] \ar[r, hook] & Z(\wt{T})  \ar[ld, "\chi"] \\
& \C^\times
\end{tikzcd}$$
commutes.
\end{lm}
\begin{proof}
The only if part is clear, and it suffices to prove the if part. Assume $\chi_J$ and $\chi$ agree on $\wt{T}_{J,\natural}$, then we get a genuine character
$$\chi_J \otimes \chi: \ Z(\wt{T}_0) \cdot Z(\wt{T}) \longrightarrow \C^\times.$$
Let $\wt{A}_0 \subset \wt{T}_0$ be a maximal abelian subgroup of $\wt{T}_0$. We extend $\chi_J \otimes \chi$ to a character
$$\chi_0'\otimes \chi:\  \wt{A}_0 \cdot Z(\wt{T}) \longrightarrow \C^\times,$$
where 
$$\chi_0': \wt{A}_0 \longrightarrow \C^\times$$
 is an extension of $\chi_J$. We can find a maximal abelian subgroup $\wt{A} \subset \wt{T}$ containing $\wt{A}_0 \cdot Z(\wt{T})$, and let 
$$\chi': \ \wt{A} \longrightarrow \C^\times$$
be extending $\chi_0' \otimes \chi$. Note that up to isomorphism class, we have
$$i(\chi_J) \simeq \Ind_{\wt{A}_0}^{\wt{T}_0} (\chi_0') \text{ and } i(\chi) \simeq \Ind_{\wt{A}}^{\wt{T}} (\chi').$$
Now Frobenius reciprocity gives that
$$ \Hom_{\wt{T}_0}(i(\chi_J), i(\chi)) = \Hom_{\wt{A}_0}( \chi_0', \Ind_{\wt{A}}^{\wt{T}}(\chi')) \ne 0,$$
as $\chi'$ extends $\chi_0'\otimes \chi$. This shows that $i(\chi_J) \subset i(\chi)|_{\wt{T}_0}$.
\end{proof}

If $\chi_J$ and $\chi$ agree on $\wt{T}_{J,\natural}$, then we denote $\chi_{J,\natural} = \chi|_{\wt{T}_{J,\natural}}$. By the construction of L-groups, we have
$$\begin{tikzcd}
{}^L \wt{T}_J \ar[r] &  {}^L \wt{T}_{J,\natural}  & {}^L\wt{T} \ar[l].
\end{tikzcd}$$
If $\chi$ and $\chi_J$ are compatible, then by the LLC for covering torus, the two parameters $\phi_{\chi_J}$ and $\phi_{\chi}$ are compatible, i.e., the following diagram commutes:
\begin{equation} \label{D:funT}
\begin{tikzcd}
{}^L \wt{T}_J \ar[r] &  {}^L \wt{T}_{J,\natural}  & {}^L\wt{T} \ar[l] \\
& \WD_F \ar[lu, "\phi_{\chi_J}"]  \ar[u]  \ar[ru, "\phi_\chi"'],
\end{tikzcd}
\end{equation}
where the middle vertical arrow is the parameter associated with $\chi_{J,\natural}$.

\begin{prop}  \label{P:func-tor}
Let $\pi:=I(\chi)$ and $\rho:=I(\chi_0)$ be irreducible genuine principal series of $\wt{G}$ and $\wt{G}_0$ respectively. Let $\tau \in \Irrg(\wt{Z(G)})$.
\begin{enumerate}
\item[(i)] If $\tau \boxtimes \rho \in \Irr(\wt{G}^\dag)$ occurs in $I(\chi)|_{\wt{G}^\dag}$, then the L-parameters $\phi_\tau$ and $\phi_\rho$ are such that the diagram in \eqref{L-comp} commutes. Moreover, if ${}^L\wt{G}_0= {}^L \wt{H}$, then $I(\chi)|_{\wt{G}_0}$ is an isotypic sum of a genuine principal series of $\wt{G}_0$.
\item[(ii)] The parameter $\phi_\pi$ of every irreducible constituent $\pi$ of $\Ind_{\wt{G}^\dag}^{\wt{G}} (\tau\boxtimes \rho)$ satisfies a commutative diagram \eqref{L-comp} involving $\phi_\tau$ and $\phi_\rho$.
 \end{enumerate}
\end{prop}
\begin{proof}
For (i), we first consider $\phi_\pi$ and $\phi_\tau$. Note that we have a natural map ${}^L\wt{T} \longrightarrow {}^L\wt{G}$, see \S \ref{SS:L-g}. Taking $J=Y_c$, we thus have a commutative diagram from \eqref{D:funT}:
\begin{equation}\label{D:c1}
\begin{tikzcd}
& {}^L\wt{Z(G)} \ar[rd] \\
\WD_F \ar[ru, "{\phi_\tau}"] \ar[rd, "{\phi_\chi}"'] \ar[rr, "{\phi_{\omega_\pi}}"] & & {}^LZ(\wt{G}) \\
& {}^L \wt{T} \ar[r, hook] \ar[ru] & {}^L\wt{G} \ar[u, "{f_{G,z}}"'].
\end{tikzcd}
\end{equation}
Similarly, taking $J=Y_0$ coupled with \eqref{D:funT} give another diagram:
\begin{equation} \label{D:c2}
\begin{tikzcd}
& {}^L\wt{T} \ar[r, hook]  \ar[d] & {}^L\wt{G} \ar[d, "{f_{G,H}}"] \\
\WD_F \ar[ru, "{\phi_\chi}"] \ar[r, "{\phi_{\chi_J}}"]  \ar[rd, "{\phi_{\chi_0}}"'] & {}^L\wt{T}_{Y_0,\natural} \ar[r, hook] & {}^L\wt{H} \\
& {}^L\wt{T}_0 \ar[u] \ar[r, hook] & {}^L\wt{G}_0 \ar[u, "{f_{G_0,H}}"'],
\end{tikzcd}
\end{equation}
where we have $\phi_{\chi_0}:=\phi_{\chi_J}$. The above two diagrams show the commutativity of \eqref{L-comp}.  If ${}^L\wt{G}_0 ={}^L\wt{H}$, then $\phi_{\chi_0}$ is uniquely determined by $\phi_\chi$, and thus 
$$i(\chi)|_{\wt{T}_0}= e \cdot i(\chi_0)$$
 is an isotypic sum of $i(\chi_0) \in \wt{T}_0$. This shows that $I(\chi)|_{\wt{G}_0} =e \cdot I(\chi_0)$ is an isotypic sum. The proof of (i) is completed.

For (ii), assume $\pi=I(\chi) \subset \Ind_{\wt{G}^\dag}^{\wt{G}}(\tau\boxtimes I(\chi_0))$, it is easy to see that \eqref{D:c1} commutes. On the other hand, using Frobenius reciprocity, one has $I(\chi_0) \subset I(\chi)|_{\wt{G}_0}$ and thus $i(\chi_0) \subset i(\chi)$. The commutativity of \eqref{D:c2} then follows from Lemma \ref{L:incs}. This concludes the proof.
\end{proof}

\subsection{Metaplectic tensor product  for $\wt{\GSp}_{2r}$} \label{SS:mtp-GSp}
In this subsection, we investigate for $\wt{\GSp}_{2r}$ an analogue of the metaplectic tensor product construction for Kazhdan--Patterson covers studied in \cite{Mez04, Tak3, Tak4, Cai1}. The parameter side interpretation for such a construction is given in \cite{Gan17} using the formalism of L-groups as in \cite{We3, We6}. More precisely, consider a Levi subgroup
$$M= \GL_{r_1} \times \GL_{r_2} \times ... \times \GL_{r_k} \subset \GL_r.$$
It is well-known (by checking on the covering tori of $\GL_{r_i}$'s for example) that the blocks $\wt{\GL}_{r_i} \subset \wt{M}$ do not commute. Thus, the representation of $\wt{M}$ can not be simply reduced to that of each $\wt{\GL}_{r_i}$. However, given with $\pi_i \in \Irrg(\wt{\GL}_{r_i})$ satisfying certain condition constrained by the central character of $Z(\wt{G})$, Mezo \cite{Mez04} gave a natural construction of a representation of $\wt{M}$. Coarsely, the construction goes through several steps as follows:
\begin{enumerate}
\item[(1)] For every $i$ and $n\in \N$, we denote 
$$\GL_{r_i}^\anga{n}=\set{g\in \GL_{r_i}: \det(g) \in F^{\times n}}.$$
One can check that for $i\ne j$, the two covering groups $\wt{\GL}_{r_i}^\anga{n}$ and $\wt{\GL}_{r_j}^\anga{n}$ commute with each other. Hence one can define
$$\wt{M}^\anga{n}:=\wt{\GL}_{r_1}^\anga{n} \times_{\mu_n} \times ... \times_{\mu_n} \wt{\GL}_{r_k}^\anga{n}.$$
For each $\pi_i \in \wt{\GL}_{r_i}$, let $\sigma_i \in \Irrg(\wt{\GL}_{r_i}^\anga{n})$ be an irreducible summand in the restriction of $\pi_i$. We have a representation
$$\sigma:= \sigma_1 \boxtimes ... \boxtimes \sigma_k$$
of $\wt{M}^\anga{n} \subset \wt{M}$.
\item[(2)] Pick an irreducible genuine character $\omega: Z(\wt{\GL}_r) \longrightarrow \C^\times$ such that
\begin{equation} \label{E:f-rel}
\omega = \sigma \text{ on } Z(\wt{\GL}_r) \cap \wt{M}^\anga{n}.
\end{equation}
This gives an irreducible representation $\omega \boxtimes \sigma $ of $Z(\wt{\GL}_r) \cdot \wt{M}^\anga{n}$.
\item[(3)] One extends $\omega \boxtimes \sigma$ as much as possible to a representation $(\omega\boxtimes \sigma)'$ of a subgroup $\wt{M}'$ such that the Mackey's irreducibility criteria are satisfied. We thus obtain
$$\tilde{\otimes}_i \pi_i:= \Ind_{\wt{M}'}^{\wt{M}} ((\omega\boxtimes \sigma)'),$$
which is an irreducible representation of $\wt{M}$. 
\end{enumerate}
It is shown in \cite{Mez04} that the representation $\tilde{\otimes}_i \pi_i$ depends only on $\pi_i$'s and $\omega$, and is independent of the intermediate choices $\sigma_i$, $\wt{M}'$ and $(\omega \boxtimes \sigma)'$. We highlight that such independence relies on the crucial property that for every $\pi \in \Irrg(\wt{M})$, one has
$${\rm supp}(\Theta_\pi) \subset Z(\wt{\GL}_r) \cdot \wt{M}^\anga{n};$$
thus, it follows from (the proof of) Proposition \ref{P:con-con} that $\Ind_{\wt{M}^\anga{n}}^{ \wt{M} }( \omega \boxtimes \sigma )$ is an isotypic sum of an irreducible representation of $\wt{M}$, which is exactly $\tilde{\otimes}_i \pi_i$ in (3) above. This observation was explicated in \cite[Proposition 4.6]{Tak3}. In view of this, we could replace (3) above by the following:
\begin{enumerate}
\item[(3)'] One has
$$\Ind_{\wt{M}^\anga{n}}^{ \wt{M} }( \omega \boxtimes \sigma ) = m \cdot \pi \text{ for a certain } \pi \in \Irrg(\wt{M}).$$
Now the representation $\tilde{\otimes}_i \pi_i:=\pi \in \Irrg(\wt{M})$ is the one we seek.
\end{enumerate}
The above construction gives a well-defined surjective map (for the surjectivity, see \cite[Lemma 4.4]{Tak3})
$$\begin{tikzcd}
\tilde{\otimes}: ( \Irrg(\wt{\GL}_{r_1}) \times ... \times \Irrg(\wt{\GL}_{r_k}) \times \Irrg(Z(\wt{\GL}_{r})))^\heartsuit \ar[r, two heads] & \Irrg(\wt{M}).
\end{tikzcd} 
$$ 
The superscript $(-)^\heartsuit$ indicates the subset of 
$$(\pi_1, ..., \pi_r, \omega)$$
satisfying the relation \eqref{E:f-rel}.  For every character $\chi_i$ of $F^\times$ which is trivial on $F^{\times n}$, replacing $\pi$ by $\pi\otimes (\chi_i \circ \det)$ gives the same representation on $\Irrg(\wt{M})$, and this accounts for the non-injectivity of the map $\tilde{\otimes}$, see \cite[Lemma 5.1]{Mez04}. This metaplectic tensor product construction was further analyzed and refined in \cite{Tak3, Tak4} especially  in the global context, see also \cite{Cai1}.

\vskip 10pt

To proceed, in this subsection, we assume $\wt{\GSp}_{2r}$ is the type I (similitudes-splitting) $n$-fold cover associated with 
$$Q(\alpha_r^\vee) = -1 \text{ and } Q(e_0) =0,$$
see \S \ref{SSS:GSp}. One has $B(e_0, e_i) = 1$ for every $i$, and $B_Q(e_c, e_0) = Q(e_c) = r$. Also,
\begin{equation} \label{GSp-L}
Y_{Q,n} =
\left\{
\begin{array}{cc}
\sum_{i=0}^r y_i e_i \in Y: \\
\bullet \quad  n|(-2y_i + y_0) \text{ for every } i,  \\
\bullet \quad  n|(y_1 + y_2 + ... + y_r).
\end{array}
\right\}.
\end{equation}
On the other hand, $Y_{Q,n}^{sc}$ is spanned by 
$$\set{\alpha_{i,Q,n}^\vee=n_2 \cdot \alpha_i^\vee: 1\lest i \lest r-1} \cup \set{\alpha_{r,Q,n}^\vee =n\alpha_r^\vee}.$$

We consider a partition 
$$\mbf{r}=(r_1, r_2, ..., r_k; r_0)$$
 of $r$ with associated Levi subgroup 
$$M_\mbf{r}=\GL_{r_1} \times \GL_{r_2} \times ... \times \GL_{r_k} \times \GSp_{2r_0}.$$
For every $j\in \N$, let
$$\GSp_{2r}^\anga{j} = \set{g\in \GSp_{2r}: e_0^*(g) \in F^{\times j}}$$
be the subgroup of $\GSp_{2r}$ with similitudes lying in $F^{\times j}$, where 
$$e_0^*: \GSp_{2r} \longrightarrow F^\times$$
is the similitude map. 
Also define
$$M_\mbf{r}^\anga{j} = \GL_{r_1} \times \GL_{r_2} \times ... \times \GL_{r_k} \times \GSp_{2r_0}^\anga{j}.$$
Accordingly, we have the covering subgroups $\wt{\GSp}_{2r}^\anga{j} \subset \wt{\GSp}_{2r}$ and $\wt{M}_\mbf{r}^\anga{j} \subset \wt{M}_\mbf{r}$. The covering blocks $\wt{\GL}_{r_i}$ in $\wt{M}_\mbf{r}$ commute with each other; however, they may not commute with $\wt{\GSp}_{2r_0}^\anga{j}$ for general $j$.

\begin{lm} \label{L:suppGSp}
\begin{enumerate}
\item[(i)] The covering group $\wt{\GSp}_{2r_0}^\anga{j}$ commutes with every $\wt{\GL}_{r_i}, 1\lest i \lest k$ if and only if $n|j$.
\item[(ii)] For every genuine representation $\pi$ of $\wt{M}_\mbf{r}$, one has
$${\rm supp}(\Theta_\pi) \subset \wt{M}_\mbf{r}^\anga{n_{(r)}}$$
for the character $\Theta_\pi$ of $\pi$.
\end{enumerate}
\end{lm}
\begin{proof}
For (i), we note that every element in $\wt{\GSp}_{2r_0}$ can be written as $g_0 \cdot e_0(a)$ with $g_0\in \wt{\Sp}_{2r_0}$ and $a\in F^\times$. We have block commutativity among the blocks $\wt{\GL}_{r_i}$ and $\wt{\Sp}_{2r_0}$, thus $g_0$ commutes with every element in $\wt{\GL}_{r_i}$. On the other hand, $e_0(a)$ commutes with every $e_{\alpha_i}(a), 1\lest i \lest r-1$, thus $e_0(a)$ commutes with $\wt{\GL}_{r_i}$ if and only if 
$$[e_0(a), e_i(x)] = (a, x)_n^{B_Q(e_0, e_i)}$$
for every $e_i \in Y_{\GL_{r_i}}$ and $x\in F^\times$. Since $B_Q(e_0, e_i) = 1$, the  above equality amounts to $a\in F^{\times n}$. 
For (ii), the proof is the same as Proposition \ref{P:Gns}  by using the above fact that $g_0$ commutes with $\wt{\GL}_{r_i}$ for every $i$.
\end{proof}

Let $M_\mbf{r} \subset \GSp_{2r}$ be a Levi subgroup as above. We assume in the rest of this subsection that 
$$\gcd(n, r_0)=1$$
which implies ${\rm supp}(\Theta_{\pi_0}) \subset \wt{\GSp}_{2r_0}^\anga{n}$ for every $\pi_0 \in \Irrg(\wt{\GSp}_{2r_0})$, by Lemma \ref{L:suppGSp}. It is also easy to check that
\begin{equation} \label{E:LZM}
\mbm{L}(Z(\wt{\GSp}_{2r}) \cap \wt{M}_\mbf{r}^\anga{n})=
\begin{cases}
\Z(ne_c) & \text{ if } 2\nmid \gcd(n,r),\\
\Z(me_c) & \text{ if } 2|\gcd(n,r).
\end{cases}
\end{equation}
Also,
\begin{equation} \label{E:ZM}
Z(\wt{\GSp}_{2r}) \cdot \wt{M}_\mbf{r}^\anga{n}=
\begin{cases}
 \wt{M}_\mbf{r}^\anga{n_{(r)}} & \text{ if } 2\nmid \gcd(n,r),\\
\wt{M}_\mbf{r}^\anga{2n_{(r)}} & \text{ if } 2|\gcd(n,r).
\end{cases}
\end{equation}
We will discuss the two cases separately: (1) $n$ is odd, (2) $n$ is even.

\subsubsection{The case of odd $n$}  \label{SSS:GSp-odd}
For $n$ odd, it is easy to obtain from \eqref{GSp-L} that 
$$\set{ne_i: 1\lest i \lest r} \cup \set{n_{(r)} \cdot e_c}$$
constitutes a $\Z$-basis for $Y_{Q,n}$. This shows that
$$\wt{\GSp}_{2r}^\vee = \set{(g, a) \in \GSpin_{2r+1} \times \GL_1: \lambda(g) = a^{\gcd(n, r)}},$$
where 
$$\lambda: \GSpin_{2r+1} \longrightarrow \GL_1$$
 is the similitude map of $\GSpin_{2r+1}$ associated with $ne_c \in Y_{Q,n}$. Since we have assumed $\gcd(n, r_0)=1$, it gives
$$\wt{\GSp}_{2r_0}^\vee = \GSpin_{2r_0 + 1}.$$
Now we describe a construction analogous to the metaplectic tensor product for $\wt{\GL}_r$. In parallel, there are three steps as follows.
\begin{enumerate}
\item[(S1)] Let $\pi_i \in \Irrg(\wt{\GL}_{r_i})$ and $\pi_0 \in \Irrg(\wt{\GSp}_{2r_0})$. Let $\sigma_0 \subset \pi_0$ be an irreducible summand in the restriction of $\pi_0$ to $\wt{\GSp}_{2r_0}^\anga{n}$. The representation
$$\sigma:= \pi_1 \boxtimes ... \boxtimes \pi_k \boxtimes \sigma_0$$
is then a representation of $\wt{M}_\mbf{r}^\anga{n}$.
\item[(S2)] Let $\omega: Z(\wt{\GSp}_{2r}) \longrightarrow \C^\times$ be a central character such that
\begin{equation} \label{E:comp01}
\omega = \sigma \text{ on } Z(\wt{\GSp}_{2r}) \cap \wt{M}_\mbf{r}^\anga{n}.
\end{equation}
We have $\omega \boxtimes \sigma \in \Irr( Z(\wt{\GSp}_{2r}) \cdot \wt{M}_\mbf{r}^\anga{n})$.
\item[(S3)] One extends $\omega \boxtimes \sigma$ to a representation $(\omega \boxtimes \sigma)'$ of a subgroup $\wt{M}_\mbf{r}' \subset \wt{M}_\mbf{r}$ such that the representation
$$\tilde{\otimes}_i \pi_i: = \Ind_{\wt{M}'}^{\wt{M}} (\omega \boxtimes \sigma)'$$
is irreducible. This representation $\tilde{\otimes}_i \pi_i$ is the sought metaplectic-tensor product of the $\pi_i$'s.
\end{enumerate}
Similar to the case of Kazhdan--Patterson covers $\wt{\GL}_r$ discussed earlier, (S3) can be replaced by an equivalent statement as follows. 
\begin{enumerate}
\item[(S3')] It follows from (ii) of Lemma \ref{L:suppGSp}, the equality \eqref{E:ZM} and Proposition \ref{P:Gns} that
$$\Ind_{ \Irr( Z(\wt{\GSp}_{2r}) \cdot \wt{M}_\mbf{r}^\anga{n}) }^{ \wt{M}_\mbf{r} } (\omega\boxtimes \sigma) = m \cdot \pi$$
is an isotypic sum of $\pi \in \Irrg(\wt{M}_\mbf{r})$. The representation $\pi$ is just $\tilde{\otimes}_i \pi_i$ described in (S3) above. In particular, fixing $\omega \boxtimes \sigma$, the representation $\tilde{\otimes}_i \pi_i$ is independent of the choice of $\wt{M}_\mbf{r}'$ and the extension $(\omega\boxtimes \sigma)'$.
\end{enumerate}
We have to justify the above construction (S1)-(S3) (or (S1), (S2) and (S3')) as follows.

\begin{lm} \label{L:Swd}
The representation $\tilde{\otimes}_i \pi_i$ is independent of the choice of $\sigma_0 \subset \pi_0$, the choice of $\wt{M}_\mbf{r}'$ and the extension $(\omega \boxtimes \sigma)'$.
\end{lm}
\begin{proof}
Again, the argument is essentially the same as in \cite{Mez04}. By the equivalence between (S3) and (S3'), it suffices to show the independence on the choice of constituent $\sigma_0$ in (S1). Every summand in the restriction of $\pi_0$ to $\wt{\GSp}_{2r_0}$ is of the form ${}^g \pi_0$, where $g_0 \in \wt{\GSp}_{2r_0}/\wt{\GSp}_{2r_0}^\anga{n} \simeq e_0(F^\times)/e_0(F^{\times n})$. Taking $g=e_0(a), a\in F^\times$, since $\Theta_{\pi_i} \subset \wt{\GL}_{r_i}^\anga{n}$ for every $\pi_i \in \Irrg(\wt{\GL}_{r_i})$ by Proposition \ref{P:Gns}, we see that
$${}^g \Theta_{\pi_i} = \Theta_{\pi_i}$$
and thus ${}^g \pi_i \simeq \pi_i$. It follows that
$$\pi_1\boxtimes ... \boxtimes \pi_k \boxtimes {}^g \sigma_0 = {}^g \sigma.$$
If $\omega$ and $\sigma$ agrees on $Z(\wt{\GSp}_{2r}) \cap \wt{M}_\mbf{r}^\anga{n}$, then $\omega = {}^g \sigma$ on this intersection subgroup as well. Now we have
$$\begin{aligned}[t]
\Ind_{Z(\wt{\GSp}_{2r}) \cdot \wt{M}_\mbf{r}^\anga{n}}^{\wt{M}_\mbf{r}} (\omega \boxtimes {}^g \sigma) & = \Ind_{Z(\wt{\GSp}_{2r}) \cdot \wt{M}_\mbf{r}^\anga{n}}^{\wt{M}_\mbf{r}} {}^g(\omega \boxtimes \sigma) \\
& = {}^g \Ind_{Z(\wt{\GSp}_{2r}) \cdot \wt{M}_\mbf{r}^\anga{n}}^{\wt{M}_\mbf{r}} (\omega \boxtimes  \sigma) \\
& =  \Ind_{Z(\wt{\GSp}_{2r}) \cdot \wt{M}_\mbf{r}^\anga{n}}^{\wt{M}_\mbf{r}} (\omega \boxtimes  \sigma) \\
& = e \cdot \tilde{\otimes}_i \pi_i,
\end{aligned}$$
where the last equality follows from Proposition \ref{P:Gns}, see also (S3) and (S3') above. This shows the independence on ${}^g\sigma_0$.
\end{proof}

We also have a reverse construction from $\Irrg(\wt{M}_\mbf{r})$ to $\prod_{i=1}^k \Irrg(\wt{\GL}_{r_i}) \times \Irrg(\wt{\GSp}_{2r_0})$
as follows.
\begin{enumerate}
\item[(RS1)] Let $\Pi \in \Irrg(\wt{M}_\mbf{r})$ be an irreducible genuine representation. We pick a summand $\Pi_0 \subset \Pi$ in the restriction of $\Pi$ to $\wt{M}_\mbf{r}^\anga{n}$, it takes the form
$$\Pi_0 = \pi_1 \boxtimes \pi_2 \boxtimes ... \boxtimes \pi_k \boxtimes \sigma_0$$
with $\pi_i \in \Irrg(\wt{\GL}_{r_i})$ and $\sigma_0 \in \Irrg(\wt{\GSp}_{2r_0}^\anga{n})$. Every summand of $\Pi|_{\wt{M}_\mbf{r}^\anga{n}}$ is of the form 
$${}^g(\Pi_0) = {}^g\pi_1 \boxtimes ... \boxtimes {}^g \pi_k \boxtimes {}^g \sigma_0$$
for some $g\in \wt{M}_\mbf{r}/\wt{M}_\mbf{r}^\anga{n} = e_0(F^\times)/e_0(F^{\times n})$. Taking $g=e_0(a), a\in F^\times$, since we have ${\rm supp}(\Theta_{\pi_i}) \subset Z(\wt{\GSp}_{2r}) \cdot \wt{\GL}_{r_i}^\anga{n}$ by Proposition \ref{P:Gns}, it follows that
$${}^g \Theta_{\pi_i} = \Theta_{\pi_i}$$
as $g$ commutes with every elements in $\wt{\GL}_{r_i}^\anga{n}$. This shows that ${}^g \pi_i \simeq \pi_i$. Thus, every constituent $\Pi_0' \subset \Pi$ is of the form
$$\Pi_0' = \pi_1 \boxtimes \pi_2 \boxtimes ... \boxtimes \pi_k \boxtimes {}^g\sigma_0$$
for some $g\in e_0(F^\times)$. In particular, we have $\pi_i \in \Irrg(\wt{\GL}_{r_i})$ uniquely determined by $\Pi$ for each $1\lest i \lest k$.
\item[(RS2)] Extend $\sigma_0$ as much as possible to a representation $\sigma_0'$ of a subgroup $\wt{\GSp}_{2r_0}' \subset \wt{\GSp}_{2r_0}$, such that the induced representation
$$\pi_0:= \Ind_{\wt{\GSp}_{2r_0}'}^{\wt{\GSp}_{2r_0}} (\sigma_0')$$
is irreducible. This gives the desired representation $\pi_0 \in \Irrg(\wt{\GSp}_{2r_0})$. Using the crucial fact that $\Theta_\pi \subset \wt{\GSp}_{2r_0}^\anga{n}$ for every $\pi\in \Irrg(\wt{\GSp}_{2r_0})$, it follows from Proposition \ref{P:Gns} that
$$\Ind_{ \wt{\GSp}_{2r_0}^\anga{n} }^{ \wt{\GSp}_{2r_0} } (\sigma_0) = m\cdot \pi_0,$$
as an isotypic sum of $\pi_0 \in \Irrg(\wt{\GSp}_{2r_0})$. Thus, $\pi_0$ is the desired representation.
\end{enumerate}

\begin{prop}  \label{P:kGSp}
Keep notations as above. 
\begin{enumerate}
\item[(i)] One has $\tilde{\otimes}_i \pi_i \simeq \tilde{\otimes}_i \pi_i'$ if and only if $\pi_i \simeq \pi_i'$ for every $0\lest i \lest k$.
\item[(ii)] Every representation $\Pi \in \Irrg(\wt{M}_\mbf{r})$ is a metaplectic tensor product of the form $\tilde{\otimes}_i \pi_i$, where the $\pi_i$'s are uniquely determined by {\rm (RS1)-(RS2)}.
\end{enumerate}
\end{prop}
\begin{proof}
For (i), it suffices to prove the ``only if" part. Consider the restriction of $\tilde{\otimes}_i \pi_i \simeq \tilde{\otimes}_i \pi_i'$ to $\wt{M}_\mbf{r}^\anga{n}$, we see that
there is $g \in e_0(F^\times)$ and $\sigma_0 \subset \pi_0, \sigma_0' \subset \pi_0'$ such that
$${}^g(\pi_1 \boxtimes ... \boxtimes \pi_k \boxtimes \sigma_0) = \pi_1' \boxtimes ... \boxtimes \pi_k' \boxtimes \sigma_0',$$
where the left hand side is in fact $\pi_1 \boxtimes ... \boxtimes \pi_k \boxtimes {}^g \sigma_0$. It follows that 
$$\pi_i \simeq \pi_i'$$
for every $1\lest i \lest k$, and also $\sigma_0' = {}^g \sigma_0$. On the other hand, we have
$$\Ind_{\wt{\GSp}_{2r_0}^\anga{n}}^{ \wt{\GSp}_{2r_0} }(\sigma_0') = e' \cdot \pi_0' \text{ and } \Ind_{\wt{\GSp}_{2r_0}^\anga{n}}^{ \wt{\GSp}_{2r_0} }(\sigma_0) = e \cdot \pi_0$$
for some $e, e' \in \N$. Since
$$\Ind_{\wt{\GSp}_{2r_0}^\anga{n}}^{ \wt{\GSp}_{2r_0} }(\sigma_0')= \Ind_{\wt{\GSp}_{2r_0}^\anga{n}}^{ \wt{\GSp}_{2r_0} }({}^g\sigma_0) = {}^g \Ind_{\wt{\GSp}_{2r_0}^\anga{n}}^{ \wt{\GSp}_{2r_0} }(\sigma_0) = \Ind_{\wt{\GSp}_{2r_0}^\anga{n}}^{ \wt{\GSp}_{2r_0} }(\sigma_0),$$
it follows that $e= e'$ and $\pi_0 \simeq \pi_0'$. 

It is clear from the above argument that every $\Pi\in \Irrg(\wt{M}_\mbf{r})$ is a metaplectic tensor product with $\pi_i$ determined as above. This gives (ii) and completes the proof.
\end{proof}

The construction in (S1)-(S3) gives a well-defined map
$$\begin{tikzcd}
\tilde{\otimes}_i: ( \prod_{i=1}^k\Irrg(\wt{\GL}_{r_i}) \times \Irrg(\wt{\GSp}_{2r_0}) \times \Irrg(Z(\wt{\GSp}_{2r})))^\heartsuit \ar[r, two heads] & \Irrg(\wt{M}_\mbf{r}).
\end{tikzcd} 
$$ 
The superscript $(-)^\heartsuit$ indicates the subset of 
$$(\pi_1, ..., \pi_k, \pi_0, \omega)$$ 
satisfying the equality \eqref{E:comp01}  in (S2). 
On the other hand, when restricted to $Z(\wt{\GSp}_{2r}) \subset \wt{M}_\mbf{r}$, the representation $\Pi$ acts by a character 
$$\omega_\Pi: Z(\wt{\GSp}_{2r}) \to \C^\times.$$
Thus, the construction in (RS1)-(RS2) gives a ``metaplectic restriction"
$$\tilde{\rm R}: \Irrg(\wt{M}_\mbf{r}) \longrightarrow   \Big(\prod_{i=1}^k \Irrg(\wt{\GL}_{r_i}) \times \Irrg(\wt{\GSp}_{2r_0}) \times \Irrg(Z(\wt{\GSp}_{2r}) ) \Big)^\heartsuit$$
given by 
$$\tilde{\rm R}(\Pi) =(\pi_1, ..., \pi_k, \pi_0; \omega_\Pi).$$
The following is immediate from Proposition \ref{P:kGSp}.

\begin{thm} \label{T:GSp}
Assume $n$ is odd and $\gcd(n, r_0)=1$. Then the map
$$\tilde{\otimes}_i:  \Big(\prod_{i=1}^k \Irrg(\wt{\GL}_{r_i}) \times \Irrg(\wt{\GSp}_{2r_0}) \times \Irrg(Z(\wt{\GSp}_{2r}) ) \Big)^\heartsuit \longrightarrow \Irrg(\wt{M}_\mbf{r})$$
is a bijection with the inverse given by $\tilde{\rm R}$, i.e. $\tilde{\rm R} \circ \tilde{\otimes}_i = {\rm id} = \tilde{\otimes}_i \circ \tilde{\rm R}$.
\end{thm}

It is expected that for the parameters associated with the representations involved in the constructions $\tilde{\otimes}$ and $\tilde{\rm R}$, one should have a corresponding identification. To proceed, we note that the dual group of $\wt{M}_\mbf{r}$ is
$$\begin{aligned}[t]
\wt{M}_\mbf{r}^\vee & = \set{(g_1, g_2, ..., g_k, g_0; z): \  \lambda(g_0) \cdot \prod_{i=1}^k  \det(g_i) = z^{g(n, r)} } \\
& \subset \prod_{i=1}^k \GL_{r_i} \times \GSpin_{2r_0 + 1} \times \GL_1(\C).
\end{aligned}$$
Assuming that we have a local Langlands correspondence for the covers, then associated with $\pi_i$'s and $\omega$ we have
$$\phi_i: \WD_F \longrightarrow \GL_{r_i} \text{ for } 1\lest i \lest k, \quad \phi_0: \WD_F \longrightarrow \GSpin_{2r_0 + 1}$$
parametrizing $\pi_i$'s for $0\lest i \lest k$, and
$$\phi_\omega: \WD_F \longrightarrow \GL_1(\C)$$
parametrizing $\omega$. This gives a parameter
$$(\phi_1 \times ... \times \phi_k \times \phi_0) \times \phi_\omega: \WD_F \longrightarrow \left(\prod_{i=1}^k \GL_{r_i} \times \GSpin_{2r_0 + 1}\right) \times \C^\times.$$

\begin{lm}
Assume the local Langlands correspondence with desiderata. Then the compatibility condition in \eqref{E:comp01} is equivalent to 
$$\prod_{i=1}^k \det(\phi_i) \cdot \lambda(\phi_0) = \phi_\omega^{\gcd(n, r)}.$$
Thus, the parameter $(\phi_1 \times ... \times \phi_k \times \phi_0) \times \phi_\omega$ factors through $\wt{M}_\mbf{r}^\vee$.
\end{lm}
\begin{proof}
It follows from \eqref{E:LZM} that 
$$Z(\wt{\GSp}_{2r}) \cap \wt{M}_\mbf{r} = \wt{Z(\GSp_{2r})^n},$$
where $Z(\wt{\GSp}_{2r}) = \wt{Z(\GSp_{2r})^{n_{(r)}}}$. If we denote by $e_{i, c} \in Y_{\GL_{r_i}}$ the natural generator of  $Y_{\GL_{r_i}, c}$, then
$$Z(\wt{\GL}_{r_i}) = \wt{Z(\GL_{r_i})^n} \text{ and } Z(\wt{\GSp}_{2r_0}) = \wt{Z(\GSp_{2r_0})^n}.$$
From the desiderata (i) of Conjecture \eqref{C:res}, the center character 
$$\omega_i: Z(\wt{\GL}_{r_i}) \longrightarrow \C^\times$$
 of $\pi_i$ has a parameter
$$\begin{tikzcd}
\det(\phi_i): \WD_F \ar[r, "{\phi_\pi}"]  & \wt{\GL}_{r_i}^\vee \ar[r, two heads, "\det"] & \wt{\GL}_{r_i}^\vee/[\wt{\GL}_{r_i}^\vee, \wt{\GL}_{r_i}^\vee] \simeq \C^\times,
\end{tikzcd}
$$
where the second map is the determinant map of $\wt{\GL}_{r_i}^\vee \simeq \GL_{r_i}$. The center character 
$$\omega_0: Z(\wt{\GSp}_{2r_0}) \longrightarrow \C^\times$$
of $\pi_0$ has a parameter
$$\begin{tikzcd}
\lambda(\phi_0): \WD_F \ar[r, "{\phi_\pi}"]  & \wt{\GSp}_{2r_0}^\vee \ar[r, two heads, "\lambda"] & \wt{\GSp}_{2r_0}^\vee/[\wt{\GSp}_{2r_0}^\vee, \wt{\GSp}_{2r_0}^\vee] \simeq \C^\times,
\end{tikzcd}
$$
where the second map is the similitude map of $\wt{\GSp}_{2r_0}^\vee \simeq \GSpin_{2r_0 + 1}$. The identity is thus clear.
\end{proof}

Let 
$$\begin{tikzcd}
i_{M}: \wt{M}_\mbf{r}^\vee \ar[r, hook] & \prod_{i=1}^k \wt{\GL}_{r_i}^\vee \times \wt{\GSp}_{2r_0}^\vee \times \C^\times
\end{tikzcd}$$
 be the natural inclusion. Assuming the local Langlands correspondence, we expect that the following diagram
 \begin{equation}
 \begin{tikzcd}
 \WD_F \ar[rr, "{ \phi_1 \times ... \times \phi_k \times \phi_0 \times \phi_\omega }"] \ar[rd, "{ \phi_\Pi }"]  & & \prod_{i=1}^k \wt{\GL}_{r_i}^\vee \times \wt{\GSp}_{2r_0}^\vee \times \C^\times \\
 & \wt{M}_\mbf{r}^\vee \ar[ru, hook, "{i_M}"].
 \end{tikzcd}
 \end{equation}
 commutes, when $\phi_\Pi$ is related to $\phi_i$'s and $\phi_\omega$ via the two metaplectic constructions $\tilde{\otimes}_i$ and $\tilde{\rm R}$.  The following is in analogue with \cite[\S 3.2, Conjecture]{Gan17}, and in fact is simpler.
\begin{conj} \label{C:F-odd}
Keep notations as above.
\begin{enumerate}
\item[(i)] The metaplectic tensor product $\tilde{\otimes}_i \pi_i$ defined above has an associated $L$-parameter exactly $(\phi_1 \times ... \times \phi_k \times \phi_0) \times \phi_\omega$, which is a parameter factoring through $i_M$ and thus is valued in $\wt{M}_\mbf{r}^\vee$.
\item[(ii)] The metaplectic restriction $\tilde{\rm R}(\Pi)$ above has $L$-parameter exactly $i_M \circ \phi_\Pi$. 
\end{enumerate}
\end{conj}

\begin{prop} \label{P:F-odd-ps}
Conjecture \ref{C:F-odd} holds when the $\pi_i$'s and $\Pi$ are genuine principal series.
\end{prop}
\begin{proof}
We denote by 
$$\wt{T}_{r_i} \subset \wt{\GL}_{r_i} \text{ and } \wt{T}_{r_0} \subset \wt{\GSp}_{2r_0}$$
the covering tori of the blocks $\wt{\GL}_{r_i}$ and $\wt{\GSp}_{2r_0}$ respectively. Every genuine principal series $\pi_i = I(i(\chi_i)) \in \Irrg(\wt{\GL}_{r_i})$ is associated with
$$\chi_i: Z(\wt{T}_{r_i}) \longrightarrow \C^\times,$$
a central genuine character; similarly, a genuine principal series $\pi_i = I(i(\chi_0)) \in \Irrg(\wt{\GSp}_{2r_0})$ is associated with 
$$\chi_0: Z(\wt{\GSp}_{2r_0}) \longrightarrow \C^\times.$$
Denote by $Y_{r_i}$ the character lattice of $\GL_{r_i}$, also by $Y_{r_0}$ that of $\GSp_{2r_0}$. 

We have
$$\mbm{L}(Z(\wt{T}_{r_i}) ) = nY_{r_i}, 1\lest i \lest k$$
and
$$\mbm{L}(Z(\wt{T}_{r_0}) ) = nY_{r_0}^{sc}\oplus \Z(ne_c),$$
where $Y_{r_0}^{sc} \subset Y_{r_0}$ is the coroot lattice of $\GSp_{2r_0}$. Since
$$\mbm{L}(Z(\wt{\GSp}_{2r}) ) = \Z n_{(r)}e_c \text{ and } \mbm{L}(Z(\wt{T})) = nY^{sc} \oplus \Z \cdot n_{(r)} e_c,$$
it follows that
$$\mbm{L}(Z(\wt{T})) =\Z n_{(r)}e_c + \left(\sum_{i=1}^k \mbm{L}(Z(\wt{T}_{r_i}) )  +\mbm{L}(Z(\wt{T}_{r_0}) )  \right).$$
Thus giving a genuine character $\chi: Z(\wt{T}) \longrightarrow \C^\times$ is equivalent to giving 
$$(\chi_{r_1}, ..., \chi_{r_k}, \chi_{r_0}; \omega) \in \Irrg(Z(\wt{T}_{r_1}) ) \times ... \times \Irrg(Z(\wt{T}_{r_k}) ) \times \Irrg(Z(\wt{T}_{r_0}) ) \times \Irrg(Z(\wt{\GSp}_{2r}) )$$
such that $\chi_{r_1} \boxtimes ... \chi_{r_k} \boxtimes \chi_{r_0}$ and $\omega$ agree on 
$$\left( Z(\wt{T}_{r_1}) \times_{\mu_n} ... \times Z(\wt{T}_{r_k}) \times_{\mu_n} Z(\wt{T}_{r_0}) \right) \cap Z(\wt{\GSp}_{2r}).$$
In particular, $\chi_{r_i}, 0\lest i \lest k$ and $\omega$ are obtained from restricting $\chi$ to the corresponding subgroups of $Z(\wt{T})$. The equality between the parameter $\chi_\Pi$ of $\Pi:=\Pi(\chi)$ and $\phi_1 \times ... \times \phi_k \times \phi_0 \times \phi_\omega$ then easily follows from functoriality of the local Langlands correspondence for covering tori with respect to restriction, as discussed in \S \ref{SS:func-ps}.
\end{proof}

\subsubsection{Remarks on the even $n$ case} \label{SSS:GSp-ev}
The metaplectic tensor product $\tilde{\otimes}$ and metaplectic restriction $\tilde{\rm R}$ for odd fold cover of $\GSp_{2r}$ do not generalize naively in the even fold cover case. Now, assuming $2|n$, we highlight the obstacles to the constructions in $\tilde{\otimes}$ and $\tilde{\rm R}$. Recall that we assume $\gcd(n, r_0)=1$.

Let $n=2m$. It follows from Proposition \ref{P:Gns} that ${\rm supp}(\Theta_{\pi_i}) \subset \wt{\GL}_{r_i}^\anga{m}$. However, there may exist $\pi_i$ such that
$${\rm supp}(\Theta_{\pi_i}) \nsubseteq \wt{\GL}_{r_i}^\anga{n}.$$
For such $\pi_i$, we may have ${}^g\pi_i \neq \pi_i$. This gives an obstruction to validating the constructions both in (S1)-(S3) and (RS1)-(RS2), as the seen from proof of Lemma \ref{L:Swd}. In fact, if $2|\gcd(n, r)$, then it follows from \eqref{E:ZM} that 
$$Z(\wt{\GSp}_{2r}) \cdot \wt{M}_\mbf{r}^\anga{n} = \wt{M}_\mbf{r}^{\anga{2n_{(r)}}},$$
which however is not equal to $\wt{M}_\mbf{r}^{\anga{n_{(r)}}}$. In this case,
\begin{equation} \label{E:ind-ev}
\Ind_{Z(\wt{\GSp}_{2r}) \cdot \wt{M}_\mbf{r}^\anga{n}}^{\wt{M}_\mbf{r}} = \bigoplus_{j \in F^\times/2} \Ind_{ \wt{M}_\mbf{r}^{\anga{n_{(r)}}} }^{\wt{M}_\mbf{r}} (\omega \boxtimes \sigma)_j = \bigoplus_j e_j\cdot \Pi_j,
\end{equation}
where $(\omega \boxtimes \sigma)_j$ is an extension of $(\omega \boxtimes \sigma)$ and 
$$F^\times/2:=e_0(F^{\times n_{(r)}}) / e_0(F^{\times 2n_{(r)}}) \simeq F^\times/F^{\times 2}.$$
The last equality in \eqref{E:ind-ev} follows from Proposition \ref{P:Gns} with $\Pi_j \in \Irrg(\wt{M}_\mbf{r})$. This gives another obstruction in the case $2|\gcd(n, r)$.

We also remark that the dual side seems to be deceivable when $n$ is even, and this provides a counterexample where the naive version of functoriality might fail for general $n$. We illustrate this below according to the parity of $\gcd(n,r)$.

In the first case, we assume $n=2m$ with $2\nmid \gcd(n, r)$. Setting
$$v_0: = n_{(r)}\cdot e_c \text{ and } v_i:=me_i + v_0/2, 1\lest i \lest r,$$
It is easy to check that 
$$Y_{Q,n}= \bigoplus_{i=0}^r \Z v_i$$
with $Y_{Q,n}^{sc}$ spanned by $\set{m\alpha_i^\vee: 1\lest i \lest r-1} \cup \set{n\alpha_r^\vee}$. Thus, for $2\nmid \gcd(n, r)$ one has
$$\wt{\GSp}_{2r}^\vee = \GSp_{2r}.$$
It follows that (given with the partition $\mbf{r}=(r_1, ..., r_k, r_0)$ of $r$)
$$\wt{M}_\mbf{r}^\vee \simeq \prod_{i=1}^k \GL_{r_i} \times \GSp_{2r_0} = \prod_{i=1}^k \wt{\GL}_{r_i}^\vee \times \wt{\GSp}_{2r_0}^\vee.$$
Analyzing the discussion in (S1)-(S3) and (RS1)-(RS2), one can define a \emph{correspondence} between  $\prod_{i=1}^k \Irrg(\wt{\GL}_{r_i}) \times \Irrg(\wt{\GSp}_{2r_0})$ and $\Irrg(\wt{M}_\mbf{r})$. However, there is no  \emph{map} (instead of correspondence) from one to the other. This indicates that a naive analogue of the functoriality might fail when considering representations of $\wt{M}_\mbf{r}$ and those of its covering blocks inside.

In the second case, we assume $2|\gcd(n, r)$, then it follows that
$$Y_{Q,n} = Y_{Q,n}^{sc} \oplus \Z\cdot v_0,$$
where $v_0= n_{(r)}e_c$ as above. In this case, 
$$\wt{\GSp}_{2r}^\vee = {\rm PGSp}_{2r} \times \GL_1$$
and thus one has a map
$$\prod_{i=1}^k \wt{\GL}_{r_i}^\vee \times \wt{\GSp}_{2r_0}^\vee \longrightarrow \wt{M}_\mbf{r}^\vee.$$
However, this map does not seem to be associated from the representation-side by a ``metaplectic tensor product" construction, the same as we remarked above for the $n=2m, 2\nmid \gcd(n, r)$ case.

\subsubsection{General metaplectic tensor product} \label{SS:gen-mtp}
From the example of Kazhdan--Patterson covers of $\GL_r$ and the odd-fold cover of $\GSp_{2r}$ discussed in this subsection, we see that the formulation of the metaplectic tensor product construction could be carried out in quite general setting, provided the following constraints are satisfied. To elaborate, assume for simplicity that
$$M=M_1 \times M_2 \times ... \times M_k$$
is a Levi subgroup of $G$. We like to have normal subgroups 
$$M_j^\natural \subset M_j, 1\lest j \lest k$$
such that every $M_j/M_j^\natural$ is a finite abelian group and satisfy the following two conditions:
\begin{enumerate}
\item[(i)] the blocks $\wt{M}_j^\natural, 1\lest j \lest k$ commute with each other,
\item[(ii)] ${\rm supp}(\Theta_\pi) \subset \prod_j \wt{M}_j^\natural$ for every $\pi \in \Irrg(\wt{M})$.
\end{enumerate}
If (i) and (ii) are satisfied, then the process described in steps (S1)--(S3) in \S \ref{SSS:GSp-odd} can be applied and one has a metaplectic tensor product construction following the same procedure. Note that the two conditions (i) and (ii) counteract against each other: for (i) it is desirable that $M_j^\natural$ is a smaller subgroup of $M_j$; on the other hand, for (ii) to be satisfied the group $M_j^\natural$ should be closer to $M_j$.
Thus, a family $\set{\wt{M}_j^\natural}_j$ satisfying both (i) and (ii) is a balancing choice, if it exists. 

For $G=\GL_r$ and $\GSp_{2r}$, the natural candidate for $M_j^\natural$ is essentially the subgroup of elements with determinant or similitude lying in $F^{\times n}$, respectively. Another reason of taking such $M_j^\natural$ is to ensure we have a natural (even conjectural) interpretation of the tensor product construction on the dual side.

\section{Local coefficients matrix under restriction} \label{S:LCM}

\subsection{Local coefficients matrix}
Let $\wt{B}=\wt{T}U$ be the Borel subgroup of $\wt{G}$, where $U$ is identified as a subgroup of $\wt{G}$ via the canonical splitting given by $e_\alpha(x) \mapsto \wt{e}_\alpha(x)$. Let $\psi: F \longrightarrow \C^\times$ be a nontrivial character. By abuse of notation, denote by
$$\psi: U \longrightarrow \C^\times$$
the unique character such that
$$\psi(e_\alpha(x)) = \psi(x)$$
for every $\alpha\in \Delta $ and $x\in F$. Let $(\pi, V_\pi) \in \Irrg(\wt{G})$ be a genuine irreducible representation of $\wt{G}$. A functional 
$$\lambda: V_\pi \longrightarrow \C$$
 is called a $\psi$-Whittaker functional if
$$\lambda( \pi(u) v) = \psi(u) \cdot \lambda(v)$$
for all $u\in U$ and $v \in V_\pi$. Denote by $\text{Wh}_\psi(\pi)$ the space of $\psi$-Whittaker functionals for $\pi$.
It follows from the work of Patel \cite{Pate}, which generalizes \cite{MW1} by M\oe glin-Waldspurger, that
$$\dim \Wh_\psi(\pi) < \infty  \text{ for every } \pi \in \Irrg(\wt{G}).$$
See \cite[Theorem I.5.2 (i)]{KP} also for a proof of this finite dimensionality for covers of $\GL_r$.

We briefly recall the definition of a local coefficients matrix as in \cite{Szp6, GSS2}. Let
$$\mbf{P}=\mbf{M} \mbf{N} \subset \mbf{G}$$ be a parabolic subgroup of $\mbf{G}$ associated with $\theta \subset \Delta$. 
Let $\sigma \in \Irrg(\wt{M})$. We have the parabolic subgroups $\wt{P} =\wt{M}N \subset \wt{G}$ and $\wt{P}_0 = \wt{M}_0 N \subset \wt{G}_0$, see \S \ref{SSS:2pic}.

By dualizing, the intertwining operator $T(w,\sigma)$ in \eqref{T(w)}, whenever holomorphic at $\sigma$, gives a linear map between finite dimensional vector spaces:
$$\begin{tikzcd}
T(w,\sigma)^*: \  \Wh_\psi(I_{\wt{P}'}^{\wt{G}}({}^w\sigma)) \ar[r]  & \Wh_\psi( I_{\wt{P}}^{\wt{G}}(\sigma)).
\end{tikzcd}$$

Let $\psi_M$ be the restriction of $\psi$ to the unipotent radical $U_M$ of the Borel subgroup $TU_M$ of $M$. Here $\psi_M$ and $w$ are compatible (in the sense of \cite[Page 51]{Sha4}). We have a natural isomorphism (see \cite{CS, Rod1}, \cite[\S 3]{Sha2} or \cite[\S 3.3]{Sha4})
$$J(\sigma): \Wh_{\psi_M}(\sigma) \longrightarrow \Wh_\psi(I_{\wt{P}}^{\wt{G}}(\sigma))$$
given by
$$J(\sigma)(\lambda)(f) = \int_{N'} \lambda( f(\wt{w}_0^{-1} n')  ) \cdot \psi(n') dn',$$
where $N'=w_0 N^- w_0^{-1}$ with $w_0 = w_l \cdot w_{l,M}$. Here $w_l$ (resp. $w_{l,M}$) is the longest
Weyl element in the Weyl group $W=W(T, G)$ (resp. $W_M=W(T, M)$).  

Similarly, one has
$$J({}^w \sigma): \Wh_{\psi_{M'}}({}^w\sigma) \longrightarrow \Wh_\psi(I_{\wt{P}'}^{\wt{G}}(\sigma)).$$
At the same time, since $V_{{}^w\sigma} =V_\sigma$, we also have a canonical identification
$$\iota: \Wh_{\psi_{M}} (\sigma)=\Wh_{\psi_{M'}} ({}^w\sigma),$$
which we briefly explained as follows. First, note that $l \in \Wh_{\psi_{M}} (\sigma)$ if and only if 
$$l(\sigma(u)(v)) = \psi_M(u) \cdot v$$
for all $u\in U_M, v\in V_\sigma$. For every $u' \in U_{M'}$ one has $w^{-1} u' w \in U_M$ and
$$\psi_{M'}(u') = \psi(u') = \psi(w^{-1} u' w) = \psi_M(w^{-1} u' w).$$
It is verified easily (cf. \cite[\S 3.2]{GSS2}) that
$$l({}^w\sigma (u') (v)) = l(\sigma(w^{-1} u' w)(v)) = \psi_M(w^{-1} u' w)\cdot v = \psi_{M'}(u')\cdot v$$
for all $u'\in U_{M'}$ and $v\in V_{{}^w\sigma} = V_\sigma$. This shows that $l \in \Wh_{\psi_{M'}}( {}^w \sigma)$ and gives the canonical identification $\iota$.

This gives a natural isomorphism of vector spaces
$$\begin{tikzcd}
\mca{C}(w, \sigma)=J({}^w\sigma) \circ \iota \circ J(\sigma)^{-1} : \  \Wh_\psi( I_{\wt{P}}^{\wt{G}}(\sigma))  \ar[r]  & \Wh_\psi(I_{\wt{P}'}^{\wt{G}}({}^w\sigma)).
\end{tikzcd}$$
We therefore obtain an endomorphism
\begin{equation} \label{End-T}
\begin{tikzcd}
\mca{T}(w,\sigma)^*= T(w,\sigma)^* \circ \mca{C}(w, \sigma) : \  \Wh_{\psi}( I_{\wt{P}}^{\wt{G}} (\sigma) ) \ar[r]  & \Wh_{\psi}( I_{\wt{P}}^{\wt{G}}(\sigma) )
\end{tikzcd}
\end{equation}
of the finite dimensional vector space $\Wh_{\psi}( I_{\wt{P}}^{\wt{G}} (\sigma) )$. A local coefficients matrix associated to $(\wt{P}, w, \sigma)$ is the matrix 
$$\mca{M}_\mfr{B}(w, \sigma)$$ representing $\mca{T}(w, \sigma)^*$ with respect to an ordered basis $\mfr{B} \subset \Wh_{\psi}( I_{\wt{P}}^{\wt{G}} (\sigma) )$.

\begin{prop} \label{P:decWh}
Fix $\sigma \in \Irrg(\wt{M})$. There is a local coefficients matrix $\mca{M}_\mfr{B}(w, \sigma)$ of the form
$$\bigoplus_{\sigma_i \subset \sigma|_{\wt{M}_0}} \mca{M}_{\mfr{B}_i}(w, \sigma_i),$$
where $\mca{M}_{\mfr{B}_i}(w, \sigma_i)$ is a local coefficients matrix associated with the triple $(\wt{P}_0, w, \sigma_i)$ and a basis $\mfr{B}_i$ of $\Wh_\psi( I_{\wt{P}_0}^{\wt{G}_0} (\sigma_i))$. In particular, for the two invariants trace and determinant, one has
\begin{equation} \label{tr-det}
{\rm Tr}(\mca{M}_\mfr{B}(w, \sigma)) = \sum_{\sigma_i \subset \sigma|_{\wt{M}_0}} {\rm Tr}(\mca{M}_{\mfr{B}_i}(w, \sigma_i)), \quad {\rm det}(\mca{M}_\mfr{B}(w, \sigma)) = \prod_{\sigma_i \subset \sigma|_{\wt{M}_0}} {\rm det}(\mca{M}_{\mfr{B}_i}(w, \sigma_i)).
\end{equation}
\end{prop}
\begin{proof} Clearly, it suffices to show that $\mca{T}(w, \sigma)^*$ decomposes as a direct sum of the operators $\mca{T}(w, \sigma_i)^*$. First, we have an isomorphism of vector spaces
\begin{equation} \label{W-dec}
\Wh_\psi(I_{\wt{P}}^{\wt{G}}(\sigma)) \simeq \bigoplus_{\sigma_i \subset \sigma|_{\wt{M}_0}} \Wh_\psi(I_{\wt{P}_0}^{\wt{G}_0}(\sigma_i)),
\end{equation}
since $\Wh_\psi(-)$ is the dual of the twisted Jacquet module, which depends only on $\psi: U \to \C^\times$ and is an exact functor. 
It follows from Lemma \ref{L:kc} that 
$$T(w, \sigma)^* = \bigoplus_{\sigma_i \subset \sigma|_{\wt{M}_0}} T(w, \sigma_i)^*.$$
Now we consider the decomposition of $\mca{C}(w, \sigma)$ over the $\mca{C}(w, \sigma_i)$'s with respect to the decomposition in \eqref{W-dec}. For this purpose, we note first for each $\sigma_i \subset \sigma|_{\wt{M}_0}$ the isomorphism
$$J(\sigma_i): \Wh_{\psi_M}(\sigma_i) \longrightarrow \Wh_\psi( I_{\wt{P}_0}^{\wt{G}_0} (\sigma_i))$$
and the decomposition
$$\Wh_{\psi_M}(\sigma) = \bigoplus_{\sigma_i \subset \sigma|_{\wt{M}_0}} \Wh_{\psi_M} (\sigma_i).$$
It is easy to see from the defining formula of $J(\sigma_i)$ that
$$J(\sigma) = \bigoplus_{\sigma_i \subset \sigma|_{\wt{M}_0}} J(\sigma_i).$$
Similar decomposition for $J({}^w \sigma)$ holds. Lastly, the canonical identification $\iota$ also decomposes:
$$\iota = \bigoplus_{\sigma_i \subset \sigma|_{\wt{M}_0}} \iota_{\sigma_i},$$
where $\iota_{\sigma_i}: \Wh_{\psi_M}(\sigma_i) \to \Wh_{\psi_{M'}}({}^w \sigma_i)$ is the canonical identification.
Thus, the isomorphism $\mca{C}(w,  \sigma)$ decomposes as a direct product of the $\mca{C}(w, \sigma_i)$'s. Hence we have
$$\mca{T}(w, \sigma)^* = \bigoplus_{\sigma_i \subset \sigma|_{\wt{M}_0}} \mca{T}(w, \sigma_i)^*,$$
which gives the desired result.
\end{proof}

The above proposition shows that the invariants associated to a local coefficients matrix $\mca{M}_\mfr{B}(w, \sigma)$ is a priori determined by all the local coefficients matrices $\mca{M}_{\mfr{B}_i}(w, \sigma_i)$ arising from the restriction of $\sigma$ to $\wt{M}_0$.  This has a striking consequence even in the case where $\sigma = i(\chi)$ is a genuine representation of $\wt{T} \subset \wt{G}$, where $\chi$ is a central character of $Z(\wt{T})$. Indeed, in this case, the two invariants 
$${\rm Tr}(\mca{M}_\mfr{B}(w_\alpha, \sigma)) \text{ and } {\rm det}(\mca{M}_\mfr{B}(w_\alpha, \sigma))$$
 associated with a simple reflection $w_\alpha$ are computed in \cite{GoSz, Szp6, GSS1, GSS2}. They are expressed in terms of the Plancherel measure and gamma or metaplectic-gamma factors associated with $i(\chi)$; in particular, these two invariants depend only on the central character $\chi$.

It is expected that ${\rm Tr}(\mca{M}_{\mfr{B}_i}(w_\alpha, \sigma_i))$ and ${\rm det}(\mca{M}_{\mfr{B}_i}(w_\alpha, \sigma_i))$ depend on $\sigma_i$ only. Note that even if $\sigma \in \wt{G}$ is an unramified representation, it is possible to have a ramified constituent $\sigma_i \subset \sigma|_{\wt{T}_0}$. In this case, the equalities in \eqref{tr-det} entail that there is a cancellation of the ramified information encoded in the ramified constituents of $\sigma|_{\wt{T}_0}$. A prototype of such examples is the pair $(\wt{\GL}_2^{(n)}, \wt{\SL}_2^{(n)})$ with even $n$, see \cite[\S 8--\S 9]{GSS2}.

A special case is the following.
\begin{cor} \label{C:isoWh}
Let $\sigma \in \Irrg(\wt{M})$. Assume that $\sigma|_{\wt{M}_0} = \sigma_0^{\oplus m}$ is an isotypic sum of $\sigma_0 \in \Irrg(\wt{M}_0)$. Then $$\dim \Wh_{\psi_M}(\sigma) = m\cdot \dim \Wh_{\psi_M}(\sigma)$$ and moreover
$\mca{M}_{\mfr{B}_0}(w, \sigma_0)^{\oplus m}$ is a local coefficients matrix for $\mca{T}(w, \sigma)^*$, where $\mca{M}_{\mfr{B}_0}(w, \sigma_0)$ is a local coefficients matrix for $\mca{T}(w, \sigma_0)^*$. In particular, 
$${\rm Tr}(\mca{M}_\mfr{B}(w, \sigma)) = m \cdot {\rm Tr}(\mca{M}_{\mfr{B}_0}(w, \sigma_0)), \quad {\rm det}(\mca{M}_\mfr{B}(w, \sigma))= {\rm det}(\mca{M}_{\mfr{B}_0}(w, \sigma_0))^m.$$
\end{cor}
One immediate application of Corollary \ref{C:isoWh} is the computation of the multiplicity $m$ by considering the Whittaker dimensions. Indeed, such multiplicity (or ramification index $e$ in the discussion of \S \ref{SSS:res}) is most often a subtle quantity to explicate.


\subsection{Genuine principal series}
We assume $p\nmid n$, and fix a splitting $s_K$ of  $\wt{G}$ over $K=\mbf{G}(O)$. We consider a $(K, s_K)$-unramified genuine principal series $I(\chi)= I(i(\chi))$ of $\wt{G}$, where
$$\chi: Z(\wt{T}) \longrightarrow \C^\times$$
is an unramified character, i.e., it is trivial on $Z(\wt{T}) \cap K$. In this case, as discussed in \S \ref{SS:res-ps}, the group 
$$\wt{A}=Z(\wt{T})\cdot \mbf{T}(O) \subset \wt{T}$$
is a maximal abelian subgroup. There is a unique extension also denoted by 
$$\tchi: \wt{A} \longrightarrow \C^\times,$$
which is trivial on $\mbf{T}(O)$. From now on, we will take
$$i(\chi) = \Ind_{\wt{A}}^{\wt{T}} (\tchi)$$
for this $\tchi$, which gives the unramified principal series $I(\chi) = I(i(\chi))$.

 It follows from Theorem \ref{T:decT} that one has 
$$I(\chi)|_{\wt{G}_0}= \val{\msc{X}_{Q,n}^\mfr{c}} \cdot \bigoplus_{\gamma \in \msc{X}_{Q,n}^\Gamma/\msc{X}_{Q,n}^\mfr{c}} \bigoplus_{\omega_{\gamma, j} \in \msc{E}(\chi, {}^\gamma \tchi_O; Z(\wt{T}_0))} I(\omega_{\gamma, j}),$$
where generically $I(\omega_{\gamma, j})$ is a multiplicity-free genuine principal series of $\wt{G}_0$  appearning in the double sum of the right hand side. We note that 
$$\dim \Wh_\psi(I(\omega_{\gamma, j})) = \val{ \msc{X}_{0,Q,n} }$$
for every $\omega_{\gamma, j} \in \msc{E}(\chi, {}^\gamma \tchi_O; Z(\wt{T}_0))$, where 
$$\msc{X}_{0,Q,n}:=Y_0/Y_{0,Q,n}.$$
 One the other hand, we have
$$\dim \Wh_\psi(I(\chi)) = \val{ \msc{X}_{Q,n} }.$$
The relations among various quotients of lattices are illustrated as follows.
\begin{equation} \label{Lat}
 \begin{tikzcd}
Y_{0,Q,n}/(Y_0 \cap Y_{Q,n}) \ar[d, hook] & & Y/Y_0  \ar[d, two heads] \\
Y_0/(Y_0 \cap Y_{Q,n}) \ar[r, hook] \ar[d, two heads] & \msc{X}_{Q,n} \ar[r, two heads, "f"] & \msc{X}_{Q,n}^\Gamma \\
\msc{X}_{0,Q,n}.
\end{tikzcd}
\end{equation}

If $Z(\wt{T}_0) \subset Z(\wt{T})$, or equivalently $Y_{0,Q,n} = Y_0 \cap Y_{Q,n}$ by Corollary \ref{C:Zinc}, then the above  diagram becomes
\begin{equation} \label{Lat-1}
\begin{tikzcd}
\msc{X}_{0,Q,n} \ar[r, hook] & \msc{X}_{Q,n} \ar[r, two heads, "f"] & \msc{X}_{Q,n}^\Gamma.
\end{tikzcd}
\end{equation}
Since $\msc{X}_{Q,n}$ is the ``moduli space" of $\Wh_\psi(I(\chi))$, i.e., there is a natural parametrization of the Whittaker space by $\msc{X}_{Q,n}$ (see \cite[\S 3.4]{GSS2}),  we see that there is a decomposition of $\msc{X}_{Q,n}$ into the disjoint union of the fibers of $f$, and thus a corresponding decomposition of $\Wh_\psi(I(\chi))$.
Each fibre of $f$ is a torsor over $\msc{X}_{0,Q,n}$ and in particular of the same size as 
$$\val{\msc{X}_{0,Q,n}} = \dim \Wh_\psi(I(\chi|_{Z(\wt{T}_0)})).$$
 This decomposition is compatible with the fact that $I(\chi)|_{\wt{G}_0} = I(\chi|_{Z(\wt{T}_0)})^{\oplus \val{\msc{X}_{Q,n}^\Gamma}}$ is an isotypic sum, see Corollary \ref{C:decTi} and Corollary \ref{C:isoWh}.

In the remaining of this section, we assume
$$\mfr{f}(\psi) = O_F$$
and investigate $\dim \Wh_\psi(\pi)$ and $\dim \Wh_\psi(\pi_0)$, where $\pi$ is an irreducible constituent of a regular or unitary unramified principal series and $\pi_0 \subset \pi|_{\wt{G}}$.

\subsection{Regular unramified principal series} \label{SS:reg-ps}
In this subsection, we consider the restriction to $\wt{G}_0$ of an irreducible constituent $\pi$ of a regular unramified genuine principal series $I(\chi)$ of $\wt{G}$. In particular, we will study how the Whittaker dimension $\dim \Wh_\psi(\pi)$ is distributed into the constituents of $\pi|_{\wt{G}_0}$.

Temporarily, we use $\wt{G}^\sharp$ to denote either $\wt{G}$ or $\wt{G}_0$, and $\wt{T}^\sharp \subset \wt{G}^\sharp$ is the covering torus. Let
$$\chi^\sharp: Z(\wt{T}^\sharp) \longrightarrow \C^\times$$
be a central character satisfying:
\begin{enumerate}
\item[--] if $\wt{T}^\sharp=\wt{T}$, then $\chi^\sharp$ is unramified;
\item[--] if $\wt{T}^\sharp = \wt{T}_0$, then $i(\chi^\sharp) \subset i(\chi)|_{\wt{T}_0}$ for an unramified $\chi: Z(\wt{T}) \longrightarrow \C^\times$.
\end{enumerate}
We highlight that in the second case above, $\chi^\sharp$ may not be unramified, unless $(\wt{G}, \wt{G}_0)$ is an isotypic pair; also, by Lemma \ref{L:incs} we have $i(\chi^\sharp) \subset i(\chi)|_{\wt{T}_0}$ if and only if $\chi^\sharp|_{\wt{T}_0\cap Z(\wt{T})} = \chi|_{\wt{T}_0\cap Z(\wt{T})}$. Denoting 
$$Y_{Q,n}^\sharp:=Y_0 \cap Y_{Q,n}$$
and let $\wt{T}_{Q,n}^\sharp$ be the preimage in $\wt{T}$ of ${\rm Im}(i_{Q,n}^\sharp) \subset T$, where
$$i_{Q,n}^\sharp: \ Y_{Q,n}^\sharp \otimes F^\times \longrightarrow Y_{Q,n}\otimes F^\times \longrightarrow Y\otimes F^\times$$
is induced from the inclusions $Y_{Q,n}^\sharp \subset Y_{Q,n} \subset Y$. One has 
$$\wt{T}_{Q,n}^\sharp \subset Z(\wt{T}).$$
Similarly, we have the subgroup $\wt{T}_{Q,n}^{sc} \subset Z(\wt{T})$ associated with $Y_{Q,n}^{sc} \subset Y_{Q,n}$. In any case, since
$$\wt{T}_{Q,n}^\sharp = \wt{T}_0 \cap Z(\wt{T}) = Z(\wt{T}_0) \cap Z(\wt{T}),$$
we see that every $\chi^\sharp$ is unramified when restricted to $\wt{T}_{Q,n}^\sharp$. In particular, the restriction
$$\chi^\sharp|_{\wt{T}_{Q,n}^{sc}}: \wt{T}_{Q,n}^{sc} \longrightarrow \C^\times$$
is always unramified.

In this subsection, we assume that $\chi^\sharp$ is regular, i.e., the stabilizer subgroup of $\chi^\sharp$ in the Weyl group is trivial.  Thus the irreducible constituents of $I(\chi^\sharp)$ are multiplicity-free. We denote its Jordan--Holder set by ${\rm JH}(I(\chi^\sharp))$. Consider
$$\Phi(\chi^\sharp):= \set{\alpha \in \Phi: \chi^\sharp(\wt{h}_\alpha(\varpi^{n_\alpha})) = q^{-1}} \subset \Phi.$$
Let $\msc{P}(\Phi(\chi^\sharp))$ be the power set of $\Phi(\chi^\sharp)$, and denote by
$$\msc{C}(X\otimes \R; \chi^\sharp)$$
the set of connected components of 
$$X\otimes \R - \bigcup_{\alpha \in \Phi(\chi^\sharp)} {\rm Ker}(\alpha^\vee).$$

\begin{prop} \label{P:Rod}
There are bijections between the three sets
$$\msc{P}(\Phi(\chi^\sharp)) \longleftrightarrow {\rm JH}(I(\chi^\sharp)) \longleftrightarrow \msc{C}(X\otimes \R; \chi^\sharp)$$
denoted by
$$S \leftrightarrow \pi_S \leftrightarrow \Gamma_S,$$
which is given as follows. First, we have
$$\Gamma_S = \set{x\in X\otimes \R: \angb{\alpha^\vee}{x} <0 \text{ if and only if } \alpha \in S}.$$
Second, the representation $\pi_S$ is characterized by its Jacquet module
$$(\pi_S)_U = \bigoplus_{w \in W_S} \delta_B^{1/2} \cdot i({}^{w^{-1}} \chi^\sharp),$$
where 
$$W_S = \set{w\in W: \Phi(\chi^\sharp)^\vee \cap w(\Phi_-^\vee) = S^\vee} \subset W.$$
\end{prop}
\begin{proof}
For linear algebraic group and a regular character $\chi$ of $T\subset G$ (not necessarily unramified), the result was shown in \cite{Rod4}. For covering groups and unramified $\chi$, it was extended in \cite{Ga6}. We explain how the argument in \cite{Ga6}, which is adopted from \cite{Rod4}, actually works for $\chi^\sharp$ here.

Indeed, the key property needed in the proof in \cite{Ga6} is the computation of the Plancherel measure $\mu(w_\alpha, i(\chi^\sharp))$
associated with $T(w_\alpha, i(\chi^\sharp))$ with $w_\alpha$ being a simple reflection. For unramified $\chi^\sharp$, this was explicitly computed in \cite{Ga1} and used in \cite[Proposition 3.5]{Ga6}; it shows that 
\begin{equation} \label{E:Plan}
\mu(w_\alpha, i(\chi^\sharp))^{-1} = \frac{ 1-q^{-1} \chi^\sharp(\wt{h}_\alpha(\varpi^{n_\alpha}))}{1- \chi^\sharp(\wt{h}_\alpha(\varpi^{n_\alpha}))} \cdot \frac{ 1-q^{-1} \chi^\sharp(\wt{h}_\alpha(\varpi^{n_\alpha}))^{-1}}{1- \chi^\sharp(\wt{h}_\alpha(\varpi^{n_\alpha}))^{-1}}.
\end{equation}
One thus concludes that the poles of $\mu(w_\alpha, i(\chi^\sharp))$ are achieved when $\alpha \in \Phi(\chi^\sharp)$ or $-\alpha \in \Phi(\chi^\sharp)$. This is the key ingredient used in the proof in \cite{Rod4}, the rest of which applies in the covering setting.

Now if $\chi^\sharp$ is for $\wt{G}_0$ with $i(\chi^\sharp) \subset i(\chi)|_{\wt{T}}$, then we can proceed in two ways in order to determine $\mu(w_\alpha, i(\chi^\sharp))$. First, one can invoke \cite[Theorem 5.1]{GoSz} directly. Alternatively, it follows from Lemma \ref{L:kc} that 
$$\mu(w_\alpha, i(\chi^\sharp)) = \mu(w_\alpha, i(\chi_0))$$
for every $\omega: Z(\wt{T}_0) \longrightarrow \C^\times$ such that $i(\omega) \subset i(\chi)|_{\wt{T}_0}$. In particular, considering the unramified $\omega$ and using the formula of $\mu(w_\alpha, i(\omega))$ in \cite{Ga1}, one also obtains the same formula \eqref{E:Plan} for $\mu(w_\alpha, i(\chi^\sharp))$. Now the same argument in \cite[Proposition 3.5]{Ga6} applies to give the desired result.
\end{proof}

For every $S \subset \Phi(\chi^\sharp)$ we denote by 
$$\pi_{\chi^\sharp, S} =\pi(\chi^\sharp)_S \in {\rm JH}(I(\chi^\sharp))$$
the associated  irreducible constituent of $I(\chi^\sharp)$.

\begin{cor}
Let $\chi: Z(\wt{T}) \to \C^\times$ be a regular unramified genuine character.
\begin{enumerate}
\item[(i)] For every $\omega: Z(\wt{T}_0) \longrightarrow \C^\times$ such that $i(\omega) \subset i(\chi)|_{\wt{T}_0}$, one has
$$\Phi(\chi) = \Phi(\omega).$$
\item[(ii)] For every $S \subset \Phi(\chi)$, we have
$$(\pi_{\chi, S})|_{\wt{G}_0}= \val{\msc{X}_{Q,n}^\mfr{c}} \cdot \bigoplus_{\gamma \in \msc{X}_{Q,n}^\Gamma/\msc{X}_{Q,n}^\mfr{c}} \bigoplus_{\omega_{\gamma, j} \in \msc{E}(\chi, {}^\gamma \tchi_O; Z(\wt{T}_0))} \pi_{\omega_{\gamma, j}, S}.$$
\end{enumerate}
\end{cor}
\begin{proof}
Part (i) follows from the proof of Proposition \ref{P:Rod}, since $\mu(w, i(\chi)) = \mu(w, i(\omega))$ for such $\omega$ and depends only on $\chi|_{\wt{T}_{Q,n}^{sc}} = \omega|_{\wt{T}_{Q,n}^{sc}}$. For (ii), we note that the representation $\pi_{\chi, S}$ is characterized by its Jacquet module, and since computing Jacquet module commutes with restriction to $\wt{G}_0$, it suffices to determine
$${\rm Res}_{\wt{T}_0}^{\wt{T}}(\pi_{\chi, S})_U = \bigoplus_{w \in W_S} \delta_B^{1/2} \cdot i({}^{w^{-1}} \chi)|_{\wt{T}_0}.$$
Since
$$i(\chi)|_{\wt{T}_0} = \val{\msc{X}_{Q,n}^\mfr{c}} \cdot \bigoplus_{\gamma \in \msc{X}_{Q,n}^\Gamma/\msc{X}_{Q,n}^\mfr{c}} \bigoplus_{\omega_{\gamma, j} \in \msc{E}(\chi, {}^\gamma \tchi_O; Z(\wt{T}_0))} i(\omega_{\gamma, j}),$$
and similar decomposition for ${}^{w^{-1}} \chi$ for each $w\in W_S$ holds, we get
$$\begin{aligned}[t]
& {\rm Res}_{\wt{T}_0}^{\wt{T}}(\pi_{\chi, S})_U \\
= & \val{\msc{X}_{Q,n}^\mfr{c}} \cdot \bigoplus_{\gamma \in \msc{X}_{Q,n}^\Gamma/\msc{X}_{Q,n}^\mfr{c}} \bigoplus_{\omega_{\gamma, j} \in \msc{E}(\chi, {}^\gamma \tchi_O; Z(\wt{T}_0))} \bigoplus_{w\in W_S} \delta_{B_0}^{1/2} \cdot i({}^{w^{-1}}\omega_{\gamma, j}) \\
= & \val{\msc{X}_{Q,n}^\mfr{c}} \cdot \bigoplus_{\gamma \in \msc{X}_{Q,n}^\Gamma/\msc{X}_{Q,n}^\mfr{c}} \bigoplus_{\omega_{\gamma, j} \in \msc{E}(\chi, {}^\gamma \chi_O; Z(\wt{T}_0))} (\pi_{\omega_{\gamma, j}, S})_U.
\end{aligned}
$$
As noted, the representation $\pi_{\omega_{\gamma, j}, S}$ is uniquely determined by its Jacquet module $(\pi_{\omega_{\gamma, j}, S})_U$, this shows (ii) and concludes the proof.
\end{proof}

For a regular $(K, s_K)$-unramified $I(\chi)$, we want to determine the Whittaker space $\Wh_\psi(\pi_{\chi, S})$ in terms of the $\Wh_\psi(\pi_{\omega_{\gamma, j}, S})$'s. As in general there is much difficulty, in the remaining part of this subsection, we will consider two special cases:
\begin{enumerate}
\item[--] when $Z(\wt{T}_0) \subset Z(\wt{T})$, and thus $\pi_{\chi, S}|_{\wt{G}_0} = \val{\msc{X}_{Q,n}^\Gamma} \cdot \pi_{\omega, S}$ is an isotypic sum, where $\omega = \chi|_{Z(\wt{T}_0)}$;
\item[--] when the pair is $(\wt{\GL}_2^{(n)}, \wt{\SL}_2^{(n)})$ with arbitrary $n$.
\end{enumerate} 
We hope to illustrate the subtleties on determining even the relation between $\dim \Wh_\psi(\pi_{\chi, S})$ and  $\dim \Wh_\psi(\pi_{\omega_{\gamma, j}, S})$.

\subsubsection{The isotypic case}
If $(\wt{G}, \wt{G}_0)$ is an isotypic pair, then we denote by 
$$\omega:=\chi|_{Z(\wt{T}_0)}$$
the unramified character of $Z(\wt{T}_0)$. We have $\pi_{\chi, S}|_{\wt{G}_0} = \val{\msc{X}_{Q,n}^\Gamma} \cdot \pi_{\omega, S}$ and
$$\dim \Wh_\psi(\pi_{\chi, S}) = \val{\msc{X}_{Q,n}^\Gamma} \cdot \dim \Wh_\psi(\pi_{\omega, S}).$$
Now we want to give an interpretation of this equality from the perspective of the short exact sequence \eqref{Lat-1}. 

\begin{lm} \label{L:pers}
Assume that $\wt{G}_0$ is a saturated cover. Then
\begin{enumerate}
\item[(i)] $\wt{G}$ is also a saturated cover, and thus both $\wt{G}_0$ and $\wt{G}$ are persistent covers;
\item[(ii)] one has $Z(\wt{T}_0) \subset Z(\wt{T})$ in this case.
\end{enumerate}
\end{lm}
\begin{proof}
Recall that we have the inclusions
$$Y_{Q,n}^{sc} \subset (Y^{sc} \cap Y_{Q,n}) \subset (Y_0 \cap Y_{Q,n}) \subset Y_{0, Q,n}.$$
The assumption implies that $Y_{Q,n}^{sc} = Y_{0, Q,n}$. This enforces the equality $(Y_0 \cap Y_{Q,n})= Y_{0, Q,n}$ and thus (ii) holds. This also enforces another equality $Y_{Q,n}^{sc} = Y^{sc} \cap Y_{Q,n}$, which exactly shows that $\wt{G}$ is saturated and gives (i).
\end{proof}

Let $s$ be any splitting, if it exists, of the short exact sequence of finite abelian groups:
\begin{equation} \label{Lat-2}
\begin{tikzcd}
\msc{X}_{0,Q,n} \ar[r, hook] & \msc{X}_{Q,n} \ar[r, two heads, "f"] & \msc{X}_{Q,n}^\Gamma \ar[l, bend left=30, "s"].
\end{tikzcd}
\end{equation}
Every such splitting gives a set  $s(\msc{X}_{Q,n}^\Gamma)$ of representatives of $\msc{X}_{Q,n}^\Gamma$ in $\msc{X}_{Q,n}$.
We have
$$\msc{X}_{Q,n} = \bigsqcup_{z' \in \msc{X}_{Q,n}^\Gamma} f^{-1}(z') = \bigsqcup_{z\in s(\msc{X}_{Q,n}^\Gamma)} (z + \msc{X}_{0,Q,n}).$$
Assume $\wt{G}_0$ is a saturated cover. Also assume that 
$$\Phi(\chi) \subset \Delta.$$
It then follows from Lemma \ref{L:pers} and \cite[Theorem 6.6]{Ga6} that
$$\dim \Wh_\psi(\pi_{\chi, S}) = \angb{\sigma_\msc{X}}{ \sigma_S }_W, \quad \dim \Wh_\psi(\pi_{\omega, S}) = \angb{ \sigma_{\msc{X}_0} }{ \sigma_S }_W.$$
Here $\sigma_S$ is a certain Kazhdan--Lusztig representation of $W$ of dimension $\val{W_S}$. On the other hand,
$$\sigma_\msc{X}: W \longrightarrow {\rm Perm}(\msc{X}_{Q,n})$$
is the permutation representation of $W$ given by $\sigma_\msc{X}(w)(y):= w[y]$, where 
$$w[y]= w(y-\rho) + \rho.$$
Similarly,
$$\sigma_{\msc{X}_0}: W \longrightarrow {\rm Perm}(\msc{X}_{0,Q,n})$$
is also given by 
$$\sigma_{\msc{X}_0}(w)(y) = w[y]$$
for every $y\in \msc{X}_{0,Q,n}$. In fact, $\sigma_{\msc{X}_0}$ is the subrepresentation of $\sigma_\msc{X}$ restricted on $\msc{X}_{0,Q,n} \subset \msc{X}_{Q,n}$.

For every coset $z + \msc{X}_{0,Q,n} \subset \msc{X}_{Q,n}$, the action $w[\cdot]$ is given by
$$w[z+ y] = w(z) + w[y] = z + (w[y] + w(z) - z)$$
for every $y\in \msc{X}_{0,Q,n}$, where $w(z)$ denotes the usual action arising from reflection.
As the image of $w(z)- z$ in $\msc{X}_{Q,n}^\Gamma$ via $f$ is trivial, we see that 
$$w(z) - z \in \msc{X}_{0,Q,n}.$$
Thus, $w[z+ y]$ lies in $z + \msc{X}_{0,Q,n}$; that is, the action $w[\cdot]$ on $z + \msc{X}_{0,Q,n}$ is well-defined. This gives a permutation representation
$$\sigma_{z + \msc{X}_0}: W \longrightarrow {\rm Perm}(z + \msc{X}_{0,Q,n}).$$
On the other hand, we consider the permutation representation
$$\sigma_{\msc{X}_0, z}: W \longrightarrow {\rm Perm}(\msc{X}_{0,Q,n})$$
given by
$$\sigma_{\msc{X}_0, z}(w)(y) = w\langle y \rangle:= w[y] + w(z) -z.$$
It is easy to check that the action $w\langle  \cdot \rangle$ is well-defined. Thus $\sigma_{\msc{X}_0, z}$ is a well-defined permutation representation of $W$ on $\msc{X}_{0,Q,n}$. In fact, we have:
\begin{enumerate}
\item[(i)] The representation $\sigma_{z + \msc{X}_0}$ is the one obtained from transporting $\sigma_{\msc{X}_0, z}$ via the translation $$\msc{X}_{0,Q,n} \to z + \msc{X}_{0,Q,n}$$ 
given by $y \mapsto z + y$. Thus, as representations of the Weyl group, $\sigma_{z+ \msc{X}_0}$ and $\sigma_{\msc{X}_0, z}$ are equivalent. In particular, we have
$$\sigma_{\msc{X}} = \bigoplus_{z \in s(\msc{X}_{Q,n}^\Gamma)} \sigma_{\msc{X}_0, z}.$$
\item[(ii)] Denote by 
$$\rho_z: W \longrightarrow {\rm Perm}(\msc{X}_{0,Q,n})$$
the representation given by the translation $\rho_z(w)(y) = y + w(z) - z$. It is clear that $\sigma_{\msc{X}_0, z} = \rho_z \circ \sigma_{\msc{X}_0}$.
\end{enumerate}

We note that the exact sequence \eqref{Lat-2} is $W$-equivariant with respect to the usual action $w(\cdot)$. Also, the action of $W$ on $\msc{X}_{Q,n}^\Gamma$ is trivial. In general, the splitting $s$ may not be $W$-equivariant. However, we believe that in our setting this is indeed the case:
\begin{conj} \label{C:dec}
Assume that $\wt{G}_0$ is a saturated cover. Then there exists a $W$-equivariant splitting $s: \msc{X}_{Q,n}^\Gamma \to \msc{X}_{Q,n}$ with respect to the Weyl action $w(\cdot)$.
Consequently, with respect to such a splitting $s$, one has  
$$\sigma_{\msc{X}_0, z} \simeq \sigma_{\msc{X}_0}$$
for every $z\in s(\msc{X}_{Q,n}^\Gamma)$ and thus $\sigma_\msc{X} =  \val{ \msc{X}_{Q,n}^\Gamma } \cdot \sigma_{\msc{X}_0 }$. In particular, $\dim \Wh_\psi(\pi_{\chi, S}) = \val{\msc{X}_{Q,n}^\Gamma} \cdot \dim \Wh_\psi(\pi_{\chi_0, S})$.
\end{conj}

If $s$ is $W$-equivariant, then 
$$w(z)-z=0$$
and thus $\rho_z(w) = {\rm id}$ for every $z\in s(\msc{X}_{Q,n}^\Gamma)$; hence, it is clear that 
$$\sigma_{\msc{X}_0, z} = \sigma_{\msc{X}_0}$$
 in this case. For a fixed splitting $s$, consider the following three statements:
\begin{enumerate}
\item[(S1)] $\sigma_{\msc{X}_0, z} \simeq \sigma_{\msc{X}_0}$ for every $z\in s(\msc{X}_{Q,n}^\Gamma)$;
\item[(S2)] $\angb{ \sigma_{\msc{X}_0, z} }{ \sigma_S }_W = \angb{ \sigma_{\msc{X}_0} }{ \sigma_S }_W$ for every $z\in s(\msc{X}_{Q,n}^\Gamma)$;
\item[(S3)] $\dim \Wh_\psi(\pi_{\chi, S}) = \val{\msc{X}_{Q,n}^\Gamma} \cdot \dim \Wh_\psi(\pi_{\chi_0, S})$.
\end{enumerate}
Clearly, we have the implications 
$${\rm (S1)} \Longrightarrow {\rm (S2)} \Longrightarrow {\rm (S3)}.$$
Here Conjecture \eqref{C:dec} asserts the strongest (S1) in our context, and is certainly compatible with (S3) which holds a priori.

\begin{eg} \label{E:rk1odd}
Consider the cover $\wt{G}= \wt{\GL}_2^{(n)}$ associated with $\bfp, \bfq \in \Z$ such that $2\bfp- \bfq=-1$, as discussed from \S \ref{SSS:GL}. Then the cover $\wt{G}_0 = \wt{\SL}_2^{(n)}$ is saturated if and only if $n$ is odd. Thus, we assume that $n$ is odd. An easy computation gives that
$$Y_{Q,n} = \set{y_1 e_1 + y_2 e_2 \in \Z e_1 \oplus \Z e_2:\  n|(4\bfp + 1) y_1 \text{ and } n|(y_1 - y_2)}.$$
Setting $d= \gcd(4\bfp + 1, n)$, we have
$$\begin{tikzcd}
\msc{X}_{0,Q,n} \ar[r, hook] & \msc{X}_{Q,n} \ar[r, two heads] & \msc{X}_{Q,n}^\Gamma, 
\end{tikzcd}$$
with 
$$\msc{X}_{0,Q, n} =\Z \alpha^\vee /(\Z n \alpha^\vee), \val{\msc{X}_{Q,n}} = n^2/d \text{ and } \val{ \msc{X}_{Q,n}^\Gamma} = n/d.$$
We claim that there is a $W$-equivariant splitting $s$ given by 
$$s(\msc{X}_{Q,n}^\Gamma) = \set{ke_c: 0\lest k \lest (n/d-1)} \subset \msc{X}_{Q,n}.$$
For this purpose, it suffices to show that 
$$ke_c \notin Y^{sc} + Y_{Q,n}$$
 for every $1\lest k \lest (n/d-1)$. Suppose not, then there exists $m\alpha^\vee \in Y^{sc}$ such that $ke_c + m\alpha^\vee \in Y_{Q,n}$. This implies that
$$(n/d)|(k+ m) \text{ and } n|(k+a - k + a).$$
As $n$ is odd, this implies that $(n/d)|k$, which is a contradiction. Thus, Conjecture \ref{C:dec} holds for this pair $(\wt{\GL}_2, \wt{\SL}_2)$. We remark that when $d=n$ (for example, if $n=3, \bfp=-1$), then $\msc{X}_{0,Q,n} = \msc{X}_{Q,n}$; in this case $\pi_{\chi, S}|_{\wt{G}_0} = \pi_{\chi_0, S}$ with multiplicity one.
\end{eg}

\begin{eg}
We consider the cover $\wt{\GSp}_4^{(3)}$ from \S \ref{SSS:GSp} such that $Q(\alpha_2^\vee) = -1$. Note that in this case, $\wt{\Sp}_4$ is saturated, as its dual group  is $\SO_5$. Using notations in \S \ref{SSS:GSp}, we see that
$$Y_{Q,n}= 
\left\{
\begin{array}{cc}
 y_1 e_1 + y_2 e_2 + y_0 e_0 \in \oplus_{i=0}^2 \Z e_i: \\
 \bullet \quad  3|(-2y_1 + y_0) \\
 \bullet \quad  3|(-2y_1 + y_0) \\
 \bullet \quad  3|(y_1 + y_2 + 2y_0 \cdot Q(e_0))
\end{array}
\right\}.
$$
There are two cases as follows.
\begin{enumerate}
\item[--] If $Q(e_0) = \pm 1 \mod 3$, then one can check easily that $\msc{X}_{0,Q,n} = \msc{X}_{Q,n}$ in this case. Thus, Conjecture \ref{C:dec} holds trivially, and we have $\pi_{\chi, S}|_{\wt{G}_0} = \pi_{\chi_0, S}$ with multiplicity one.
\item[--] If $Q(e_0) = 0 \mod 3$, then $Y_{Q,n} = 3 Y$. In this case, one has 
$$\msc{X}_{0,Q,n} = Y_0/ 3 Y_0.$$
On the other hand, $\msc{X}_{Q,n} = Y/3Y$. Thus, a $W$-equivariant splitting is given by
$$s(\msc{X}_{Q,n}^\Gamma) = \set{k e_c: 0\lest k \lest 2}.$$
\end{enumerate}
In either case, we see that Conjecture \ref{C:dec} holds.
\end{eg}

\subsubsection{Covers of $\GL_2$ and $\SL_2$}
We consider the pair $(\wt{\GL}_2, \wt{\SL}_2)$ of $n$-fold covers associated with $\bfp$ and $\bfq$ such that  $Q(\alpha^\vee) = 2\bfp - \bfq = -1$. For odd $n$, this was discussed in Example \ref{E:rk1odd}; thus, we assume in this subsection that 
$$n=2m$$
 is even. Again, we set 
 $$d:=\gcd(4\bfp + 1, n).$$
 It is easy to see the diagram \eqref{Lat} becomes
\begin{equation} \label{Lat-rk2}
 \begin{tikzcd}
\Z(m\alpha^\vee)/\Z(n\alpha^\vee) \ar[d, hook] & & Y/Y_0  \ar[d, two heads] \\
\Z\alpha^\vee/\Z (n\alpha^\vee) \ar[r, hook] \ar[d, two heads] & \msc{X}_{Q,n} \ar[r, two heads] & \msc{X}_{Q,n}^\Gamma & \msc{X}_{Q,n}^\mfr{c} \ar[l, hook'] \\
\Z \alpha^\vee/\Z (m\alpha^\vee) \ar[r, equal] & \msc{X}_{0,Q,n},
\end{tikzcd}
\end{equation}
where $\val{\msc{X}_{Q,n}} = n^2/d$ and $\val{ \msc{X}_{Q,n}^\Gamma} = n/d$. Also, it is easy to obtain
$$Y^\mfr{c} = \set{y_1 e_1 + y_2 e_2 \in Y: \ 2|(y_1 - y_2)},$$
and thus we see that
$$[\msc{X}_{Q,n}^\Gamma: \msc{X}_{Q,n}^\mfr{c}] = 2.$$
It follows that 
$$\msc{X}_{Q,n}^\Gamma/\msc{X}_{Q,n}^\mfr{c} = \set{i e_1: 0\lest i \lest 1}.$$
For every $\gamma_i \in \msc{X}_{Q,n}^\Gamma/\msc{X}_{Q,n}^\mfr{c}$, we have 
$$\msc{E}(\chi, {}^\gamma \tchi_O; Z(\wt{T}_0)) = \set{ \omega_{\gamma, j}: \ 0 \lest j \lest 1 }.$$
Setting $\gamma_i:= i e_1$. It follows from Theorem \ref{T:decT} (see also \cite[Theorem 8.15]{GSS2}) that
$$I(\chi)|_{\wt{G}_0} = (m/d) \cdot  \bigoplus_{0\lest i \lest 1} \bigoplus_{ 0\lest j \lest 1} I(\omega^0_{\gamma_i, j}).$$
Also, the two representations $I(\omega_{\gamma_0, j}), j=0, 1$ are the two $(K_0, s_K)$-unramified constituents of $I(\chi)$. Again, for every $S \subset \Phi(\chi)$, we have
\begin{equation} \label{S-deco}
\pi_{\chi, S}|_{\wt{G}_0} =  (m/d) \cdot  \bigoplus_{0\lest i \lest 1} \bigoplus_{ 0\lest j \lest 1} \pi(\omega_{\gamma_i, j})_S,
\end{equation}
where we have written $\pi(\omega_{\gamma_i, j})_S$ instead of $\pi_{\omega_{\gamma_i, j}, S}$ for notational convenience.

In the rest of this subsection, we assume that $\Phi(\chi) = \Delta^\vee =\set{\alpha^\vee}$. In this case, we have
$$\begin{tikzcd}
\pi_{\chi, \emptyset} \ar[r, hook] & I(\chi) \ar[r, two heads]  & \pi_{\chi, \Delta},
\end{tikzcd}
$$
where $\pi_{\chi, \emptyset}$ is the covering analogue of the Steinberg representation, and $\pi_{\chi, \Delta}$ is the theta representation associated with $\chi$. Since $\wt{\GL}_2$ is always a saturated cover and thus persistent (see \cite[Lemma 2.7]{Ga6}), it follows from \cite[Theorem 6.6]{Ga6} that
\begin{enumerate}
\item[--] $\dim \Wh_\psi(\pi_{\chi, \emptyset}) = \angb{ \sigma_\msc{X} }{ \mbm{1} }_W = m(n+1)/d$,  which is equal to the number of $W$-orbits in $\msc{X}_{Q,n}$;
\item[--] $\dim \Wh_\psi(\pi_{\chi, \Delta}) = \angb{ \sigma_\msc{X} }{ \varepsilon_W }_W = m(n-1)/d$, which is equal to the number of free $W$-orbits in $\msc{X}_{Q,n}$.
\end{enumerate}
Here $\mbm{1}$ (resp. $\varepsilon_W$) denotes the trivial character (resp. sign character) of $W$.

\begin{thm} \label{T:GLSL2}
Let $\chi: Z(\wt{T}) \to \C^\times$ be an unramified genuine character such that $\Phi(\chi) = \Delta$. We always have 
$$\dim \Wh_\psi(\pi(\omega_{\gamma_i, j})_\emptyset) + \dim \Wh_\psi(\pi(\omega_{\gamma_i, j})_\Delta) = \dim \Wh_\psi(I(\omega_{\gamma_i, j})) = m.$$
Moreover, the following fold.
\begin{enumerate}
\item[(i)] If $4|n$ (i.e. $2|m$), then
$$ \dim \Wh_\psi(\pi(\omega_{\gamma_i, j})_\Delta)  = 
\begin{cases}
m/2 & \text{ for } i=0,  j\in \set{0, 1},\\
m/2 & \text{ for } i=1, j=0, \\
(m-2)/2 & \text{ for } i=1, j=1.
\end{cases}
$$
\item[(ii)] If $n=2m$ with $m$ odd, then
$$ \dim \Wh_\psi(\pi(\omega_{\gamma_i, j})_\Delta)  = 
\begin{cases}
(m+1)/2 & \text{ for } i=0,  j=0,\\
(m-1)/2 & \text{ for } i=0,  j=1,\\
(m-1)/2 & \text{ for } i=1, j \in \set{0, 1}.
\end{cases}
$$
\end{enumerate}
In either case, $\pi(\omega_{\gamma_0, j})_\Delta, j=0, 1$ are the two $(K_0, s_K)$-unramified theta representations associated with $I(\omega_{\gamma_0, j})$.
\end{thm}
\begin{proof}
We first consider the case $m$ is even. In this case, the cover $\wt{\SL}_2^{(n)}$ is persistent. As $I(\omega_{\gamma_0, j}^0)$ for $j\in \set{0, 1}$ is unramified, it thus follows from \cite[Proposition 6.2]{Ga6} that
\begin{equation} \label{E:4un}
\dim \Wh_\psi(\pi(\omega_{\gamma_0, j})_\Delta) = \angb{ \sigma_{\msc{X}_0} }{ \varepsilon_W }_W = \frac{m}{2},
\end{equation}
which is equal to the number of free $W$-orbits in $\msc{X}_{0,Q,n}$. To deal with $\dim \Wh_\psi(\pi(\omega_{\gamma_1, j})_\Delta)$, we consider a local coefficients matrix $\mca{M}(w, i(\omega_{\gamma_1, j}))$ associated with the map
$$\mca{T}(w, i(\omega_{\gamma_1, j}))^*:  \Wh_\psi(I(\omega_{\gamma_1, j})) \longrightarrow \Wh_\psi(I(\omega_{\gamma_1, j})),$$
which arises from the intertwining operator 
$$T(w, i(\omega_{\gamma_1, j})): I( i(\omega_{\gamma_1, j})) \longrightarrow I({}^w i(\omega_{\gamma_1, j})).$$
One has
$$\dim \Wh_\psi(\pi(\omega_{\gamma_1, j})_\Delta) = \text{ rank of } \mca{T}(w, i(\omega_{\gamma_1, j}))^*.$$
To have a description of $\mca{M}(w, i(\omega_{\gamma_1, j}))$, we first reconcile some notations used in \cite{GSS2}. Let
$$\xi: F^\times \longrightarrow \C^\times$$
be a character such that
$$\xi(x)^n = \chi(\wt{h}_\alpha(x^{n})).$$
Now we have from \cite[\S 9.5.3]{GSS2} (especially the proof of Theorem 9.13 there) the following.
\begin{enumerate}
\item[--] The matrix  $\mca{M}(w, i(\omega_{\gamma_1, j}))$, by swapping certain rows and columns, becomes a block-diagonal matrix with $m/2$-blocks and each block is of size two by two. In fact, each block is associated with a free $W$-orbit in $\msc{X}_{0,Q,n}$. We label these two by two blocks as 
$$\mca{M}_k, 1\lest k \lest (m/2).$$
Every $\mca{M}_k$ is a non-zero matrix, and thus in particular ${\rm rank}(\mca{M}_k) \gest 1.$
\item[--] There is a special two by two block labelled by $\mca{M}_1$, which takes the form
$$\mca{M}_1 = 
\left(\begin{matrix}
0 &  \gamma(1, \xi^{-n}) \\
\pmb{\beta}(\omega_{\gamma_1, j} , s, \psi) & 0
\end{matrix}\right),
$$
where
$$\pmb{\beta}(\omega_{\gamma_1, j} , s, \psi)= \frac{ (1- q^{-1/2}\cdot (\xi \eta_u)^{-m_{\gamma_\psi}}) (1- q^{-1/2} \xi^{m_{\gamma_\psi}}) }{ (1- (\xi \eta_u)^{m_{\gamma_\psi}})(1- \xi^{-m_{\gamma_\psi}}) }.$$
Here $u\in O^\times$ is a representative of the unique nontrivial coset of $O^\times/O^{\times 2}$, and $\eta_u(a) = (u, a)_n$ for $a\in F^\times$; also,
$$m_{\gamma_\psi} := m \cdot \gamma_\psi(\varpi) \in \set{\pm m},$$
where $\gamma_\psi$ is the Weil-index associated with $\psi$. 
\item[--] For every $2\lest k \lest (m/2)$, one has 
$$\det(\mca{M}_k) =_{\C^\times} \mu(w, i(\chi))^{-1},$$
where $=_{\C^\times}$ means the equality holds modulo some elements in $\C^\times$ (in fact modulo an element in $\set{\pm \xi^{-m}}$ as shown in the proof of  \cite[Theorem 9.13]{GSS2}), and $\mu(w, i(\chi))$ denotes the Plancherel measure.
\end{enumerate}
Since $\Phi(\chi)=\Delta$, we see that $\mu(w, i(\chi))^{-1}=0$ and thus
$${\rm rank}(\mca{M}_k) = 1 \text{ for } 2\lest k \lest (m/2).$$
Also, we have 
$$\gamma(1, \xi^{-n}) = (1- q^{-1} \xi^{-n})/(1- \xi^n),$$
which equals $0$ in this case. Let ${\rm sgn}(\xi) \in \set{\pm 1}$ be such that
$$\xi(\varpi)^m = {\rm sgn}(\xi) \cdot q^{-1/2}.$$
Then it is not hard to see that
$$\pmb{\beta}(\omega_{\gamma_1, j} , s, \psi)=0 \Longleftrightarrow {\rm sgn}(\xi) = \eta_{u, (2)}(\varpi),$$
where $\eta_{u, (2)}(x) = (u, x)_2$. Thus, we could label $\omega_{\gamma_1, 0}$ to be such that ${\rm sgn}(\xi) = - \eta_{u, (2)}(\varpi)$ and thus $\pmb{\beta}(\omega_{\gamma_1, 0} , s, \psi) \ne 0$; in this case, ${\rm rank}(\mca{M}_1) = 1$ and 
$$\dim \Wh_\psi(\pi(\omega_{\gamma_1, 0})_\Delta) = 1 + (m-2)/2= m/2.$$
On the other hand, $\omega_{\gamma_1, 1}$ is such that $\pmb{\beta}(\omega_{\gamma_1, 1}^0 , s, \psi)=0$, and thus 
$$\dim \Wh_\psi(\pi(\omega_{\gamma_1, 1})_\Delta) = (m-2)/2.$$
This completes the proof for the case $4|n$.

Now we assume that $m=2l+1$ is odd. Again, we first consider the unramified theta representation $\pi(\omega_{\gamma_0, j})_\Delta, j\in \set{0, 1}$. Note that the $n$-fold cover $\SL_2$ is not persistent  in this case, and thus formula \eqref{E:4un} does not hold. The $W$-orbits in $\msc{X}_{0,Q,n} = \Z(\alpha^\vee)/\Z(m\alpha^\vee)$ are
$$\set{\mca{O}_{i\alpha^\vee}: 1\lest i \lest l} \cup \set{\mca{O}_{-l\alpha^\vee}}$$
with 
$$\mca{O}_{i\alpha^\vee} = \set{i\alpha^\vee, (1-i)\alpha^\vee} \text{ for } 1\lest i \lest l \text{ and } \mca{O}_{-l\alpha^\vee} = \set{-l\alpha^\vee}.$$
Every free orbit $\mca{O}_{i\alpha^\vee}, 1\lest i \lest l$ supports a one-dimensional subspace of $\Wh_\psi(\pi(\omega_{\gamma_0, j})_\Delta)$ for both $j\in \set{0, 1}$. On the other hand, it follows from \cite[Theorem 3.14 and \S 4]{Ga2} that $\mca{O}_{-l\alpha^\vee}$ supports a one-dimensional subspace of $\Wh_\psi(\pi(\omega_{\gamma_0, j})_\Delta)$ for exactly one of the unramified characters in $\set{\omega_{\gamma_0, j}: j=0, 1}$; we assume without loss of generality that 
$$\Wh_\psi(\pi(\omega_{\gamma_0, 0})_\Delta)_{\mca{O}_{-l\alpha^\vee}} = 1, \text{ and thus } \Wh_\psi(\pi(\omega_{\gamma_0, 1})_\Delta)_{\mca{O}_{-l\alpha^\vee}} = 1.$$
This gives that
\begin{equation} \label{E:2un}
\Wh_\psi(\pi(\omega_{\gamma_0, 0})_\Delta) = l+ 1 = (m+1)/2, \text{ and } \Wh_\psi(\pi(\omega_{\gamma_0, 1})_\Delta)_{\mca{O}_{-l\alpha^\vee}} = l = (m-1)/2.
\end{equation}
To determine $\Wh_\psi(\pi(\omega_{\gamma_1, j})_\Delta)$, we consider a local coefficients matrix 
$$\mca{M}'(w, i(\omega_{\gamma_1, j}))$$ of size $m\times m$ associated with the map $\mca{T}(w, i(\omega_{\gamma_1, j}))^*$ as above. Again, let $\xi: F^\times \to \C^\times$ be the character such that $\xi(x)^n = \chi(\wt{h}_\alpha(x^{n}))$ for $x\in F^\times$. We may assume that $\xi$ is ramified and $\xi^2$ is unramified. Such a matrix $\mca{M}'(w, i(\omega_{\gamma_1, j}))$  is computed in  the proof of \cite[Theorem 9.12]{GSS2} with the following properties.
\begin{enumerate}
\item[--] The matrix $\mca{M}'(w, i(\omega_{\gamma_1, j}))$, by swapping certain rows and columns, becomes a block-diagonal matrix with every block associated to a $W$-orbits in $\msc{X}_{0,Q,n}$. Thus, we have $(m+1)/2$-many blocks $\mca{M}_i'$, each associated with $\mca{O}_{i\alpha^\vee}, 1\lest i \lest l$ or $i=-l$. Also the size of $\mca{M}_i'$ is exactly $\val{\mca{O}_{i\alpha^\vee}}$ for every such $i$.

\item[--] For every $1\lest i \lest l$, the matrix $\mca{M}_i'$ is a non zero two by two matrix with 
$$\det(\mca{M}_i') =_{\C^\times} \mu(w, i(\chi))^{-1}.$$
On the other hand,
$$\mca{M}_{-l}'= \tilde{\gamma}(1, \xi^{-m}, \psi),$$
which is the metaplectic-gamma factor defined in \cite{Szp3}.
\end{enumerate}
Now, since $\Phi(\chi) = \Delta$, we see that $\mu(w, i(\chi))^{-1}=0$ and thus
$${\rm rank}(\mca{M}_i') = 1 \text{ for every } 1\lest i \lest l.$$
On the other hand, one has from \cite[Theorem A.1]{GoSz} (see also \cite[Theorem 4.4]{GSS2}) that
$$\tilde{\gamma}(1, \xi^{-m}, \psi)=_{\C^\times} \frac{ \gamma(1/2, \xi^m, \psi) }{ \gamma(0, \xi^n, \psi_2) },$$
where $\psi_2(x) = \psi(2x)$. Since $\xi$ is ramified with $\xi^2$ unramified, we see that $\xi^m$ is ramified and therefore
$$\gamma(1/2, \xi^m, \psi) \in \C^\times.$$
Also, $\xi^n$ is unramified with $\xi^n(\varpi) = q^{-1}$, and this implies that
$$\tilde{\gamma}(1, \xi^{-m}, \psi)=_{\C^\times} \frac{ 1-q^{-1} \xi^{-n}(\varpi) }{ 1- \xi^n(\varpi)} =0.$$
This shows that $\mca{M}_{-l}' = 0$ for either $j=0, 1$, and thus
$$\dim \Wh_\psi(\pi(\omega_{\gamma_1, j})_\Delta) = \frac{m-1}{2}$$
for both $j=0, 1$. 

We remark that the above consideration also applies to the unramified $\pi(\omega_{\gamma_0, j})_\Delta, j=0, 1$. Indeed, in this case, we may assume that $\xi$ is unramified. Thus,
$$\tilde{\gamma}(1, \xi^{-m}, \psi)= \frac{ \gamma(1/2, \xi^m, \psi) }{ \gamma(0, \xi^n, \psi_2) }= \frac{1-q^{1/2} \xi^m(\varpi)}{ 1-q^{-1/2} \xi^{-m}(\varpi) } \cdot \frac{ 1-q^{-1} \xi^{-n}(\varpi) }{ 1- \xi^n(\varpi)}.$$
Recall that $\xi^m(\varpi) = {\rm sgn}(\xi) \cdot q^{-1/2}$, and therefore
$$ \tilde{\gamma}(1, \xi^{-m}, \psi)= 0 \Longleftrightarrow  {\rm sgn}(\xi)=-1.
$$
This also gives the desired equalities for $\dim \Wh_\psi(\pi(\omega_{\gamma_0, j})_\Delta), j=0, 1$. In any case, the proof of the case $m$ is odd is completed.
\end{proof}

\begin{rmk}
We always have the equality
$$\sum_{0\lest i, j\lest 1} \dim \Wh_\psi(\pi(\omega_{\gamma_i, j})_\Delta) = n-1.$$
For $n=4$ this equality and  \eqref{E:4un} together enforce $\dim \Wh_\psi(\pi(\omega_{\gamma_1, j})_\Delta)=0$ for both $j=0, 1$. For $n=2$, the above equality coupled with \eqref{E:2un} also imply that 
$$\dim \Wh_\psi(\pi(\omega_{\gamma_1, j})_\Delta)=0 \text{ for } j=0, 1.$$
However, for $n$ large, it is necessary to have an explicit description of a local coefficients matrix, as used in the proof of Theorem \ref{T:GLSL2}.
\end{rmk}

\subsection{Unitary unramified principal series} \label{SS:uni-ps}
In this subsection, we investigate the decomposition of an irreducible constituent of a unitary $(K, s_K)$-unramified principal series under the restriction from $\wt{G}$ to $\wt{G}_0$, and also how the Whittaker dimension behaves with respect to the restriction.

We first recall some notations and results following \cite{Ga7}. Let $I(\chi)$ be a unitary $(K, s_K)$-unramified genuine principal series of $\wt{G}$. Denote by 
$$R_\chi \subset W(\chi)$$
 the R-group associated with $I(\chi)$, where $W(\chi) \subset W$ is the stabilizer subgroup of $\chi$. One has an algebra isomorphism
$$\C[R_\chi] \simeq {\rm End}_{\wt{G}}(I(\chi)),$$
given by 
$$R_\chi \ni w \mapsto \mca{A}(w, i(\chi)),$$
where 
$$\mca{A}(w, i(\chi)) = \gamma(w, \chi) \cdot T(w, i(\chi))$$
is the normalized intertwining operator. Here $\gamma(w, \chi)$ is the gamma factor associated with $w\in W$. One knows that $R_\chi$ is abelian (see \cite{Luo3} or \cite[Theorem 4.6]{Ga7}) and thus there is a multiplicity-free decomposition 
$$I(\chi) = \bigoplus_{\tau \in \Irr(R_\chi)} \pi_\tau,$$
where $\Irr(R_\chi) = \Hom(R_\chi, \C^\times)$ denotes the characters of $R_\chi$. One has
$$\mca{A}(w, i(\chi))|_{\pi_\tau}  = \tau(w) \cdot {\rm id}$$
for every $w\in R_\chi$.

On the dual side, let 
$$\phi_\chi: \WD_F \longrightarrow {}^L\wt{T} \longrightarrow {}^L\wt{G}$$
be the parameter associated with $I(\chi)$. Recall the component group
$$\mca{S}_\chi:= \mca{S}_{\phi_\chi}$$
of the centralizer of ${\rm Im}(\phi_\chi)$ modulo $Z(\wt{G})^\vee$. By using the results of Keys \cite{Key3}, it is shown in 
\cite[Theorem 4.9]{Ga7} that
\begin{equation} \label{E:R=S}
R_\chi \simeq \mca{S}_\chi.
\end{equation}
Note that the same consideration applies to unitary unramified principal series of $\wt{G}_0$, and analogous results hold.

Consider the decomposition from Theorem \ref{T:decT} (using slightly different notation)
\begin{equation} \label{E:dec-ij}
I(\chi)|_{\wt{G}_0} = \val{\msc{X}_{Q,n}^\mfr{c} } \cdot \bigoplus_{\gamma_i } \bigoplus_{j} I(\omega_{\gamma_i, j}),
\end{equation}
where 
$$\msc{X}_{Q,n}^\Gamma/\msc{X}_{Q,n}^\mfr{c}=\set{\gamma_i}_i.$$
For every $\gamma_i$ we have 
$$\msc{E}(\chi, {}^{\gamma_i} \tchi_O; Z(\wt{T}_0))=\set{\omega_{\gamma_i, j}}_j,$$
which always has size $\val{Y_{0,Q,n}/(Y_0\cap Y_{Q,n})}$. For simplicity of notations, we may even write
$$\omega_{i, j} := \omega_{\gamma_i, j}.$$

For every constituent $I(\omega_{i, j}):= I(\omega_{\gamma_i, j})$, we have the L-parameter $\phi_{\omega_{i, j}}$ which fits into the following commutative diagram:
\begin{equation} \label{D:uni}
\begin{tikzcd}
W_F \ar[r, "{\phi_\chi}"] \ar[rd, "{\phi_{\omega_{i,j}}}"'] & {}^L \wt{G} \ar[rd, "{f_{G,H}}"]  \\
& {}^L \wt{G}_0 \ar[r, "f_{G_0, H}"']  & {}^L \wt{H}.
\end{tikzcd}
\end{equation}
Recall that we defined
$$\phi^\diamondsuit = f_{G,H} \circ \phi_\chi = f_{G_0, H} \circ \phi_{\omega_{i, j}}$$
and thus have
$$\mca{S}^\diamondsuit := \mca{S}_{\phi^\diamondsuit}.$$

We denote
$$\Omega_\chi=\set{\omega_{i, j} \in \Hom_{\rm gen}(Z(\wt{T}_0), \C^\times): I(\omega_{i,j}) \subset I(\chi)|_{\wt{G}_0} },$$
and have from Lemma \ref{L:incs} that
$$\Omega_\chi= \set{\omega_{i, j} \in \Hom_{\rm gen}(Z(\wt{T}_0), \C^\times): \omega_{i,j}|_{\wt{T}_{Q,n}^\sharp} = \chi|_{\wt{T}_{Q,n}^\sharp} }.$$
We also denote
$$\begin{aligned}[t]
\Omega^{\rm un}_\chi & = \msc{E}(\chi, \tchi_O; Z(\wt{T}_0)) \\
& =\set{\omega_{i, j} \in \Hom_{\rm gen}(Z(\wt{T}_0), \C^\times): I(\omega_{i,j}) \subset I(\chi)|_{\wt{G}_0}  \text{ and is $(K_0, s_K)$-unramified}}
\end{aligned}$$
For every $\omega^\flat \in \Omega^{\rm un}_\chi$, set
$$\Omega^{\rm un}_\chi(\omega^\flat) := \set{{}^w (\omega^\flat):  w\in W} \cap \Omega^{\rm un}_\chi =  \set{{}^w (\omega^\flat):  w\in W} \cap \msc{E}(\chi, \tchi_O; Z(\wt{T}_0)) .$$
For $\omega^\flat \in \Omega^{\rm un}_\chi$ there are two natural inclusions (see \S \ref{SS:Lspec} and \cite[\S 4.3]{Ga7})
$$\begin{tikzcd}
\mca{S}_\chi \ar[rd, hook] \\
\mca{S}_{\omega^\flat} \ar[r, hook] &  \mca{S}^\diamondsuit.
\end{tikzcd}
$$

Note that $\mca{S}^\diamondsuit$ depends only on $\phi_\chi$ and not on choice of $\omega^\flat$. We write 
$$\chi_{\rm re}:= \chi|_{\wt{T}_{Q,n}^\sharp} = \omega^\flat|_{\wt{T}_{Q,n}^\sharp},$$
where $\mbm{L}(\wt{T}_{Q,n}^\sharp) = Y_0 \cap Y_{Q,n}$.  Since $Y_{Q,n}^{sc} \subset (Y_0 \cap Y_{Q,n})$, it follows from \eqref{E:R=S} and \cite[\S 4.2]{Ga7} that one has an isomorphism
$$\mca{S}^\diamondsuit/\mca{S}_{\omega^\flat} \simeq W(\chi_{\rm re})/W(\omega^\flat). $$

\begin{lm} \label{L:tors}
For every $\omega^\flat \in \Omega^{\rm un}_\chi$, one has
$$\Omega^{\rm un}_\chi(\omega^\flat) = \set{{}^w(\omega^\flat): w\in W(\chi_{\rm re})},$$
with ${\rm Stab}_{W(\chi_{\rm re})}(\omega^\flat) = W(\omega^\flat)$. Thus, $\Omega^{\rm un}_\chi(\omega^\flat)$ is a torsor over $\mca{S}^\diamondsuit/\mca{S}_{\omega^\flat}$.
\end{lm}
\begin{proof}
It follows from the definition of $\Omega^{\rm un}_\chi(\omega^\flat)$ that
$$\begin{aligned}[t]
& {}^w (\omega^\flat) \in \Omega^{\rm un}_\chi(\omega^\flat) \\
\Longleftrightarrow & \text{ ${}^w (\omega^\flat)$ is an unramified extension of $\chi_{\rm re}$ } \\
\Longleftrightarrow & \  {}^w (\omega^\flat)|_{Z(\wt{T})\cap \wt{T}_0} = \chi_{\rm re} = \omega^\flat|_{Z(\wt{T}) \cap \wt{T}_0} \\
\Longleftrightarrow & \ w\in W(\chi_{\rm re}).
\end{aligned}$$
It is clear that the stabilizer subgroup  of $\omega^\flat$ is $W(\omega^\flat) \subset W(\chi_{\rm re})$. It follows that $\Omega^{\rm un}_\chi(\omega^\flat)$ is a torsor over $\mca{S}^\diamondsuit/\mca{S}_{\omega^\flat}$.
\end{proof}
We have the short exact sequence
$$\begin{tikzcd}
\mca{S}_\chi/(\mca{S}_\chi \cap \mca{S}_{\omega^\flat}) \ar[r, hook] & \mca{S}^\diamondsuit/\mca{S}_{\omega^\flat} \ar[r, two heads] & \mca{S}^\diamondsuit/(\mca{S}_\chi \cdot \mca{S}_{\omega^\flat}).
\end{tikzcd}$$
The action of $\mca{S}_\chi$ on $\Omega^{\rm un}_\chi(\omega^\flat)$ factors through $\mca{S}_\chi/(\mca{S}_\chi \cap \mca{S}_{\omega^\flat})$, and every such orbit is in fact a torsor over $\mca{S}_\chi/(\mca{S}_\chi \cap \mca{S}_{\omega^\flat})$ by Lemma \ref{L:tors}. Thus, we have a decomposition
$$\Omega^{\rm un}_\chi(\omega^\flat) = \bigsqcup_{k=1}^{\val{ \mca{S}^\diamondsuit/(\mca{S}_\chi \cdot \mca{S}_{\omega^\flat}) }} \mca{O}(\omega^\flat)_k$$
into a disjoint union of $\mca{S}_\chi$-orbits, where every $\mca{O}(\omega^\flat)_k$ is a torsor over $\mca{S}_\chi/(\mca{S}_\chi \cap \mca{S}_{\omega^\flat})$.

To proceed, we analyze the space $\Hom_{\wt{G}_0}(\pi_\rho, \pi_\tau)$ with $\rho \in \Irr(\mca{S}_{\omega^\flat}), \tau \in \Irr(\mca{S}_\chi)$ in more detail by adapting ideas from \cite{Key3}. To set up some notations:
\begin{enumerate}
\item[--] $f: A \longrightarrow B$ is a ring homomorphism;
\item[--] $W_A$ is a $A$-module, and $W_B, V$ are both $B$-modules;
\item[--] setting $E= {\rm End}_B(V)$, then $\Hom_B(W_B, V)$ and $\Hom_A(W_A, V_f)$ are both $E$-modules, where $V_f = V$ is view as a $A$-module via $f$.
\end{enumerate}
It is proved in \cite[Page 53]{Key3} that if  $W_B$ is a direct summand of $V$, then the natural homomorphism
\begin{equation} \label{E:G-S}
\Hom_A(W_A, (W_B)_f)  \longrightarrow \Hom_E( \Hom_B(W_B, V), \Hom_A(W_A, V_f) )
\end{equation}
of vector spaces is an isomorphism.
We apply the above to the following data
$$A=C_c^\infty[\wt{G}_0],\  B = C_c^\infty [\wt{G}],\  V=I(\chi),\  W_A = \pi_\rho,\  W_B = \pi_\tau$$
to obtain
\begin{equation} \label{E:bridge}
\Hom_{\wt{G}}(\pi_\rho, \pi_\tau|_{\wt{G}_0}) = \Hom_{\C[\mca{S}_\chi]}\left(\tau, \Hom_{\wt{G}_0}(\pi_\rho, I(\chi)|_{\wt{G}_0}) \right),
\end{equation}
where we note that 
$$
\tau \simeq \Hom_{\wt{G}}(\pi_\tau, I(\chi))
$$
 as $\C[\mca{S}_\chi]$-module.
In view of the decomposition of $I(\chi)|_{\wt{G}_0}$ in \eqref{E:dec-ij}, we first observe that (see \cite[Theorem 2.9]{BZ2})
$$\Hom_{\wt{G}_0}(\pi_\rho, I(\omega_{i,j})) \ne 0 \Longleftrightarrow \omega_{i,j} \in \Omega^{\rm un}(\omega^\flat).$$

Denoting
$$\Pi(\omega^\flat):= \bigoplus_{\omega_{i, j} \in \Omega^{\rm un}(\omega^\flat)} I(\omega_{i, j})$$
and
$$\Pi(\mca{O}(\omega^\flat)_k):= \bigoplus_{\omega_{i, j} \in \mca{O}(\omega^\flat)_k} I(\omega_{i, j}),$$
one clearly has
$$\Pi(\omega^\flat) = \bigoplus_k \Pi(\mca{O}(\omega^\flat)_k).$$
Also,
$$\Hom_{\wt{G}_0}(\pi_\rho, I(\chi)|_{\wt{G}_0}) =\val{\msc{X}_{Q,n}^\mfr{c}} \cdot \bigoplus_{k} \Hom_{\wt{G}_0}(\pi_\rho, \Pi(\mca{O}(\omega^\flat)_k)).$$

The action of $w\in \mca{S}_\chi$ on  $\Hom_{\wt{G}_0}(\pi_\rho, I(\chi)|_{\wt{G}_0})$ is given by transporting the action of $\mca{A}(w, \chi)$ on $I(\chi)$. Also, we have a decomposition 
$$\mca{A}(w, \chi) = \val{\msc{X}_{Q,n}^\mfr{c} } \cdot \bigoplus_{i, j} \mca{A}(w, \omega_{i, j}),$$
where 
$$\mca{A}(w, \omega_{i,j}): I( \omega_{i, j}) \longrightarrow I({}^w \omega_{i, j})$$
is the normalized intertwining operator of each $I( \omega_{i,j} )$. We see that the restriction $\mca{A}(w, \chi)|_{\Pi(\mca{O}(\omega^\flat)_k)}$ is well-defined for every $k$. Thus, the space
$$\Hom_{\wt{G}_0}(\pi_\rho, \Pi(\mca{O}(\omega^\flat)_k))$$
affords a representation of $\C[\mca{S}_\chi]$ of dimension $\val{ \mca{O}(\omega^\flat)_k) }$, and $\mca{S}_\chi$ also acts on the $\val{\Omega^{\rm un}_\chi(\omega^\flat)}$-dimensional space
$$\Hom_{\wt{G}_0}(\pi_\rho, \Pi(\omega^\flat)) = \bigoplus_k  \Hom_{\wt{G}_0}(\pi_\rho, \Pi(\mca{O}(\omega^\flat)_k)).$$

\begin{thm} \label{T:uni-func}
Let $I(\chi)$ be a unitary $(K, s_K)$-unramified genuine principal series of $\wt{G}$ and $I(\omega^\flat)$ an $(K_0, s_K)$-unramified constituent of $I(\chi)|_{\wt{G}_0}$. Keep notations as above. Then as representations of $\mca{S}_\chi$,
\begin{equation} \label{E:uni-rep}
\Hom_{\wt{G}_0}(\pi_\rho, \Pi(\mca{O}(\omega^\flat)_k)) \simeq  \Ind_{\mca{S}_\chi \cap \mca{S}_{\omega^\flat}}^{\mca{S}_\chi} (\rho)
\end{equation}
for every $k$. Consequently, 
\begin{equation} \label{E:uni-mul}
\dim \Hom_{\wt{G}_0}(\pi_\rho, \pi_\tau) =\val{ \msc{X}_{Q,n}^\mfr{c} } \cdot \angb{ \Ind_{\mca{S}_{\omega^\flat}}^{ \mca{S}^\diamondsuit } \rho }{  \Ind_{\mca{S}_\chi}^{ \mca{S}^\diamondsuit } \tau }_{\mca{S}^\diamondsuit}
\end{equation}
for every $\tau \in \Irr(\mca{S}_\chi)$ and $\rho \in \Irr(\mca{S}_{\omega^\flat})$.
\end{thm}
\begin{proof}
First, 
$$\Hom_{\wt{G}_0}(\pi_\rho, \Pi(\mca{O}(\omega^\flat)_k)) = \bigoplus_{\omega_{i, j} \in \mca{O}(\omega^\flat)_k} \Hom_{\wt{G}_0}(\pi_\rho, I(\omega_{i,j})).$$
Since $I(\omega_{i,j}) \simeq I(\omega^\flat)$ for every $\omega_{i,j} \in \mca{O}(\omega^\flat)_k$,
it is easy to check that as representations of $\mca{S}_\chi \cap \mca{S}_{\omega^\flat}$, one has
 $$\Hom_{\wt{G}_0}(\pi_\rho, I(\omega_{i,j}) \simeq \rho|_{ \mca{S}_\chi \cap \mca{S}_{\omega^\flat} }.$$
Moreover, as $\mca{O}(\omega^\flat)_k$ is a torsor over $\mca{S}_\chi/(\mca{S}_\chi \cap \mca{S}_{\omega^\flat})$ for every $k$, the representation of $\mca{S}_\chi$ on $\Hom_{\wt{G}_0}(\pi_\rho, \Pi(\mca{O}(\omega^\flat)_k))$ is indeed $\Ind_{\mca{S}_\chi \cap \mca{S}_{\omega^\flat}}^{\mca{S}_\chi} \rho$. This gives \eqref{E:uni-rep}.

Now to prove \eqref{E:uni-mul}, we see that as representations of $\mca{S}_\chi$,
$$\begin{aligned}[t]
\Hom_{\wt{G}_0}(\pi_\rho, I(\chi)|_{\wt{G}_0}) & =\val{\msc{X}_{Q,n}^\mfr{c}} \cdot \bigoplus_{k} \Hom_{\wt{G}_0}(\pi_\rho, \Pi(\mca{O}(\omega^\flat)_k)) \\
& =  \val{\msc{X}_{Q,n}^\mfr{c}} \cdot  \val{ \mca{S}^\diamondsuit/(\mca{S}_\chi \cdot \mca{S}_{\omega^\flat}) } \cdot  \Ind_{\mca{S}_\chi \cap \mca{S}_{\omega^\flat}}^{\mca{S}_\chi} \rho.
\end{aligned}$$
On the other hand, the Frobenius reciprocity gives
$$\angb{ \Ind_{\mca{S}_{\omega^\flat}}^{ \mca{S}^\diamondsuit } \rho }{  \Ind_{\mca{S}_\chi}^{ \mca{S}^\diamondsuit } \tau }_{\mca{S}^\diamondsuit}
= \angb{  \Res_{\mca{S}_\chi}^{ \mca{S}^\diamondsuit } \Ind_{\mca{S}_{\omega^\flat}}^{ \mca{S}^\diamondsuit } \rho }{ \tau }_{\mca{S}_\chi},$$
and by Mackey's theory one has
$$\Res_{\mca{S}_\chi}^{ \mca{S}^\diamondsuit } \Ind_{\mca{S}_{\omega^\flat}}^{ \mca{S}^\diamondsuit } \rho 
= \bigoplus_{s\in \mca{S}^\diamondsuit/(\mca{S}_\chi \cdot \mca{S}_{\omega^\flat})} \Ind_{\mca{S}_\chi \cap \mca{S}_{\omega^\flat}}^{\mca{S}_\chi} \rho
=\val{ \mca{S}^\diamondsuit/(\mca{S}_\chi \cdot \mca{S}_{\omega^\flat}) } \cdot \Ind_{\mca{S}_\chi \cap \mca{S}_{\omega^\flat}}^{\mca{S}_\chi} \rho.$$
This coupled with \eqref{E:bridge} give the desired equality \eqref{E:uni-mul}. The proof is completed.
\end{proof}

\begin{cor} \label{C:uni-ac}
Let $I(\chi)$ be a unitary $(K, s_K)$-unramified genuine principal series of $\wt{G}$ and $I(\omega) \subset I(\chi)|_{\wt{G}}, \omega \in \Omega(\chi)$ an arbitrary constituent. Then
\begin{equation*}
\dim \Hom_{\wt{G}_0}(\pi_\rho, \pi_\tau) =\val{ \msc{X}_{Q,n}^\mfr{c} } \cdot \angb{ \Ind_{\mca{S}_{\omega}}^{ \mca{S}^\diamondsuit } \rho }{  \Ind_{\mca{S}_\chi}^{ \mca{S}^\diamondsuit } \tau }_{\mca{S}^\diamondsuit}
\end{equation*}
for every $\tau \in \Irr(\mca{S}_\chi)$ and $\rho \in \Irr(\mca{S}_{\omega})$. In particular, this verifies Conjecture \ref{C:res} (ii) for unitary unramified principal series of $\wt{G}$ with $e(I(\chi)) = \val{ \msc{X}_{Q,n}^\mfr{c} }$.
\end{cor}
\begin{proof}
Assume that $\omega = \omega_{\gamma, j}$ in the notation of \eqref{E:dec-ij}. Since $I(\chi)$ is $(K, s_K)$-unramified, it is also $(\gamma\cdot K, \gamma \cdot {s_K})$-unramified. Since the genuine principal series $I(\omega)$ is $(\gamma\cdot K_0, \gamma \cdot s_K)$-unramified in this case (see Remark \ref{R:all-unram}), we can apply Theorem \ref{T:uni-func} to obtain the desired equality regarding $\dim \Hom_{\wt{G}_0}(\pi_\rho, \pi_\tau)$. The rest is clear. 
\end{proof}

A special case is the following
\begin{cor} \label{C:uni-iso}
Assume that $(\wt{G}, \wt{G}_0)$ is an isotypic-pair. Let $I(\chi)$ be a unitary unramified principal series of $\wt{G}$ and let $\omega^\flat=\chi|_{Z(\wt{T}_0)}$. Then
$$\dim \Hom_{\wt{G}_0}(\pi_\rho, \pi_\tau) = \val{ \msc{X}_{Q,n}^\Gamma } \cdot \angb{\rho|_{ \mca{S}_\chi } }{ \tau }_{\mca{S}_\chi}.$$
\end{cor}
\begin{proof}
For an isotypic pair $(\wt{G}, \wt{G}_0)$, one has $\wt{G}_0^\vee = H^\vee$ and thus $\mca{S}_{\omega^\flat} = \mca{S}^\diamondsuit$. In particular,  $\mca{S}_\chi \subset \mca{S}_{\omega^\flat}$. The result follows from Theorem \ref{T:uni-func}.
\end{proof}

\begin{rmk}
Corollary \ref{C:uni-iso} confirms for unitary unramified genuine principal series part (ii) of Conjecture \ref{C:res}. In fact, Theorem \ref{T:uni-func} generalizes this to the case where $(\wt{G}, \wt{G}_0)$ may not be an isotypic-pair. 
\end{rmk}

It follows from Corollary \ref{C:uni-iso} that for an isotypic pair $(\wt{G}, \wt{G}_0)$, one has
\begin{equation} \label{E:uni-Wd}
\dim \Wh_\psi(\pi_\tau) = \val{\msc{X}_{Q,n}^\Gamma} \cdot \sum_{\rho \in \Irr(\mca{S}_{\omega^\flat})} \angb{ \rho|_{\mca{S}_\chi} }{ \tau }_{\mca{S}_\chi} \cdot \dim \Wh_\psi(\pi_\rho)
\end{equation}
for every $\tau \in \Irr(\mca{S}_\chi)$ and $\rho\in \Irr(\mca{S}_{\omega^\flat})$. We want to justify the equality \eqref{E:uni-Wd} from \cite[Theorem 5.6]{Ga7} which we briefly recall as follows:
\begin{enumerate}
\item[--] there is a representation 
$$\sigma_0^{\rm Wh}: \mca{S}_{\omega^\flat} \longrightarrow \GL(\C^{ \val{\msc{X}_{0,Q,n} } })$$
which is given by $\sigma_0^{\rm Wh}(w) = \mca{A}(w, \omega^\flat)^*$, the induced isomorphism of $\Wh_\psi(I(\omega^\flat))$ of dimension $\val{\msc{X}_{0,Q,n}}$. One has
$$\dim \Wh_\psi(\pi_\rho) = \angb{ \rho }{ \sigma_0^{\rm Wh} }_{ \mca{S}_{\omega^\flat} }$$
for every $\rho \in \Irr(\mca{S}_{\omega^\flat})$.
\end{enumerate}
Note that one also has a representation
\begin{equation} \label{E:rho}
\sigma^{\rm Wh}: \mca{S}_\chi \longrightarrow \GL(\C^{\val{\msc{X}_{Q,n}} })
\end{equation}
given by $\sigma^{\rm Wh}(w):=\mca{A}(w, \chi)^*$ for $w\in \mca{S}_\chi$, and similarly
\begin{equation} \label{E:tau}
\dim \Wh_\psi(\pi_\tau) = \angb{ \tau }{ \sigma^{\rm Wh} }_{ \mca{S}_\chi }
\end{equation}
for every $\tau \in \Irr(\mca{S}_\chi)$.
The decomposition
$$
\Wh_\psi(I(\chi)) = \bigoplus_{i=1}^{ \val{\msc{X}_{Q,n}^\Gamma} } \Wh_\psi(I(\omega^\flat)),
$$
which is compatible with the decomposition of $\mca{A}(w, \chi) = \bigoplus_i \mca{A}(w, \omega^\flat)$, shows that
\begin{equation} \label{E:dec-sW}
\sigma^{\rm \Wh} \simeq \val{ \msc{X}_{Q,n}^\Gamma } \cdot (\sigma_0^{\rm Wh})|_{\mca{S}_\chi}.
\end{equation}
It is easy to see that the three equalities \eqref{E:rho}, \eqref{E:tau} and \eqref{E:dec-sW} together imply \eqref{E:uni-Wd} as well.

Now we explain relations among \eqref{E:uni-Wd}, Conjecture \ref{C:dec} and \cite[Conjecture 5.3]{Ga7}. Thus, for the rest of this subsubsection, we assume that $G_0$ is simply-connected and $\wt{G}_0$ is a saturated cover. We have 
\begin{enumerate}
\item[--] (\cite[Conjecture 5.3]{Ga7}) the equality $\dim \Wh_\psi(\pi_\rho) = \angb{ \rho }{ \sigma_{\msc{X}_0} }_{\mca{S}_{\omega^\flat}}$ holds for every $\rho \in \Irr(\mca{S}_{\omega^\flat})$.
\end{enumerate}
Here 
$$\sigma_{\msc{X}_0}: \mca{S}_{\omega^\flat} \into W \longrightarrow {\rm Perm}(\msc{X}_{0,Q,n})$$
is the permutation representation given by $\sigma_{\msc{X}_0}(w)=w[\cdot]$. The following is straightforward.

\begin{prop} \label{P:c-rel}
Consider a pair $(\wt{G}, \wt{G}_0)$ with $G_0$ being simply-connected and $\wt{G}_0$ a saturated cover.
\begin{enumerate}
\item[(i)] Assume \cite[Conjecture 5.3]{Ga7}, then
$$\dim \Wh_\psi(\pi_\tau)= \val{\msc{X}_{Q,n}^\Gamma} \cdot \angb{\tau}{ (\sigma_{\msc{X}_0})|_{\mca{S}_\chi} }_{\mca{S}_\chi}$$
for every $\tau\in \Irr(\mca{S}_\chi)$.
\item[(ii)] Assume both \cite[Conjecture 5.3]{Ga7} and Conjecture \ref{C:dec}, then
 $$\dim \Wh_\psi(\pi_\tau) = \angb{\tau }{ \sigma_\msc{X} }_{\mca{S}_\chi}$$
 for every $\tau \in \Irr(\mca{S}_\chi)$.
\end{enumerate}
\end{prop}
Note that if $I(\chi)$ is irreducible, then the equalities in Proposition \ref{P:c-rel} hold trivially and unconditionally.
We give two examples below in the non isotypic-pair case.

\begin{eg}
Let $\wt{G}= \wt{\GL}_2^{(n)}$ be associated with $\bfp=0, \bfq=1$, and $n=2m$ is even.
One has $$Y_{0, Q,n} = \Z(m\alpha^\vee), \ Y_{Q,n} = nY, \ Y_0 \cap Y_{Q,n} = nY_0.$$
Thus,
$$\wt{G}^\vee = \GL_2, \ \wt{G}_0^\vee = \SL_2 \text{ and } H^\vee = \PGL_2.$$
For every $(K, s_K)$-unramified $I(\chi)$, one has
$$\Omega^{\rm un}_{\chi} = \set{\omega_{0, 0}, \omega_{0, 1}}.$$
Note that we always have $\mca{S}_\chi =\set{1}$ and $\mca{S}_{\omega_{0, j}}=\set{1}$ for $j=0, 1$. Regarding $\mca{S}^\diamondsuit$, there are two cases:
\begin{enumerate}
\item[--] if $\chi(\wt{h}_\alpha(\varpi^n)) \ne -1$, then $\mca{S}^\diamondsuit=\set{1}$ and thus $\Omega^{\rm un}_\chi(\omega_{0, j}) = \set{\omega_{0, j}}$ for each $j$.
\item[--] if $\chi(\wt{h}_\alpha(\varpi^n)) = -1$, then $\mca{S}^\diamondsuit=W$ and we have $\Omega^{\rm un}_\chi(\omega_{0, j}) = \Omega^{\rm un}_\chi$ for $j=0, 1$.
\end{enumerate}
In the second case above, on the dual side the relative Satake parameter ${\rm sat}(\omega_{0, j})$ (in the sense of \cite{GG} if $m$ is odd) for $\omega_{0, j}$ is 
$${\rm sat}(\omega_{0, j}) =
\left(\begin{matrix}
\sqrt{-1}^{2j+1} &  0 \\
0 & \sqrt{-1}^{-(2j+1)}
\end{matrix}\right) \in \wt{G}_0^\vee.$$
This example is easily generalizable to the pair $(\wt{\GL}_r^{(n)}, \wt{\SL}_r^{(n)})$ with $r|n$. For such a pair, one always has $\mca{S}_\chi = \mca{S}_{\omega^\flat} = \set{1}$. However, the group $\mca{S}^\diamondsuit$ will be more complicated.
\end{eg}

\begin{eg}
Consider the pair of double covers $(\wt{G}, \wt{G}_0) = (\wt{\GSp}_{2r}^{(2)}, \wt{\Sp}_{2r}^{(2)})$ as from \S \ref{SSS:GSp}, see also \cite{Szp4-1, Szp5} and \cite[\S 5.5]{Ga7}. Regarding the dual groups, one has
$$  \wt{G}^{\vee} = \begin{cases}
\GSp_{2r}(\C), \text{  if $r$ is odd;} \\
{\rm PGSp}_{2r}(\C)  \times \GL_1(\C),  \text{  if $r$ is even.}  
\end{cases} 
$$
On the other hand, we have
$$\wt{G}_0^\vee = \Sp_{2r}, \ \wt{H}^\vee = {\rm PGSp}_{2r}.$$
Assume $r$ is even. It follows from \cite[Page 399]{Key2} that the only nontrivial $\mca{S}_\chi = R_\chi$ is $\set{1, w} \simeq \Z/2\Z$ which is generated by
$$w:= w_{\alpha_1} w_{\alpha_3} ... w_{\alpha_{r-1}},$$
with the character $\chi$ satisfying
$$\chi(\wt{h}_{\alpha_i}(\varpi^{n_{\alpha_i}})) = -1 \text{ for all } i=2k-1, 1\lest k \lest r/2.$$
For such $\chi$, one has $\Omega_{\chi}^{\rm un} = \set{\omega_{0, 0}, \omega_{0, 1}}$; also,
$$\mca{S}_\chi = \mca{S}^\diamondsuit = \set{1, w} \text{ and } \mca{S}_{\omega^\flat} = \set{1}$$
for every $\omega^\flat \in \Omega^{\rm un}_\chi$. 
In this case, for any $j\in \set{0, 1}$,
$$\Omega^{\rm un}_\chi(\omega_{0, j}) = \Omega^{\rm un}_\chi,$$
which is a torsor over $\mca{S}^\diamondsuit$.
It follows from Theorem \ref{T:uni-func} that
$$\dim_{\wt{G}_0}(I(\omega_{0, j}), \pi_\tau) = \val{ \msc{X}_{Q,n}^\mfr{c} } \cdot \angb{\tau}{ \C[\mca{S}_\chi]}_{\mca{S}_\chi} = \val{ \msc{X}_{Q,n}^\mfr{c} } = 1$$
for every $\tau \in \Irr(\mca{S}_\chi)$ and $j=0, 1$.
\end{eg}

\section{Unramified L-packets} \label{S:unLP}
In this section, we consider the internal structure of an unramified L-packets, the linear algebraic analogue has been investigated in \cite{Mis2}. The following two topics are pertaining to our discussion:
\begin{enumerate}
\item[(i)] the parametrization of elements inside an L-packet with respect to changing the hyperspecial maximal compact subgroups of $G$,
\item[(ii)] the variation of the Whittaker dimension of elements inside an unramified L-packet with respect to different orbits of the Whittaker datum. 
\end{enumerate}
Regarding (ii), some earlier work include \cite{Kuo1, Kuo2, Mis1}. In this section, we adopt some ideas from these works in the covering setting. The fact that the three groups $G$, its $n$-fold cover $\wt{G}$ and the principal endoscopic group $G_{Q,n}$ all play a role here necessitates a careful analysis of certain relations among them, which allows us to postulate some natural problems and prove certain results for covers. For this purpose, we first briefly recall the results for (split) linear algebraic groups.

\subsection{Unramified L-packets for linear groups} \label{S:un-L-lin}
Let $G$ be the $F$-rational points of a linear algebraic group $\mbf{G}$, and $T$ the $F$-points of $\mbf{T}$. For a character
$$\uchi: T \longrightarrow \C^\times$$
which is trivial on $T_0:=\mbf{T}(O)$, one obtains on the dual side an unramified parameter
$$\phi_{\uchi}: \WD_F \longrightarrow {}^L T$$
which is trivial on $I_F \subset W_F$ and $\SL_2(\C) \subset \WD_F$. We call $\phi_{\uchi}$ an unramified L-parameter valued in ${}^LT$. Post-composing $\phi_{\uchi}$ with the natural inclusion 
$${}^LT \longrightarrow {}^L G$$
gives a parameter which is still denoted by 
$$\phi_{\uchi}: W_F \longrightarrow {}^L G.$$
The local Langlands correspondence postulates that the L-packet $\mca{L}(\phi_{\uchi})$ associated with $\phi_{\uchi}$ consists exactly of the subquotients of the principal series $I(\uchi)$, which are $K$-unramified with respect to a hyperspecial maximal compact subgroup $K$ of $G$, see \cite[\S 10.4]{Bor}. Moreover, it is expected that there is a bijection
$$\mca{L}(\phi_{\uchi}) \longleftrightarrow \Irr(\mca{S}(\phi_{\uchi})),$$
where $\mca{S}(\phi_{\uchi})$ is the component group of the image of the parameter $\phi_{\uchi}$ in ${}^LG$.

If we conveniently index the packet as
$$\mca{L}(\phi_{\uchi}) = \set{\pi(\phi_{\uchi}, \rho): \rho \in \Irr(\mca{S}(\phi_{\uchi}))},$$
then as remarked each $\pi(\phi_{\uchi}, \rho)$ is $K$-unramified for some hyperspecial subgroup $K \subset G$. It is natural to ask how the indexing by $\rho$ varies with respect to changing the hyperspecial maximal compact subgroup $K$. Following \cite{Mis2}, we describe the answer to this question which embodies part of the internal structure of $\mca{L}(\phi_{\uchi})$.

Consider the root datum
$$(X, \Phi, \Delta; \ Y, \Phi^\vee, \Delta^\vee)$$
of the group $\mbf{G}$, and also recall the root lattice $X^{sc} = \Z[\Delta]$ and coroot lattice $Y^{sc}=\Z[\Delta^\vee]$. One has the standard affine Weyl group
$$W_{\rm a} = X^{sc} \rtimes W$$
acting on the vector space $X^{sc} \otimes_\Z \R$. On the other hand, there is also a natural action of the extended affine Weyl group
$$\tilde{W}_{\rm a} = X \rtimes W$$
on  $X \otimes_\Z \R$. Defining $\Gamma:=\tilde{W}_{\rm a}/W_{\rm a}$, one has a split short exact sequence
$$\begin{tikzcd}
W_{\rm a} \ar[r, hook] & \tilde{W}_{\rm a} \ar[r, two heads] & \Gamma
\end{tikzcd}$$
with a splitting given by
$$s(\Gamma) = \set{x\in \tilde{W}_{\rm a}: \  x\cdot C = C},$$
where $C \subset X\otimes \R$ is an alcove for the action of $\tilde{W}_{\rm a}$ on $X\otimes \R$.
That is, we have
$$\tilde{W}_{\rm a} = W_{\rm a} \rtimes \Gamma$$
with $$\Gamma \simeq X/X^{sc}.$$
Note that $Z(G) = \Hom(\Gamma, F^\times)$. For every lattice $L$, we write
$$L_\Q := L\otimes \Q$$
for the $\Q$-vector space. It is easy to see that
$$\Gamma^{\rm tor} = (X \cap X^{sc}_\Q)/X^{sc},$$
where $\Gamma^{\rm tor} \subset \Gamma$ denotes the torsion-subgroup. Denote by $Y_{ad}$ the cocharacter group of $\mbf{G}_{ad}$, the adjoint quotient group of $\mbf{G}$. Then $Y_{ad} \subset \Q[\Delta^\vee]$ is the lattice of coweights, which is $\Z$-dual to $\Z[\Delta]$. Consider
$$\widehat{\Gamma}^\dag:= {\rm Coker}(Y \longrightarrow Y_{ad}),$$
we see that there is an embedding
$$\widehat{\Gamma}^\dag \into \widehat{\Gamma}:=\Hom(\Gamma, \Q/\Z).$$
In fact, the composite
$$\begin{tikzcd}
\widehat{\Gamma}^\dag \ar[r, hook] & \widehat{\Gamma} \ar[r, two heads] & \widehat{\Gamma^{\rm tor}}
\end{tikzcd}$$
is an isomorphism, see \cite[Lemma 11]{Mis2}.

Consider the set
$$\mca{K}=\set{G_x: \ x \text{ is a hyperspecial point in } \msc{B}(G) }/_\sim$$
 of conjugacy classes of hyperspecial maximal compact subgroup of $G$, where $\msc{B}(G)$ is the Bruhat-Tits building associated with $\mbf{G}$. Here $\mca{K}$ is a torsor over $\widehat{\Gamma}^\dag$ with the action $y\cdot G_x, y\in \widehat{\Gamma}^\dag$ inherited from the conjugation action of $G_{ad}:=\mbf{G}_{ad}(F)$ on $G$. One has 
 $$y\cdot G_x = G_{y \cdot x} \in \mca{K},$$
 where $y\cdot x$ is the action of $G_{ad}$ on $\msc{B}(G)$, see \cite[\S 2.5]{Tits}.

On the other hand, there is also a natural map from $\widehat{\Gamma}^\dag$ to $\mca{S}(\phi_{\uchi})$ given as follows, by reducing to the case of unitary $\uchi$ first (see \cite[\S 10.3]{Mis2}). More precisely, denoting $\phi:=\phi_{\uchi}$ for simplicity, it gives a commuting pair 
$$\phi_o,\  \phi_+: W_F \longrightarrow {}^L T \subset {}^LG$$
 of parameters such that $\phi_o$ is tempered and 
$$\phi(a) = \phi_o(a) \cdot \phi_+(a)$$
for every $a\in W_F$. Let 
$$\uchi_+: T \longrightarrow \C^\times$$ be the character attached to $\phi_+$. Then the set
$$\Delta_M:=\set{\alpha\in \Delta: \ \val{\uchi_+(\alpha^\vee(\varpi))} = 1} \subset \Delta$$
gives a Levi subgroup $\mbf{M} \subset \mbf{G}$ such that ${\rm Im}(\phi) \subset M^\vee$ and
$$\mca{S}_{M^\vee}(\phi_o) = \mca{S}_{G^\vee}(\phi).$$
Moreover, it is shown (see \cite[Lemma 15]{Mis2}) that one has two embeddings
$$\begin{tikzcd}
\mca{S}_{M^\vee}(\phi_o) \ar[r, hook] & \Gamma_M^{\rm tor} \ar[r, hook] & \Gamma_G^{\rm tor}.
\end{tikzcd}$$
This shows in particular that the group $\mca{S}(\phi)$ is abelian. By applying $\Hom(-, \Q/\Z)$ we obtain a surjection (writing $\widehat{\Gamma}^\dag=\widehat{\Gamma}_G^\dag$  and omitting $G^\vee$ in $\mca{S}_{G^\vee}(\phi)$)
\begin{equation}
\begin{tikzcd}
f_\Gamma: \widehat{\Gamma}^\dag \ar[r, two heads] & \Irr(\mca{S}(\phi)).
\end{tikzcd}
\end{equation}

\begin{thm}[{\cite[Theorem 1]{Mis2}}] \label{T:Mish}
For every conjugacy class $K \in \mca{K}$ and every $y \in \widehat{\Gamma}^\dag$, the representation $\pi(\phi_{\uchi}, \rho) \in \mca{L}(\phi_{\uchi})$ is K-unramified if and only if $\pi(\phi_\chi, f_\Gamma(y) \otimes \rho)$ is $y\cdot K$-unramified. 
\end{thm}

If we choose $K \in \mca{K}$, and require that  $\pi(\phi_{\uchi}, \mbm{1}) \in {\rm JH}(I(\uchi))$ is the unique $K$-unramified constituent of $I(\chi)$, then the above theorem implies that $\pi(\phi_{\uchi}, \rho), \rho \in \Irr(\mca{S}(\phi_{\uchi}))$ is the unique $y\cdot K$-unramified constituent in $I(\uchi)$, where $y\in f_\Gamma^{-1}(\rho)$ is any lifting of $\rho$.

\subsection{Variation of $(K, s_K)$ for covers} \label{SS:varK}  

\subsubsection{Splittings over $K$} \label{SSS:spl-K}
Now we consider an $n$-fold cover $\wt{G}$ of $G$, and still denote by $\mca{K}$ the set of conjugacy classes of hyperspecial maximal compact subgroup of $G$.

\begin{lm} \label{L:K-spl}
With the assumption that $\gcd(n, p)=1$ and $\eta_n: Y^{sc} \to F^\times/F^{\times n}$ is trivial, the group $\wt{G}$ splits over every hyperspecial maximal compact subgroup $K \in \mca{K}$.
\end{lm}
\begin{proof}
We presents two arguments. First, since $K$ is hyperspecial, one has $K=\mbf{G}_O(O)$ for a smooth group scheme $\mbf{G}_O$ over ${\rm Spec}(O)$. By \cite[Construction 12.11]{BD}, one knows that $\wt{K} \subset \wt{G}$ is the pull-back of a residual extension 
$$\begin{tikzcd}
\mu_n \ar[r, hook] & \wt{\mbf{G}_O(\kappa)} \ar[r, two heads] & \mbf{G}_O(\kappa).
\end{tikzcd}
$$
The proof of \cite[Proposition 4.1]{GG}, which was written for $\mbf{G}(O)$, also applies for such $K$ and gives that the above extension splits, if we assume the triviality of $\eta_n$. Thus, we have a splitting of $\wt{K}$ over $K$. Note that the existence of such a splitting also follow from \cite[Theorem 4.3]{We4}.

Alternatively, we note that $\wt{G}$ splits over $K_\natural=\mbf{G}(O)$ by \cite[Theorem 4.2]{GG}, and thus we let $s: K_\natural \into \wt{G}$ be a fixed splitting. The conjugation action of $G_{ad}$ on $G$ extends to $\wt{G}$ (see \cite{BD}), and the group $G_{ad}$ acts transitively on $\mca{K}$. We see that for $t \cdot K_\natural \in \mca{K}, t\in G_{ad}$, one has a splitting
$$t \cdot s:\  t \cdot K_\natural \into \wt{G}$$
given by 
$$(t\cdot s)(t \cdot k):= t \cdot s(k) \cdot t^{-1}, k\in K_\natural.$$
This is a well-defined splitting of $\wt{G}$ over the hyperspecial maximal compact subgroup $y\cdot K_\natural \subset G$.
\end{proof}

We impose the same assumption in the Lemma \ref{L:K-spl} throughout the remaining part of this section. For every $K\in \mca{K}$, the set
$${\rm Spl}(\wt{G}, K)$$ 
of splittings of $\wt{G}$ over $K$ is a torsor over $\Hom(K, \mu_n)$. Since $G= T\cdot G_\der$ and thus $K= (T\cap K) \cdot (G_\der \cap K)$, we see that every homomorphism $h\in \Hom(K, \mu_n)$ must be trivial on $K\cap G_\der$. Hence, the restriction map
$$\begin{tikzcd}
\Hom(K, \mu_n) \ar[r, hook] & \Hom(T \cap K, \mu_n)
\end{tikzcd}$$
is an injection, where $T\cap K = \mbf{T}(O)$. As the restriction of $h$ to $T_\der \subset T$ is also trivial, and the above injection factors through the inclusion 
$$\Hom(\mbf{T}(O)/\mbf{T}_\der(O), \mu_n) \into \Hom(\mbf{T}(O), \mu_n)$$
 we have
$$\begin{tikzcd}
\iota_K: \Hom(K, \mu_n) \ar[r, hook] & \Hom(\mbf{T}(O)/\mbf{T}_\der(O), \mu_n) \ar[r, hook] & \Hom(\mbf{T}(O), \mu_n).
\end{tikzcd}$$

Note that we do not consider the whole set $\Hom(K, \mu_n)$, but only a subset of ``admissible" elements constrained as follows. 
Let $(K, s_K)$ be such that $K\in \mca{K}$ and $s_K \in {\rm Spl}(\wt{G}, K)$. If a representation $\pi$ is $(K, s_K)$-unramified, then it is also $(K, h\otimes s_K)$-unramified if $\iota_K(h)$ is trivial on $\mbf{T}_{Q,n}(O) \subset T$, where by abuse of notation we use $\mbf{T}_{Q,n}(O)$ to denote the image of $\mbf{T}_{Q,n}(O)$ in $\mbf{T}(O)$ with respect to the natural map $T_{Q,n} \to T$. For this reason, we define
$$\Hom(K,\mu_n)^\natural:= \iota_K^{-1}\left( \Hom(\mbf{T}(O)/\mbf{T}_{Q,n}(O), \mu_n)\right).$$
We thus have natural injections (still denoted by $\iota_K$)
$$\begin{tikzcd}
\iota_K: \Hom(K, \mu_n)^\natural \ar[r, hook] & \Hom(\mbf{T}(O)/\mbf{T}_\der(O) \cdot \mbf{T}_{Q,n}(O), \mu_n) \ar[r, hook] & \Hom(\mbf{T}(O)/\mbf{T}_{Q,n}(O), \mu_n).
\end{tikzcd}$$

Since $\wt{A}:=C_{\wt{T}}(\mbf{T}(O)) = \mbf{T}(O) \cdot Z(\wt{T})$ is a maximal abelian subgroup of $\wt{T}$, the commutator map gives a perfect pairing
$$[-,-]: \wt{T}/\wt{A} \times \mbf{T}(O)/\mbf{T}_{Q,n}(O) \longrightarrow \mu_n,$$
which gives a group isomorphism
$$\Hom(\mbf{T}(O)/\mbf{T}_{Q,n}(O) , \mu_n) \longrightarrow Y/Y_{Q,n} \text{ given by } [y(\varpi), -] \leftrightarrow y.$$
By composing with $\iota_K$, we obtain a natural injection
\begin{equation}
\begin{tikzcd}
\iota_K^\natural: \Hom(K, \mu_n)^\natural \ar[r, hook] & Y/Y_{Q,n}. 
\end{tikzcd}
\end{equation}
We denote 
$$(Y/Y_{Q,n})_K^\natural:={\rm Im}(i_K^\natural) \subset Y/Y_{Q,n}$$
and also for $z\in (Y/Y_{Q,n})_K^\natural$, we write
$$f_z\in \Hom(K, \mu)^\natural$$
for the unique element such that $i_K^\natural(f_z) = z$.

If we define 
\begin{equation} \label{YKnat}
Y_K^\natural = \set{y\in Y: B_Q(y, z) \in n\Z \text{ for all } z\in Y_0},
\end{equation}
which clearly contains $Y_{Q,n}$, then one has the inclusions
$$(Y/Y_{Q,n})_K^\natural \into Y_K^\natural/Y_{Q,n} \into \msc{X}_{Q,n}.$$
The first inclusion is often an identity. It is also easy to see that there is a commutative diagram
\begin{equation} \label{CD:big}
\begin{tikzcd}
Y_{0,Q,n}/Y_0 \cap Y_{Q,n} \ar[r, hook] \ar[d, hook] & (Y/Y_{Q,n})_K^\natural \ar[r, "{p_\Gamma}"] \ar[d, hook] & \msc{X}_{Q,n}^\mfr{c} \ar[d, hook] \\
Y_0/Y_0 \cap Y_{Q,n} \ar[d, two heads] \ar[r, hook] & \msc{X}_{Q,n} \ar[r, two heads, "{p_\Gamma}"] & \msc{X}_{Q,n}^\Gamma  \\
\msc{X}_{0,Q,n} \ar[r, equal] & Y_0/Y_{0,Q,n},
\end{tikzcd}
\end{equation}
where the map $p_\Gamma: (Y/Y_{Q,n})_K^\natural \to \msc{X}_{Q,n}^\mfr{c}$ obtained from restriction is well-defined in view of \eqref{D:mfr-c} and \eqref{D:Xc}, and the middle horizontal and left vertical maps are short exact sequences.

\subsubsection{A relation between $\widehat{\Gamma_G^{\rm tor}}$ and $\widehat{\Gamma_{G_{Q,n}}^{\rm tor}}$}

For $G$ we have the following commutative diagram:
$$\begin{tikzcd}
X_{Q,n}/X & X_{Q,n}/X_G^{sc} \ar[l, two heads] & X/X_G^{sc} \ar[l, hook'] \ar[r, equal] & \Gamma_G \\
(X_{Q,n} \cap X_{G,\Q}^{sc})/(X\cap X_{G,\Q}^{sc}) \ar[u, hook] & (X_{Q,n}\cap X_{G,\Q}^{sc})/X_G^{sc}  \ar[l, two heads] \ar[u, hook]& (X\cap X_{G,\Q}^{sc}) / X_G^{sc} \ar[l, hook'] \ar[u, hook] \ar[r, equal] & \Gamma_G^{\rm tor} . \ar[u, hook]
\end{tikzcd}$$
Here $X_{Q,n}=\Hom(Y_{Q,n}, \Z)$ and we have used the subscript $G$ in $X_{G,\Q}^{sc}=X_G^{sc} \otimes \Q$ (for example) to emphasize the underlying group involved. Similarly, one has for a Levi subgroup $M$ the commutative diagram
$$\begin{tikzcd}
X_{Q,n}/X & X_{Q,n}/X_G^{sc} \ar[l, two heads] & X/X_M^{sc} \ar[l, hook'] \ar[r, equal] & \Gamma_M \\
(X_{Q,n} \cap X_{M,\Q}^{sc})/(X\cap X_{M,\Q}^{sc}) \ar[u, hook] & (X_{Q,n}\cap X_{M,\Q}^{sc})/X_M^{sc}  \ar[l, two heads] \ar[u, hook]& (X\cap X_{M,\Q}^{sc}) / X_M^{sc} \ar[l, hook'] \ar[u, hook] \ar[r, equal] & \Gamma_M^{\rm tor} . \ar[u, hook]
\end{tikzcd}$$
The above two diagrams are compatible, and for the bottom two exact sequences this is illustrated from the following diagram
$$\begin{tikzcd}
X_{Q,n}/X & & \\
(X_{Q,n} \cap X_{G,\Q}^{sc})/(X\cap X_{G,\Q}^{sc}) \ar[u, hook]  & (X_{Q,n}\cap X_{G,\Q}^{sc})/X_G^{sc}  \ar[l, two heads] &  \Gamma_G^{\rm tor} \ar[l, hook'] \\
(X_{Q,n} \cap X_{M,\Q}^{sc})/(X\cap X_{M,\Q}^{sc}) \ar[u, hook] & (X_{Q,n}\cap X_{M,\Q}^{sc})/X_M^{sc}  \ar[l, two heads] \ar[u, hook]& \Gamma_M^{\rm tor} \ar[l, hook'] \ar[u, hook].
\end{tikzcd}$$
Denote by 
$$P(L) \subset \Q[\Delta^\vee]$$
 the $\Z$-dual of the lattice $L \subset \Q[\Delta]$. Applying $\Hom(-, \Q/\Z)$ to the above diagram, we obtain
 $$\begin{tikzcd}
Y/Y_{Q,n} \ar[d, two heads] & & \\
P(X\cap X_{G,\Q}^{sc})/P(X_{Q,n} \cap X_{G,\Q}^{sc}) \ar[r, hook]  & P(X_G^{sc})/P(X_{Q,n}\cap X_{G,\Q}^{sc}) \ar[r, two heads] &  \widehat{\Gamma_G^{\rm tor}}.
\end{tikzcd}$$

For the principal endoscopic group $G_{Q,n}$ of $\wt{G}$, one has the group $\Gamma_{G_{Q,n}}$ and $\Gamma_{G_{Q,n}}^{\rm tor}$ defined analogously as for $\Gamma_G$ and $\Gamma_G^{sc}$. Thus,
$$\Gamma_{G_{Q,n}} = X_{Q,n}/X_{Q,n}^{sc} \text{ and } \Gamma_{G_{Q,n}}^{\rm tor} = (X_{Q,n} \cap X_{G,\Q}^{sc})/X_{Q,n}^{sc},$$
where the last equality follows from the fact that $X_{G_{Q,n}, \Q}^{sc} = X_{G,\Q}^{sc}$. This gives that
$$\widehat{\Gamma_{G_{Q,n}}^{\rm tor}} = P(X_{Q,n}^{sc}) /P(X_{Q,n}\cap X_{G,\Q}^{sc}).$$

The discussion in \S \ref{S:un-L-lin} applies to $G_{Q,n}$. In particular, the set of conjugacy classes of hyperspecial maximal compact subgroup of $G_{Q,n}$ is a torsor over $\widehat{\Gamma_{G_{Q,n}}^{\rm tor}}$.
Moreover, for every genuine character $\chi: Z(\wt{T}) \longrightarrow \C^\times$, which is $s_{Q,n}(\mbf{T}(O))$-unramified with respect to a splitting $s_{Q,n}$ of $\mbf{T}(O)$, there is the associated parameter
$$\phi_\chi: W_F \longrightarrow {}^L\wt{T} \into {}^L\wt{G}.$$
One has a surjection
$$\widehat{\Gamma_{G_{Q,n}}^{\rm tor}} \onto \Irr(\mca{S}(\phi_\chi)).$$

To relate the two group $\widehat{\Gamma_{G}^{\rm tor}}$ and $\widehat{\Gamma_{G_{Q,n}}^{\rm tor}}$, we note that there is a natural map
$$\varphi: P(X_G^{sc})/P(X_{Q,n}\cap X_{G,\Q}^{sc}) \longrightarrow \widehat{\Gamma_{G_{Q,n}}^{\rm tor}}$$
given as follows. Recall that $X_{Q,n}^{sc}$ is generated by $\set{\alpha/n_\alpha: \alpha\in \Delta}$, and thus there is a group homomorphism
$$\varphi_0: X_{Q,n}^{sc} \longrightarrow X_G^{sc}$$
uniquely determined by
$$\varphi_0(\alpha/n_\alpha) = \alpha \text{ for every } \alpha \in \Delta.$$
This gives a well-defined map $\varphi_0$ as in
\begin{equation} \label{D:r-01}
\begin{tikzcd}
(X_{Q,n} \cap X_{G,\Q}^{sc})/(X \cap X_{G,\Q}^{sc}) & (X_{Q,n} \cap X_{G,\Q}^{sc})/X_G^{sc} \ar[l, two heads] & (X \cap X_{G,\Q}^{sc})/X_G^{sc} \ar[l, hook']  \ar[d, equal]\\
\Gamma_{G_{Q,n}}^{\rm tor} \ar[r, equal] & (X_{Q,n} \cap X_{G,\Q}^{sc})/X_{Q,n}^{sc}. \ar[u, hook, "\varphi_0"] & \Gamma_G^{\rm tor}
\end{tikzcd}
\end{equation}
By taking applying $\Hom(-, \Q/\Z)$, we obtain the desired homomorphism $\varphi$. In fact, let
$$\set{\omega_{\alpha_i}} \subset \Q[\Delta^\vee]$$
be the set of fundamental coweights such that
$$\angb{\omega_{\alpha_i}}{\alpha_j} = \delta_{ij}$$
for every $\alpha_j\in \Delta$. We have
$$P(X_G^{sc}) = \Z[\omega_{\alpha}: \alpha\in \Delta].$$
Similarly, 
$$P(X_{Q,n}^{sc}) = \Z[n_\alpha \omega_\alpha: \alpha\in \Delta]$$
Then $\varphi$ is induced from the map 
$$P(X_G^{sc}) \longrightarrow P(X_{Q,n}^{sc})$$
given by 
$$\omega_\alpha \mapsto n_\alpha w_\alpha$$ for every $\alpha \in \Delta$. In particular, $\varphi$ is surjective. 

As a summary of the above discussion, we have the following diagram
\begin{equation} \label{CD:K1}
 \begin{tikzcd}
Y/Y_{Q,n} \ar[d, two heads, "h"] & & \\
P(X\cap X_{G,\Q}^{sc})/P(X_{Q,n} \cap X_{G,\Q}^{sc}) \ar[r, hook] \ar[d, two heads, "\varphi"]  & P(X_G^{sc})/P(X_{Q,n}\cap X_{G,\Q}^{sc}) \ar[d, two heads, "\varphi"] \ar[r, two heads] &  \widehat{\Gamma_G^{\rm tor}} \ar[d, two heads, "\varphi"] \\
H_\varphi \ar[r, hook]& P(X_{Q,n}^{sc}) /P(X_{Q,n}\cap X_{G,\Q}^{sc}) \ar[d, equal] \ar[r, two heads] & \widehat{\Gamma_{G_{Q,n}}^{\rm tor}}/H_\varphi \\
& \widehat{\Gamma_{G_{Q,n}}^{\rm tor}}.
\end{tikzcd}
\end{equation}
Here by definition $H_\varphi$ is the image of $P(X\cap X_{G,\Q}^{sc})/P(X_{Q,n} \cap X_{G,\Q}^{sc})$ via $\varphi$. The following compressed diagram will play a key role in our subsequent discussion:
\begin{equation} \label{CD:K2}
\begin{tikzcd}
\Hom(K, \mu_n)^\natural \ar[r, hook, "{\iota_K^\natural}"] & Y/Y_{Q,n} \ar[d, two heads, "{\varphi \circ h}"] & & \widehat{\Gamma_G^{\rm tor}} \ar[d, two heads, "\varphi"] \\
& H_\varphi \ar[r, hook]& \widehat{\Gamma_{G_{Q,n}}^{\rm tor}} \ar[d, two heads, "{q_\chi}"] \ar[r, two heads] & \widehat{\Gamma_{G_{Q,n}}^{\rm tor}}/H_\varphi  \ar[l, bend right=50, dashed, "{s_\varphi?}"']  \\
& & \Irr(\mca{S}(\phi_\chi)).
\end{tikzcd}
\end{equation}

\subsubsection{The packet $\mca{L}(\phi_\chi)$ versus unramifiedness}
For every $K\in \mca{K}$ and $y\in \widehat{\Gamma_G^{\rm tor}}$ the action $y\cdot K$ is realized by the conjugacy action of $y(\varpi) \in T_{\rm ad}$ on $K$. For every $f\in \Hom(K, \mu_n)$, by transport of structure, we define
$$y\cdot f: \ y\cdot K \longrightarrow \mu_n$$
by
$$(y\cdot f)(y(\varpi)\cdot k\cdot y(\varpi)^{-1}):= f(k).$$
Recall that $G_{ad}$ and thus $T_{ad}$ acts on $\wt{G}$. As used in Lemma \ref{L:K-spl}, if $s_K: K \into \wt{G}$ is a splitting, then we obtain a splitting 
$$y\cdot s_K: \ y\cdot K \longrightarrow \wt{G}$$
of $y\cdot K \in \mca{K}$ given by
$$(y\cdot s_K)(y(\varpi)\cdot k \cdot y(\varpi)^{-1}):= y(\varpi) \cdot s_K(k) \cdot y(\varpi)^{-1}.$$

\begin{dfn} \label{D:assoc}
For $K, K' \in \mca{K}$, two splittings $s_K: K\into \wt{G}$ and $s_{K'}: K'\into \wt{G}$ are called associated if 
they agree on $\mbf{T}_{Q,n}(O)$.
\end{dfn}
In particular, if $K=K'$, then $s_K$ and $s_K' = f\otimes s_K, f\in \Hom(K, \mu_n)$ are associated only if $f\in \Hom(K, \mu_n)^\natural$.

\begin{conj} \label{C:varK}
There is a canonical splitting $s_\varphi$ of $\widehat{\Gamma_{G_{Q,n}}^{\rm tor}}$ over $\widehat{\Gamma_{G_{Q,n}}^{\rm tor}}/H_\varphi$, as depicted in \eqref{CD:K2}. Moreover, there is a bijection
$$\begin{tikzcd}
\mca{L}(\phi_\chi) \ar[r, leftrightarrow] & \Irr(\mca{S}(\phi_\chi))
\end{tikzcd}$$
which we denote conveniently by $\pi(\phi_\chi, \rho) \leftrightarrow \rho$, such that
\begin{enumerate}
\item[--] for every $K\in \mca{K}$ and splitting $s_K: K \into \wt{G}$ with $s_K = s_{Q,n}$ on $\mbf{T}(O)$, 
\item[--] for every $f_z \in \Hom(K, \mu_n)^\natural$,
\item[--] for every $y \in \widehat{\Gamma_G^{\rm tor}}$ with $y\cdot s_K$ being associated with $s_K$,
\end{enumerate}
one has that $\pi(\phi_\chi, \rho)$ is $(K, s_K)$-unramified if and only if $\pi(\phi_\chi, \gamma_{z,y} \cdot \rho)$ is $(y\cdot K, (y\cdot f_z)\otimes (y\cdot s_K))$-unramified, where
$$\gamma_{z,y} = \big(\varphi\circ h \circ i_K^\natural(f_z)\big) \cdot \big( s_\varphi \circ \varphi(y) \big) \in \widehat{\Gamma_{G_{Q,n}}^{\rm tor}}.$$
\end{conj}
Note that $y\cdot f_z \in \Hom(y\cdot K, \mu_n)^\natural$ and in fact $y\cdot f_z \otimes y\cdot s_K = y\cdot (f_z \otimes s_K)$.

\begin{rmk}
There is an action of $\widehat{\Gamma_G^{\rm tor}}$ on the set
$$\tilde{\mca{K}}= \set{(K, s_K): K\in \mca{K} \text{ and $s_K$ is a splitting of $K$ into } \wt{G}}/_\sim,$$
where $\sim$ means modulo the conjugation action of $G$. The action is still free, but will be not be transitive in general. This implies that  the notion of unramifiedness is more delicate for covering groups.
\end{rmk}

\subsubsection{Different models of $I(\chi)$}
 To facilitate later computations, we provide some further analysis of a $(K, s_K)$-unramified genuine principal series $I(\chi)$ of $\wt{G}$, especially on its different concrete realizations. Recall that $\chi: Z(\wt{T}) \to \C^\times$ is a central character and 
 $$\tchi: \wt{A} \longrightarrow \C^\times$$ an extension to a maximal abelian subgroup $\wt{A} \subset \wt{T}$. The isomorphism class $i(\chi)$ only depends on $\chi$, and not on the chosen $\wt{A}$ and the extension $\tchi$. The commutator pairing
 $$[-, -]: \wt{T}/Z(\wt{T}) \times \wt{T}/Z(\wt{T}) \longrightarrow \C^\times$$
 is non-degenerate, and thus every character of $\wt{T}/Z(\wt{T})$ is of the form $[-, t]$ for some $t\in \wt{T}$. In particular, if we fix $\wt{A}$ and consider two extensions $\tchi, \tchi'$ of $\chi$ to $\wt{A}$,  we get a character
 $$\tchi/\tchi': \wt{A}/Z(\wt{T}) \longrightarrow \C^\times,$$
then there exists $t\in \wt{T}$ such that
 $$\tchi' = {}^t\tchi,$$
 where 
 $${}^t\tchi (x) = \tchi(t^{-1} x t)=[x, t] \cdot \tchi(x).$$
 For any fixed lifting $\wt{t} \in \wt{T}$ of $t$, we have an isomorphism
 $$\phi_{\wt{t}}: \Ind_{\wt{A}}^{\wt{T}}(\tchi) \longrightarrow \Ind_{\wt{A}}^{\wt{T}} ({}^t\tchi)$$
given by
 $$\phi_{\wt{t}}(f)(x) = f(\wt{t}^{-1}x).$$
 Hence, there is an induced isomorphism
  \begin{equation} \label{Iso-2m}
  \Phi_{\wt{t}}: \Ind_{\wt{A}U}^{\wt{G}} (\tchi) \longrightarrow \Ind_{\wt{A}U}^{\wt{G}} ({}^t\tchi)
  \end{equation}
given by
$$ \Phi_{\wt{t}}(f)(g) = f(\wt{t}^{-1}g) \text{ for } g\in \wt{G},$$
 where we note that the two sides are different models for the same isomorphism class $I(\chi)$.

\subsection{Whittaker datum varied} \label{SS:psi-var}
We consider in this subsection how the Whittaker dimension of an element inside an unramified L-packet varies with respect to different Whittaker data. 

\subsubsection{Linear case} First we recall the results in the linear case, see \cite{Kuo1, Kuo2, GGP1, Kal4, Mis1}.
Let
$$\mfr{w}=(B, U, \psi)$$
be a fixed Whittaker datum such that $\psi: U \longrightarrow \C^\times$ is a non-degenerate character. The group $T=\mbf{T}(F)$ acts on the set of such non-degenerate characters $\psi$ and gives finite orbits. On the other hand, from the short exact sequence
$$Z(\mbf{G}) \into \mbf{G} \onto \mbf{G}_{ad}$$
we obtain by Galois cohomology
$$\begin{tikzcd}
Z(G) \ar[r] & G \ar[r] & G_{ad} \ar[r] & H^1(F, Z(\mbf{G})) \ar[r] & ... ,
\end{tikzcd}$$
where $G$ and $G_{ad}$ denote the $F$-points of $\mbf{G}$ and $\mbf{G}_{ad}$ respectively. One has an embedding
$$G_{ad}/G \into H^1(F, Z(\mbf{G}))$$
and also an isomorphism
$$T_{ad}/T \simeq G_{ad}/G,$$
which acts simply and transitively on the set of $T$-orbits of non-degenerate characters of $U$. We denote the action by ${}^t \mfr{w}$ for every $t\in T_{ad}$, or simply ${}^t \psi$, since the action of $t$ is only on $\psi$.

On the other hand, for every unramified $\uchi$, there is an embedding
$$\mca{S}(\phi_{\uchi}) \into X/X^{sc},$$
and by applying $\Hom(-, \Q/\Z)$ we obtain the surjection
$$Y_{ad}/Y = P(X^{sc})/Y \onto \Irr(\mca{S}(\phi_{\uchi})),$$
where $Y_{ad} = P(X^{sc})$ is the cocharacter lattice of $\mbf{G}_{ad}$. Note that the valuation map $v_F: F^\times \to \Z$ induces a well-defined map
$$v_F: T_{ad}/T \longrightarrow Y_{ad}/Y.$$
We thus have a surjection
\begin{equation} \label{TacS}
\begin{tikzcd}
T_{ad}/T \ar[r, two heads] &  \Irr(\mca{S}(\phi_{\uchi})).
\end{tikzcd}
\end{equation}
In fact, by using Tate duality, one can define a natural map (see \cite[\S 3.1]{Mis1}) 
$$H^1(F, Z(\mbf{G})) \longrightarrow \Irr(\mca{S}(\phi_{\uchi})),$$
the composite of which with the inclusion $T_{ad}/T \into H^1(F, Z(\mbf{G}))$ gives rise to \eqref{TacS}.

It is shown in \cite[Theorem 12]{Mis1} that for an unramified L-packet $\mca{L}(\phi_{\uchi})$ associated with an unramified $\uchi$, the representation $\pi(\phi_{\uchi}, \rho)$ is $\psi$-generic if and only if $\pi(\phi_{\uchi}, t\cdot \rho)$ is ${}^t \psi$-generic, for every $t\in T_{ad}$. Here $\psi$-genericity of $\pi \in \Irr(G)$ means that
$$\dim \Wh_\psi(\pi) = 1.$$ 
Moreover, $\pi(\phi_{\uchi}, t\cdot \rho)$ is the unique ${}^t \psi$-generic element in $\mca{L}(\phi_{\uchi})$ for the non-degenerate character ${}^t \psi$, see \cite{Var1, Ato1} and references therein.

\subsubsection{Speculations for coverings}
It is natural to consider the analogous question for covering groups and an unramified L-packet $\mca{L}(\phi_\chi)$, by viewing $U$ as a subgroup of $\wt{G}$ via the canonical splitting. However, similar to the earlier discussion in the preceding subsection, we encounter the problem that on the one hand the Whittaker datum has a natural action by $T_{ad}/T$, depending on $G$; on the other hand, the $\mca{S}$-group pertains to the dual group $\wt{G}^\vee$ and thus to the principal endoscopic group $G_{Q,n}$. It is therefore crucial to relate the two aspects.

We denote by $G_{Q,n,ad}$ and $T_{Q,n,ad}$ the $F$-rational points of the quotient adjoint group $\mbf{G}_{Q,n,ad}$ of $\mbf{G}_{Q,n}$ and the split torus $\mbf{T}_{Q,n,ad}$ inside $\mbf{G}_{Q,n,ad}$, respectively. We ask when is there a natural map
$$T_{ad}/T \longrightarrow T_{Q,n,ad}/T_{Q,n}?$$
In fact, for our purpose, it suffices to consider if there is a natural map
$$X_{Q,n}/X_{Q,n}^{sc} \longrightarrow X/X^{sc},$$
or equivalently, a map
$$Z(\mbf{G}) \longrightarrow Z(\mbf{G}_{Q,n}).$$

For simplicity of discussion, we assume that $\mbf{G}$ is a semsimple group through the rest of this subsection. In this case, one has from \eqref{D:r-01} the following diagram
\begin{equation}
\begin{tikzcd}
X_{Q,n}/X & X_{Q,n}/X^{sc} \ar[l, two heads] & X/X^{sc} \ar[l, hook', "j_0"'] \\
& X_{Q,n}/X_{Q,n}^{sc} \ar[u, hook, "\varphi_0"] \ar[ru, hook, dashed, "{\tilde{\varphi}_0}"']
\end{tikzcd}
\end{equation}
Taking $\varphi_0$ to be the natural candidate, the question is whether it factors through $j_0$. In the special case of $G$ being almost simple and simply-connected, the answer is as follows.

\begin{prop} \label{P:sc-OK}
Assume $\mbf{G}$ is almost simple and simply-connected. If 
$$(G, \wt{G}^\vee)\ne (\Spin_{2r+1}, \SO_{2r+1}),$$
where the equality holds exactly when $r$ is odd and $n_{\alpha} \equiv 2 \mod 4$ for every short coroot $\alpha^\vee$ of $\Spin_{2r+1}$, then the map $\varphi_0$ factors through $j_0$, i.e.,
$$\varphi_0(X_{Q,n}/X_{Q,n}^{sc}) \subset X/X^{sc}.$$
\end{prop}
\begin{proof}
For every $x = \sum_{\alpha\in \Delta} c_\alpha \alpha \in X_{Q,n}, c_\alpha \in \Q$, we want to show that 
$$\varphi_0(x) = \sum_{\alpha \in \Delta} c_\alpha n_\alpha \alpha$$ 
lies in $X$, or equivalently,
$$\angb{\varphi_0(x)}{y} = \sum_{\alpha\in \Delta} c_\alpha n_\alpha \angb{\alpha}{y}$$
lies in $\Z$ for every $y\in Y$. Our argument is essentially a case by case discussion. First, assume $\mbf{G}$ is simply-laced. Then $\alpha \mapsto n_\alpha$ is constant on $\Delta$ and we have
$$Y_{Q,n}^{sc} = n_\alpha \cdot Y \subset Y_{Q,n}.$$
In this case, 
$$\angb{\varphi_0(x)}{y} = \angb{x}{n_\alpha y} \in \Z,$$
since $x\in X_{Q,n}$ and $n_\alpha y\in Y_{Q,n}$. Thus, we are left with $\mbf{G}$ being of type $B_r, C_r$. Note that though $F_4$ and $G_2$ have two lengths for the simple roots, we always have $X_{Q,n}=X_{Q,n}^{sc}$ for such groups, and thus there is nothing to check. The two nontrivial cases will be for $B_r, C_r$ and when $[X_{Q,n}: X_{Q,n}^{sc}]=2$, i.e., when the dual group $\wt{G}^\vee$ of $\wt{G}$ is of adjoint type. In this case $Y_{Q,n} = Y_{Q,n}^{sc}$.

First, if $\mbf{G} = \Sp_{2r}$ and $\wt{G}^\vee$ is adjoint, then $G_{Q,n} = \Sp_{2r}$ as well. In particular, $G_{Q,n}$ has the same root system type as $G$, and thus $\alpha \mapsto n_\alpha$ is constant on $\Delta$. In this case, $$Y_{Q,n}^{sc} = n_\alpha \cdot Y^{sc} = Y.$$ 
Hence $\angb{\varphi_0(x)}{y} \in \Z$ for all $x\in X_{Q,n}$ and $y\in Y$, and this gives $\varphi_0(x) \in X$.

Second, if $\mbf{G} = \Spin_{2r+1}$ and $\wt{G}^\vee$ is adjoint, then there are exactly two possibilities as follows. Assume that $\alpha_r^\vee$ is unique long simple coroot in $\Delta^\vee$, and $\alpha_i^\vee \in \Delta^\vee$ is short for $1\lest i \lest r-1$.
$$\wt{G}^\vee =
\begin{cases}
{\rm PGSp}_{2r} & \text{ if $n_{\alpha_1}$ is odd},\\
\SO_{2r+1} & \text{ if $r$ is odd and $n_{\alpha_1} \equiv 2 \mod 4$}.
\end{cases}
$$
Now if $\wt{G}^\vee ={\rm PGSp}_{2r}$, then $G_{Q,n}=\Spin_{2r+1}$ and thus $\alpha \mapsto n_\alpha$ is constant on $\Delta$. In this case, the same argument as above gives that $\varphi_0(x)\in X$ for every $x\in X_{Q,n}$. Now if $\wt{G}^\vee=\SO_{2r+1}$, then we get $G_{Q,n}=\Sp_{2r}$ with 
$$X_{Q,n}^{sc} = \Z[\alpha_1/n_{\alpha_1}] \oplus \Z[\alpha_2/n_{\alpha_2}] \oplus ... \oplus \Z[\alpha_r/n_{\alpha_r}]$$
where $n_{\alpha_r}= n_{\alpha_1}/2$, and also
$$X_{Q,n}^{sc} = \Z[\alpha_1/n_{\alpha_1}] \oplus \Z[\alpha_2/n_{\alpha_1}] \oplus ... \oplus \Z[\alpha_r/n_{\alpha_1}]$$
Thus, $\varphi_0(\alpha_r/n_{\alpha_1}) = \alpha_r/2 \notin X$. This completes the proof.
\end{proof}

To proceed with our discussion, we assume $\varphi_0$ factors through $j_0$ and thus one has a well-defined map
$$\tilde{\varphi}_0: X_{Q,n}/X_{Q,n}^{sc} \longrightarrow X/X_G^{sc}.$$
Note that since we assumed $G$ is semisimple, one has $\Gamma_G = \Gamma^{\rm tor}$, and similarly for $G_{Q,n}$. For every unramified genuine character $\chi: Z(\wt{T}) \to \C^\times$, this gives a commutative diagram
\begin{equation} \label{E:com1}
\begin{tikzcd}
\widehat{\Gamma_G} \ar[r, equal] & P(X^{sc})/Y \ar[d, two heads, "\varphi"] & T_{ad}/T \ar[l, two heads, "{v_F}"'] \ar[r] & H^1(F, Z(\mbf{G})) \ar[d] \\
\widehat{\Gamma_{G_{Q,n}}} \ar[r, equal] & P(X_{Q,n}^{sc})/Y_{Q,n} \ar[rd, two heads, "{q_\chi}"'] & T_{Q, n, ad}/T_{Q,n} \ar[d, two heads] \ar[l, two heads, "{v_F}"'] \ar[r] & H^1(F, Z(\mbf{G}_{Q,n})) \ar[ld] \\
& & \Irr(\mca{S}(\phi_\chi)),
\end{tikzcd}
\end{equation}
where $\varphi$ is induced from $\varphi_0$ and the right vertical map arises from $\tilde{\varphi}_0^*: Z(\mbf{G}) \longrightarrow Z(\mbf{G}_{Q,n})$ and the functoriality of Galois cohomology. By our assumption on $\varphi_0$, the map 
$$\varphi: P(X^{sc})/Y_{Q,n} \onto P(X_{Q,n}^{sc})/Y_{Q,n}$$
 as in \eqref{CD:K1} factors through the quotient $\varphi: P(X^{sc})/Y_{Q,n} \onto \widehat{\Gamma_G}$, giving rise to the map (still denoted by) $\varphi$ in \eqref{E:com1}. For convenience, we denote by
 $$\zeta_\chi:=q_\chi \circ \varphi \circ v_F: T_{ad}/T \onto \Irr(\mca{S}(\phi_\chi))$$
 for the composed surjection.

\begin{conj} \label{C:varW}
Assume that $\tilde{\varphi}_0: X_{Q,n}/X_{Q,n}^{sc} \longrightarrow X/X_G^{sc}$ is well-defined. Then for every Whittaker data $\mfr{w}=(B, U, \psi)$ and $(K, s_K)$-unramified genuine principal series $I(\chi)$, one has
$$\dim \Wh_{\psi}(\pi(\phi_\chi, \rho)) =\dim \Wh_{{}^t \psi}(\pi(\phi_\chi, \zeta_\chi(t) \otimes \rho))$$
for every $\pi(\phi_\chi, \rho) \in \mca{L}(\phi_\chi)$ and $t\in T_{ad}/T$.
\end{conj}

\begin{rmk} \label{R:ano-1}
It seems to us that the map $\varphi_0: X_{Q,n}/X_{Q,n}^{sc} \to X_{Q,n}/X^{sc}$ considered here is the most natural one, though its image may not be in $X/X^{sc}$, as shown in Proposition \ref{P:sc-OK} even for $G$ almost simple and simply-connected. On the one hand, one could definitely consider any other ``natural" map $\varphi_0': X_{Q,n}/X_{Q,n}^{sc} \to X_{Q,n}/X^{sc}$ whose image lies in $X/X^{sc}$, for example, $x \mapsto n x$. However, such a map may not be injective and thus the induced map 
$$\varphi': P(X^{sc})/Y \to P(X_{Q,n}^{sc})/Y_{Q,n}$$
may not be surjective. It is expected that the analogue of Conjecture \ref{C:varW} may not hold for such $\varphi'$. Indeed, if $\varphi_0'$ is the trivial map, then the action of $t\in T_{ad}/T$ on $\Irr(\mca{S}(\phi_\chi))$ via $\varphi'$ is trivial; however, $\dim \Wh_{{}^t \psi}(\pi(\phi_\chi, \rho))$ may not equal to $\dim \Wh_{\psi}(\pi(\phi_\chi, \rho))$. On the other hand, we expect that a cover $\wt{G}$ such that $G_{Q,n}$ has simply-connected derived group but of a different root system type is quite unusual and exhibits many interesting properties. The pair 
$$(G, \wt{G}^\vee) = (\Spin_{2r+1}, \SO_{2r+1})$$
 in Proposition \ref{P:sc-OK} is such an example. Another example is 
 $$(G, \wt{G}^\vee) = (\GSp_{2r}, {\rm PGSp}_{2r} \times \GL_1)$$
  for even fold cover of $\GSp_{2r}$ with $r$ even, see \S  \ref{SSS:GSp-ev}. For this latter example, we will have a more detailed analysis in the last section \S \ref{S:GSp} of this paper.
\end{rmk}

\subsubsection{Scattering matrix for $\mfr{f}(\psi) = \mfr{p}_F$}
To facilitate some computations later, we recall briefly the scattering matrix associated with $T(w, i(\chi))$ when
$$\mfr{f}(\psi) = \mfr{p}_F.$$
We use freely the notations in \cite{GSS2} and those in earlier sections of this paper. There is a square matrix
$$[\tau(w, \chi, \gamma, \gamma')]_{\gamma, \gamma'\in \wt{T}/\wt{A}}$$
of size $\val{Y/Y_{Q,n}}$ representing the operator
$$T(w, \chi; r_w^{\rm un})^*: \Wh_\psi(I({}^w\chi)) \longrightarrow \Wh_\psi(I(\chi))$$
such that
$$T(w, \chi; r_w^{\rm un})^*(\lambda_{\gamma}^{{}^w\chi}) = \sum_{\gamma'\in \wt{T}/\wt{A}} \tau(w, \chi, \gamma, \gamma') \cdot \lambda_{\gamma'}^{\chi}.$$
In particular, one has a local scattering matrix 
$$\mca{S}_\mfr{R}(w, \chi):=[\tau(w, \chi, \s_{y_1}, \s_y)]_{y_1, y\in \mfr{R}}$$
indexed by $\mfr{R}$ for every ordered set $\mfr{R} \subset Y$ of representatives $Y/Y_{Q,n}$. Some immediate properties are:

\begin{enumerate}
\item[$\bullet$] For $w\in W$ and $\wt{z}, \wt{z}'\in \wt{A}$, the identity holds:
$$\tau(w, \chi, \gamma \cdot \wt{z}, \gamma' \cdot \wt{z}')=({}^w \chi)^{-1}(\wt{z}) \cdot \tau(w, \chi, \gamma, \gamma') \cdot \chi(\wt{z}').$$
\item[$\bullet$] For $w_1, w_2 \in W$ such that $l(w_2w_1)=l(w_2) + l(w_1)$, one has
$$\tau(w_2w_1, \chi, \gamma, \gamma')=\sum_{\gamma''\in \wt{T}/\wt{A}} \tau(w_2, {}^{w_1}\chi, \gamma, \gamma'') \cdot \tau(w_1, \chi, \gamma'', \gamma'),$$
which is referred to as the cocycle relation.
\end{enumerate}
To explicate the matrix for a simple reflection $w=w_\alpha$ associated with $\alpha\in \Delta$, we fix the Haar measure $\mu$ of $F$ such that $\mu(O_F)=1$; thus, $\mu(O_F^\times)=1-1/q$. The Gauss sum is defined by
$$G_\psi(a, b)=\int_{O^\times_F} (u, \varpi)_n^a \cdot \psi(\varpi^b u) \mu(u), \quad a, b\in \Z.$$
It satisfies (compared with \cite[\S 3.6]{GSS2}, noting that $\psi$ has conductor $\mfr{p}_F$) 
\begin{equation*}
G_\psi(a, b)=
\begin{cases}
0 & \text{ if } b<0, \\
1-1/q & \text{ if } n| a, b\gest 1,\\
0 &\text{ if } n\nmid a, b\gest 1, \\
-1/q &\text{ if } n|a, b=0,\\
G_\psi(a, 0) \text{ with } |G_\psi(a,0)|=q^{-1/2} &\text{ if } n\nmid a, b=0.
\end{cases}
\end{equation*}
One has $\overline{G_\psi(a, b)}=(-1, \varpi)_n^a \cdot G_\psi(-a, b)$. For every $k\in \Z$, we write
$$\mathbf{g}_{\psi}(k):=G_{\psi}(k, 0).$$


\begin{prop} \label{P:tau-p}
Suppose $\gamma=\s_{y_1}$ is represented by $y_1$ and $\gamma'=\s_y$ by $y$. Then we can write 
$$\tau(w_\alpha, \chi, \gamma, \gamma')=\tau^1(w_\alpha, \chi, \gamma, \gamma') + \tau^2(w_\alpha, \chi, \gamma, \gamma')$$ with the following properties:
\begin{enumerate}
\item[$\bullet$] $\tau^i(w_\alpha, \chi,\gamma \cdot \wt{z}, \gamma' \cdot \wt{z}')=({}^{w_\alpha} \chi)^{-1}(\wt{z}) \cdot \tau^i(w_\alpha, \chi, \gamma, \gamma') \cdot \chi(\wt{z}') \text{ for every } \wt{z}, \wt{z}'\in \wt{A}$,
\item[$\bullet$] $\tau^1(w_\alpha, \chi, \gamma, \gamma')=0$  unless  $y_1 \equiv y \mod Y_{Q,n}$,
\item[$\bullet$] $\tau^2(w_\alpha, \chi, \gamma, \gamma')=0$   unless $y_1 \equiv w_\alpha(y) \mod Y_{Q,n}$.
\end{enumerate}
Moreover, 
\begin{enumerate}
\item[$\bullet$] if $y_1= y$, then 
$$\tau^1(w_\alpha, \chi, \gamma, \gamma')=(1-q^{-1}) \frac{\chi (\wt{h}_\alpha(\varpi^{n_\alpha}))^{k_{y,\alpha}}}{1-\chi (\wt{h}_\alpha(\varpi^{n_\alpha}))}, \text{ where } k_{y,\alpha}=\ceil{\frac{1+\angb{y}{\alpha}}{n_\alpha}};$$
\item[$\bullet$] if $y_1=w_\alpha(y)$, then
$$\tau^2(w_\alpha, \chi, \gamma, \gamma') = (-1, \varpi)_n^{\angb{y}{\alpha}\cdot D(y,\alpha^\vee)} \cdot \g(\angb{y}{\alpha}Q(\alpha^\vee)).$$
\end{enumerate}
\end{prop}
\begin{proof}
The computation in \cite{Mc2} (with refinement given in \cite{Ga2}) applies here mutatis mutantis, by noting that $\psi$ has conductor $\mfr{p}_F$ instead of $O_F$ as in the references mentioned. We omit the details of the computation.
\end{proof}

\subsection{Some simple examples} \label{SS:varKW-eg}
We illustrate on the previous discussion, especially on Conjectures \ref{C:varK} and Conjecture \ref{C:varW} by considering some simple examples. These include covers of  $\SL_2, \SO_3$ and $\GL_r$. We focus only on unitary $(K, s_K)$-unramified principal series of these groups.

\subsubsection{Covers of $\SL_2$}
Let $\wt{\SL}_2$ be the $n$-fold cover of $\SL_2$ associated with 
$$Q(\alpha^\vee) = -1.$$
Let $K=\SL_2(O)$ and 
$$K'=e \cdot K = e(\varpi) \cdot K\cdot e(\varpi)^{-1},$$
where $e:=\alpha^\vee/2 \in Y_{ad}$ is a generator for the cocharacter lattice for $\PGL_2$. Here $e \in \widehat{\Gamma_G^{\rm tor}}$ is the nontrivial element and 
$$\mca{K} =\set{K, K'}.$$
Consider the two unique splittings $s_K$ and $s_{K'}$ of $K$ and $K'$ respectively, one has 
\begin{equation} \label{E:2K}
s_K' =[e(\varpi), -] \otimes s_K
\end{equation}
on $\mbf{T}(O)$, where 
\begin{equation} \label{E:SL-ad-cm}
[e(\varpi), \alpha^\vee(a^k)]= (\varpi, a)_n^{-k}
\end{equation}
for all $a\in F^\times$ and $k\in \Z$, see \cite[\S 4.13]{BD}.  There are two cases depending on the parity of $n$. 

First, if $n$ is odd, then it is easy to see
$$Y_{Q,n}= Y_{Q,n}^{sc}=\Z[n\alpha^\vee],\  X_{Q,n}=\Z[\alpha/2n], \ X_{Q,n}^{sc}=\Z[\alpha/n].$$
This immediately gives that the diagram in \eqref{CD:K1} becomes
$$ \begin{tikzcd}
\Z[\alpha^\vee]/\Z[n\alpha^\vee] \ar[d, "\simeq"] & & \\
\Z[\alpha^\vee]/\Z[n\alpha^\vee] \ar[r, hook] \ar[d, two heads, "\varphi"]  & \Z[\alpha^\vee/2]/\Z[n\alpha^\vee] \ar[d, two heads, "\varphi"] \ar[r, two heads] & \Z[\alpha^\vee/2]/\Z[\alpha^\vee] \ar[d, two heads, "\varphi"] \ar[r, equal] &   \widehat{\Gamma_G^{\rm tor}} \\
\set{0} \ar[r]& \Z[n\alpha^\vee/2]/\Z[n\alpha^\vee] \ar[d, equal] \ar[r, "\simeq"] & \Z[n\alpha^\vee/2]/\Z[n\alpha^\vee] \\
& \widehat{\Gamma_{G_{Q,n}}^{\rm tor}} \ar[r, two heads] & \Irr(\mca{S}(\phi_\chi)),
\end{tikzcd}$$
where $\varphi(x) = n\cdot x$ for middle vertical map. The existence of the splitting $s_\varphi$ is clear.  As mentioned, we concentrate on the case where $\chi$ is unramified and unitary. 

If $n|k$, or equivalently $k\alpha^\vee \in Y_{Q,n}$ and thus $k\alpha^\vee(a)\in Z(\wt{T})$, then 
$$[e(\varpi), \alpha^\vee(a^k)]=1.$$
This shows that $s_K$ and $s_K'$ agree on $\mbf{T}_{Q,n}(O)$, and thus are associated in the sense of Definition \ref{D:assoc}. It follows that in this case $I(\chi)$ is both $(K, s_K)$ and $(K', s_{K'})$-unramified. 

If $I(\chi)$ is unitary unramified and irreducible, then 
$$\mca{L}(\phi_\chi) = \set{I(\chi)}.$$
It is easy to see that Conjecture \ref{C:varK} holds. On the other hand, if $I(\chi)$ is a unitary $(K, s_K)$-unramified and reducible, then
it occurs exactly when 
$$\chi_\alpha:=\chi(\wt{h}_\alpha(\varpi^{n_\alpha})) = -1$$
and thus $R_\chi = W$. In particular, ${}^{w_\alpha} \chi = \chi$.  Note that we have an isomorphism $R_\chi \simeq \mca{S}(\phi_\chi)$ (see \cite{Ga7}) and thus we always make such an identification. In this case,
$$\mca{L}(\phi_\chi) = {\rm JH}(I(\chi)) = \set{\pi(\phi_\chi, \mbm{1}), \pi(\phi_\chi, \varepsilon_W)},$$
where $\varepsilon_W$ is the nontrivial character of $\mca{S}(\phi_\chi)$.

\begin{prop} \label{P:SL2-K}
There is exactly one $\pi(\phi_\chi, \rho)$ which is $(K, s_K)$-unramified, while the other is $(K', s_{K'})$-unramified. In particular, this verified Conjecture \ref{C:varK}.
\end{prop}
\begin{proof} The proof uses the same idea as in the linear case, that is, to compute the coefficient of the intertwining operator applied to the unramified vector. More precisely, recall that the normalized intertwining operator 
$$\mca{A}(w_\alpha, \chi) = \gamma(w_\alpha, \chi)\cdot T(w_\alpha, \chi; r_{w_\alpha}^{\rm un})$$
acts on $\pi(\phi_\chi, \rho) \subset I(\chi)$ by $\rho(w_\alpha) \cdot \id$ for each $\rho \in \Irr(\mca{S}(\phi_\chi)) =\Irr(W)$.

Note that as character of $\wt{A}$, we have
$${}^{e(\varpi)} \tchi =[-, e(\varpi)] \otimes \tchi.$$
However, the problem is that $e(\varpi) \notin \wt{T}$. Setting
$$t:=\varpi^{(n+1)\alpha^\vee/2} \in \wt{T},$$
it is easy to see
$${}^t \tchi = {}^{e(\varpi)} \tchi,$$
since $[-,t] = [-, e(\varpi)]$ as a function on $T$. Let $\tchi=\chi \boxtimes \tchi_O$ be the unique extension of $\chi$ such that $\tchi_O \circ s_K = \mbm{1}$ on $\mbf{T}(O)$. We see that
$${}^t \tchi \circ s_{K'}|_{\mbf{T}(O)} =\mbm{1}$$
and 
\begin{equation} \label{s2K}
s_{K'} = [t, -] \otimes s_K
\end{equation}
on $\mbf{T}(O)$. Here $\tchi$ and ${}^t \tchi$ are both extensions of $\chi$ and give rise to two models 
$$\Ind_{\wt{A}U}^{\wt{G}}(\tchi),\  \Ind_{\wt{A}U}^{\wt{G}}({}^t\tchi)$$ of the principal series $I(\chi)$. For any fixed lifting $\wt{t}$ of $t$, we have the isomorphism
$$\Phi_{\wt{t}}: \Ind_{\wt{A}U}^{\wt{G}} (\tchi) \longrightarrow \Ind_{\wt{A}U}^{\wt{G}} ({}^t \tchi)$$
given by $\Phi_{\wt{t}}(f)(g) = f(\wt{t}^{-1} g)$, see \eqref{Iso-2m}.

Without loss of generality, we assume that $\pi(\phi_\chi, \mbm{1})$ is $(K, s_K)$-unramified and thus 
$$\mca{A}(w_\alpha, \chi)(f_{K}) = f_K,$$
where $f_{K} \in \Ind_{\wt{A}U}^{\wt{G}}(\tchi)$ is the normalized $(K, s_K)$-unramified vector. We want to compute the scalar arising from $\mca{A}(w_\alpha, \chi)$ applied to the $(K', s_{K'})$-unramified vector. Note that in the model $\Ind_{\wt{A}U}^{\wt{G}}({}^t\tchi)$ the $(K', s_{K'})$-unramified vector is given by
$$f_{K'}(b\cdot s'(k')) =
\begin{cases}
\delta_B^{1/2}(t) \cdot \tchi(a) & \text{ if } b=a u \text{ with } a\in \wt{A}, u\in U \text{ and } k'\in K',\\
0 & \text{ otherwise.}
\end{cases} 
$$
Thus, $\Phi_t^{-1}(f_{K'}) \in \Ind_{\wt{A}U}^{\wt{G}}(\tchi)$ is the $(K', s_{K'})$-unramified vector.  We want to determine the scalar $c\in \C$ such that
$$T(w_\alpha, \chi; r_{w_\alpha}^{\rm un})(\Phi_{\wt{t}}^{-1}(f_{K'}) ) = c\cdot \Phi_{\wt{t}}^{-1}(f_{K'}).$$
We see that
$$\begin{aligned}[t]
c = & T(w_\alpha, \chi; r_{w_\alpha}^{\rm un})(\Phi_{\wt{t}}^{-1}(f_{K'}) ) (\wt{t}^{-1})  \\
= & \int_U \Phi_{\wt{t}}^{-1}(f_{K'})(\wt{w}_\alpha^{-1} u \wt{t}^{-1}) du \\
 = & \int_U f_{K'}(\wt{t} \cdot \wt{w}_\alpha^{-1} u \cdot \wt{t}^{-1}) du \\
 = & \int_U f_{K'}(t \cdot \wt{w}_\alpha^{-1} t^{-1} \wt{w}_\alpha \cdot \wt{w}_\alpha^{-1} \cdot t u t^{-1}) du \\
 = & \int_U f_{K'}(\wt{h}_\alpha(\varpi^{n}) \cdot \wt{h}_\alpha(\varpi) \cdot \wt{w}_\alpha^{-1} \cdot t u t^{-1}) du,
 \end{aligned}
$$
where the last equality follows from \eqref{W-act}. We have
$$c  =q \cdot \chi_\alpha \cdot \int_U f_{K'}(\wt{h}_\alpha(\varpi) \cdot \wt{w}_\alpha^{-1} \cdot u) du.$$
Now to compute the integral, we split $U$ into $\mfr{p}_F$ and $U- \mfr{p}_F$.
If $u\in \mfr{p}_F$, it is easy to see that $\wt{h}_\alpha(\varpi) \wt{w}_\alpha^{-1} u \in s_{K'}(K')$ and thus
$$\int_{\mfr{p}_F} f_{K'}(\wt{h}_\alpha(\varpi) \cdot \wt{w}_\alpha^{-1} \cdot u) du = d(\mfr{p}) = q^{-1}.$$
On the other hand,
$$\begin{aligned}[t]
& \int_{U-\mfr{p}_F} f_{K'}(\wt{h}_\alpha(\varpi) \cdot \wt{w}_\alpha^{-1} \cdot \wt{e}_\alpha(u)) du \\
= &   \int_{U-\mfr{p}} f_{K'}(\wt{h}_\alpha(\varpi) \cdot \wt{h}_\alpha(u^{-1}) \cdot \wt{e}_\alpha(-u) \cdot \wt{e}_{-\alpha}(-u^{-1}) ) du \\
= & \sum_{k\lest 0} \int_{u\in O^\times} f_{K'}(\wt{h}_\alpha(\varpi) \cdot \wt{h}_\alpha(\varpi^{-k} u^{-1}) \cdot \wt{e}_\alpha(-u)) d(\varpi^k u)   \\
= & \sum_{k\lest 0} \val{\varpi^k} \int_{u\in O^\times} f_{K'}(\wt{h}_\alpha(\varpi) \cdot \wt{h}_\alpha(\varpi^{-k})  \wt{h}_\alpha(u^{-1}) \cdot (u, \varpi)_n^{kQ(\alpha^\vee)} ) d(u)   \\
= & \sum_{k\lest 0} \val{\varpi^k} \int_{u\in O^\times} f_{K'}(\wt{h}_\alpha(\varpi) \cdot \wt{h}_\alpha(\varpi^{-k})  s_K(h_\alpha(u^{-1})) \cdot (u, \varpi)_n^{kQ(\alpha^\vee)} ) d(u).
\end{aligned}$$
Now by \eqref{s2K} we have
$$s_{K'}(h_\alpha(u^{-1})) = s_K(h_\alpha(u^{-1})) \cdot (u, \varpi)_n^{Q(\alpha^\vee)}.$$
This gives that the above integral over $U- \mfr{p}_F$ equals
$$\begin{aligned}[t]
& \sum_{k\lest 0} \val{\varpi^k} \int_{u\in O^\times} f_{K'}(\wt{h}_\alpha(\varpi) \cdot \wt{h}_\alpha(\varpi^{-k})  s_{K'}(h_\alpha(u^{-1})) \cdot (u, \varpi)_n^{(k-1)Q(\alpha^\vee)} ) d(u) \\
= & \sum_{k\lest 0, n|(k-1)} \val{\varpi^k} \int_{u\in O^\times} f_{K'}(\wt{h}_\alpha(\varpi^{1-k}) ) d(u) \\
= & q^{-1} \frac{\chi_\alpha (1-q^{-1})}{1-\chi_\alpha}.
\end{aligned}$$
Combining the above we get that
$$c= \chi_\alpha \cdot c_{\sf gk}(w_\alpha, \chi) = (-1) \frac{1-q^{-1} \chi_\alpha}{ 1-\chi_\alpha }$$
and therefore in our setting
$$\mca{A}(w_\alpha, \chi)(f_{K'}) = (-1) \cdot f_{K'}.$$
This concludes the proof.
\end{proof}

Second, we consider the case $n=2m$ is even, which gives then
$$Y_{Q,n}=\Z[m\alpha^\vee], \ Y_{Q,n}^{sc}=\Z[n\alpha^\vee],\  X_{Q,n}=X_{Q,n}^{sc}=\Z[\alpha/n].$$
In this case, diagram \eqref{CD:K1} becomes
$$ \begin{tikzcd}
\Z[\alpha^\vee]/\Z[m\alpha^\vee] \ar[d, "\simeq"] & & \\
\Z[\alpha^\vee]/\Z[m\alpha^\vee] \ar[r, hook] \ar[d, two heads, "\varphi"]  & \Z[\alpha^\vee/2]/\Z[m\alpha^\vee] \ar[d, two heads, "\varphi"] \ar[r, two heads] & \Z[\alpha^\vee/2]/\Z[\alpha^\vee] \ar[d, two heads, "\varphi"] \ar[r, equal] &   \widehat{\Gamma_G^{\rm tor}} \\
\set{0} \ar[r]& \Z[m\alpha^\vee]/\Z[m\alpha^\vee] \ar[d, equal] \ar[r] & \set{0} \\
& \widehat{\Gamma_{G_{Q,n}}^{\rm tor}} \ar[r, two heads] & \Irr(\mca{S}(\phi_\chi)),
\end{tikzcd}$$
As $\widehat{\Gamma_{G_{Q,n}}^{\rm tor}}=\set{0}$, every $\mca{L}(\phi_\chi)$ is a singleton and if $\chi$ is unitary, then 
$$\mca{L}(\phi_\chi)=\set{I(\chi)}.$$
Note that in this case, if $I(\chi)$ is $(K, s_K)$-unramified, then it is not $(K', s_{K'})$-unramified. Indeed, for every $z=i m\alpha^\vee \in Y_{Q,n}$ and $u\in O^\times$, it follows from \eqref{E:2K} and \eqref{E:SL-ad-cm} that
$$s_{K'}(z\otimes u) =(\varpi, u)_n^{-mi } \cdot s_K(z\otimes u),$$
where $(\varpi, u)_n^{mi} = (\varpi, u)_2^i$ may not be trivial in general. Thus, the two splittings $s_K$ and $s_{K'} = y\cdot s_K$ are not associated. This agrees with Conjecture \ref{C:varK}.

\vskip 10pt

Regarding the Whittaker dimension with respect to varying Whittaker datum, we assume that for the conductor,
$$\mfr{f}(\psi) = O_F.$$
For simplicity of notation, we denote ${}^e \psi:= {}^{e(\varpi)} \psi$ where $e(\varpi) \in T_{ad}/T$ such that
$$\mfr{f}({}^e \psi) = \mfr{p}_F.$$
Consider the $n$-fold cover $\wt{\SL}_2$ such that $n$ is odd and 
$$I(\chi) = \pi(\phi_\chi, \mbm{1}) \oplus \pi(\phi_\chi, \varepsilon_W),$$
where the labelling is such that $\pi(\phi_\chi, \mbm{1})$ is the unique $(K, s_K)$-unramified constituent, and thus $\pi(\phi_\chi, \varepsilon_W)$ is $(K', s_{K'})$-unramified by Proposition \ref{P:SL2-K}.

\begin{prop} \label{P:SL2-W}
Keep notations and conventions as above, one has
$$\dim \Wh_\psi(\pi(\phi_\chi, \mbm{1})) = \dim \Wh_{^e\psi}(\pi(\phi_\chi, \varepsilon_W))  = (n+1)/2$$
and 
$$\dim \Wh_\psi(\pi(\phi_\chi, \varepsilon_W)) = \dim \Wh_{^e\psi}(\pi(\phi_\chi, \mbm{1}))  = (n-1)/2.$$
In particular, Conjecture \ref{C:varW} is verified in this case.
\end{prop}
\begin{proof}
We first note that for any non-degenerate character $\psi: U \longrightarrow \C^\times$, there is a representation (see \cite[\S 5.4]{Ga7})
$$\sigma^{\rm Wh_\psi}: R_\chi \longrightarrow \GL(\Wh_\psi(I(\chi)))$$
given by 
$$\sigma^{\rm Wh_\psi}(w) =\mca{A}(w, i(\chi))^*,$$
the induced map from $\mca{A}(w, i(\chi)) = \gamma(w, \chi) \cdot T(w, \chi; r_w^{\rm un})$, where $\gamma(w, \chi)$ is the gamma-factor associated with $w$ and $i(\chi)$. It is proved in \cite[Theorem 5.6]{Ga7} that one has
\begin{equation} \label{E:dWh}
\dim \Wh_\psi(\pi(\phi_\chi, \rho)) = \angb{\rho}{\sigma^{\rm Wh_\psi}}_{R_\chi}.
\end{equation}
If $\mfr{f}(\psi)=O_F$ and $\pi(\phi_\chi, \mbm{1})$ is $(K, s_K)$-unramified, as is assumed here, then it is shown in \cite[Theorem 7.1]{Ga7} that 
$$\sigma^{\rm Wh_\psi} =\frac{n+1}{2} \mbm{1} \bigoplus \frac{n-1}{2} \varepsilon_W,$$
which immediately gives that
$$\dim \Wh_\psi(\pi(\phi_\chi, \mbm{1})) = \frac{n+1}{2} \text{ and } \dim \Wh_\psi(\pi(\phi_\chi, \varepsilon_W)) = \frac{n-1}{2}.$$
The proof relies on an explicit form of the scattering matrix for $\mfr{f}(\psi) = O$, as described in \cite{Ga2} and \cite{GSS2}.

Now to compute $\sigma^{\Wh_{{}^e\psi}}$, we use Proposition \ref{P:tau-p} instead, since $\mfr{f}({}^e\psi) =\mfr{p}_F$. We take the set
$$\mfr{R} = \set{i\alpha^\vee: -(n-1)/2 \lest i \lest (n-1)/2} \subset Y$$
of representatives of $Y/Y_{Q,n}$. In view of Proposition \ref{P:tau-p}, by permuting the set $\mfr{R}$ properly, the scattering matrix 
$$[\tau(w_\alpha, \chi, \s_z, \s_y)]_{z, y\in \mfr{R}}$$ is equal to a block-diagonal matrix with $(n+1)/2$ blocks $\mca{M}_i: 0\lest i \lest (n-1)/2$ such that
\begin{enumerate}
\item[--] for $i=0$ we get the size-one matrix
$$\mca{S}_\mfr{R}(w_\alpha, \chi)_{0}=[\tau(w_\alpha, \chi, \s_0, \s_0)]=\frac{\chi_\alpha - q^{-1}}{1-\chi_\alpha} = -\frac{1+ q^{-1}}{2},$$
where the second equality follows from $\chi_\alpha = -1$ and Proposition \ref{P:tau-p}.
\item[--] for each $1\lest i \lest (n-1)/2$, we have the two by two matrix $\mca{S}_\mfr{R}(w_\alpha, \chi)_{i\alpha^\vee}=[\tau(w_\alpha, \chi, \s_z, \s_y)]$ with $z, y\in \set{\pm i\alpha^\vee}$, which by Proposition \ref{P:tau-p} is equal to 
\begin{equation}
\begin{aligned}
\mca{S}_\mfr{R}(w_\alpha, \chi)_{i\alpha^\vee}  = & 
\left(\begin{matrix}
\chi_\alpha(1-q^{-1})/(1-\chi_\alpha) & {\bf g}_{^t\psi^{-1}}(2i)) \\
{\bf g}_{^t\psi^{-1}}(-2i) & (1-q^{-1})/(1-\chi_\alpha)
\end{matrix}\right)  \\
= & 
\left(\begin{matrix}
-(1-q^{-1})/2 & {\bf g}_{^t\psi^{-1}}(2i)) \\
{\bf g}_{^t\psi^{-1}}(-2i) & (1-q^{-1})/2
\end{matrix}\right) .
\end{aligned}
\end{equation}
Here the left-upper and right-lower entry is $\tau(w_\alpha, \chi, \s_{i\alpha^\vee}, \s_{i\alpha^\vee})$ and  $\tau(w_\alpha, \chi, \s_{-i\alpha^\vee}, \s_{-i\alpha^\vee})$ respectively.
\end{enumerate}
Note that in this case
$$\gamma(w_\alpha, \chi)= \frac{1- \chi_\alpha}{1-q^{-1} \chi_\alpha} = \frac{2}{1+ q^{-1}}.$$
This coupled with the above equalities for the $\mca{S}_\mfr{R}(w_\alpha, \chi)_{i\alpha^\vee}$'s  gives that
$$\sigma^{\Wh_{^t\psi}} = \frac{n+1}{2} \varepsilon_W \bigoplus \frac{n-1}{2} \mbm{1},$$
and it thus follows from \eqref{E:dWh} that 
$$\dim \Wh_{{}^e\psi}(\pi(\phi_\chi, \mbm{1})) = \frac{n-1}{2} \text{ and } \dim \Wh_{{}^e\psi}(\pi(\phi_\chi, \varepsilon_W)) = \frac{n+1}{2},$$
as desired. This completes the proof.
\end{proof}

\begin{rmk} \label{R:Sp2r}
Results in this subsection generalize to covers of $\Sp_{2r}$ without too much difficulty. Indeed, for even-fold cover every unitary $(K, s_K)$-unramified $I(\chi)$ is irreducible. If $n$ is odd, then $I(\chi)$ could be reducible with $R_\chi=\set{1, w_{\alpha_r}}$, where $\alpha_r$ is the unique long simple root. The same proof as Proposition \ref{P:SL2-K} applies to $\wt{\Sp}_{2r}$ to give the same result. Write $I(\chi) = \pi(\phi_\chi, \mbm{1}) \oplus \pi(\phi_\chi, \varepsilon_W)$, assuming $\pi(\phi_\chi, \mbm{1})$ is $(K, s_K)$-unramified. Regarding the Whittaker dimension, one
has 
$$\dim \Wh_\psi(\pi(\phi_\chi, \mbm{1})) = \dim \Wh_{^e\psi}(\pi(\phi_\chi, \varepsilon_W))  = \frac{n^{r-1}(n+1)}{2}$$
and 
$$\dim \Wh_\psi(\pi(\phi_\chi, \varepsilon_W)) = \dim \Wh_{^e\psi}(\pi(\phi_\chi, \mbm{1}))  = \frac{n^{r-1}(n-1)}{2}.$$
For these equalities, one can either apply  Rodier's heredity coupled with Proposition \ref{P:SL2-W}, or argue as in the proof of Proposition \ref{P:SL2-W} directly.
\end{rmk}

\subsubsection{Covers of $\SO_3$}  \label{SSS:SO3}
Let $Y=\Z\cdot e$ be the cocharacter lattice of $\SO_3$ with $\alpha^\vee=2e$  generating the co-root lattice $Y^{sc}$. Let 
$$Q: Y\to \Z$$ be the Weyl-invariant quadratic form such that $Q(e)=1$. Thus, $Q(\alpha^\vee)=4$. We get
$$\alpha_{Q,n}^\vee= \frac{n}{\text{gcd}(4, n)}\alpha^\vee \text{ and thus } Y_{Q,n}^{sc} = n_{(4)} \alpha^\vee.$$
On the other hand, 
$$Y_{Q,n} = \Z \frac{n}{\text{gcd}(2, n)} e =\Z[n_{(2)} e].$$
Let $\chi$ be a unitary unramified genuine character of $Z(\wt{T})$. The diagram \eqref{CD:K1} becomes
$$ \begin{tikzcd}
\Z[e]/\Z[n_{(2)} e] \ar[d, "\simeq"] & & \\
\Z[e]/\Z[n_{(2)} e] \ar[r, hook, "\simeq"] \ar[d, two heads, "\varphi"]  & \Z[e]/\Z[n_{(2)} e] \ar[d, two heads, "\varphi"] \ar[r, two heads] & \set{0} \ar[d, two heads, "\varphi"] \ar[r, equal] &   \widehat{\Gamma_G^{\rm tor}} \\
\Z[n_{(4)}e]/\Z[n_{(2)} e] \ar[r, hook, "\simeq"]& \Z[n_{(4)}e]/\Z[n_{(2)} e] \ar[d, equal] \ar[r, two heads] & \set{0} \\
& \widehat{\Gamma_{G_{Q,n}}^{\rm tor}} \ar[r, two heads] & \Irr(\mca{S}(\phi_\chi)),
\end{tikzcd}$$
where one has $\varphi(e)=n_{(4)} \cdot e$. It is shown in \cite[\S 8.1]{Ga7} that $\mca{S}(\phi_\chi) = W$ if and only if $4|n$ and $\uchi_\alpha$ is a non-trivial quadratic character; otherwise $\mca{S}(\phi_\chi) = \set{1}$.

Note that $K=\PGL_2(O)=\SO_3(O) \in \mca{K}$ is the unique conjugacy class of hyperspecial maximal compact subgroup. 
If we fix a splitting $s_K: K \into \wt{G}$, then the other splitting $s_K \otimes f_\xi$ is given by the twist
$$f_\xi: \PGL_2(O) \longrightarrow \mu_n$$
where
$$f_\xi=\xi \circ \det: \GL_2(O) \longrightarrow \mu_n$$
for a quadratic character 
$$\xi: O^\times/O^{\times 2} \longrightarrow \mu_n.$$
In particular,
\begin{enumerate}
\item[$\bullet$] if $n$ is odd, then $\Hom(K, \mu_n)=\set{1}$ as $\xi$ is always trivial, and thus there is only one splitting $s_K$ of $K$;
\item[$\bullet$] if $n=2k$ with $k$ odd, then $\Hom(K, \mu_n) =\set{\mbm{1}, f_\xi}$, and for the unique nontrivial $\xi$ two splittings $ s_K\otimes f_\xi$ and $s_K$ are not associated;
\item[$\bullet$] if $4|n$, then we have $\Hom(K, \mu_n) =\set{\mbm{1}, f_\xi}$ for the unique nontrivial $\xi$, and in the case the two splittings $s_K \otimes f_\xi$ and $s_K$ are associated, i.e., $f_\xi: \mbf{T}(O) \to \mu_n$ is trivial on $\mbf{T}_{Q,n}(O)$.
\end{enumerate}
From the perspective of \S \ref{SSS:spl-K}, we see that
$$(Y/Y_{Q,n})^\natural \simeq \Z n_{(4)} e/\Z n_{(2)}e.$$
If we consider an $(K, s_K)$-unramified $I(\chi)$, then there are two cases as follows.
\begin{enumerate}
\item[$\bullet$] First, if $4\nmid n$, then every unitary $(K, s_K)$-unramified $I(\chi)$ is irreducible and one has 
$$\mca{L}(\phi_\chi) = \set{I(\chi)}.$$
then $\Hom(K, \mu_n)^\natural=\set{1}$.
In particular, if $n=2k$ with $k$ odd, then $I(\chi)$ is not $(K, s_K \otimes f_\xi)$-unramified, where $\Hom(K, \mu_n) =\set{\mbm{1}, f_\xi}$. In this case, Conjecture \ref{C:varK} is vacuously true since $\Hom(K, \mu_n)^\natural = \set{1}$.
\item[$\bullet$] Second, if $4|n$, then $\Hom(K, \mu_n)^\natural = \Hom(K, \mu_n) = \set{\mbm{1}, f_\xi}$, and the two splittings $s_K$ and $s_K \otimes f_\xi$ are associated. 
Thus, if $I(\chi)$ is irreducible, then it is both $(K, s_K)$ and $(K, s_K \otimes f_\xi)$-unramified; in this case, Conjecture \ref{C:varK} holds as well.
\end{enumerate}
In fact, for $4|n$ the nontrivial element in $\Hom(K, \mu_n)^\natural$ is
$$f_\xi = f_{ne/4}: \mbf{T}(O) \longrightarrow \mu_n,$$
explicitly given by 
$$f_{ne/4}(y\otimes u) =[\varpi^{ne/4}, y\otimes u] = (\varpi, u)_n^{nB_Q(e, y)/4}.$$

\begin{prop}  \label{P:SO3}
Keep notations as above. For a unitary $(K, s_K)$-unramified principal series $I(\chi)$, if 
$$\mca{L}(\phi_\chi) = {\rm JH}(I(\chi)) = \set{\pi(\phi_\chi, \mbm{1}), \pi(\phi_\chi, \varepsilon)},$$
which occurs exactly for $4|n$ and that $\chi_\alpha$ is the non-trivial quadratic character, then $\pi(\phi_\chi, \rho)$ is $(K, s_K)$-unramified for exactly one $\rho \in \Irr(\mca{S}(\phi_\chi))$, and the other is $(K, s_K\otimes f_\xi)$-unramified. In particular, Conjecture \ref{C:varK} holds in this case.
\end{prop} 
\begin{proof}
For simplicity, in the proof we write $s_K':=s_K \otimes f_\xi$. We first remark that the setting is different from Proposition \ref{P:SL2-K}, since here we only have one conjugacy class $K$ of hyperspecial maximal compact subgroup, but equipped with two splittings. However, the idea of proof in Proposition \ref{P:SL2-K} applies. That is, we will consider the eigenvalue of the normalized intertwining operator $\mca{A}(w_\alpha, \chi)$ applied to the normalized unramified vectors $f_{s_K}$ and $f_{s_K'}$.

We take the model $\Ind_{\wt{A}}^{\wt{T}} (\tchi)$ for $i(\chi)$, where 
$$\tchi: \wt{A} \longrightarrow \C^\times$$ is trivial on $s_K(\mbf{T}(O))$. In this case $\Ind_{\wt{A}U}^{\wt{G}}(\tchi)$ is $(K, s_K)$-unramified and the normalized $(K, s_K)$-unramified vector $f_{s_K} \in \Ind_{\wt{A}U}^{\wt{G}}(\tchi)$ is given by
\begin{equation} \label{F:sK}
f_{s_K}(b\cdot s_K(k)) =
\begin{cases}
\delta_B^{1/2}(t) \cdot \tchi(a) & \text{ if } b=t u \text{ with } t\in \wt{A}, u\in U \text{ and } k\in K,\\
0 & \text{ otherwise.}
\end{cases} 
\end{equation}
One has 
$$\mca{A}(w_\alpha, \chi)(f_{s_K}) = f_{s_K}.$$
To compute $\mca{A}(w_\alpha, \chi)(f_{s_K'})$, we first note that in the model  $\Ind_{\wt{A}U}^{\wt{G}}(\tchi)$ above, the normalized $(K, s_K')$-unramified vector $f_{s_K'}$ is not given by the formula \eqref{F:sK} above.

Recall that the character 
$$s_K'/s_K: \mbf{T}(O)/\mbf{T}_{Q,n}(O) \longrightarrow \mu_n$$
is given by
$$(s_K'/s_K)(k) = \xi \circ \det(k) = [e(\varpi^{n/4}), \det(k)]$$
for every $k\in \mbf{T}(O)$. Fix a lifting 
$$t:=\wt{e(\varpi^{n_\alpha})} \in \wt{T}$$ of $e(\varpi^{n_\alpha})\in T$, where $n_\alpha=n/4$. We see that the model $\Ind_{\wt{A}U}^{\wt{G}} (^t \tchi)$ contains the normalized $(K, s_K')$-unramified vector $f_{s_K'}$ given by
$$
f_{s_K'}(b\cdot s_K'(k)) =
\begin{cases}
\delta_B^{1/2}(a) \cdot (^t\tchi)(a) & \text{ if } b=a u \text{ with } a\in \wt{A}, u\in U \text{ and } k\in K,\\
0 & \text{ otherwise.}
\end{cases} 
$$
In view of the $\wt{G}$-isomorphism (see \eqref{Iso-2m})
$$\Phi_t: \Ind_{\wt{A}U}^{\wt{G}} (\tchi) \longrightarrow \Ind_{\wt{A}U}^{\wt{G}} ({}^t\tchi)$$
we see that $\Phi_t^{-1}(f_{s_K'})$ is the $(K, s_K')$-unramified vector in the model $\Ind_{\wt{A}U}^{\wt{G}} (\tchi)$. We want to compute the constant $c\in \C$ such that 
$$T(w_\alpha, \chi; r_{w_\alpha}^{\rm un})\big( \Phi_t^{-1}(f_{s_K'}) \big) = c \cdot \Phi_t^{-1}(f_{s_K'}).$$
By noting that $\Phi_t^{-1}(f_{s_K'})(t^{-1}) = f_{s_K'}(1)=1$, one obtains
$$\begin{aligned}[t]
c=  T(w_\alpha, \chi; r_{w_\alpha}^{\rm un})(\Phi_t^{-1}(f_{s_K'}))(t^{-1}) = & \int_U (\Phi_t^{-1}(f_{s_K'}))(\wt{w}_\alpha^{-1} u t^{-1}) du \\
 = & \int_U f_{s_K'}(t \cdot (w_\alpha^{-1} t^{-1} w_\alpha)\cdot \wt{w}_\alpha^{-1} t u t^{-1}) du \\
 = & \int_U f_{s_K'}(\wt{h}_\alpha(\varpi^{n_\alpha}) \cdot \wt{w}_\alpha^{-1} t u t^{-1}) du \\
 = & \chi_\alpha \cdot \int_U f_{s_K'}(\wt{w}_\alpha^{-1} u) du.
 \end{aligned}
$$
Now similar (and in fact simpler) argument as in Proposition \ref{P:SL2-K} gives that the last integral is equal to $c_{\sf gk}(w_\alpha, \chi)$ and thus
$$c= \chi_\alpha \cdot c_{\sf gk}(w_\alpha, \chi) = (-1) \gamma(w_\alpha, \chi)^{-1} = (-1) \frac{1+q^{-1}}{2}.$$
This shows that $\mca{A}(w_\alpha, \chi)(f_{s_K'}) = - f_{s_K'}$ as desired and the proof is completed.
\end{proof}

Consider the Whittaker dimension, note that:
\begin{enumerate}
\item[--] if $4\nmid n$, then the hypothesis of Conjecture \ref{C:varW} is satisified, or equivalently the map $\varphi$ is trivial. However, in this case, $I(\chi)$ is always irreducible, and thus Conjecture \ref{C:varW} holds trivially.
\item[--] if $4|n$, then the cover $\wt{\SO}_3$ does not fit into the hypothesis of Conjecture \ref{C:varW}. On the other hand, since $T_{ad}/T =\set{1}$, we see that a priori one expects
\begin{equation} \label{E:SO-4W}
\dim \Wh_\psi(\pi(\phi_\chi, \rho)) = \dim\Wh_{{}^t\psi}(\pi(\phi_\chi, \rho))
\end{equation}
for every $\rho\in \Irr(\mca{S}(\chi_\chi))$ and $t\in T$. 
\end{enumerate}
For $4|n$ we verify below the equality \eqref{E:SO-4W} for reducible $I(\chi)$, for which $\mca{S}(\phi_\chi) = W$. We assume $\mfr{\psi} = O_F$ and $t= e(\varpi) \in T$.

First, writing $n=2k$ with $k$ even, we have the ordered set
$$\mfr{R}= \set{ie: -k/2 + 1 \lest i \lest k/2}$$
of representatives of $Y/Y_{Q,n}$. To apply Proposition \ref{P:tau-p},  we consider the action $w(-)$ on $Y/Y_{Q,n}$. There are two trivial orbits $\set{0}$ and $\set{ke/2}$ in $Y/Y_{Q,n}$. The remaining elements in $Y/Y_{Q,n}$ are formed by free $W$-orbits with respect to the action $w(-)$.  Accordingly, up to a permutation of elements in $\mfr{R}$, we see that the scattering matrix 
$$[\tau(w_\alpha, \chi, \s_z, \s_y)]_{z, y\in \mfr{R}}$$
 is a block-diagonal matrix with each block $\mca{S}_\mfr{R}(w_\alpha, \chi)_{ie}$ described as follows.
\begin{enumerate}
\item[--] There is a size-one block 
$$\mca{S}_\mfr{R}(w_\alpha, \chi)_{0} = \tau(w_\alpha, \chi, \s_0, \s_0)= \frac{\chi_\alpha - q^{-1}}{1-\chi_\alpha} = - \frac{1+ q^{-1}}{2}$$
associated with the trivial $W$-orbit $\set{0}$; similarly, associated to $\set{ke/2}$ one has another size-one block
$$\mca{S}_\mfr{R}(w_\alpha, \chi)_{ke/2} = \tau(w_\alpha, \chi, \s_{ke/2}, \s_{ke/2})= \frac{(1-q^{-1}) \chi_\alpha^2}{1-\chi_\alpha} - q^{-1}\chi_\alpha =  \frac{1+ q^{-1}}{2}.$$
\item[--] For every $1\lest i \lest k/2-1$, one has a size-two block
\begin{equation}
\mca{S}_\mfr{R}(w_\alpha, \chi)_{ie} =
\left(\begin{matrix}
-(1-q^{-1})/2 & {\bf g}_{^t\psi^{-1}}(i)) \\
{\bf g}_{^t\psi^{-1}}(-i) & (1-q^{-1})/2
\end{matrix}\right),
\end{equation}
where the left-upper entry and right-lower entry are equal to $\tau(w_\alpha, \chi, \s_{ie}, \s_{ie})$  and $\tau(w_\alpha, \chi, \s_{-ie}, \s_{-ie})$ respectively.
\end{enumerate}
From the above, we immediately get (writing ${}^e\psi$ for ${}^{e(\varpi)} \psi$)
$$\sigma^{\Wh_{^e\psi}} = (k/2) \cdot \mbm{1} \oplus (k/2) \cdot \varepsilon_W = \sigma^{\Wh_\psi},$$
where the second equality follows from \cite[Proposition 8.3]{Ga7}. Hence, 
$$\dim \Wh_\psi(\pi(\phi_\chi, \rho)) = \dim \Wh_{{}^e\psi}(\pi(\phi_\chi, \rho)) = n/4$$
for any $\rho \in \set{\mbm{1}, \varepsilon_W}= \Irr(\mca{S}(\phi_\chi)$, as expected.

\subsubsection{Covers of $\GL_r$}
It seems to be a folkloric result that a Zelevinsky-type classification in terms of segements of reducibility points of parabolic induction from supercuspidal representations holds for covers of $\GL_r$. However, the classical proof of Zelevinsky relies crucially on the fact that the Whittaker model of the linear $\GL_2$ is unique, and thus the approach is not directly adaptable in the covering settings. If one restricts to the representations with unique Whittaker model, then one can obtain similar results by using the same method as in \cite{BZ1, BZ2, Zel}.

On the other hand, it is expected that every unitary unramified genuine principal series of $\wt{\GL}_r$ is irreducible. Recall that every Brylinski--Deligne cover is associated with $B_Q$ such that
$$B(e_i, e_i) = 2\bfp \text{ and }  B(e_i, e_j) = \bfq \text{ for } i \ne j.$$
One has $Q(\alpha_i) = 2\bfp - \bfq$.

\begin{prop} \label{P:GL-triv}
Let $\wt{\GL}_r$ be an arbitrary Brylinski--Deligne cover of $\GL_r$ such that $n_\alpha\cdot Y \subset Y_{Q,n}$. Then one has 
$$\Gamma_{G_{Q,n}}^{\rm tor} = \set{1}.$$
Hence, every unitary unramified genuine principal series of such cover $\wt{\GL}_r$ is irreducible.
\end{prop}
\begin{proof}
In view of the embedding $\mca{S}(\phi_\chi) \into \Gamma_{G,{Q,n}}^{\rm tor}$ for every unitary unramified $\chi$, it suffices to prove the triviality of the latter group. 
Recall that
$$\Gamma_{G_{Q,n}}^{\rm tor}  = (X_{Q,n} \cap X_{G,\Q}^{sc})/X_{Q,n}^{sc}.$$
Assume 
$$x\in \sum_{i=1}^{r-1} c_i \cdot \alpha_i \in X_{Q,n} \cap X_{G,\Q}^{sc}$$
with $c_i \in \Q$ and $\alpha_i = e_i^* - e_{i+1}^*$ for each $i$. Since $n_\alpha Y \subset Y_{Q,n}$ by assumption, we have $n_\alpha e_k \in Y_{Q,n}$ for every $1\lest k \lest r-1$. Then we have $\angb{x}{n_\alpha e_k} \in \Z$, that is,
$$n_\alpha c_1\in \Z \text{ and } n_\alpha(c_{i+1} - c_i) \in \Z$$
for every $1\lest i \lest r-1$. This gives that $c_i \in \Z/n_\alpha$ and thus $x\in X_{Q,n}^{sc}$.
\end{proof}

It is clear that Proposition \ref{P:GL-triv} applies to the cases when $n_\alpha = n$ or $\bfq=0$, in particular to the Kazhdan--Patterson covers and Savin covers discussed in \S \ref{SSS:GL}. 

\begin{rmk}
It is  expected that the last assertion in Proposition \ref{P:GL-triv} holds for arbitrary Brylinski--Deligne cover of $\GL_r$, without the assumption $n_\alpha Y \subset Y_{Q,n}$, i.e., we expect the equality $\mca{S}(\phi_\chi)=\set{1}$ for every unitary unramified $\chi$. However, the group $\Gamma_{G_{Q,n}}^{\rm tor}$ may not be trivial in general. For example, consider the Brylinski--Deligne cover $\wt{\GL}_2^{(4)}$ associated with
$$\bfp=1, \ \bfq=2 \text{ and } n=4.$$
Then it is easy to check that $n_\alpha =1$ and $\Gamma_{G_{Q,n}}^{\rm tor} = (\Z\alpha/2)/\Z\alpha \simeq \Z/2\Z$ in this case.
\end{rmk}
%

\section{Analysis for covers of $\GSp_{2r}$} \label{S:GSp}
In this section, we analyze the unitary unramified principal series $I(\chi)$ of $\wt{G}=\GSp_{2r}$ of type I (see \S \ref{SSS:GSp}) for the following three aspects:
\begin{enumerate}
\item[(i)] the determination of the irreducible components $\pi(\phi_\chi, \rho), \rho \in \Irr(\mca{S}(\phi_\chi))$ of an $(K, s_K)$-unramified $I(\chi)$, determine the pair $(K', s_{K'})$ with respect to which $\pi(\phi_\chi, \rho)$ is unramified,
\item[(ii)] determining the $\psi$-Whittaker dimension of each constituent $\pi(\phi_\chi, \rho)$,
\item[(iii)] investigation of the restriction of each $\pi(\phi_\chi, \rho)$ to $\wt{\Sp}_{2r}$ and the pair $(K_0, s_{K_0})$ with respect to which a constituent is unramified, and also how the Whittaker dimension of each constituent varies with respect to different $\psi$-Whittaker datum.
\end{enumerate}
Recall that the $n$-fold cover $\wt{\GSp}_{2r}$ of type I is associated with the quadratic form $Q: Y \to \Z$ such that
$$Q(\alpha_r^\vee)=-1 \text{ and } Q(e_0) =0,$$
where we follow the notations in \S \ref{SSS:GSp} closely.

We fix the hyperspecial maximal compact subgroup $K=\GSp_{2r}(O)$ and a splitting $s_K: K \into \wt{G}$. Then for every 
$$f: K \longrightarrow \mu_n,$$
one obtains a splitting $f\otimes s_K$. 
The homomorphism $f$ factors through the similitude map
$${\rm sim}: \GSp_{2r} \longrightarrow F^\times$$
and thus corresponds to a homomorphism
$$\xi: O^\times/O^{\times n} \longrightarrow \mu_n$$
such that $f = \xi \circ {\rm sim}$. 
We are interested in the subgroup $\Hom(K, \mu_n)^\natural \subset \Hom(K, \mu_n)$, that is, those $f_z$ with $z\in (Y/Y_{Q,n})_K^\natural$. For every $z\in (Y/Y_{Q,n})_K^\natural$, as a function on $\mbf{T}(O) \subset K$, one has
$$f_z(y\otimes u) = [y\otimes u, z(\varpi)],$$
and in particular
$$f_z(e_0(u)) = [e_0(u), z(\varpi)]=(u, \varpi)_n^{B_Q(e_0, z)}$$
and
\begin{equation} \label{E:Y-shp}
1= f_z(\alpha_i^\vee(u))=[\alpha_i^\vee(u), t]=(u, \varpi)_n^{B_Q(\alpha_i^\vee, z)}
\end{equation}
for every $u\in O^\times$ and $1\lest i \lest r$. The equalities in \eqref{E:Y-shp} are equivalent to
$$B_Q(\alpha_i^\vee, z) \in n\Z.$$
for all $1\lest i\lest r$, or equivalently,
$$B_Q(e_i, z) \in n\Z$$
for every $1\lest i \lest r$. We have in this case $(Y/Y_{Q,n})_K^\natural = Y_K^\natural/Y_{Q,n}$, where
$$Y_K^\natural:=\set{z\in Y: B_Q(z, e_i) \in n\Z \text{ for all } i}$$
and $Y_{Q,n}$ is given in \eqref{GSp-L}. Setting $z=\sum_{i=1}^r y_i e_i$, it is easy to see that
\begin{equation} \label{YK-n}
Y_K^\natural =
\left\{
\begin{array}{cc}
\sum_{i=0}^r y_i e_i \in Y: \\
\bullet \quad  n|(-2y_i + y_0) \text{ for every } i
\end{array}
\right\}.
\end{equation}

Let $I(\chi)$ be a $(K,s_K)$-unramified principal series. It is also $(K, f_z\otimes s_K)$-unramified for every $z\in (Y/Y_{Q,n})_K^\natural$. On the other hand, for any $y\in Y$, consider
$$K_y:= y(\varpi) \cdot K \cdot y(\varpi)^{-1}.$$
The representation $I(\chi)$ is unramified with respect to 
 $$(K_y, y \cdot s_K)$$
 as well. Indeed, since $[y(\varpi), \mbf{T}_{Q,n}(O)]=1$, we see that the splittings $y \cdot s_K: K_y \into \wt{G}$ and $s_K: K \into \wt{G}$ agree on $\mbf{T}_{Q,n}(O)$. We also have
 $$(Y/Y_{Q,n})_K^\natural \simeq (Y/Y_{Q,n})_{K_y}^\natural$$
 for every $y\in Y$. Combining the above, we see $I(\chi)$ is $(K_y, f_z \otimes y \cdot s_K)$-unramified for every $z\in (Y/Y_{Q,n})_K^\natural$ and $y\in Y$.
 
We set 
$$K_0 = \Sp_{2r} \cap K= \Sp_{2r}(O),$$ 
and for any $y\in Y$ denote
$$K_{y,0} = K_z \cap \Sp_{2r} = y(\varpi) \cdot  K_0 \cdot y(\varpi)^{-1}.$$
In particular, we write
$$K' = K_{e_0} \text{ and thus } K_0' = K_{e_0, 0}.$$
It follows that $$\mca{K}_0=\set{K_0, K_0'}$$
is the set of conjugacy classes of hyperspecial maximal compact subgroup of $\Sp_{2r}$.
 
For the questions in (i)--(iii), we will give a seperate discussion according to the parity of $n$. 
\subsection{Odd fold cover of $\GSp_{2r}$} 
If $n$ is odd, then it is given in \S \ref{SSS:GSp-odd} that $Y_{Q,n}$ has a $\Z$-basis
$$\set{ne_i: 1\lest i \lest r} \cup \set{n_{(r)} \cdot e_c}$$
and also that
\begin{equation} \label{E:GSp-d1}
\wt{\GSp}_{2r}^\vee = \set{(g, a) \in \GSpin_{2r+1} \times \GL_1: \lambda(g) = a^{\gcd(n, r)}}.
\end{equation}
One has
$$\val{\msc{X}_{Q,n}} = 2n^r \cdot n_{(r)}.$$
Since $n$ is odd, we have $Y_{0, Q,n}= Y_{Q,n}^{sc}$, which is then equal to $Y_0\cap Y_{Q,n}$. In any case, we have for odd $n$ the identity 
$$\msc{X}_{Q,n}^\Gamma = \msc{X}_{Q,n}^\mfr{c}.$$
Moreover, it is easy to see that
$$\set{e_i: 1\lest i \lest r} \cup \set{n_{(r)} e_c}$$
constitutes a basis for $Y_0 + Y_{Q,n}$, and therefore $\val{\msc{X}_{Q,n}^\Gamma} = 2\cdot n_{(r)}$. One has in this case
$$(Y/Y_{Q,n})_K^\natural =\set{i \cdot e_c: \ 0\lest i \lest n_{(r)} -1}.$$
The diagram \eqref{CD:big}  becomes
\begin{equation} \label{CD:GSp-odd}
\begin{tikzcd}
&  & (Y/Y_{Q,n})_K^\natural \ar[r, "{p_\Gamma}"] \ar[d, hook] & \msc{X}_{Q,n}^\mfr{c} \ar[d, equal] \\
0 \ar[r] & \msc{X}_{0,Q,n} \ar[r, hook] & \msc{X}_{Q,n} \ar[r, two heads, "{p_\Gamma}"] & \msc{X}_{Q,n}^\Gamma \ar[r] & 0.
\end{tikzcd}
\end{equation}
One can check that the image of $(Y/Y_{Q,n})_K^\natural$ in $\msc{X}_{Q,n}^\mfr{c}$ via $p_\Gamma$ is of index 2.

Every unitary $(K, s_K)$-unramified genuine principal series $I(\chi)$ is irreducible, in view of \eqref{E:GSp-d1}.  As remarked, for every $y\in Y$ and  $z \in (Y/Y_{Q,n})_K^\natural$,  the representation $I(\chi)$ is also  unramified with respect to
 $$(K_y, f_z \otimes y(\varpi) \cdot s_K).$$

The pair $(\wt{G}, \wt{G}_0)$ is an isotypic pair, and one has the natural map $f_{G, G_0}$ between the $L$-groups arising from the composites as in the following diagram
$$\begin{tikzcd}
\wt{G}^\vee \ar[r, hook] \ar[d, "{f_{G, G_0}}"] &  \GSpin_{2r+1} \times \GL_1 \ar[d, two heads] \\
\wt{G}_0^\vee \ar[r, equal] & \SO_{2r+1}.
\end{tikzcd}$$
By the choice of a distinguished genuine character $\chi_\psi: Z(\wt{T}) \longrightarrow \C^\times$, which by restriction gives a distinguished genuine character $\chi_\psi^0: Z(\wt{T}) \to \C^\times$, one can extend $f_{G, G_0}$ to be a homomorphism of $L$-groups
$$f_{G, G_0}: {}^L\wt{G} \longrightarrow {}^L\wt{G}_0.$$
We have
$$I(\chi)|_{\wt{G}_0} = \val{\msc{X}_{Q,n}^\Gamma} \cdot I(\omega),$$
where $\omega=\chi|_{Z(\wt{T}_0)}$.  For any non-degenerate $\psi$, one has
$$\dim \Wh_\psi(I(\chi)) = 2\cdot n^r \cdot n_{(r)} \text{ and } \dim \Wh_\psi(I(\omega)) = n^r.$$

Note that $I(\omega)$ might be reducible. Indeed, if $\chi_{\alpha_r} = -1$, then 
$$I(\omega) = \pi(\phi_\omega, \mbm{1}) \oplus \pi(\phi_\omega, \varepsilon)$$
where $\varepsilon$ is the nontrivial character of $\mca{S}(\phi_{\omega}) = \set{1, w_{\alpha_r}}$. 
Let $\mca{K}_0 = \set{K_0, K_0'}$ be the above set of conjugacy classes of hyperspecial maximal compact subgroups of $\Sp_{2r}$, as above. Let 
$$s_0: K_0 \into \wt{\Sp}_{2r} \text{ and } s_0': K_0' \into \wt{\Sp}_{2r}$$
be the two unique splittings of $K_0$ and $K_0'$ respectively.

\begin{thm} \label{T:GSp-odd}
For $\wt{\GSp}_{2r}^{(n)}$ with odd $n$ and a unitary $(K, s_K)$-unramified $I(\chi)$, one has
$$I(\chi)|_{\wt{G}_0} = \val{\msc{X}_{Q,n}^\Gamma} \cdot I(\omega)$$
where $\omega = \chi|_{Z(\wt{T}_0)}$. If $I(\omega) = \pi(\phi_\omega, \mbm{1}) \oplus \pi(\phi_\omega, \varepsilon)$ is reducible, and we assume that $\pi(\phi_\omega, \mbm{1})$ is $(K_0, s_0)$-unramified, then $\pi(\phi_\omega, \varepsilon)$ is $(K_0', s_0')$-unramified. Regarding the Whittaker dimension, for an additive character $\psi$ with $\mfr{f}(\psi) = O_F$, one has
$$\dim \Wh_\psi(\pi(\phi_\omega, \mbm{1})) =\frac{n^r + n^{r-1}}{2},\  \dim \Wh_\psi(\pi(\phi_\omega,\varepsilon)) =\frac{n^r - n^{r-1}}{2}$$
and also
$$\dim \Wh_{{}^t\psi}(\pi(\phi_\omega,\mbm{1})) =\frac{n^r - n^{r-1}}{2},\  \dim \Wh_{{}^t\psi}(\pi(\phi_\omega, \varepsilon)) =\frac{n^r + n^{r-1}}{2},$$
where $t:=e_0(\varpi)$.
\end{thm}
\begin{proof}
The result here is essentially the content of Remark \ref{R:Sp2r}. Indeed, the assertion that $\pi(\phi_\omega,\varepsilon)$ is $(K_0', s_0')$-unramified follows from the same argument as in Proposition \ref{P:SL2-K}. The equalities regarding the $\psi$-Whittaker dimensions follow from \cite{Ga7}, while those for the ${}^{t}\psi$-Whittaker dimension follows from the same argument in Proposition \ref{P:SL2-W}, by noting $\mfr{f}({}^{e_0}\psi) =\mfr{p}_F$.

Alternatively, to prove the equalities on Whittaker dimensions, as noted in Remark \ref{R:Sp2r}, one may reduce (essentially) to the $\wt{\SL}_2^{(n)}$ by parabolic induction. Indeed, let $\wt{M} \subset \wt{P} \subset \wt{\Sp}_{2r}$ be the maximal Levi subgroup associated with $\Delta -\set{\alpha_{r-1}}$. Then one has 
$$\wt{M} \simeq \wt{\GL}_{r-1} \times_{\mu_n} \wt{\Sp}_2$$
and also
$$\omega = \chi_1 \boxtimes \mu.$$
If $\chi_{\alpha_r} =\omega_{\alpha_r} = -1$, equivalently $\mu_{\alpha_r}=-1$, then one has a decomposition of the genuine principal series
$$I(\mu) = \pi_0(\phi_\mu, \mbm{1}) \oplus \pi_0(\phi_\mu, \varepsilon)$$
of $\Sp_{2}$, such that
$$\pi(\phi_\omega, \rho) = \Ind_{\wt{P}}^{\wt{\Sp}_{2r}} I(\chi_1) \boxtimes \pi_0(\phi_\mu,\rho)$$
for $\rho \in \set{\mbm{1}, \varepsilon}=\Irr(\mca{S}(\phi_\chi))$.
Since $\dim_{\psi}(I(\chi_1)) = n^r = \dim_{^t \psi}(I(\chi_1))$, we see that the desired equalities follow directly from Rodier's heredity and the results in Proposition \ref{P:SL2-W}. This also concludes the proof.
\end{proof}

Regarding Conjecture \ref{C:varK} and \ref{C:varW}, we see that the three groups
$$\Hom(K, \mu)^\natural, \ \widehat{\Gamma_G^{\rm tor}} \text{ and } \widehat{ \Gamma_{G_{Q,n}}^{\rm tor} }$$
are all trivial. Also, every unitary $(K, s_K)$-unramified $I(\chi)$ is irreducible. Hence, we see that Conjecture \ref{C:varK} and Conjecture \ref{C:varW} hold trivially.

\subsection{Even fold cover of $\GSp_{2r}$}  
In this subsection, we assume $n=2m$ is even. We fix a unitary $(K, s_K)$-unramified $I(\chi)$ of $\wt{\GSp}_{2r}$, it follows from Theorem \ref{T:decT} that
$$I(\chi)|_{\wt{G}_0}= \val{\msc{X}_{Q,n}^\mfr{c}} \cdot \bigoplus_{\gamma \in \msc{X}_{Q,n}^\Gamma/\msc{X}_{Q,n}^\mfr{c}} \bigoplus_{\omega_{\gamma, j} \in \msc{E}(\chi, {}^\gamma \tchi_O; Z(\wt{T}_0))} I(\omega_{\gamma, j}),$$
where 
$$\val{\msc{E}(\chi, {}^\gamma \tchi_O; Z(\wt{T}_0))} = \val{Y_{0,Q,n}/Y_0 \cap Y_{Q,n}}=2.$$
We also have
$$\val{\msc{X}_{0,Q,n}} = m^r.$$
There are two situations according to the parity of $r$.

\subsubsection{When $r$ is odd}  In this case, as discussed in \S \ref{SSS:GSp-ev}, the set
$$\set{v_0:=n_{(r)} \cdot e_c} \cup \set{me_i + v_0/2: 1\lest i \lest r}$$
constitutes a basis for $Y_{Q,n}$ and
$$\wt{\GSp}_{2r}^\vee = \GSp_{2r}.$$
Thus, every unitary $(K, s_K)$-unramified $I(\chi)$ is irreducible. Moreover,
$$\val{Y/Y_{Q,n}} = 2\cdot m^r \cdot n_{(r)}.$$
It is also easy to see that $Y_0 + Y_{Q,n} \subset Y$ is the $\Z$-span of 
$$\set{e_i: 1\lest i \lest r} \cup \set{v_0/2}$$
and thus
$$\msc{X}_{Q,n}^\Gamma =Y/(Y_0 + Y_{Q,n}) = \set{i e_0: \ 0\lest i \lest n_{(r)}-1},$$
which has size $n_{(r)}$. It is straightforward to compute from its definition that
$$\msc{X}_{Q,n}^\mfr{c} = \set{i e_0: \ 2| i \text{ and } 0\lest i \lest n_{(r)}-1}$$
and thus
$$\msc{X}_{Q,n}^\Gamma/\msc{X}_{Q,n}^\mfr{c} = \set{0, e_0}.$$
We have
\begin{equation} \label{CD:GSp-c2}
\begin{tikzcd}
& 0 \ar[d] \\
& Y_{0,Q,n}/Y_0 \cap Y_{Q,n} \ar[r, hook] \ar[d, hook, "m^r"] & (Y/Y_{Q,n})_K^\natural \ar[r, two heads, "{p_\Gamma}"] \ar[d, hook, "{m^{r-1} \cdot n_{(r)}}"] & \msc{X}_{Q,n}^\mfr{c} \ar[d, hook, "2"] \\
0 \ar[r] & Y_0/Y_0 \cap Y_{Q,n} \ar[d, two heads, "2"] \ar[r, hook, "{n_{(r)}}"] & \msc{X}_{Q,n} \ar[r, two heads, "{2m^r}"] & \msc{X}_{Q,n}^\Gamma \ar[r] & 0  \\
& \msc{X}_{0,Q,n}  \ar[d] \ar[r, equal] & Y_0/Y_{0,Q,n} \\
& 0,
\end{tikzcd}
\end{equation}
where the numbers around an arrow indicate the ratio of the orders of the two groups as the domain and codomain of the pertaining map; for example $2m^r=\val{\msc{X}_{Q,n}}/\val{\msc{X}_{Q,n}^\Gamma}$. Since $\Z e_c/ Y_{Q,n} \subset (Y/Y_{Q,n})_K^\natural$, we see that $p_\Gamma$ is surjective.

From the above, we obtain more explicitly
$$I(\chi)|_{\wt{G}_0}= \frac{n_{(r)}}{2}\cdot (I(\omega_{0, 0}) \oplus I(\omega_{0, 1}))\bigoplus \frac{n_{(r)}}{2}\cdot (I(\omega_{e_0, 0}) \oplus I(\omega_{e_0, 1})).$$
Note that $I(\omega_{0, i}), i =0, 1$ are both $(K_0, s_0)$-unramified. 
We show that $I(\omega_{e_0, i}), i =0 ,1$ are both $(K_0', s_0')$-unramified. Indeed, we have
$$i(\omega_{e_0, i}) = \Ind_{\wt{A}_0}^{\wt{T}_0} (\omega_{e_0, i} \boxtimes {}^{e_0(\varpi)} \tchi_O) \in \Irr(\wt{T}),$$
where
$${}^{e_0(\varpi)} \tchi_O: \mbf{T}_0(O) \longrightarrow \C^\times$$
is the twisted character. Since the character $\tchi_O$ restricted to $s_0(\mbf{T}_0(O))$ is trivial and by definition
$$s_0'(k') = e_0(\varpi) \cdot s_0\big(e_0(\varpi)^{-1} \cdot k' \cdot e_0(\varpi)\big) \cdot e_0(\varpi)^{-1}$$
for every $k'\in K_0'$, it gives that
$${}^{e_0(\varpi)} \tchi_O(s_0'(k)) = \tchi_O \circ s_0(k) = 1$$ 
for every $k\in \mbf{T}_0(O)$. Thus, by the Satake isomorphism, every $I(\omega_{e_0, i})$ for $j =0, 1$ is $(K_0', s_0')$-unramified.

Since the dual group $\wt{\Sp}_{2r}^\vee= \Sp_{2r}$, we have $\mca{S}(\phi_{\omega_{\gamma, j}}) = \set{1}$ and thus every $I(\omega_{\gamma, i})$ is irreducible. We give a summary of the above discussion.

\begin{thm} \label{T:nev-rod}
Assume $n=2m$ and $r$ is odd. Then every $(K, s_K)$-unramified unitary principal series $I(\chi)$ of $\wt{\GSp}_{2r}$ is irreducible,  and one has
$$I(\chi)|_{\wt{G}_0}= \frac{n_{(r)}}{2}\cdot (I(\omega_{0, 0}) \oplus I(\omega_{0, 1}))\bigoplus \frac{n_{(r)}}{2}\cdot (I(\omega_{e_0, 0}) \oplus I(\omega_{e_0, 1}))$$
where $I(\omega_{0, j}), j=0, 1$ is $(K_0, s_0)$-unramified and $I(\omega_{e_0, j}), j=0, 1$ is $(K_0', s_0')$-unramified. Moreover, for every $\gamma \in \msc{X}_{Q,n}^\Gamma/\msc{X}_{Q,n}^\mfr{c}$ and $j\in \set{0, 1}$ one has
$$\dim \Wh_\psi( I(\omega_{\gamma, j})) = m^r$$
for any non-degenerate character $\psi$.
\end{thm}

Similar to the odd-fold cover of $\GSp_{2r}$, we see that if $n$ is even with $r$ odd, the two groups $\widehat{\Gamma_G^{\rm tor}}$ and $\widehat{\Gamma_{G_{Q,n}}^{\rm tor}}$ are trivial. This shows that Conjecture \ref{C:varK} and Conjecture \ref{C:varW} hold trivially as well.

\subsubsection{When $r$ is even}  In this case, we have $Y_{Q,n} = Y_{Q,n}^{sc} \oplus \Z v_0$ with $v_0=n_{(r)}\cdot e_c$ as above. Thus, 
$$\set{m(e_i-e_{i+1}): 1\lest i \lest r-1} \cup \set{ne_r} \cup \set{v_0}$$
constitutes a basis for $Y_{Q,n}$, and also
$$\wt{\GSp}_{2r}^\vee = {\rm PGSp}_{2r} \times \GL_1,$$
see \S \ref{SSS:GSp-ev}. One has
$$\val{\msc{X}_{Q,n}} = 4m^r \cdot n_{(r)}.$$
It is easy to see that
$Y_0 + Y_{Q,n}$ is the $\Z$-span of 
$$\set{e_i: 1\lest i \lest r} \cup \set{v_0=n_{(r)} \cdot e_c}.$$
Thus
$$\msc{X}_{Q,n}^\Gamma =\set{ie_c + c e_0: \ 0\lest i \lest n_{(r)}-1 \text{ and } 0\lest c \lest 1}$$
with $\val{\msc{X}_{Q,n}^\Gamma} = 2\cdot n_{(r)}$. In view of the map
$$\mfr{c}: \msc{X}_{Q,n}^\Gamma \longrightarrow \Hom(\mbf{T}_0(O) \cap Z(\wt{T}_0), \mu_n),$$
we see that
$$\msc{X}_{Q,n}^\mfr{c}:={\rm Ker}(\mfr{c}) = \set{ie_c: 0\lest i \lest n_{(r)}-1} \subset \msc{X}_{Q,n}^\Gamma.$$
We have
\begin{equation} \label{CD:GSp-c3}
\begin{tikzcd}
& 0 \ar[d] \\
& Y_{0,Q,n}/Y_0 \cap Y_{Q,n} \ar[r, hook] \ar[d, hook, "m^r"] & (Y/Y_{Q,n})_K^\natural \ar[r, two heads, "{p_\Gamma}"] \ar[d, hook, "{m^{r-1} \cdot n_{(r)}}"] & \msc{X}_{Q,n}^\mfr{c} \ar[d, hook, "2"] \\
0\ar[r] & Y_0/Y_0 \cap Y_{Q,n} \ar[d, two heads, "2"] \ar[r, hook, "{2n_{(r)}}"] & \msc{X}_{Q,n} \ar[r, two heads, "{2m^r}"] & \msc{X}_{Q,n}^\Gamma \ar[r] & 0  \\
& \msc{X}_{0,Q,n}  \ar[d] \ar[r, equal] & Y_0/Y_{0,Q,n} \\
& 0
\end{tikzcd}
\end{equation}
where the map $p_\Gamma$ is surjective. In particular, 
$$I(\chi)|_{\wt{G}_0}= n_{(r)}\cdot (I(\omega_{0, 0}) \oplus I(\omega_{0, 1}))\bigoplus n_{(r)}\cdot (I(\omega_{e_0, 0}) \oplus I(\omega_{e_0, 1})).$$

Note that $I(\chi)$ itself might be reducible, exactly when $\chi$ satisfies that
$$\chi_{\alpha_i} = -1 \text{ for all } i=2k-1, 1\lest k \lest r/2.$$
In this case $\mca{S}(\chi_\chi)=\set{1, w}$ with $w= w_{\alpha_1} w_{\alpha_3} ... w_{\alpha_{r-1}}$ and we write
$$I(\chi) = \pi(\phi_\chi, \mbm{1}) \oplus \pi(\phi_\chi, \varepsilon),$$
where $\varepsilon \in \Irr(\mca{S}(\phi_\chi))$ is the nontrivial character.

\begin{prop} \label{P:GSp-ee}
Assume $I(\chi)$ is unitary $(K, s_K)$-unramified and reducible, and assume $\pi(\phi_\chi, \mbm{1}) \subset I(\chi)$ is the unique $(K, s_K)$-unramified constituent. Then $\pi(\phi_\chi, \varepsilon)$ is the unique $(K, f_{me_r} \otimes s_{K})$-unramified constituent.
\end{prop}
\begin{proof} 
First note that $me_r \in (Y/Y_{Q,n})_K^\natural$, in view of \eqref{YK-n}. The idea of the proof is the same as Proposition \ref{P:SO3}, though here we are dealing with an intertwining operator associated with 
$$w=w_{\alpha_1} w_{\alpha_3} ... w_{\alpha_{r-1}}.$$
Let $\tchi$ be a genuine character of $\wt{A}$ extending $\chi$ and trivial on $s_K|_{\mbf{T}(O)}$. For the proof, we set 
$$t=me_r(\varpi) \text{ and } s_K'=f_{me_r}\otimes s_K.$$
Then ${}^t \tchi$ also extends $\chi$ and is trivial on $s_K'$ restricted to $\mbf{T}(O)$. 

Let $\wt{t}$ be any lifting of $t$. One has the isomorphism
$$\Phi_{\wt{t}}: {\rm Ind}_{\wt{A}U}^{\wt{G}}(\tchi) \longrightarrow \Ind_{\wt{A}U}^{\wt{G}} ({}^t \tchi)$$
between the two models given by 
$$\Phi_{\wt{t}}(f)(g) = f(\wt{t}^{-1} g).$$ 
Let $f_{s_K'} \in \Ind_{\wt{A}U}^{\wt{G}} ({}^t \tchi)$ be the $(K, s_K')$-unramified vector, which gives $\Phi_{\wt{t}}^{-1}(f_{s_K'}) \in {\rm Ind}_{\wt{A}U}^{\wt{G}}(\tchi)$. We want to compute the number $c\in \C$ such that
$$T(w, \tchi; r_{w_\alpha}^{\rm un})\big( \Phi_{\wt{t}}^{-1}(f_{s_K'}) \big) = c \cdot \Phi_{\wt{t}}^{-1}(f_{s_K'}).$$

We have the following commutative diagram
$$\begin{tikzcd}
I(\tchi) \ar[rr, "{T(w, \tchi)}"] \ar[d, "{\Phi_{\wt{t}}}"] & & I({}^{w} \tchi) \ar[d, "{\Phi_{wtw^{-1}}}"] \ar[rrd, "{\Phi_{\wt{t}}}"] \\
I(^t \tchi) \ar[rr, "{T(w, {}^t\tchi)}"] & & I(^w (^t \tchi)) \ar[rr, "{ \Phi_{twt^{-1} w^{-1}} }"'] & &  I(^t(^w \tchi)),
\end{tikzcd} $$
where
$${}^w \tchi = \tchi.$$
Let $c(w, {}^t\tchi) \in \C$ be the number such that
$$\Phi_{\wt{t}} \circ T(w, \tchi) \circ \Phi_{\wt{t}}^{-1}( f_{s_K'} ) = c(w, {}^t \tchi) \cdot f_{s_K'}.$$
Thus, we get
$$c(w, {}^t \tchi)  = T(w, {}^t\tchi)( f_{s_K'} )(wtw^{-1} t^{-1}).$$
Now we see that
$$wtw^{-1} t^{-1} = w_{r-1} \cdot \wt{me_r(\varpi)}  \cdot w_{r-1}^{-1} \cdot \wt{me_r(\varpi)}^{-1} = \wt{h}_{\alpha_{r-1}}(\varpi^m) \in Z(\wt{T}).$$
It then follows that
$$c(w, {}^t \tchi) = {}^w \chi( \wt{h}_{\alpha_{r-1}}(\varpi^m)  ) \cdot c_{\sf gk}(w, \chi) = (-1) \cdot c_{\sf gk}(w, \chi).$$
This shows that the eigenvalue of 
$$\mca{A}(w, \chi): I(\tchi) \longrightarrow I(\tchi)$$
 is equal to $-1$ on the $(K, s_K')$-normalized unramified vector $\Phi_t^{-1}(f_{s_k'}) \in I(\tchi)$. This concludes the proof.
\end{proof}

Regarding Conjecture \ref{C:varK}, we have $\widehat{\Gamma_G^{\rm tor}}=\set{1}$ and $\widehat{\Gamma_{G_{Q,n}}^{\rm tor}}=\Z/2\Z$, where the nontrivial element is equal to $\varphi \circ h(me_r)$ with the function $\varphi \circ h$ given as in \eqref{CD:K2}. Thus, we see that Proposition \ref{P:GSp-ee} implies Conjecture \ref{C:varK} for $z=me_r$ (while $y=1$). It is not hard to verify Conjecture \ref{C:varK} for any other $z'\in \Hom(K, \mu)^\natural$. Thus, Conjecture \ref{C:varK} holds for even fold cover of $\GSp_{2r}$, where $r$ is also even.

On the other hand, the hypothesis of Conjecture \ref{C:varW} is not satisfied for such cover of $\GSp_{2r}$. We analyze the Whittaker dimension $\dim \Wh_\psi(\pi(\phi_\chi, \rho))$ as follows. It follows from the analysis in \S \ref{SS:uni-ps} that 
$$I(\omega_{0, 0}) \simeq I(\omega_{0, 1})$$
since $\omega_{0, 1} = {}^w \omega_{0, 0}$, and that these two principal series are both $(K_0, s_0)$-unramified. On the other hand, we have
$$I(\omega_{e_0, 0}) \simeq I(\omega_{e_0, 1}),$$
which are both $(K_0', s_0')$-unramified. Note that Theorem \ref{T:uni-func} implies that $\pi(\phi_\chi, \mbm{1})|_{\wt{G}_0}$ and $\pi(\phi_\chi, \varepsilon)|_{\wt{G}_0}$  both contain $n_{(r)}\cdot I(\omega_{0,0})$. Since $I(\chi)$ is also $(K_{e_0}, e_0 \cdot s_K)$-unramified, it is easy to see that $I(\omega_{e_0, i})$ is $(K_0', s_0')$-unramified. Theorem \ref{T:uni-func} then also implies $n_{(r)}\cdot I(\omega_{e_0, 0})$ is contained in both $\pi(\mbm{1})$ and $\pi(\varepsilon)$. Thus, as representations of $\wt{G}_0$, we have
$$\pi(\phi_\chi, \mbm{1})|_{\wt{G}_0} \simeq \pi(\phi_\chi, \varepsilon)|_{\wt{G}_0} \simeq n_{(r)} \cdot I(\omega_{0, 0}) \bigoplus n_{(r)} \cdot I(\omega_{e_0, 0}) \simeq n_{(r)} \cdot I(\omega_{0, 1}) \bigoplus n_{(r)} \cdot I(\omega_{e_0, 1}).$$
For every nondegenerate $\psi$, we have
$$\dim \Wh_\psi(\pi(\phi_\chi, \mbm{1})) = \dim \Wh_\psi(\pi(\phi_\chi, \varepsilon)) = 2m^r \cdot n_{(r)}.$$
We give a summary of this as follows.

\begin{thm} \label{T:GSp-ee}
For $\wt{\GSp}_{2r}^{(n)}$ with $n$ and $r$ both even, a unitary $(K, s_K)$-unramified principal series $I(\chi)$ has the decomposition
$$I(\chi)|_{\wt{G}_0}= n_{(r)}\cdot (I(\omega_{0, 0}) \oplus I(\omega_{0, 1}))\bigoplus n_{(r)}\cdot (I(\omega_{e_0, 0}) \oplus I(\omega_{e_0, 1})),$$
where $I(\omega_{0, j})$ is $(K_0, s_0)$-unramified and $I(\omega_{e_0, j})$ is $(K_0', s_0')$-unramified. Here every $I(\omega_{\gamma, j})$ is irreducible. If $I(\chi) =\pi(\phi_\chi, \mbm{1}) \oplus \pi(\phi_\chi, \varepsilon)$ is reducible, then 
$$\pi(\phi_\chi, \mbm{1})|_{\wt{G}_0} \simeq \pi(\phi_\chi, \varepsilon)|_{\wt{G}_0} \simeq n_{(r)} \cdot I(\omega_{0, 0}) \bigoplus n_{(r)} \cdot I(\omega_{e_0, 0}) \simeq n_{(r)} \cdot I(\omega_{0, 1}) \bigoplus n_{(r)} \cdot I(\omega_{e_0, 1}).$$
Regarding the Whittaker dimension, one has
$$\dim \Wh_\psi(I(\omega_{\gamma, j})) = m^r$$
for every non-degenerate $\psi$.
\end{thm}

\begin{eg}
Consider $\wt{\GSp}_4^{(4)}$, i.e., $n=4$ and $r=2$. In this case, $n_{(r)}=2$. We have
$$\msc{X}_{Q,n}^\Gamma=\set{0, e_c, e_0, e_c + e_0} \simeq \Z/2\Z \times \Z/2\Z.$$
and
$$\msc{X}_{Q,n}^\mfr{c} = \set{0, e_c}.$$
On the other hand, we see 
$$(Y/Y_{Q,n})_K^\natural=\set{0, 2e_2, e_c, e_c - 2e_2} \simeq \Z/2\Z \times \Z/2\Z$$
with 
$$p_\Gamma^{-1}(0)=\set{0, 2e_2} \text{ and } p_\Gamma^{-1}(e_c) = \set{e_c, e_c-2e_2}.$$
When the $(K, s_K$)-unramified $I(\chi)$ is reducible, we have
$$\pi(\phi_\chi, \mbm{1})|_{\wt{G}_0} \simeq \pi(\phi_\chi,\varepsilon)|_{\wt{G}_0} \simeq 2 \cdot I(\omega_{0, 0}) \bigoplus 2\cdot I(\omega_{e_0, 0}) \simeq 2 \cdot I(\omega_{0, 1}) \bigoplus 2 \cdot I(\omega_{e_0, 1})$$
with 
$$\dim \Wh_\psi(\pi(\phi_\chi,\mbm{1})) = \dim \Wh_\psi(\pi(\phi_\chi,\varepsilon)) = 16.$$
Assume $\pi(\phi_\chi,\mbm{1})$ is $(K, s_K)$-unramified, then it is also $(K, f_{e_c}\otimes s_K)$-unramified, and that $\pi(\phi_\chi, \varepsilon)$ is both $(K, f_{2e_2}\otimes s_K)$ and $(K, f_{e_c - 2e_2}\otimes s_K)$-unramified.
\end{eg}

\begin{rmk}
Here we observe again that covers satisfying $$\wt{G}^\vee \simeq G_{ad}$$ seem to exhibit quite special properties. For instance, they do not satisfy the hypothesis of Conjecture \ref{C:varW}. One such example here is the even fold cover of $\GSp_{2r}$ with $r$ being even. A previous example, as noted in Remark \ref{R:ano-1}, is the $n$-fold cover of $\Spin_{2r+1}$ with $r$ odd and $n_\alpha \equiv 2 \mod 4$ for every short coroot $\alpha^\vee$. The $n$-fold cover of $\SO_3$ studied in \S \ref{SSS:SO3} , if $4|n$, is also of such type.
\end{rmk}

\begin{rmk}
For the double cover $\wt{\GSp}_{2r}^{(2)}$ of $\GSp_{2r}$, many results in this section are already proved in \cite{Szp5}, even without assuming the tame condition $p\nmid n$. Indeed, the restriction problem for general genuine principal series of $\wt{\GSp}_{2r}^{(2)}$ is analyzed systematically in loc. cit. by utilizing a natural maximal abelian subgroup $\wt{A} \subset \wt{T}$ chosen independent of the residual characteristic $p$ of $F$, see \cite[Lemma 2.1]{Szp5}.
\end{rmk}


\end{document}